\documentclass{article}[11pt]
\usepackage[english]{babel}

\usepackage{amsmath}
\usepackage{amsthm}
\usepackage{amssymb}
\usepackage{amsfonts}
\usepackage{dsfont}
\usepackage{tikz-cd}
\usepackage{tikz}
\usepackage{pgfplots}
\usetikzlibrary{hobby}
\usepackage{tikz}
\usetikzlibrary{arrows,decorations.markings}
\usepackage{euflag}

\usepackage{hyperref}
\usepackage[a4paper, margin=0.75in]{geometry}

\newtheorem{theorem}{Theorem}[subsection]
\newtheorem{corollary}[theorem]{Corollary}
\newtheorem{lemma}[theorem]{Lemma}
\newtheorem{remark}{Remark}[subsection]
\newtheorem{proposition}[theorem]{Proposition}
\newtheorem{definition}[theorem]{Definition}

\counterwithin{equation}{subsection}

\newcommand*{\real}{\mathbb{R}\mathrm{e}~}

\title{A Grushin problem for Bismut's hypoelliptic Laplacian}
\author{Francis~Nier\thanks{LAGA, Universit{\'e} de Sorbonne-Paris Nord (Paris XIII), 99 avenue
J.B.~Cl{\'e}ment, F-93430~Villetaneuse. nier@math.univ-paris13.fr}\\
Xingfeng~Sang\thanks{LAGA, Universit{\'e} de Sorbonne-Paris Nord (Paris XIII), 99 avenue
J.B.~Cl{\'e}ment, F-93430~Villetaneuse. sang@math.univ-paris13.fr}\\
Francis~White\thanks{Department of Mathematics, University of California Irvine, 400 Physical Sciences Quad, Irvine CA-92697. fgwhite@uci.edu}}

\begin{document}
\maketitle

\textbf{Abstract}\\

The name ``Grushin problem'' refers here to the variation of Schur complement technique introduced by J.~Sj{\"o}strand in his early works, which is now a commonly used tool in spectral analysis.  Recently Q.~Ren and Z.~Tao proposed such an approach for the analysis of the low lying eigenvalues in the large friction limit for a simple scalar kinetic model. Inspired by this recent work and with the introduction of functional spaces  adapted to the analysis of geometric Kramers-Fokker-Planck operators in a previous article, the authors now study  the combined asymptotic analysis of Bismut's hypoelliptic Laplacian, in the high friction $b\to 0^{+}$ and possibly low temperature $h\to 0^{+}$ regimes.\\


\tableofcontents
\section{Introduction}
\label{sec:intro}
\subsection{Problem and motivations}
\label{sec:motivations}
In \cite{Bis041}\cite{Bis042}, J.M.~Bismut  introduced the hypoelliptic Laplacian which allows to extend to $p$-forms the generator of the semigroup associated with the Langevin process. The Witten Laplacian, self-adjoint and elliptic, corresponds to the simpler description of the Brownian motion, proposed by Einstein and where the temperature denoted by $h>0$ is essentially the only parameter. The full  Langevin process written here in the euclidean case:
$$
dq=pdt\quad,\quad dp=-\frac{1}{h}\partial_{q}V dt -\frac{1}{b}p dt +\frac{1}{\sqrt{b}}dW
$$
involves actually the two independent parameters $h>0$ and $b>0$, where $h>0$ can be interpreted (after rescaling) as a temperature and $\frac{1}{b}$ as a friction parameter.
The low temperature limit $h\to 0^{+}$ is known as the semiclassical limit for the semiclassical Witten Laplacian $\Delta_{V,h}$ and many works have been devoted to its analysis  after the seminal articles \cite{Wit}\cite{HeSj4}\cite{CFKS} or to  its consequences for the theory of topological invariant of manifolds (see e.g. \cite{Zha}\cite{BiZh}). It also has obvious relationships with all the asymptotic results of Freidlin-Wentzell theory (see e.g. \cite{FrWe}) and its development in the study of simulated annealing in the late 70's (see e.g. \cite{HKS} or \cite{Mic}). We refer the reader to \cite{Ber} for additional references and a historical background and to \cite{LeSt} for  the presentation of more recent applications and issues for the design of effective algorithms in molecular dynamics.\\
It was rapidly shown in \cite{BiLe} that the large friction limit $b\to 0^{+}$ (and $h>0$ fixed) of Bismut's hypoelliptic Laplacian\,, is related to the Witten or Hodge Laplacian on the base manifold. It is summarized by the commonly used terminology of ``overdamped Langevin process'' for
Einstein's description of the brownian motion.\\
Motivated by the applications to molecular dynamics or kinetic theory many works have been devoted in the last twenty years to the accurate computations of small eigenvalues of such operators, with the aim of providing quantitative information about the trend to the equilibrium. Even in the elliptic, self-adjoint and purely semiclassical framework of the semiclassical Witten Laplacian new questions arouse concerned with the accurate computation of spectral element in various geometrical or topological landscape and possibly with boundary value problems.
In particular, in \cite{LNV2} it was proved that when the potential $V\in \mathcal{C}^{\infty}(Q;\mathbb{R}) $ has finitely many critical values and $Q$ is a closed manifold, $\dim~Q=d$\,, the various exponentially small scales of low-lying spectrum of the semiclassical Witten Laplacian's $ \Delta_{V,h} $ in the limit $h\to 0^{+}$ are determined by a topological object: the persistent homology bar code of the function $V$\,.\\
One question addressed in this text is whether a similar result holds for Bismut's hypoelliptic Laplacian in the limit $b\to 0^{+}$ and $h\to 0^{+}$\,.\\
Before giving our main result, let us recall, what is known about similar problems:
\begin{itemize}
\item The accurate description of exponentially small eigenvalues for semi-classical Witten Laplacians, in connection with Eyring-Kramers asymptotics, the generalized Arrhenius law, or the study of quasistationnary distributions, has been studied or used in   \cite{BEGK} \cite{BGK} \cite{HKN} \cite{HeNi2} \cite{Lep1} \cite{LeNi} \cite{DLLN} \cite{LeNe} \cite{LNV1} \cite{LNV2} and references therein.
\item For the Langevin process, the semiclassical regime, which after a rescaling corresponds to $b\propto \sqrt{h}$ and $h\to 0^{+}$\,, for functions ($0$-forms) in the euclidean space with a Morse potential $V$ was considered in \cite{HHS}. An accurate study of the tunnel effect, with microlocal analytic techniques, led to a full asymptotic description of the bottom spectrum under the above assumptions.
  
\item A similar asymptotic framework was considered in the Ph.D thesis of S.~Shen (in \cite{She}) for Bismut's hypoelliptic Laplacian on the cotangent of a closed riemannnian manifold with $b\propto \sqrt{h}$\,, $h\to 0^{+}$\,, the potential $V$ is a Morse function and the metric is euclidean in Morse coordinates around critical points.
\item In \cite{BLM} the analysis of H{\'e}rau-Hitrik-Sj{\"o}strand was extended to a more general class of still scalar ($0$-forms) semiclassical non self-adjoint and subelliptic operators. Such methods have been developped in \cite{Nor1}\cite{Nor2} for other relevant kinetic scalar models where the diffusive part is no longer given by a harmonic oscillator hamiltonian but by a possibly non local operator in the momentum variable.
\item In \cite{BFLS} the authors considered the scalar operator for the Langevin dynamics in the euclidean space but with rather general kinetic energy and potential function. They discuss according to the friction and temperature parameter, the size of the spectral gap (or resolvent estimate). Their variational (so called ``hypocoercive'') method is combined with a Schur complement method which is reminiscent of the formal calculations of \cite{BiLe}-Chap~17.
\item In \cite{ReTa} Ren and Tao developed a Grushin problem  approach for a simple kinetic model in a high friction  limit $\gamma=\frac{1}{b}\to +\infty$\,. Their operator is $\mathcal{Y}-\gamma\Delta_{V}^{\mathbb{S}}$ on the cosphere bundle $S^{*}Q=\left\{(q,p)\in T^{*}Q, g^{ij}(q)p_{i}p_{j}=1\right\}$\,, where $\mathcal{Y}$ is the hamiltonian vector field of the geodesic flow and $\Delta_{V}^{\mathbb{S}}$ is the vertical Laplace-Beltrami operator on the spherical fiber.
\end{itemize}
The results of S.~Shen in \cite{She} are up to now the only accurate asymptotic results on $p$-forms for Bismut's hypoelliptic Laplacian in the combined limit $b\to 0^{+}$ and $h\to 0^{+}$\,, and it is done under some restricted
assumptions.
We note that the works of \cite{BiLe} and \cite{She} are also concerned with the convergence of generalized determinants in connection with Ray-Singer metrics on determinant bundles and other  topological invariants, by developing the strategy of \cite{BiZh}\cite{Zha}.\\
Additionally it is known from the various studies of the elliptic case, i.e. the semiclassical Witten Laplacian, that $1$-forms can be extremely useful even if one is only interested in the scalar case (degree $0$). This is due to  the supersymmetric argument:
 if $\omega$ is an eigenvector in degree $p$ then for a Hodge type operator, $(\mathbf{d}+\mathbf{d}^{*})^{2}=\mathbf{d}\mathbf{d}^{*}+\mathbf{d}^{*}\mathbf{d}$\,, $\mathbf{d}\omega\neq 0$ (resp. $\mathbf{d}^{*}\omega\neq 0$)\,, $\mathbf{d}\omega$ (resp. $\mathbf{d}^{*}\omega$) is an eigenvector of degree $p+1$ (resp. $p-1$)\,.
For these reasons, it is very natural to explore the accurate description of the small eigenvalues in the combined asymptotic regimes $b\to 0^{+}$ and $h\to 0^{+}$\,.\\
Although our previous work, was initially intended to the study of Bismut's hypoelliptic Laplacian with boundary conditions, it rapidly appeared after we heard of Ren and Tao article \cite{ReTa}\,, that our functional framework should allow a rather straightforward transposition of their method. Briefly said, it suffices to replace the total Laplacian $\Delta_{q,p}$ on the total space $S^{*}Q$ by the operator $W^{2}_{\theta}$ introduced in \cite{NSW} for the definition of global Sobolev spaces adapted to the analysis of Bismut's hypoelliptic Laplacian. This combined with various explicit geometric formulas in \cite{Bis05}\cite{BiLe} finally convinced us that an accurate description in the double asymptotics $b\to 0^{+}$ and $h\to 0^{+}$ ( it works for $h=1$) and for a general potential $V\in \mathcal{C}^{\infty}(Q;\mathbb{R})$ should be accessible.

\subsection{Main result and comments}
\label{sec:mainres}
We are concerned with the spectral and semigroup properties of Bismut's hypoelliptic Laplacian, denoted here  by $B_{\pm,b,\frac{1}{h}V}$ on $X=T^*Q$, where $b,h>0$ are parameters\,. Actually, the operator $B_{\pm,b,V}$ is equal to $2(\mathfrak{A}_{\phi_{b},\pm \mathcal{H}}^{\prime})^{2}$ with the presentation of \cite{BiLe}-Section~2 and the additional parameter $h>0$ is introduced by replacing $V$ by $\frac{1}{h}V$\,. We refer to Subsection~\ref{sec:BihypLap} below for a detailed presentation. The semiclassical Witten Laplacian on the closed base manifold $Q$ is given by $\Delta_{V,h}=(d_{V,h}+d_{V,h}^{*})^{2}$ with $d_{V,h}=e^{-\frac{V}{h}}(hd)e^{\frac{V}{h}}=hd+dV\wedge$ and we refer the reader to Subsection~\ref{sec:hermbund} and Subsection~\ref{sec:scalings} for various unitarily equivalent presentations adapted to our problem. We use the $h-$dependent version of the double exponent Sobolev spaces $\tilde{\cal W}_h^{s_1,s_2}$ introduced in Definition~\ref{def:sobsp} for $h=1$ and in  Definition~\ref{def:Ws1s2h} for $h\in ]0,1]$\,.\\

The data of our problem are the spectrum of the semiclassical Witten Laplacian $\mathrm{Spec}(\Delta_{V,h})=\mathrm{Spec}(\Delta_{V^{h},1})$\,, where $V^h(q)=\frac{1}{h}V(hq)$ is defined on a dilated manifold (see Subsection~\ref{sec:scalings}), and the parameters $b,h\in ]0,1]$\,. 

The following definition makes sense if one considers asymptotic regimes where $h\to 0^{+}$ and when $V\in \mathcal{C}^{\infty}(Q;\mathbb{R})$ has a finite number of critical values.
\begin{definition}\label{def:rhoh}
  The parameter $\varrho_{h}\in]0,1]$ parametrized by $h\in ]0,1]$ measures a spectral gap for $\Delta_{V,h}$ according to
  \begin{eqnarray*}
    &&\mathrm{Spec}(\frac{1}{2}\Delta_{V,h})\cap [0,\varrho_{h}]\subset [0,e^{-\frac{c}{h}}] \subset [0,\frac{\varrho_{h}}{2}] \\
       \text{and}&& \mathrm{Spec}(\frac{1}{2}\Delta_{V,h})\cap ]\varrho_{h},+\infty[\subset [4\varrho_{h},+\infty[ 
  \end{eqnarray*}
  for all $h\in ]0,1]$\,. We call $\mathcal{N}_{\pm}(V)$ the rank of $1_{[0,\varrho_{h}]}(\frac{1}{2}\Delta_{V,h})$ and $\mathcal{N}_{\pm}^{(p)}(V)$ the rank of $1_{[0,\varrho_{h}]}(\frac{1}{2}\Delta_{V,h}^{(p)})$ for $p\in \left\{0,\ldots,d\right\}$\,, where the $\pm$ sign refers to the choice of the line bundle $F_{+}=Q\times\mathbb{C}$ or $F_{-}=(Q\times \mathbb{C})\otimes\mathbf{or}_{Q}$\,.\\
  For every $p\in \{0,\dots , d\} $ the eigenvalues of $\frac{1}{2}\Delta_{V,h}^{(p)}$ in $ [0,\varrho_h ]$\,, repeated with multiplicity, are labelled by $ \tilde{\lambda}_{\pm,j,h}^{(p)}(V) $\,, $1\leq j\leq \mathcal{N}_{\pm}(V)$\,, in the increasing order.
\end{definition}

It was proved in \cite{HeSj4} (resp. in \cite{LNV2}) that one can take $\varrho_{h}=ch$ with $c>0$ (resp. $\varrho_{h}=e^{-\frac{\varepsilon}{h}}$ with $\varepsilon>0$ arbitrarily small) when $V\in \mathcal{C}^{\infty}(Q;\mathbb{R})$ is a Morse function (resp. has a finite number of critical values). Additionally the number of eigenvalues of $\frac{1}{2}\Delta_{V,h}^{(p)}$ counted with multiplicities, in  $[0,\varrho_{h}]$\,, is fixed for $h\in ]0,h_{0}]$\,, $h_{0}>0$ small enough, and determined by the topological properties of the sublevel sets of $V$\,, via Morse theory, or more generally via the barcode of persistent homology. We note also that the Poincar{\'e} duality implies $\mathcal{N}_{+}^{(p)}(V)=\mathcal{N}_{-}^{(d-p)}(-V)$ for every $p\in \{0,\dots,d\}$\,.\\

\begin{definition}
          For every $p\in \{0,\dots,2d\}$, the eigenvalues of $B_{\pm,b,\frac{V}{h}}^{(p)}$ lying in $D(0,\frac{\varrho_h}{h^2})$ , repeated according to their algebraic multiplicity, will be denoted by $(\lambda_{\pm,j,h}^{(p)})_{1\leq j\leq \mathcal{N}_{\pm}^{(p)}}$\,.
          The characteristic space 
          $$E_{\pm,b,h}^{(p)} = \mathrm{Ran}\bigg(\frac{1}{2i\pi}\int_{|z|= \frac{\varrho_h}{h^2} } (z-B_{\pm,b,\frac{V}{h}}^{(p)})^{-1}\,dz\bigg)$$ 
          has the dimension $ \mathcal{N}_{\pm}^{(p)} = \dim(E_{\pm,b,h}^{(p)})$\,. When $B_{\pm,b,\frac{V}{h}}^{(p)}\big|_{E_{\pm,b,h}^{(p)}}$ is diagonalizable (see Theorem~\ref{th:main}-\textbf{a)}),  a basis of eigenvectors is written 
          $ (u_{\pm,j,h}^{(p)})_{1\leq j \leq \mathcal{N}_{\pm}^{(p)}}$ and its $L^2$ dual basis is denoted by $ (v_{\pm,j,h}^{(p)})_{1\leq j \leq \mathcal{N}_{\pm}^{(p)}} $\,.
        \end{definition}

      \begin{theorem}
        \label{th:main}
        Let $g$ be a metric on $Q$ and let $V\in\mathcal{C}^{\infty}(Q;\mathbb{R})$ be a potential function with finitly many critical values.
In the following statements $C_s\geq 1$ denotes a large enough constant determined by $s\in\mathbb{R}$\,.
  \begin{description}
  \item[a)] When $bC_0\leq h\varrho_h \leq h$ all the eigenvalues of $B_{\pm,b,\frac{V}{h}} $ with real part below $\frac{\varrho_h}{h^2}$ are real and non negative:
    $$ \mathrm{Spec}(B_{\pm,b,\frac{V}{h}})\cap \{z\in\mathbb{C}\,, \real z \leq \frac{\varrho_h}{h^2}\} = \mathrm{Spec}(B_{\pm,b,\frac{V}{h}}) \cap [0,\frac{\varrho_h}{h^2}]\,.$$
   In degree $p\in\{0,\ldots,2d\}$\,, their number, counted with multiplicity, is given by $\mathcal{N}_{\pm}^{(p)}=\mathcal{N}_{\pm}^{(p-\frac{d}{2}\pm \frac{d}{2})}(V)$\,, which is $0$ if $p>d$ (resp. $p<d$) in the $+$ case (resp. $-$ case). Poincar{\'e} duality  implies $\lambda_{+,j,h}^{(p)}=\lambda_{-,j,h}^{(2d-p)}$\,. Additionally the restricted operator $B_{\pm,b,\frac{V}{h}}^{(p)}\big|_{E_{\pm,b,h}^{(p)}} $ is diagonalizable.
   \item[b)] Under the stronger assumption $ bA^4C_0 \leq h\varrho_h \leq h $ with $A\geq C_0$\,, the  comparison between the Witten Laplacian and Bismut's hypoelliptic Laplacian of the low lying spectrum, is given by
      $$
\forall p\in \{0,\ldots,2d\}, \forall j\in \{1,\ldots,\mathcal{N}_{\pm}^{(p)}\}\,,\quad
(1+C_0 A^{-1/2})^{-1}\frac{\tilde{\lambda}_{\pm,j,h}^{(p-\frac{d}{2}\pm \frac{d}{2})}(V)}{h^2}\leq  \lambda_{j,\pm,h}^{(p)}\leq (1+C_0 A^{-1/2})\frac{\tilde{\lambda}_{\pm,j,h}^{(p-\frac{d}{2}\pm \frac{d}{2})}(V)}{h^2}\,.
$$
  \item[c)] When $bC_s\leq h \varrho_h$\,, the semigroup $(e^{-tB_{\pm,b,\frac{V}{h}}})_{t>0}$ satisfies:
  \begin{eqnarray*}
  &&e^{-tB_{\pm,b,\frac{V}{h}}} = \sum_{p\in \{0,\dots,2d\}}\sum_{j} e^{-t\lambda_{\pm,j,h}^{(p)}} | u_{\pm,j,h}^{(p)}  \rangle \langle v_{\pm,j,h}^{(p)} | + R_h(t) \, \\
    \text{with}&& \|R_h(t)\|_{\mathcal{L}(\tilde{\cal W}^{0,s}_h;\tilde{\cal W}^{0,s}_h)} \leq  \frac{1}{b^2}(h^2+\frac{1}{t})e^{-t\frac{\varrho_h}{h^2}} \\
    \text{and} && \max ( \|u_{\pm,j,h}^{(p)}\|_{\tilde{\cal W}^{0,s}_h}, \|v_{\pm,j,h}^{(p)}\|_{\tilde{\cal W}^{0,s}_h} ) \leq C_s \,,
    \end{eqnarray*}
for suitably normalized basis of eigenvectors $(u_{\pm,j,h}^{(p)})_{1\leq j\leq \mathcal{N}_{\pm}^{(p)}}$\,.
  \end{description}
  
\end{theorem}

The second statement \textbf{b)} says in particular that in the limit $h\to 0^{+}$ the eigenvalues of Bismut's hypoelliptic Laplacian have the same exponentially small asymptotic behaviour as the eigenvalues of the Witten Laplacian. The latter were shown in \cite{LNV2} to be related to the bar codes of persistent homology.
\begin{corollary}
\label{cor:persis}
When $V\in \mathcal{C}^{\infty}(Q;\mathbb{R})$ has finitely many critical values and
under the condition $C_0^5 b\leq h\varrho_h$\,, the eigenvalues $(\lambda_{+,j,h}^{(p)})_{1\leq j\leq \mathcal{N}_{\pm}^{(p)}}$ satisfy $\lim_{h\to 0^+}-h\log(\lambda_{+,j,h}^{(p)})=2\ell_j^{(p)}$\,,
where $\ell_j^{(p)}$ is the length of a bar, indexed by $j$\,,  with a degree $p$ endpoint in the bar code associated with $V$\,. The $-$ case is obtained by Poincar{\'e} duality with $\lambda_{-,j,h}^{(p)}=\lambda_{+,j,h}^{(2d-p)}$\,.
\end{corollary}
\noindent\textbf{Comments:}
\begin{itemize}
\item  Actually all these results are related with resolvent estimates of which the accurate formulation is given in Proposition~\ref{pr:estimateForTheConvergence}, after introducing other intermediate quantities.
\item In order to write a general statement, we preferred to  express things in terms of the non explicit spectral gap $\varrho_{h}$ of Definition~\ref{def:rhoh}. For a general function $V\in \mathcal{C}^{\infty}(Q;\mathbb{R})$ with finitely many critical values the result of \cite{LNV2} says that one can take $\varrho_{h}=e^{-\frac{\varepsilon}{h}}$ with $\varepsilon>0$ arbitrarily small. But when one knows better the geometry of the critical sets an algebraic expression $\varrho_{h}=h^{\nu}$ can be obtained. The basic example is when $V(x)=x^{n}$ in $\mathbb{R}$, in which case a simple rescaling argument gives $\varrho_{h}=h^{2\frac{n-1}{n}}$\,.
\item When the potential $V$ is a Morse function, the condition $C_{0}b\leq \varrho_{h}h$ says $b\leq ch^{2}$, which is stronger than the condition $b\leq c\sqrt{h}$ suggested by the works of S.~Shen \cite{She} and H{\'e}rau-Hitrik-Sj{\"o}strand \cite{HHS}\,, where they considered $b\propto \sqrt{h}$\,. Actually our method relies on the elimination of the potential term, as a perturbative term, after the rescaling $\phi_{h}:Q\mapsto Q^{h}=\frac{1}{h}Q$ of Subsection~\ref{sec:scalings}. A similar analysis of what is proposed here, could be developed with better treatment of the Morse potential function. Instead of the above dilation take $\phi_{\sqrt{h}}:Q\mapsto Q^{\sqrt{h}}=\frac{1}{\sqrt{h}}Q$ and use on $Q^{\sqrt{h}}$ a partition of unity in riemannian balls of radius $M\sqrt{h}$ with $M\geq 1$ large enough. With a more inclusive description of the scalar quadratic model in every ball, which takes better into account the quadratic Taylor approximation of the potential $V^{\sqrt{h}}(q)=\frac{1}{\sqrt{h}}V(\sqrt{h}q)$\,, subelliptic estimates of \cite{NSW} can be improved in particular by using the accurate quantitative estimates of \cite{BNV} for quadratic Kramers-Fokker-Planck operators in the euclidean space. In the end the rescaling leads to the comparison of the spectral gap for $\Delta_{V^{\sqrt{h}},1}$ on $Q^{\sqrt{h}}$\,, which by unitary equivalence is equal to $\frac{1}{h}\varrho_{h}\propto 1$ and the rescaled parameter $\frac{b}{\sqrt{h}}$ instead of $\frac{b}{h}$\,. One then recovers the natural condition $b\leq c\sqrt{h}$\,.
  This is just a sketch and an accurate spectral analysis remains to be done.
  In this article, we preferably considered a $\mathcal{C}^{\infty}$-function without assuming that it is a Morse function, in order to highlight the generality of the Grushin problem approach.
\item In \cite{LNV2} results were given for non smooth potentials, in particular when $V$ is a  Lipschitz subanalytic function. This more general case  is not considered here and it would require a specific analysis, which could follow partly the strategy presented here.
\item Theorem~\ref{th:main} is a digest of what can be deduced from the Grushin problem method. Many intermediate resolvent estimates can be used and maybe improved for other purposes.
\item Finally, we have not considered as in \cite{BiLe} and \cite{She} the convergence of generalized determinants. Actually, for topological invariants which do not depend on the riemannian metric, the simplifying assumptions of \cite{She} suffice for a general treatment. It is not clear that a more accurate and general analysis would bring relevant improvements.
\end{itemize}
\subsection{Outline of the article}
\label{sec:outline}

The geometric framework, the operators, the various scalings and the functional spaces are defined in Section~\ref{sec:framework}. 
Remember that Bismut's hypoelliptic Laplacian is a second order non self-adjoint and non elliptic operator acting on differential forms defined 
on the total space $X$ of the cotangent bundle $T^*Q$ of the closed riemannian manifold $(Q,g)$\,. The definition of the hypoelliptic Laplacian in\cite{Bis05}\cite{BiLe}, the associated Weitzenb\"ock formula,
and the introduction of adapted functional spaces, strongly relies on the horizontal and vertical decomposition $T(T^*Q)=TX=T^HX\oplus T^VX$  recalled in subsection \ref{sec:horizontalVerticalDecomposition}.
The exact definition of the Witten Laplacian and the hypoelliptic Laplacian are given in Subsections~\ref{sec:hermbund} and \ref{sec:BihypLap}.
An $h$-dependent change of scale is introduced in Subsection~\ref{sec:scalings}. This allows to get easily uniform constants with respect to $h\in ]0,1]$ in all the subelliptic estimates which are used in the text.
The first of these subelliptic estimates in Subsection~\ref{sec:hyplapSubEll} is an adaptation of the general results of \cite{NSW} to the present framework. Although Theorem~\ref{th:main} is expressed for the operator $B_{\pm,b,\frac{1}{h}V}=B_{\pm,(Q,g,\frac{V}{h},b)}$ on $X=T^*Q$\,, all the analysis of this text is carried out on the dilated geometry of Subsection~\ref{sec:scalings} with the operator $B_{\pm,b',V^h}=B_{\pm,(Q^h,g^h,V^h,b')}$, $b'=\frac{b}{h}$ and where the $\prime$ is dropped afterward) where the $h$-dependence is easier to track.\\

In Section~\ref{sec:modifiedOperators}, various perturbations or modifications of the operator $B_{\pm,b,V^h} $ are considered, which have no spectrum around $0$\,. For such new operators, resolvent and possibly subelliptic estimates are specified. One of them, denoted by $B_{\pm,b,V^h} + Q_{A,L,V^h}$, is directly inspired from the work \cite{ReTa} of Q.~Ren and Z.~Tao. In a crucial way, this section aims at providing subelliptic estimates for $B_{\pm,b,V^h} + Q_{A,L,V^h}$ with a uniform lower bound with respect to $b,h\in ]0,1]$ for a large enough new additional parameter $A\geq 1$\,. Due to the new complexity of our problem but also in order to improve Ren-Tao lower bounds, this is done in two steps with the intermediate operator $B_{\pm,b,V^h}+A^2\pi_{0,\pm}$ easier to handle, especially if one uses the \underline{maximal}~subelliptic estimates.\\

The writing of a Grushin problem in Section~\ref{sec:Grushin} allows an accurate comparison of the resolvents $(B_{\pm,b,V^h}+Q_{A,L,V^h}-z)^{-1}$\,, $(B_{\pm,b,V^h}-z)^{-1}$\,, $(\Delta_{V^h,1}-z)^{-1}$ and $(\Delta_{V^h,1}+\tilde{Q}_{A,L,V^h}-z)^{-1}$\,.
Although the resolvent estimates of Section~\ref{sec:modifiedOperators} can be written with uniform constants which are independent of the Sobolev exponent $s\in \mathbb{R}$\,,  the range of validity for the parameters $b,h,A>0$ actually depends on this Sobolev exponent $s$\,. Attention must be paid to the formal calculations which are not done in the general distributional setting but rather in an arbitrarily fixed range of Sobolev exponents $s\in [s_{\min},s_{\max}]$\,.\\

The proof of Theorem~\ref{th:main} is achieved in Section~\ref{sec:specconseq}. 
The resolvent comparison in Section~\ref{sec:Grushin} and the spectral information of the semiclassical Witten Laplacian, 
summarized in Definition~\ref{def:rhoh}\,, provide the first accurate localization of the spectrum of $B_{\pm,b,V^h}$ around $0$\,, with new accurate estimates for the resolvent and  the semigroup.
We finally use the Hodge structure and the PT-symmetry property, $r^* B_{\pm,b,V^h}r^*=B_{\pm,b,V^h}^*$ with $r^*$ a unitary involution, in order to make an accurate comparison between the eigenvalues of $B_{\pm,b,V^h}$\,, identified now as the squared singular values of a restricted differential,  and the eigenvalues of the Witten Laplacian $\Delta_{V^h,1}$\,, with $\mathrm{Spec}(\Delta_{V^h,1})=\mathrm{Spec}(\Delta_{V,h})$\,.

\section{Framework}
\label{sec:framework}
\subsection{Total space  $X$ of the cotangent bundle }
\label{sec:total}
Let $(Q,g^{TQ})$ be a closed Riemannian manifold of dimension $d$, $\nabla^{LC}$ the associated Levi-Civita connection and let $X=T^{*}Q$ be the total space of the cotangent bundle. On one side, the total space $X$ is a symplectic manifold with the canonical symplectic form $\sigma$\,. One the other side, the kinetic energy  function is globally defined by
\begin{equation}
  \label{eq:kinecticEnergy}
  \forall x=(q,p)\in T_{q}^{*}Q ,\quad \mathcal{H}(x) = \frac{|p|_{q}^{2}}{2} = \frac{1}{2}g^{T^{*}Q}(p,p) .
\end{equation}
Then the hamiltonian vector field $\mathcal{Y}$  of the geodesic flow is given by
\begin{equation}
  \label{eq:definitionOfHamiltonianVectorFieldOfTheKinecticEnergy}
  d^{X}\mathcal{H} + \mathbf{i}_\mathcal{Y}\sigma = 0 .
\end{equation}
The scalar vertical harmonic oscillator $\mathcal{O}$ is the self-adjoint differential operator defined with its maximal domain in $L^{2}(X,dqdp;\mathbb{C})$ by
\begin{equation}
  \label{eq:definitionOfO}
\frac{1}{2} (g^{TQ}(D_p,D_p)+ g^{T^*Q}(p,p))\geq \frac{d}{2}\mathrm{Id}\quad \text{with}~D_p=\frac{1}{i}\partial_p
\end{equation}
where $-\Delta_{\mathrm{Vert}}=g^{TQ}(D_p,D_p)$ is the fiberwise vertical Laplacian.

\subsection{The horizontal-vertical decomposition}\label{sec:horizontalVerticalDecomposition}
The Levi-Civita connection induces a splitting of the tangent space and cotangent space of $X$ given by
\begin{equation}\label{eq:horizontalVerticalDecomposition}
  TX = (TX)^{H} \oplus (TX)^{V} \simeq \pi^{*}(TQ \oplus T^{*}Q) \quad ; \quad T^{*}X = (T^{*}X)^{H} \oplus (T^{*}X)^{V}\simeq \pi^{*}(T^{*}Q \oplus TQ)
\end{equation}
where $\pi: X=T^{*}Q \rightarrow Q$ is the  natural projection, $(TX)^{H} \simeq \pi^{*}(TQ)$ is the horizontal distribution and $(TX)^{V} = \ker(d\pi) \simeq \pi^{*}(T^{*}Q)$ is the vertical distribution. Once a frame $u_{1},u_{2},\dots,u_{d}$ and the associated coframe $u^{1},\dots, u^{d}$ are locally chosen, we take a copy of those two frames $ \hat{u}_{1},\hat{u}_{2},\dots,\hat{u}_{d} $ and $\hat{u}^{1},\dots,\hat{u}^{d}$ where  the above identification is written
$$(TX)^{H} \simeq \mathrm{Span}(u_{1},\dots, u_{d}) \quad ; \quad (TX)^{V} \simeq \mathrm{Span}(\hat{u}^{1}, \dots , \hat{u}^{d}) $$
and
$$ (T^{*}X)^{H} \simeq \mathrm{Span}(u^{1},\dots, u^{d}) \quad ; \quad (T^{*}X)^{V} \simeq \mathrm{Span}(\hat{u}_{1},\dots, \hat{u}_{d}) .$$
In the rest of the text we will use the above identification with the following additional conventions

\begin{itemize}
\item When $u_{i}$'s  are associated to a coordinate system on $Q$ i.e. $u_{i}  = \frac{\partial}{\partial q^{i}} $ for all $i\in \{ 1,\dots , d\}$ then we use the notation
$$\pi^{*}(u_{i}) = e_{i} = \frac{\partial}{\partial q^{i}} + \Gamma_{ij}^{k}p_{k}\frac{\partial}{\partial p_{j}} \in (TX)^{H} \quad ; \quad \pi^{*}(\hat{u}^{i}) = \hat{e}^{i} = \frac{\partial}{\partial p_{i}} \in (TX)^{V}$$
and
$$ \pi^{*}(u^{i})=e^{i} = dq^{i} \in (T^{*}X)^{H} \quad ; \quad \pi^{*}(\hat{u}_{i}) = \hat{e}_{i} = dp_{i} - \Gamma_{ij}^{k}p_{k}dq^{j} \in (T^{*}X)^{V}  .$$
Where $\Gamma_{ij}^{k}$ denote the Christoffel symbol for the Levi-Civita connection, defined by $ \nabla^{LC}_{\frac{\partial}{\partial q^{i}}}\frac{\partial}{\partial q^{j}} = \Gamma_{ij}^{k}(q) \frac{\partial}{\partial q^{k}} $.
\item When $u_{i}$'s is an (local) orthonormal frame of $TQ$ we will use the notation
\begin{eqnarray}
 \label{eq:fihfi1}
 &&\pi^{*}(u_{i})=f_{i}\in (TX)^{H} \quad : \quad  \pi^{*}(\hat{u}^{i})=\hat{f}^{i}\in (TX)^{V}
 \\
 \label{eq:fihfi2}
 \text{and}&&
   \pi^{*}(u^{i})=f^{i} \in (T^{*}X)^{H} \quad ; \quad   \pi^{*}(\hat{u}_{i}) = \hat{f}_{i} \in (T^{*}X)^{V} .
\end{eqnarray}
\end{itemize}
  Passing from one writing to another is simply given by a section $P$ of the fiber bundle $GL(TQ)$ above $Q$. Indeed for all $i\in \{1,\dots,d\}$
  $$ u_{i} = P(q)_{i}^{j}\frac{\partial}{\partial q^{j}} .$$
  The following relations hold on $TX$ and $T^{*}X$
  \begin{align*}
     f_{i}  & =  P(q)_{i}^{j}e_{j}        &  \hat{f}^{i}  & = (P(q)^{-1})_{j}^{i} \hat{e}^{j} ,\\
  \textrm{and}  \quad   f^{i}  & =  (P(q)^{-1})^{i}_{j}e^{j}  &   \hat{f}_{i}  & =  P(q)_i^{j}\hat{e}_{j}.
  \end{align*}
With this decomposition
  \begin{itemize}
  \item The metric $g^{TX}$ on $TX$ is defined as $g^{TX}=g^{TQ} \oplus^{\bot} g^{T^{*}Q}$ with respect to the decomposition~\eqref{eq:horizontalVerticalDecomposition}. The frame $ f_{1},\dots,f_{d},\hat{f}^{1},\dots, \hat{f}^{d} $ is an orthonormal frame with respect to $g^{TX}$.
 Similarly we define the metric $g^{T^*X} = g^{T^*Q}\oplus^{\bot} g^{TQ}$ on the cotangent space $T^*X$ of $X$. For the exterior algebra we use $\Lambda T^*X\simeq (\Lambda T^*Q)\otimes(\Lambda TQ)$ as a vector space and $g^{\Lambda T^*X}=g^{\Lambda T^*X}\otimes g^{\Lambda TX}$\,. With the orthonormal frames  $ f^{1},\dots,f^{d},\hat{f}_{1},\dots, \hat{f}^{d} $, an orthonormal frame of $\Lambda T^*X$ is given by $(f^I\wedge \hat{f}_J)_{I,J\subset\{1,\ldots,d\}}$.
  \item The hamiltonian vector field $\mathcal{Y}$ defined by \eqref{eq:definitionOfHamiltonianVectorFieldOfTheKinecticEnergy} can be written
    $$ \mathcal{Y} = g^{T^{*}Q}(e^{i},e^{j}) p_{j} e_{i} =\sum_{i=1}^{d} \tilde{p}_{i}f_{i} $$
    where $p=p_{i}dq^{i} = \tilde{p}_{i}f^{i}\in T_{q}^{*}Q$.
  \item The vertical Laplacian equals
    $$ \Delta^{V} = g^{TQ}(e_{i} ,e_{j} ) \hat{e}^{i}\hat{e}^{j} = \sum_{i=1}^{d} (\hat{f}^{i})^{2} .$$
    \item  The tautological connection on $ TX $ and $T^*X$, extended to $\Lambda T^*X$ or $\Lambda T^*X\otimes \pi^*\mathbf{or}(Q)$, is defined by the following formula
    $$\begin{array}{rclcrcl}\label{eq:definitionOfTautologicalConnection}
      \nabla^{TX}_{e_i}e_j & = & \Gamma_{ij}^k e_k   & ; &    \nabla^{TX}_{e_i}\hat{e}^j & = & - \Gamma_{ik}^j\hat{e}^k ,\\ 
       \nabla^{TX}_{ \hat{e}^i}e_j & = &  0  & ; & \nabla^{TX}_{\hat{e}^i}\hat{e}^j & = & 0 \\
       \textrm{and} \quad \nabla^{T^*X}_{e_i}e^j & = & -\Gamma_{ik}^je^k & ; & \nabla^{T^*X}_{e^i}\hat{e}_j & = & \Gamma_{ij}^k\hat{e}_k , \\
       \nabla^{T^*X}_{\hat{e}^i}e^j & = & 0 & ; & \nabla^{T^*X}_{\hat{e}^i} \hat{e}_j & = & 0.
    \end{array}$$
  \end{itemize}
  
  \subsection{Hermitian trivial bundle $F$ over $Q$}
\label{sec:hermbund}  
  Although Bismut's theory of the hypoelliptic Laplacian in \cite{Bis041}\cite{Bis042}\cite{Bis05} works in a much more general framework, we focus here on the simpler case which makes the connection with the standard semiclassical Witten Laplacian on the base manifold $Q$. Namely we work with the trivial bundle  $F=Q\times \mathbb{C} $ on the base manifold $Q$, equiped with the hermitian metric $g^{F} = \exp(-2 V(q))d\bar{z} \otimes dz$ and the trivial connection $\nabla^{F} = d^{Q}$. When smooth duality arguments are used on a non-oriented manifold $Q$, the trivial bundle $Q\times \mathbb{C}$ must be replaced by $(Q\times\mathbb{C})\otimes \mathbf{or}_Q $ where $\mathbf{or}_Q$ is the orientation bundle on $Q$. Locally nothing is changed.\\
By following Bismut's notations, set
  $$ \omega(\nabla^{F},g^{F}) = (g^{F})^{-1} \nabla^{F}g^{F} = -2dV,$$
which is here a real scalar $1-$form on $Q$. The adjoint connection $\nabla^{F*}$  of $\nabla^{F}$ with respect to $g^{F}$ equals
  $$ \nabla^{F*} = d^{Q} - 2dV $$
and the associated unitary connection $\nabla^{F,u}$  is
  $$ \nabla^{F,u} = d^{Q} - dV .$$
  Contrary to the general case studied in \cite{Bis05}\cite{BiLe}, here the unitary connection $\nabla^{F,u}$ is flat 
 since its curvature $R^{F}$ is given by $R^{F} =-\frac{1}{4} \omega(\nabla^{F},g^{F})\wedge\omega(\nabla^{F},g^{F})= - dV\wedge dV = 0$.\\
In Bismut work and more generally for a probabilistic approach, the natural $L^2$-space is $L^2(Q,d\mathrm{Vol}_g;\Lambda T^*Q\otimes F)$ where the notation recalls the non trivial metric $g^{F}=e^{-2V(q)}$ on $F\simeq Q\otimes C$, in the $L^2$-scalar product
$$
\langle u,v\rangle_{L^2(Q,d\mathrm{Vol}_g;\Lambda T^*Q\otimes F)}=\int_Q g^{\Lambda T^*Q}(\overline{u},v)~e^{-2 V(q)}~d\mathrm{Vol}_g(q).
$$
For the accurate spectral analysis it is simpler to work in the standard $L^2$-space, $L^2(Q,d\mathrm{Vol}_g;\Lambda T^*Q\otimes \mathbb{C})$, with the scalar product
$$
\langle u,v\rangle_{L^2(Q,d\mathrm{Vol}_g;\Lambda T^*Q\otimes \mathbb{C})}=\int_Q g^{\Lambda T^*Q}(\overline{u},v)~d\mathrm{Vol}_g(q).
$$
Passing from one formulation to the other via the unitary multiplication by $e^{\pm\frac{V(q)}{h}}$ is summarized by the following table.
\renewcommand{\arraystretch}{1.3}
\begin{table}[h]
  \centering
    \begin{tabular}{|c|c|c|} 
    \hline
     Functional space &  $L^{2}(Q,d\mathrm{Vol}_{g};\Lambda^{\cdot}T^*Q\otimes F)$ & $L^{2}(Q,d\mathrm{Vol}_{g};\Lambda^{\cdot}T^*Q\otimes\mathbb{C})$  \\  \hline
      Sections & $v=e^{V}u$ & $u=e^{-V}v$  \\ \hline
      metric &   $g^F= exp(-2V)$ & 1  \\ \hline
     Connection & $\nabla^{F} = d^{Q}$ & $ d^{Q} + dV$   \\ \hline
     Endomorphism $\omega$ & $\omega(\nabla^{F},g^{F} ) = -2dV $ & $\omega(\nabla^{F},g^{F} ) = -2dV $  \\ \hline
     Adjoint connection & $\nabla^{F*} = d^{Q}-2dV$ & $ d^{Q} - dV $  \\ \hline
      unitary connection & $ \nabla^{F,u} $ & $ d^{Q} $  \\ \hline
      differential & $d^Q$& $d^Q+dV\wedge =:d_{V,1}$\\\hline
      codifferential& $d^{Q,F,*}=e^{2V}d^{Q,*}e^{-2V}=d^{Q,*}+2\mathbf{i}_{\nabla V}$& $d^{Q,*}+\mathbf{i}_{\nabla V}=:d_{V,1}^*$\\ \hline
      Hodge/Witten Laplacian & $\square^{Q,F} = (d^Q+d^{Q,F,*})^2$ & $\Delta_{V,1}=(d_{V,1}+d_{V,1}^*)^2$  \\ \hline 
    \end{tabular}
    \caption{Correspondance of $L^2$ spaces}
    \label{tab:exp}
  \end{table}

We recall the formulas
  \begin{eqnarray*}
    \square^{Q,F}&=&(d^Qd^{Q,*}+d^{Q,*}d^Q)+2\mathcal{L}_{\nabla V}\,,\\
    d_{V,1}&=&e^{-V}(d)e^{V}=d+dV\wedge\quad,\quad d^{*}_{V,1}=e^{V}(d^{*})e^{-V}=(d_{V,1})^{*}=d^{*}+\mathbf{i}_{\nabla V}\,,\\
     \Delta_{V,1}&=&(d_{V,1}+d_{V,1}^*)^2=(d_{V,1}d^*_{V,1}+d^*_{V,1}d_{V,1})=(d^Qd^{Q,*}+d^{Q,*}d^Q) +|\nabla V|^2+(\mathcal{L}_{\nabla V}+\mathcal{L}_{\nabla V}^*)\,.
  \end{eqnarray*}
  The subscript $1$ in $d_{V,1}$, $d^{*}_{V,1}$ and $\Delta_{V,1}$ refers to the specific case $h=1$ for the semiclassical Witten differential, codifferential and Laplacian:
  \begin{eqnarray*}
    && d_{V,h}=e^{-\frac{V}{h}}(hd)e^{\frac{V}{h}}\quad,\quad d_{V,h}^{*}=e^{\frac{V}{h}}(hd)^{*}e^{-\frac{V}{h}}\\
    && \Delta_{V,h}=(d_{V,h}+d_{V,h}^{*})^{2}=h^{2}(dd^{*}+d^{*}d)+|\nabla V|^{2}+h(\mathcal{L}_{\nabla V}+\mathcal{L}_{\nabla V}^{*})\,.
  \end{eqnarray*}
  Within the presentation of Table~\ref{tab:exp} the semiclassical regime can be introduced by simply replacing $V$ by $\frac{V}{h}$ and by choosing the metric $g^{F}=e^{-\frac{2V}{h}}$\,. This actually leads to
  $$
\square^{F}=(d^{Q}d^{Q,*}+d^{Q,*}d^Q)+\frac{2}{h}\mathcal{L}_{\nabla V}
$$
in $L^{2}(Q,d\mathrm{Vol}_{g};\Lambda T^{*}Q\otimes F)$\,, transformed in the $L^{2}(Q;d\mathrm{Vol}_{g};\Lambda T^{*}Q\otimes \mathbb{C})$ picture into
$$
(d^{Q}d^{*,Q}+d^{*,Q}d^{Q})+\frac{1}{h^{2}}|\nabla V|^{2}+\frac{1}{h}(\mathcal{L}_{\nabla V}+\mathcal{L}_{\nabla V}^{*})=\frac{1}{h^{2}}\Delta_{V,h}\,.
$$
We will explain in the specific Subsection~\ref{sec:scalings} how the semiclassical asymptotic regime, or more generally $h\in]0,1]$\,, can be easily introduced in the analysis of geometric Kramers-Fokker-Planck operators of \cite{NSW}, where the parameter $h\in]0,1]$ was actually not considered. 
  
\subsection{Functional spaces on $X$}
\label{sec:Funcsp}
The isomorphism of vector bundles $\mathcal{E}$ and $\pi^*(\underbrace{\Lambda T^*Q\otimes \Lambda TQ\otimes F}_{=E})$ is provided by the horizontal-vertical decomposition \eqref{eq:horizontalVerticalDecomposition} of $\Lambda T^*X\otimes\pi^*F$ . With this identification, the vector bundle
  $\mathcal{E}$ is endowed with the metric $\pi^*(g^{\Lambda T^*Q}\otimes g^{\Lambda TQ}\otimes g^F)$  where we recall $F=Q\times\mathbb{C}$ (or possibly $F=(Q\times \mathbb{C})\otimes \mathbf{or}_Q)$ and $g^F=e^{-\frac{2V(q)}{h}}d\bar{z}\otimes dz$\,. The pulled back vector bundle will be denoted by  $\mathcal{E}=\Lambda T^*X\otimes \pi^*F$ and depending on the case $\mathcal{E}_{+}=\Lambda T^*X\otimes \mathbb{C}$ and $\mathcal{E}_-=\Lambda T^*X\otimes \mathbb{C}\otimes \pi^*(\mathbf{or}_Q)$. With the symplectic volume denoted by $dqdp=d\mathrm{Vol}_{g\oplus^{\perp}g^{-1}}$, the associated $L^2$ space, denoted by $L^2(X,dqdp;\mathcal{E})$ and equal to $L^2(X,e^{-\frac{2V(q)}{h}}dqdp;\mathcal{E}_\pm)$, is given by the hermitian scalar product
  $$
  \langle u,v\rangle_{L^2(X,dqdp;\mathcal{E})}=\int_X g^{\Lambda T^*X}(\overline{u},v) e^{-\frac{2V(q)}{h}}~dqdp~.
  $$
After setting $\tilde{u}=e^{-\frac{V(q)}{h}}u$ and $\tilde{v}=e^{-\frac{V(q)}{h}}v$\,, it can be replaced by the standard $L^2(X,dqdp;\mathcal{E}_\pm)$ with the scalar product
 $$
  \langle \tilde{u},\tilde{v}\rangle=\int_X g^{\Lambda T^*X}(\overline{\tilde{u}},\tilde{v})~dqdp=\langle u,v\rangle_{L^2(X,dqdp;\mathcal{E}_{\pm})}~.
  $$
Those $L^2$ spaces and the Schwartz space of rapidly decaying (as $p\to \infty$) smooth sections, $\mathcal{S}(X;\mathcal{E})$ and $\mathcal{S}(X;\mathcal{E}_\pm)$, coincide with the obvious density result. The  first $L^2$-norm depends on $h>0$ while the second $L^2$-norm does not change with $h>0$ and is more convenient here.\\

\begin{center}
    We work in $L^2(X,dqdp;\mathcal{E}_\pm)$.
    \end{center}

When necessary, formulas of \cite{Bis05}\cite{BiLe} written in $L^2(X,dqdp;\mathcal{E})$ with the corresponding scalar product and duality, will be translated later by extending the general rules of Table \ref{tab:exp}

For the analysis it is more convenient to work with a local presentation on the base manifold $Q$ of the functional spaces and associated differential operators.
\begin{definition}
\label{de:localization}
Let $\sum_{j=1}^{J}\theta_j^2(q)\equiv 1$ be a quadratic partition of unity on $Q$\,, such that above a neighborhood $\mathcal{V}_{\theta,j}$ of every $\mathrm{supp}\,\theta_j$\,, there are smooth dual orthonormal frames $(u^1_j,\ldots,u^d_j)$ of $T^*Q\big|_{\mathcal{V}_{\theta,j}}$ and $(u_{j,1},\ldots,u_{j,d})$ of $TQ\big|_{\mathcal{V}_{\theta,j}}$. Set $f^i=\pi^*(u^i)\in (T^*X)^H$ and $ \hat{f}_i=\pi^*(u_i)\in (T^*X)^V$ according to \eqref{eq:fihfi2}.\\
For $F=Q\times \mathbb{C}$ or $F=(Q\times\mathbb{C})\otimes\mathbf{or}_Q$, let $\mathcal{I}_{\theta,Q}$ and $\mathcal{I}_{\theta,X}$ denote the product of isometries:
\begin{alignat}{7}
 \mathcal{I}_{\theta,Q}: && L^2(Q,d&\mathrm{Vol}_g;\Lambda T^*Q\otimes F)&\to && \mathop{\oplus}_{1\leq j\leq J}^{\perp} L^2(\mathcal{V}_{\theta,j},d\mathrm{Vol}_g&; (\Lambda T^*Q\otimes F)\big|_{\mathcal{V}_{\theta,j}})
 &&\to &
 \mathop{\oplus}_{\substack{1\leq j\leq J\\I\subset \{1,\ldots,d\}}}^{\perp} L^2(\mathcal{V}_{\theta,j},d\mathrm{Vol}_g;\mathbb{C})
 \\
&&s&&\mapsto &&(\theta_j s)_{1\leq j\leq J}&&&\mapsto &
\mathcal{I}_{\theta,Q}s=(s_{j,I})_{\substack{1\leq j\leq J\\ I\subset \{1,\ldots,d\}}}
\end{alignat}
with
$$
\theta_j s=\sum_{I\subset\{1,\ldots, d\}}s_{j,I}(q) u^{I}_j\quad,\quad u^{j,I}=u_j^{i_1}\wedge\ldots\wedge u_j^{i_{|I|}}\,,
$$
and 
\begin{alignat}{7}
    \mathcal{I}_{\theta,X}:&&L^2(X,dqdp&;\mathcal{E}_{\pm})&\to &&\mathop{\oplus}_{1\leq j\leq J}^{\perp} L^2(\pi^*(\mathcal{V}_{\theta,j}),&dqdp;\mathcal{E}_{\pm}\big|_{\pi^*(\mathcal{V}_{\theta,j})})
 && \to&
 \mathop{\oplus}_{\substack{1\leq j\leq J\\I,K\subset \{1,\ldots,d\}}}^{\perp} L^2(\pi^*(\mathcal{V}_{\theta,j}),dqdp;\mathbb{C})\\
 &&s&&\mapsto &&(\theta_j s)_{1\leq j\leq J}&&&\mapsto &(s_{j,I}^K)_{\substack{1\leq j\leq J\\ I,K\subset \{1,\ldots,d\}}}\,,
\end{alignat}
with
$$
\theta_j s=\sum_{I,K\subset\{1,\ldots, d\}}s_{j,I}^K(q) f^{I}_j\wedge \hat{f}_{j,K}\quad,
\quad f_j^{I}=f_j^{i_1}\wedge\ldots\wedge f_j^{i_{|I|}}\quad,\quad \hat{f}_{j,K}=\hat{f}_{j,k_1}\wedge\ldots\wedge \hat{f}_{j,k_{|K|}}
$$
\end{definition}
Let us gather obvious properties of the isometries $\mathcal{I}_{\theta,Q}$ and $\mathcal{I}_{\theta,X}$:
\begin{itemize}
    \item The adjoints of $\mathcal{I}_{\theta,Q}$ and $\mathcal{I}_{\theta,X}$ are given by
    \begin{eqnarray*}
    &&\mathcal{I}_{\theta,Q}^*\left[(s_{j,I})_{\substack{1\leq j\leq J\\
    I\subset\{1,\ldots,d\}}}\right]=\sum_{\substack{1\leq j\leq J\\
    I\subset\{1,\ldots,d\}}} \theta_j(q)s_{j,I}u^{I}_j\,,\\
    \text{and}&&
    \mathcal{I}_{\theta,X}^*\left[(s_{j,I}^K)_{\substack{1\leq j\leq J\\
    I,K\subset\{1,\ldots,d\}}}\right]=\sum_{\substack{1\leq j\leq J\\
    I\subset\{1,\ldots,d\}}} \theta_j(q)s_{j,I}^K f^{I}_j\wedge\hat{f}_{j,K}\,.\\
    \end{eqnarray*}
    \item The isometry $\mathcal{I}_{\theta,Q}$  is continuous from $\mathcal{C}^{\infty}(Q;\Lambda T^*Q\otimes F)$ to 
    $\mathop{\oplus}_{\substack{1\leq j\leq J\\ I\subset\{1,\ldots d\}}}\mathcal{C}^{\infty}_0(\mathcal{V}_{\theta,j};\mathbb{C})$, 
    resp. $\mathcal{I}_{\theta,X}$ is continuous from $\mathcal{S}(X:\mathcal{E}_{\pm})$ to 
    $\mathop{\oplus}_{\substack{1\leq j\leq J\\ I\subset\{1,\ldots d\}}}\mathcal{S}(\pi^*(\mathcal{V}_{\theta,j});\mathbb{C})$ 
    while the supports satisfy $\mathrm{supp}\,s_{j,I}^K\subset \pi^*(\mathrm{supp}\,\theta_j)$, with
    \begin{eqnarray*}
     &&\mathcal{I}^*_{\theta,Q}\mathcal{I}_{\theta,Q}=\mathrm{Id}_{L^2}\quad
     ,\quad \mathcal{I}^*_{\theta,Q}\mathcal{I}_{\theta,Q}\big|_{\mathcal{C}^{\infty}(Q;\Lambda T^*Q\otimes F)}=\mathrm{Id}_{\mathcal{C}^{\infty}(Q;\Lambda T^*Q\otimes F)}\,,\\
     \text{and}
     &&
    \mathcal{I}^*_{\theta,X}\mathcal{I}_{\theta,X}=\mathrm{Id}_{L^2}\quad
     ,\quad \mathcal{I}^*_{\theta,X}\mathcal{I}_{\theta,X}\big|_{\mathcal{S}(X;\mathcal{E}_\pm)}=\mathrm{Id}_{\mathcal{S}(X;\mathcal{E}_{\pm})}\,.
 \end{eqnarray*}
 \item The vertical harmonic oscillator hamiltonian given by \eqref{eq:definitionOfO} satisfies as a self-adjoint operator
 \begin{equation}
 \label{eq:locO}
 \mathcal{O}=\mathcal{I}_{\theta,X}^*\left[\frac{-\Delta_{\mathrm{Vert}}+|p|_q^2}{2}\otimes \mathrm{Id}_{\mathbb{C}^{J\times 2^{2d}}}\right]\mathcal{I}_{\theta,X}
 \end{equation}
 with the functional calculus given by
 \begin{equation}
 \label{eq:locfO}
 f(\mathcal{O})=\mathcal{I}_{\theta,X}^*\left[f(\frac{-\Delta_{\mathrm{Vert}}+|p|_q^2}{2})\otimes \mathrm{Id}_{\mathbb{C}^{J\times 2^{2d}}}\right]\mathcal{I}_{\theta,X}
 \end{equation}
 for any Borel function $f:\mathbb{R}\to \mathbb{C}$\,.\\
 The vertical degree $N^V$ written locally as $\sum_{i=1}^d\hat{f}_{j,i}\wedge \mathbf{i}_{\hat{f}_{j,i}}$ is diagonal according to 
 \begin{equation}
\label{eq:locNV}
 N^V=\mathcal{I}_{\theta,X}^*\left[\mathop{\oplus}_{K\subset\{1,\ldots,K\}}^{\perp}|K|\mathrm{Id}_{\mathbb{C}^{J\times 2^d}}\right]\mathcal{I}_{\theta,X}\,.
\end{equation}
\end{itemize}
We recall now the general definition of the global Sobolev spaces $\tilde{\mathcal{W}}^{s_1,s_2}(X;\mathcal{E}_\pm)$, $(s_1,s_2)\in \mathbb{R}^2$, introduced in \cite{NSW}.\\
With the horizontal-vertical decomposition \eqref{eq:horizontalVerticalDecomposition} and the metric $g^{TQ}$, the horizontal scalar Laplacian (see \cite{BeBo}) is given by
$$\Delta_H =g^{ij}(q)(e_ie_j - \Gamma_{ij}^k e_k ) = (e_i)^*\circ g^{ij}(q)\circ e_j\,,
$$
while the vertical scalar harmonic oscillator operator $\mathcal{O}$ has already been introduced in \eqref{eq:definitionOfO}.\\
The scalar operator $W^2$  is defined as the closure in $L^{2}(X,dqdp;\mathbb{C})$ of the differential operator $C_{g}-\Delta_{H}+C_{g}\mathcal{O}^2: \mathcal{S}(X;\mathbb{C}) \rightarrow \mathcal{S}(X;\mathbb{C})$ for $C_{g}\geq 1$ large enough. The operator $W^2$ is self-adjoint and  $(W^2,\mathcal{O})$ is a pair of commuting self-adjoint operators.\\
The non scalar version $W^2_{\theta}$ is modelled on the scalar version after using the quadratic partition of unity on $Q$\,, $\sum_{j=1}^{J}\theta_j^2(q)\equiv 1$ and the isometry $\mathcal{I}_{\theta,X}$\,. It is given 
by 
\begin{equation}
\label{eq:definitionW2theta}
    W^2_{\theta}=\mathcal{I}_{\theta,X}^*\left[W^2\otimes\mathrm{Id}_{\mathbb{C}^{J\times 2^{2d}}}\right]\mathcal{I}_{\theta,X}
    =\sum_{j=1}^J\theta_{j}(q)\circ W^{2}_{sc,j}\circ \theta_{j}(q)\,.
\end{equation}
where $W^2_{sc,j}$ is defined by using the connection $\nabla^j$ which is trivial in the orthonormal frame $(f_j^1,..,f_j^d,\hat{f}_{j,1},..,\hat{f}_{j,d})$\,.\\
Again for $C_g\geq 1$ large enough, $(W^2_{\theta},\mathcal{O})$ is a pair of strongly commuting self-adjoint operators in $L^2(X,dqdp;\mathcal{E}_\pm)$. We refer the reader to \cite{NSW} for details.
 
\begin{definition}[Sobolev Spaces]
  \label{def:sobsp}
    For all $s_1,s_2 \in \mathbb{R}$, the double exponent Sobolev space $\tilde{\mathcal{W}}^{s_1,s_2}(X,dqdp;\mathcal{E}_{\pm})$ is defined by
    $$ \tilde{\mathcal{W}}^{s_{1},s_{2}}(X; \mathcal{E}_\pm) = \{ u \in \mathcal{S}'(X; \mathcal{E}_\pm), \mathcal{O}^{\frac{s_{1}}{2}} (W^{2}_{\theta})^{s_{2}/2}u \in L^{2}(X,dqdp;\mathcal{E}_\pm) \} .$$
    The norm is defined as $ \|u\|_{ \tilde{\mathcal{W}}^{s_{1},s_{2}}(X;\mathcal{E}_\pm)} = \| \mathcal{O}^{\frac{s_{1}}{2}}(W^{2}_{\theta})^{\frac{s_{2}}{2}}u \|_{L^{2}}$. For simplicity,  those spaces will often be denoted by $ \tilde{\mathcal{W}}^{s_{1},s_{2}} $\,.
  \end{definition}
  The pseudodifferential calculus associated with $W^2_{\theta}$ was introduced in \cite{NSW}  where the order of operators is recalled here:
  $$
  p_i, D_{p_{i}}~(1/2)\quad,\quad \mathcal{O}, e_i  ~(1)\quad,\quad \nabla_{\mathcal{Y}}^{\mathcal{E}_\pm}~(3/2)\quad,\quad (W^2_{\theta})^{s/2} ~ (s)\,,
  $$
  and it says in particular
  \begin{align*}
&&\tilde{\mathcal{W}}^{0,s_{2}+\frac{s_1}{2}}&\subset \tilde{\mathcal{W}}^{s_{1},s_{2}}\subset \tilde{\mathcal{W}}^{0,s_{2}}\,,&\\
&& \tilde{\mathcal{W}}^{0,s_{2}}&\subset \tilde{\mathcal{W}}^{-s_{1},s_{2}}\subset \tilde{\mathcal{W}}^{0,s_{2}-\frac{s_1}{2}}&\quad\text{for}~s_1\geq 0, s_2\in \mathbb{R}\,,
  \end{align*}
  and
  $$
  \mathop{\cap}_{s_2\in\mathbb{R}} \tilde{\mathcal{W}}^{s_1,s_2}=\mathcal{S}(X;\mathcal{E}_{\pm})\quad \mathop{\cup}_{s_2\in \mathbb{R}}\tilde{\mathcal{W}}^{s_1,s_2}=\mathcal{S}'(X;\mathcal{E}_{\pm})\quad
  \text{for~all}~s_1\in \mathbb{R}.
  $$

We end this section by adding some notations and by recalling some functional analysis properties.

\begin{definition}
  \label{de:adjoints}
  For a continuous operator $A:\mathcal{S}(X;\mathcal{E}_{\pm})\to \tilde{\mathcal{W}}^{0,s}(X;\mathcal{E}_{\pm})$\,, which is closable in the Hilbert space $\tilde{\mathcal{W}}^{{0,s}}(X;\mathcal{E}_{\pm})$\,, its closure will be denoted by $\overline{A}^{s}$ while $\overline{A}=\overline{A}^{0}$\,.\\
  Its formal adjoint for the $\tilde{\mathcal{W}}^{0,s}(X;\mathcal{E}_{\pm})$-scalar product will be written $A^{\prime,s}:\tilde{\mathcal{W}}^{0,s}(X;\mathcal{E})
  \to \mathcal{S}'(X;\mathcal{E}_{\pm})$\,, with $A'=A^{\prime,0}$\,. The same notation will be used for its restriction to $\mathcal{S}(X;\mathcal{E}_{\pm})$ instead of $A^{\prime,s}\big|_{\mathcal{S}(X;\mathcal{E}_{\pm})}$\,.\\
  Its adjoint for the $\tilde{\mathcal{W}}^{0,s}(X;\mathcal{E}_{\pm})$ will be denoted by $A^{*,s}:D(A^{*,s})\to \tilde{\mathcal{W}}^{0,s}(X;\mathcal{E}_{\pm})$ with $u\in D(A^{*,s})$ characterized by
$$
\exists C_{u}\geq 0\,,\quad \forall v\in  \mathcal{S}(X;\mathcal{E})\,, |\langle u\,,\, A v\rangle_{\tilde{\mathcal{W}}^{0,s}}|\leq C_{u,s}\|v\|_{\tilde{\mathcal{W}}^{0,s}}\,.
$$
Again the simpler notation $A^{*}=A^{*,0}$ is reserved for the case $s=0$\,. 
\end{definition}
Because $(W^{2}_{\theta})^{s/2}:\tilde{\mathcal{W}}^{0,s}(X;\mathcal{E}_{\pm})\to L^{2}(X,dqdp;\mathcal{E}_{\pm})$ is unitary, while it is a continuous automorphism of  $\mathcal{S}(X;\mathcal{E}_{\pm})$ (resp. $\mathcal{S}'(X;\mathcal{E}_{\pm})$),  the study of $A:\mathcal{S}(X;\mathcal{E}_{\pm})\to \tilde{\mathcal{W}}^{0,s}(X;\mathcal{E}_{\pm})$ is equivalent to the one of
$$
A_{s}=(W^{2}_{\theta})^{s/2}A(W^{2}_{\theta})^{-s/2}: \mathcal{S}(X;\mathcal{E}_{\pm})\to L^{2}(X,dpdp;\mathcal{E}_{\pm})\,.
$$
This is in particular convenient when $A:\mathcal{S}(X;\mathcal{E})\to \mathcal{S}(X;\mathcal{E})$ is continuous. Actually $\overline{A}^{s}=(W^{2}_{\theta})^{-s/2}\overline{A_{s}}^{0}(W^{2}_{\theta})^{s/2}$ and we can simply work in $L^{2}(X,dqdp;\mathcal{E}_{\pm})$ with the family of densely defined operators $(A_{s})_{s\in\mathbb{R}}$ as we already did in the proof of Proposition~\ref{pr:main1sttext}.\\
We deduce in particular the formulas:
\begin{eqnarray}
  \label{eq:formadjconj1}
   A^{\prime,s} &= &
     \left[(W^{2}_{\theta})^{-s/2}A_{s}(W^{2}_{\theta})^{s/2}\right]^{\prime,s}=(W^{2}_{\theta})^{-s}\left[(W^{2}_{\theta})^{-s/2}A_{s}(W^{2}_{\theta})^{s/2}\right]^{\prime}(W^{2}_{\theta})^{s}   \label{eq:formadjconj2} \\
   &  = &(W^{2}_{\theta})^{-s/2}A_{s}'(W^{2}_{\theta})^{s/2} \nonumber\\
(A^{\prime,s})_{s}& = & A_{s}'\\
  \label{eq:adjconj}
  \text{and} \quad 
   (A^{*,s})_s & = & A_{s}^{*}\,.
\end{eqnarray}
In all of our cases the operator $A$ and its formal adjoint $A'$ are continuous from $\mathcal{S}(X;\mathcal{E}_{\pm})$ to itself. Alternatively $A$ is continous from $\mathcal{S}(X;\mathcal{E}_{\pm})$ to itself and from $\mathcal{S}'(X;\mathcal{E}_{\pm})$ to itself.
We always have
$$
\overline{A^{\prime,s}}^{s}\subset A^{*,s} \quad\text{and}\quad
\overline{A_{s}'}\subset A_{s}^{*}
$$
in the sense that $\overline{A^{\prime,s}}^{s}$ is the minimal extension of $A^{\prime,s}\big|_{\mathcal{S}(X;\mathcal{E}_{\pm})}$ while $A^{*,s}$ is its maximal extension. Under the above assumption the case of equality is treated via the equivalence
$$
\left(\overline{A^{\prime,s}}^{s}=A^{*,s}\right)\Leftrightarrow \left(
A_{s}^{'}=A_{s}^{*}\right)\,.
$$
Remember that accretive operators are closable and with an additional positive constant they are one to one and have a closed range. Essential maximal accretivity, under the above assumptions, means exactly  $A^{\prime,s}=A^{*,s}$ or,  equivalently, $A'_{s}=A^{*}_{s}$\,.

  \subsection{Bismut's hypoelliptic Laplacian}
\label{sec:BihypLap} 
  We do present here neither the construction of the hypoelliptic Laplacian as a deformed Hodge type operator, nor the various various versions of it which are presented in \cite{Bis05}\cite{BiLe}.
  We directly start with the Weitzenb\"ock formula for the version denoted by $2\mathfrak{A}'_{\phi_b,\pm \mathcal{H}}$ in \cite{BiLe}-p~32 formulas (2.3.12)(2.3.13). According to \cite{BiLe}-p32 (see formula (2.3.14)) it makes sense as an operator acting on $\mathcal{S}(X;\mathcal{E})$ and the formal adjoints are computed with the scalar product of  $L^2(X,dqdp; \mathcal{E})=L^2(X,e^{-2V(q)}dqdp;\mathcal{E}_\pm)$\,. For this presentation the parameter $h\in ]0,1]$ is not yet considered but it suffices to replace like in Subsection~\ref{sec:hermbund} the potential $V$ by $\frac{V}{h}$ and various equivalent representations are explained in Subsection~\ref{sec:scalings}.\\
  Formulas (2.3.12)(2.3.13) of \cite{BiLe} say for a local orthonormal frame $f_{1},\dots,f_{d}$ of $TQ$:
  \begin{equation}
  \label{eq:Weitz}
             2\mathfrak{A}_{\phi_b,\pm \mathcal{H}}'^{2} = \frac{1}{b^{2}} \alpha_{\pm}^{\prime} + \frac{1}{b} \beta_{\pm}^{\prime} + \gamma_{\pm}^{\prime} 
    \end{equation}

    where
    \begin{align}
    \label{eq:alphaprime}
      \alpha_{\pm}^{\prime}  = & \frac{1}{2}(- \Delta^{V} + |p|_{g}^{2} \pm (2\hat{f}_{i}\mathbf{i}_{\hat{f}^{i}} - d))  ,  \\
     \label{eq:betaprime}
      \beta_{\pm}^{\prime}  =  & - ( \pm \nabla^{\Lambda^{\cdot}T^{*}X\otimes \pi^*F,u}_{\mathcal{Y}} - (f_{i}V) \nabla^{\Lambda^{\cdot} T^{*}X}_{\hat{f}^{i}}) ,   \\
      \label{eq:gammaprime}
      \gamma_{\pm}^{\prime}  =  & - \frac{1}{4} \left\langle R^{TQ}(f_{i},f_{j})f_{k},f_{\ell} \right\rangle (f^{i} - \hat{f}_{i})(f^{j}-\hat{f}_{j})\mathbf{i}_{f_{k} + \hat{f}^{k}}\mathbf{i}_{f_{\ell} + \hat{f}^{\ell}}  \\
                      & \quad  - \left( \pm\left\langle R^{TQ}(p,f_{i})p,f_{j} \right\rangle - (f_{i}(f_{j}V) +\tilde{\Gamma}_{ij}^{k} f_{k}V) \right)(f^{i} - \hat{f}_{i})\mathbf{i}_{f_{j} + \hat{f}^{j}}\,,    \end{align}
and $\tilde{\Gamma}_{ij}^{k}(q) = f^{k}(\nabla_{f_{i}}^{TQ}f_{j})$  the Christoffel symbol expressed in this frame.\\
In order to have good duality arguments when the base manifold $Q$ is not oriented, the vector bundle $\mathcal{E}$ must be $\Lambda T^*X\otimes \mathbb{C}$ in the $+$ case and $\Lambda T^*X\otimes \mathbb{C}\otimes \pi^*(\mathbf{or}_Q)$ in the $-$ case.\\
In \cite{Bis05}, Propositions~3.14 also provides the formula
\begin{equation}
\label{eq:algebraicIdentityForHodgeLaplacianprime}
    \pi_{0,\pm}(\gamma_{\pm}^{\prime} - \beta_{\pm}^{\prime} \alpha_{\pm}^{\prime,-1} \beta_{\pm}^{\prime} )\pi_{0,\pm} = \frac{\square^{Q,F}}{2}\,,
\end{equation}
where $\pi_{0,\pm}$ is the orthogonal projection on the kernel of $\ker(\alpha^{\prime}_{\pm})$\,. In the formula \eqref{eq:algebraicIdentityForHodgeLaplacianprime}, there is an identification between operators acting on $\mathrm{Ran}\,\pi_{0,\pm}$ and operators defined on the base manifold $Q$ which is detailed below. Let us keep for the moment the notations of \cite{Bis05}\cite{BiLe}.

In our framework, i.e. when we work in $L^2(X,dqdp;\mathcal{E}_\pm)$, it suffices to conjugate all the operators according to $A\mapsto e^{-V}Ae^{V}$\,. We obtain
\begin{equation}
  \label{eq:WeitzbL2}
             B_{\pm,b,V} =2 e^{-V}\mathfrak{A}_{\phi_b,\pm \mathcal{H}}'^{2}e^{V}= \frac{1}{b^{2}} \alpha_{\pm} + \frac{1}{b} \beta_{\pm} + \gamma_{\pm}
    \end{equation}

    where
    \begin{align}
    \label{eq:alpha}
      \alpha_{\pm}  = &\alpha_{\pm}^{\prime}= \frac{1}{2}(- \Delta^{V} + |p|_{g}^{2} \pm (2\hat{f}_{i}\mathbf{i}_{\hat{f}^{i}} - d))  ,  \\
     \label{eq:betaprelim}
      \beta_{\pm}  =  & e^{-V(q)}\beta_{\pm}^{\prime}e^{V(q)}=- e^{-V(q)}( \pm \nabla^{\Lambda^{\cdot}T^{*}X\otimes \pi^*F,u}_{\mathcal{Y}} - (f_{i}V) \nabla^{\Lambda^{\cdot} T^{*}X}_{\hat{f}^{i}})e^{V(q)} ,   \\
      \label{eq:gamma}
      \gamma_{\pm}  =  &\gamma_{\pm}^{\prime} = - \frac{1}{4} \left\langle R^{TQ}(f_{i},f_{j})f_{k},f_{\ell} \right\rangle (f^{i} - \hat{f}_{i})(f^{j}-\hat{f}_{j})\mathbf{i}_{f_{k} + \hat{f}^{k}}\mathbf{i}_{f_{\ell} + \hat{f}^{\ell}}  \\
                      & \quad  - \left( \pm\left\langle R^{TQ}(p,f_{i})p,f_{j} \right\rangle - (f_{i}(f_{j}V) +\tilde{\Gamma}_{ij}^{k} f_{k}V) \right)(f^{i} - \hat{f}_{i})\mathbf{i}_{f_{j} + \hat{f}^{j}}\,.   \end{align}
The only non trivial calculation is for $\beta_{\pm}$. According to Table~\ref{tab:exp}, $e^{-V(q)}\nabla^{F,u}e^{V(q)}=d^Q$ and we obtain
$$
e^{-V(q)} \nabla^{\Lambda^{\cdot}T^{*}X\otimes \pi^*F,u} e^{-V(q)}=\nabla^{\mathcal{E}_\pm}
$$
where $\nabla^{T^*X}$ is the tautological connection on $T^*X$ associated with the Levi-Civita connection on $TQ$ and $\nabla^{\mathcal{E}_\pm}$ is the exterior algebra extension.\\
We obtain
\begin{equation}
\label{eq:beta}
\beta_{\pm}=- ( \pm \nabla^{\mathcal{E}_\pm}_{\mathcal{Y}} - (f_{i}V) \nabla^{\mathcal{E}_\pm}_{\hat{f}^{i}})
\end{equation}
Because $\alpha_{\pm}^{\prime}$ commutes with $e^{\pm V(q)}$ the kernel of $\alpha_{\pm}$ and the orthogonal projections $\pi_{0,\pm}$ are not changed. 
By Table~\ref{tab:exp} we also know
$$
e^{-V(q)}\square^{Q,F}e^{V(q)}=\Delta_{V,1}\,.
$$
The formula \eqref{eq:algebraicIdentityForHodgeLaplacianprime} becomes
$$
  \pi_{0,\pm}(\gamma_{\pm} - \beta_{\pm} \alpha_{\pm}^{-1} \beta_{\pm} )\pi_{0,\pm} =  \frac{1}{2}\Delta_{V,1}\,,
  $$
  and when the potential $V$ is replaced by $\frac{V}{h}$\,, $h\in ]0,1]$\,,
\begin{equation}
 \label{eq:algebraicIdentityForWittenLaplacianprime}
   \pi_{0,\pm}(\gamma_{\pm} - \beta_{\pm} \alpha_{\pm}^{-1} \beta_{\pm} )\pi_{0,\pm} = \frac{1}{2h^2} \Delta_{V,h}\,,
\end{equation}
where $\Delta_{V,h}=(d_{V,h}+d_{V,h}^*)^2$ is the semiclassical Witten Laplacian.\\

Another property proved in \cite{Bis05}\cite{BiLe} which will be useful for proving $\mathrm{Spec}(B_{\pm,b,V})\subset [0,+\infty[$ for $b>0$ small enough, is related with the Hodge structure of $B_{\pm,b,V}=2\mathfrak{A}_{\phi_{b},\pm\mathcal{H}}'^{2}$ that we briefly recall here.
\begin{definition}
\label{de:Bismutnot}~
\begin{itemize}
\item The tensorial operations $\lambda_0$ and $\mu_0$, expressed in the orthonormal frames $(f_i,\hat{f}^i,f^i,\hat{f}_i)_{1\leq i\leq d}$\,, are  $\lambda_0=f^i\wedge \mathbf{i}_{\hat{f}^i}$ (resp. $\mu_0=\hat{f}_i\wedge \mathbf{i}_{f_i}$) which increases the horizontal (resp. vertical) degree by $1$ and decreases the  vertical (resp. horizontal) degree by $1$\,.
As nilpotent elements of $\mathrm{End}(\Lambda T^*X) $, their exponential $e^{\pm \lambda_0} $ (resp. $e^{\pm \mu_0} $) are polynomials.
    \item For $a\in \mathbb{R}$\,, $r_a:X\to X$ is given by $r_a(q,p)=(q,ap)$ and $r_a^*:\mathcal{S}(X;\mathcal{E}_\pm)\to \mathcal{S}(X;\mathcal{E}_\pm)$ is the natural pull-back. The simpler notations $r$ and $r^*$ will be used for the isometric involutions obtained for $a=-1$\,. The linear map $K_a:\mathcal{S}(X;\mathcal{E}_{\pm})\to \mathcal{S}(X;\mathcal{E}_{\pm})$ is given by $K_a (s_I^J(q,p) f^I\hat{f}_J)=a^{d/2} s_I^J(q,ap) (f^I \hat{f}_J)\big{|}_{x}$ with a trivial action in the bases $(f^I,\hat{f}_J)_{I,J\subset\{1,\ldots,d\}}, \text{at}~x=(q,p)$ and $(f^I,\hat{f}_J)_{I,J\subset\{1,\ldots,d\}}$\,, at $~x=(q,ap)$\,.
    \item The hermitian form $\langle~,~\rangle_{r}$ on $\mathcal{S}(X;\mathcal{E}_{\pm})$ is given by
    $$
    \langle u\,,\, v\rangle_r=\langle u\,,\, r^* v\rangle\,.
    $$
\end{itemize}
\end{definition}
The operator $\mathfrak{A}^{'}_{\pm,b}$ equals
\begin{eqnarray}
\label{eq:BHodge}
&& B_{\pm,b,V}=2\mathfrak{A}_{\phi_b,\pm\mathcal{H}}'^{2}=
2(\frac{\delta_{\pm,b,V}+ \delta_{\pm,b,V}^{*,r}}{2})^2=\frac{1}{2}[\delta_{\pm,b,V}\delta_{\pm,b,V}^{*,r}+\delta_{\pm,b,V}^{*,r}\delta_{\pm,b,V}]\,,\\
\label{eq:deltabV}\text{with}&&
\delta_{\pm,b,V}=K_{{b}}e^{-\mu_0} e^{-V}d^{X}_{\pm\frac{1}{b^{2}}\mathcal{H}}e^{V}e^{\mu_0}K_{b}^{-1}
                                =e^{-\mu_0}e^{\mp\mathcal{H}-V}(K_{b}d^{X} K_{b}^{-1})e^{\pm\mathcal{H}+V}e^{\mu_0}\,,
  \\
  \text{and}&&
               \label{eq:deltabV*}
\delta_{\pm,b,V}^{*,r}=e^{-\lambda_0}e^{\pm\mathcal{H}+V}K_{b}d^{X,*}K_{b}^{-1}e^{\mp\mathcal{H}-V}
e^{+\lambda_0}\,,
\end{eqnarray}
where $d^{X,*}$ stand for the standard Hodge codifferential for the metric $\pi^{*}(g\oplus g^{-1})$ on $TX=T(T^{*}Q)$\,.\\
The operator $\delta_{\pm,b,V}^{*,r}$ is actually the $\langle~,~\rangle_{r}$-formal adjoint of $\delta_{\pm,b,V}$:
$$
\forall u,v\in \mathcal{S}(X;\mathcal{E}_{\pm}),\quad
\langle u\,,\, r^{*}(\delta_{\pm,b,V}^{*,r})v\rangle=\langle \delta_{\pm,b,V}u\,,\, r^{*}v\rangle\,.
$$
The important properties for us are $\delta_{\pm,b,V}^{2}=0$\,, $(\delta_{\pm,b,V}^{*,r})^{2}=0$ and the fact that $\frac{1}{2}B_{\pm,b,h}$ is the square of the $\langle~,~\rangle_{r}$ symetric operator $\mathfrak{A}_{\phi_{b},\pm \mathcal{H}}'$\,.\\
Let us explain how \eqref{eq:BHodge}\eqref{eq:deltabV} and \eqref{eq:deltabV*} written in our setting are deduced from the formulas (2.1.23) (2.1.24) and (2.1.28) of \cite{BiLe} (see also Section~2 and Section~3 of \cite{Bis05}):
\begin{itemize}
\item The factors $e^{\pm V}$ come from our choice of scalar product $\langle~,~\rangle$ instead of $\langle~,~\rangle_{L^{2}(X,dqdp;\mathcal{E})}$ and the correspondance of Table~\ref{tab:exp}. Once this is settled, this factor can be forgotten for the comparison with the formulas of \cite{BiLe}.
\item For a general $b>0$\,, the formula (2.1.28)-\cite{BiLe}
  $$
  \mathfrak{A}'_{\phi_{b},\pm \mathcal{H}}=
  K_{b}\mathfrak{A}'_{\phi_{1},\pm r_{\frac{1}{b}}^{*}\mathcal{H}}K_{b}^{-1}
  =K_{b}\mathfrak{A}'_{\phi_{1},\pm \frac{1}{b^{2}}\mathcal{H}}K_{b}^{-1}
  $$
  allows to extend the formulas (2.1.22)(2.1.23)-\cite{BiLe} written for the case $b=1$ to the general case. Because $K_{b}$ commutes with $\mu_{0}$ (and $\lambda_{0}$) this provides the formula \eqref{eq:deltabV}. Because $K_{b}$ commutes with $r^{*}$ it implies that $\delta_{\pm,b,V}^{*,r}$ is the $\langle~,~\rangle_{r}$-formal adjoint of $\delta_{\pm,b,V}$\,.
\item Finally the explicit expression of $\delta_{\pm,b,V}$ is obtained after using the property that $\lambda_{0}$ is the $\langle~,~\rangle$-formal adjoint of $\mu_{0}$\,, and the involutive identity $r^{*}\lambda_{0}r^{*}=-\lambda_{0}$\,.
\end{itemize}
Like in \cite{BiLe}-page 32 but with now the $L^2(X,dqdp;\Lambda T^*X)$ scalar product $\langle~,~\rangle$, we recall the elementary functional properties of $\alpha_{\pm}$ and $\beta_{\pm}$\,. Meanwhile, we make  the identifications hidden in \eqref{eq:algebraicIdentityForHodgeLaplacianprime} and \eqref{eq:algebraicIdentityForWittenLaplacianprime} more explicit by using the isometries $\mathcal{I}_{\theta,X}$ and $\mathcal{I}_{\theta,Q}$ of Definition~\ref{de:localization}.
  \begin{itemize}
  \item The operator $\alpha_{\pm}=\mathcal{O}\pm (N^V-d/2)$ is self-adjoint on its domain $\tilde{\mathcal{W}}^{2,0}(X;\mathcal{E}_\pm)$,. By using the fiberwise change of variable $\tilde{p}_i=(\sqrt{g(q)})^{ij}p_j$ the Hilbert space $L^2(X,dqdp;\mathcal{E}_{\pm})$ can be written as the direct integral
  $$
  L^2(X,dqdp;\mathcal{E}_{\pm})=\int_Q^{\oplus}L^2(\mathbb{R}^d,d\tilde{p};\mathbb{C}^{2^{2d}})~d\mathrm{Vol}_g(q)
  $$
  if we notice $dqdp=|\det(g(q))|^{1/2}~dq d\tilde{p}$\,. In this direct integral representation, $\alpha_{\pm}$ is nothing but
  $$
  \alpha_{\pm}=\int_Q^\oplus \frac{-\Delta_{\tilde{p}}+|\tilde{p}|^2}{2}\otimes \mathrm{Id}_{\mathbb{C}^{2^{2d}}}\pm (N_V-d/2)~d\mathrm{Vol}_g(q)
  $$
  where $\frac{-\Delta_{\tilde{p}}+|\tilde{p}|^2}{2}=\sum_{i=1}^d \frac{-\partial^2_{\tilde{p}_j}+\tilde{p}_j^2}{2}$ is the euclidean scalar harmonic oscillator.
  Therefore the spectrum of $ \alpha_{\pm} $  equals $\mathbb{N}$. The kernel of $\alpha_{+}$ is given by horizontal forms times $\exp\big(-\frac{|p|_{q}^{2}}{2}\big)$ and the kernel of $\alpha_{-}$ is given by the exterior product of horizontal forms with a top vertical form times $\exp\big(-\frac{|p|_{q}^{2}}{2}\big)$.
    \item With the orthogonal projection $ \pi_{0,\pm} $ on the kernel of $\alpha_{\pm}$ and $1-\pi_{0,\pm} = \pi_{\bot,\pm}$ its orthogonal complement, we have
      \begin{eqnarray*}
    &&    L^2(X,dqdp;\mathcal{E}_{\pm})=\ker{\alpha_{\pm}} \mathop{\oplus}^\perp \mathrm{Ran}{\,\alpha_{\pm}}
    =\mathrm{Ran}\,\pi_{0,\pm}\mathop{\oplus}^\perp \mathrm{Ran}{\,\pi_{\perp,\pm}}
    \\
    \text{and}
     &&\mathcal{S}(X;\mathcal{E}_{\pm})=
      \big(\mathrm{Ran}\,\pi_{0,\pm}\cap\mathcal{S}(X;\mathcal{E}_{\pm})\big) \oplus \big(\mathrm{Ran}{\,\pi_{\bot,\pm}}\cap \mathcal{S}(X;\mathcal{E}_{\pm})\big)
    \end{eqnarray*}
    while the functional calculus says that  $\alpha_{\pm}:\mathrm{Ran}\pi_{\bot,\pm}\cap\tilde{\mathcal{W}}^{2,0}(X;\mathcal{E}_{\pm})\to \mathrm{Ran}\,\pi_{\bot,\pm}$ is invertible with the norm of $\alpha_{\pm}^{-1}\pi_{\bot,\pm}$ equal to $1$.
    \item The differential operator $\beta_{\pm}$ maps $\ker{\alpha_{\pm}}\cap \mathcal{S}(X;\mathcal{E}_{\pm})=\mathrm{Ran}\,{\pi_{0,\pm}}\cap \mathcal{S}(X;\mathcal{E}_{\pm})$ into $\mathrm{Ran}\,\alpha_{\pm}\cap \mathcal{S}(X;\mathcal{E}_{\pm})=\mathrm{Ran}\,\pi_{\bot,\pm}\cap \mathcal{S}(X;\mathcal{E}_{\pm})$.
    \item With \eqref{eq:locO}\eqref{eq:locfO}\eqref{eq:locNV} we can write
    $$
    f(\alpha_{\pm})=\mathcal{I}_{\theta,X}^*\left[\mathop{\oplus}_{K\subset\{1,\ldots,d\}}f\left(\frac{-\Delta_{\mathrm{Vert}}+|p|_q^2}{2}\pm (|K|-d/2)\right)\otimes \mathrm{Id}_{\mathbb{C}^{J\times {2d}}}\right]\mathcal{I}_{\theta,X}
    $$
    for any Borel function $f:\mathbb{R}\to\mathbb{C}$\,. In particular for $f=1_{\{0\}}$ we obtain
    $$
    \pi_{0,\pm}=\mathcal{I}_{\theta,X}^*\left[1_{\{0\}}\big(\frac{-\Delta_{\mathrm{Vert}}+|p|_q^2-d}{2}\big)1_{\{0\}}\big(|K|-d/2\pm d/2\big)\right]\mathcal{I}_{\theta,X}\,.
    $$
    The kernel of the harmonic oscillator $\frac{-\Delta_{\mathrm{Vert}}+|p|_q^2-d}{2}$ equals $\mathbb{C}\frac{e^{-\frac{|p|_q^2}{2}}}{\pi^{d/4}}$ with
    $$
    \int_{\mathbb{R}^d}|\frac{e^{-\frac{|p|_q^2}{2}}}{\pi^{d/4}}|^2~dp=|\det(g(q)|^{1/2}\,.
    $$
    We deduce that
    \begin{equation}
        \label{eq:U+theta}
    U_{+,\theta}=\mathcal{I}_{\theta,X}^*\left[\frac{e^{-\frac{|p|_q^2}{2}}}{\pi^{d/4}}\times\right]\mathcal{I}_{\theta,Q}
    \end{equation}
    is a unitary transform $U_{+,\theta}:L^2(Q,d\mathrm{Vol}_g;\Lambda T^*Q\otimes \mathbb{C})\to \mathrm{Ran}\,\pi_{+,0}=\ker(\alpha_+)$ in the $+$ case.
    In the $-$ case, we choose $\eta\in \mathcal{C}^{\infty}(X;\Lambda^d (T^*X)^V\otimes\pi^*(\mathbf{or}_Q))$  to be a normalized non vanishing section, which can be written locally as $\hat{f}_{j,1}\wedge\ldots\wedge\hat{f}_{j,d}$ with the suitable orientation. Then the unitary transform $U_{-,\theta}:L^2(Q,d\mathrm{Vol}_g;\Lambda T^*Q\otimes \mathbb{C}\otimes \mathbf{or}_{Q})\to \mathrm{Ran}\,\pi_{-,0}=\ker(\alpha_-)$ is given by 
    \begin{equation}
    \label{eq:U-theta}
    U_{-,\theta}=\mathcal{I}_{\theta,X}^*\left[\frac{e^{-\frac{|p|_q^2}{2}}}{\pi^{d/4}}\times\right](\mathcal{I}_{\theta,Q}\wedge\eta).
    \end{equation}
    When $\mathcal{I}_{\theta,Q}(s)=(s_{j,I})_{\substack{1\leq j\leq J\\ I\subset\{1,\ldots,d\}}}$ for $s\in L^2(Q,d\mathrm{Vol}_g;\Lambda T^*Q\otimes F)$ ($F=Q\times \mathbb{C}$ in the $+$ case and $F=(Q\times\mathbb{C}\otimes \mathbf{or}_Q)$ in the $-$ case) we get
    \begin{eqnarray}
     \label{eq:U+thetaexpl}
     &&
     U_{+,\theta}s=\sum_{\substack{1\leq j\leq J\\ I\subset\{1,\ldots,d\}}} \theta_j(q)s_{j,I}(q)\frac{e^{-\frac{|p|_q^2}{2}}}{\pi^{d/4}} f_{j}^I\\
     \text{and}&&
     U_{-,\theta}s=\sum_{\substack{1\leq j\leq J\\ I\subset\{1,\ldots,d\}}} \theta_j(q)s_{j,I}(q)\frac{e^{-\frac{|p|_q^2}{2}}}{\pi^{d/4}} f_{j}^I\wedge \hat{f}_{j,1}\wedge \ldots \wedge \hat{f}_{j,d}\,.
     \label{eq:U-thetaexp}
    \end{eqnarray}
    When $\mathcal{I}_{\theta,X}(s')=(s_{j,I}^K)_{\substack{1\leq j\leq J\\I,K\subset\{1,\ldots,d\}}}$ for $s'\in L^2(X,dqdp;\mathcal{E}_{\pm})$ we get
    \begin{alignat}{1}
     U_{+,\theta}^{-1}s'=U_{+,\theta}^*s'  & =\sum_{\substack{1\leq j\leq J\\ I\subset\{1,\ldots,d\}}} \theta_j(q)\left(\int_{T^*_q Q}\frac{e^{-\frac{|p|_q^2}{2}}}{\pi^{d/4}} s_{j,I}^{\emptyset}(q,p)~dp\right) u_{j}^I \label{eq:U+theta*expl}\\
      \label{eq:U-theta*exp}
      \text{and} \quad
     U_{-,\theta}^{-1}s' =U_{-,\theta}s'  & = \sum_{\substack{1\leq j\leq J\\ I\subset\{1,\ldots,d\}}} \theta_j(q)\left(\int_{T^*_q Q}\frac{e^{-\frac{|p|_q^2}{2}}}{\pi^{d/4}} s_{j,I}^{\{1,\ldots,d\}}(q,p)~dp\right) u_{j}^I\wedge\hat{f}_{j,1}\wedge\ldots\wedge \hat{f}_{j,d}\,.
    \end{alignat}
The unitary map $U_{\pm,\theta}:L^2(Q,d\mathrm{Vol}_g;\Lambda T^*Q\otimes F)\to \mathrm{Ran}\,\pi_{0,\pm}=\ker(\alpha_{\pm})$ clearly induces an isomorphism depending on the case:
\begin{eqnarray*}
    U_{+,\theta}&:&\mathcal{C}^\infty(Q;\Lambda^{\cdot}T^{*}Q\otimes \mathbb{C})  \rightarrow  \mathcal{S}(X; \mathcal{E}_+) \cap \ker \alpha_+\\ 
    U_{-,\theta}&:&\mathcal{C}^\infty(Q;\Lambda^{\cdot}T^{*}Q\otimes\mathbb{C}\otimes\mathbf{or}_Q) 
     \rightarrow  \mathcal{S}(X; \mathcal{E}_-) \cap \ker \alpha_- ~.
\end{eqnarray*}
Other functional spaces can be considered. With those notations, formula \eqref{eq:algebraicIdentityForWittenLaplacianprime} means precisely
\begin{equation}
     \label{eq:algebraicIdentityForWittenLaplacianUtheta}
     U_{\pm,\theta}^{-1}[\pi_{0,\pm}(\gamma_{\pm} - \beta_{\pm} \alpha_{\pm}^{-1} \beta_{\pm} )\pi_{0,\pm}] U_{\pm,\theta}= \frac{1}{2h^2} \Delta_{V,h} \,.
\end{equation}
Notice also that $e_i (e^{-\frac{|p|_q^2}{2}}\ a(q))=e^{-\frac{|p|_q^2}{2}}\frac{\partial a}{\partial q^i}(q)$ implies
\begin{equation}
\label{eq:LaplBelscUtheta}
U_{\pm,\theta}^{-1}[\pi_{0,\pm} W^2_{\theta}\pi_{0,\pm}] U_{\pm,\theta}=\mathcal{I}_{\theta, Q}^*\left[(C+C\frac{d^2}{4}-\frac{1}{2}\Delta_Q)\otimes\mathrm{Id}_{\mathbb{C}^{J\times 2^d}}\right]\mathcal{I}_{\theta,Q}
\end{equation}
where $\Delta_{Q}$ is the scalar Laplace-Beltrami operator on $Q$\,.
       \end{itemize}
\begin{lemma}
\label{le:nabYpi0}
There exists a constant $C_{g,\theta}\geq 1$ determined by the metric $g$ and the quadratic partition of unity $\sum_{j=1}^{J}\theta_j^2(q)\equiv 1$ such that
$$
C_{g,\theta}^{-1}\|u\|_{\tilde{\mathcal{W}}^{0,1}}^2\leq \|u\|^2_{L^2}+\|\nabla_{\mathcal{Y}}^{\mathcal{E}_{\pm}}u\|_{L^2}^2\leq C_{g,\theta}\|u\|_{\tilde{\mathcal{W}}^{0,1}}^2
$$
holds for all $u\in \mathcal{S}(X;\mathcal{E}_{\pm})\cap \ker \alpha_{\pm}$\,.
\end{lemma}

\begin{proof}
 We first notice that for any connection $\nabla$ and all $u\in \mathcal{S}(X;\mathcal{E}_{\pm})\cap \ker \alpha_{\pm}$\,,
 $$
\nabla_{\mathcal{Y}}u=
e^{-\frac{|p|_q^2}{4}}
\nabla_{\mathcal{Y}}e^{\frac{|p|_q^2}{4}}\pi_{0,\pm}u\,.
$$
Because $\nabla^{\mathcal{E}_{\pm}}_{\mathcal{Y}}$ is a first order differential operator  we have
$$
\|\nabla_{\mathcal{Y}}^{\mathcal{E}_{\pm}}u\|_{L^2}^2=\sum_{j=1}^{J}\|\nabla_{\mathcal{Y}}^{\mathcal{E}_{\pm}}(\theta_j(q)u)\|_{L^2}^2
-\sum_{j=1}^J \|(\mathcal{Y}\theta_j)u\|_{L^2}^2\,.
$$
With the local formula 
$$
\mathcal{Y}\theta_j=g^{ik}(q)p_k (\partial{q^i}\theta_j)(q)=e^{-\frac{|p|_q^2}{4}}g^{ik}p_k (\partial{q^i}\theta_j)(q)e^{\frac{|p|_q^2}{4}}
$$
and $\|e^{\frac{|p|_q^2}{4}}\pi_0\|_{\mathcal{L}(L^2)}\leq C_g$ we obtain
$$
C_{g,\theta,1}^{-1}\sum_{j=1}^{J}\|\theta_j(q) u\|^2_{L^2}+\|\nabla_{\mathcal{Y}}^{\mathcal{E}_{\pm}}(\theta_j(q)u)\|_{L^2}^2
\leq 
\|u\|^2_{L^2}+\|\nabla_{\mathcal{Y}}^{\mathcal{E}_{\pm}}u\|_{L^2}^2
\leq \sum_{j=1}^{J}\|\theta_j(q) u\|^2_{L^2}+\|\nabla_{\mathcal{Y}}^{\mathcal{E}_{\pm}}(\theta_j(q)u)\|_{L^2}^2\,.
$$

Above the neighborhood $\mathcal{V}_{\theta,j}\supset \mathrm{supp}\,\theta_j$, we use the connection $\nabla^j$ which is trivial in the local orthonormal frame 
$(f_j^1\ldots f_j^d,\hat{f}_{j,1},\ldots,\hat{f}_{j,d})$\,.\\The relation
$$
\nabla_{\mathcal{Y}}^{\mathcal{E}_{\pm}}-\nabla_{\mathcal{Y}}^{j}=e^{-\frac{|p|_q^2}{4}}g^{ik}(q)p_k (\nabla^{\mathcal{E}_{\pm}}_{e_i}-\nabla^{j}_{e_i})e^{\frac{|p|_q^2}{4}}\,
$$
allows the same comparison which leads to 
$$
C_j^{-1}\left(\|\theta_j u\|^2_{L^2}+\|\nabla_{\mathcal{Y}}^{\mathcal{E}_{\pm}}(\theta_j u)\|_{L^2}^2\right)
\leq 
\|\theta_j u\|^2_{L^2}+\|\nabla_{\mathcal{Y}}^{j}(\theta_j u)\|_{L^2}^2\leq
C_j\left(\|\theta_j u\|^2_{L^2}+\|\nabla_{\mathcal{Y}}^{\mathcal{E}_{\pm}}(\theta_j u)\|_{L^2}^2\right)\,.
$$
With the trivial connection $\nabla^j$ in the orthonormal frame $(f_j^1,\ldots,f_j^d, \hat{f}_{j,1},\ldots,\hat{f}_{j,d})$ the estimate of the middle term is reduced to the computation of $\|\theta_j u\|^2_{L^2}+\|\mathcal{Y}(\theta_j u)\|_{L^2}^2$ for $u=\pi^{-d/4}e^{-\frac{|p|_q^2}{2}} a(q)$ for $a\in \mathcal{C}^{\infty}_0(\mathcal{V}_{\theta,j};\mathbb{C})$. We compute
$$
\|\mathcal{Y}(\pi^{-d/4}e^{-\frac{|p|_q^2}{2}} \theta_j a)\|_{L^2}^2=\int_{Q}\left[\int_{\mathbb{R}^d}|g^{ik}(q)p_k (\partial_{q^i} (\theta_j a)|^2 \pi^{-d/2}e^{-g^{ik}(q)p_ip_k}~dp\right]~dq
=\frac{1}{2}\int_Q |\nabla^g_q (\theta_j a)|^2 d\mathrm{Vol}_g(q).
$$
Similarly the definition $W^2_\theta=\mathcal{I}_{\theta,X}^*\left[W^2\otimes\mathrm{Id}_{\mathbb{C}^{J\times 2^{2d}}}\right]\mathcal{I}_{\theta,X}$ reduces the problem to the computation of 
$$
\langle \theta_j u\,,\, W^2 \theta_j u\rangle=\langle \theta_j u\,,\, [C-\Delta_H+C\mathcal{O}^2] \theta_j u\rangle
$$
with $u=\pi^{-d/4}e^{-\frac{|p|_q^2}{2}} a(q)$\,, $a\in \mathcal{C}^{\infty}_0(\mathcal{V}_{\theta,j};\mathbb{C})$\,. We obtain
like in \eqref{eq:LaplBelscUtheta}
$$
\langle \theta_j u\,,\, W^2 \theta_j u\rangle =(C+C\frac{d^2}{4})\|\theta_j a\|_{L^2}^2 + \frac{1}{2}\int_Q |\nabla^g_q (\theta_j a)|^2 d\mathrm{Vol}_g(q)
$$
and this ends the proof.
\end{proof}

\subsection{Scalings}
\label{sec:scalings}
When we consider semiclassical Witten Laplacians,  it is natural to introduce semiclassical Sobolev spaces. Accordingly the space $\tilde{\mathcal{W}}^{0,s}(X;\mathcal{E}_{\pm})$ has to be defined with an $h$-dependent norm. There are various transformations on the operators, Witten's and Bismut's Laplacian, and on the functional spaces which allow to reduce the $h$-dependent problem, $h\in]0,1]$\,, to the case $h=1$\,. This simplifies the asymptotic analysis with respect to the pair of parameters $b>0,h>0$\,. In particular,   the initial subelliptic estimates  of \cite{NSW}, where only the parameter $b>0$ was considered, can be easily translated into a $b,h$-dependent version.\\
\textbf{Semiclassical Witten Laplacian:} The semiclassical Witten Laplacians $\Delta_{V,h}$ on the riemannian manifold $(Q,g)$ can be given several equivalent presentations. It is better to think in terms of the four data $(Q,g,V,h)$ where $(Q,g)$ is the riemannian manifold $V\in \mathcal{C}^{\infty}(Q;\mathbb{R})$ is the potential function and $h\in ]0,1]$ is the semiclassical parameter.\\
The semiclassical Witten Laplacian equals
\begin{eqnarray*}
  &&\Delta_{V,h}=\Delta_{(Q,g,V,h)}= (d_{V,h}+d_{V,h}^{*,g})^2\\
  \text{where}&& d_{V,h}=e^{-\frac{V}{h}}(hd)e^{\frac{V}{h}}=hd+dV\wedge\quad
                 d_{V,h}^{*,g}=e^{-\frac{V}{h}}(hd^{*,g})e^{\frac{V}{h}}=hd^{*,g}+\mathbf{i}_{\nabla_{g}V}\,,
\end{eqnarray*}
where the subscripts recall that the Hodge star operator, the codifferential and the gradient all depend on the chosen metric $g$\,.\\
Relations between the following differential operators acting on $\mathcal{C}^{\infty}(Q;\Lambda T^{*}Q\otimes F_{\pm})$ can be written:
\begin{eqnarray*}
  &&d^{*,\frac{1}{h^{2}}g}=h^{2}d^{*,g}\quad,\quad
     \nabla_{\frac{g}{h^{2}}}V=h^{2}\nabla_{g}V\\
  && d_{V,h}=hd_{\frac{V}{h},1}
     \quad,\quad d_{V,h}^{*,g}=\frac{1}{h}d_{\frac{V}{h},1}^{*,\frac{g}{h^{2}}}\quad,\quad \Delta_{(Q,g,V,h)}=\Delta_{(Q,\frac{g}{h^{2}},\frac{V}{h},1)}\,.
\end{eqnarray*}
For the $L^{2}$-spaces we note that
$$
d\mathrm{Vol}_{\frac{1}{h^{2}}g}=h^{-d}d\mathrm{Vol}_{g}\quad
    \int_{Q} \langle s,s' \rangle_{\frac{1}{h^{2}}g}~d\mathrm{Vol}_{\frac{1}{h^{2}}g}=\int_{Q}(h^{2})^{\mathrm{deg}s-d/2}\langle s\,,s'\rangle_{g}~d\mathrm{Vol}_{g}
    $$
    and the map $s\mapsto h^{d/2-\mathrm{deg}\,s}s$ is a unitary map from $L^{2}_{_{g}}(Q;\Lambda T^{*}Q\otimes F_{\pm})$ onto $L^{2}_{\frac{1}{h^{2}}g}(Q;\Lambda T^{*}Q\otimes F_{\pm})$\,. Semiclassical Sobolev spaces are defined by replacing derivatives of vector fields with $g$-norms bounded by $1$ \,, by vector fields with $\frac{1}{h^{2}}g$-norms bounded by  $1$ or $g$-norms of size $\mathcal{O}(h)$\,. By using a Laplace type operator $\Delta_{(Q,g,0,1)}$ or $H_{0,g}=\sum_{j=1}^{J}\theta_{j}(q)\Delta_{sc,g}\theta_{j}(q)$ the semiclassical Sobolev norms are given by
$$
\|u\|_{H^{s,h}_{g}(Q)}=\|(1+h^{2}H_{0,g})^{s/2}u\|_{L^{2}_{g}(Q)}=\|h^{\frac{d}{2}-\mathrm{deg}}(1+H_{0,\frac{1}{h^{2}}g})^{s/2}u\|_{L^{2}_{\frac{g}{h^{2}}}(Q)}\,.
$$
Another introduction of the scaling relies on the fact that $(Q,g)$ can be isometrically embedded in the euclidean space $(\mathbb{R}^{d_{Q}};g_{\mathbb{R}^{d_{Q}}})$\,, according to Nash embedding theorem (see e.g. \cite{Gro}).  This isometric embedding can be done such that $d_{g^{\mathbb{R^{d_{Q}}}}}(0,Q)=1$
and  one may consider the homothetic transformation of $Q$ with center $0$ and ratio $\frac{1}{h}$\,, $Q^{h}=\frac{1}{h}Q$ or $Q^{h}=\phi_{h}Q$ with $\phi_{h}(q)=\frac{1}{h}q$ for $q\in \mathbb{R}^{d_{Q}}$\,. The tangent and conormal vector bundle $TQ$\,, $N^{*}Q=\left\{v \in T^{*}\mathbb{R}^{d_{Q}}\big|_{Q}\,,\quad \forall t\in TQ\,, v.t=0\right\}$ are well defined and the euclidean metric $g^{\mathbb{R}^{d_{Q}}}$ allows to identify $$T^{*}Q=\left\{v\in T^{*}\mathbb{R}^{d_{Q}}\big|_{Q}\,,\quad \forall w\in N^{*}Q\,, \quad (g^{\mathbb{R}^{d_{Q}}})^{-1}(v,w)=0\right\}$$\,.
The same can be done with $Q^{h}$ which is endowed with the metric $g^{h}=g^{\mathbb{R}^{d_{Q}}}\big|_{TQ^{h}\times TQ^{h}}$\,. Then the riemannian manifold $(Q^{h},g^{h})$ is isometric to $(Q,\frac{1}{h^{2}}g)$ and when $H_{0,g^{h}}=\sum_{j=1}^{J}\theta_{j}(h.)\Delta_{sc, g^{h}}\theta_{j}(h.)$ and $V^{h}(q)=\frac{1}{h}V(hq)$ we obtain
\begin{eqnarray*}
  && \phi_{h}^{*}\Delta_{(Q^{h},g^{h},V^{h},1)}\phi_{h,*}=\Delta_{(Q,\frac{1}{h^{2}}g,\frac{1}{h}V,1)}=\Delta_{(Q,g,V,h)}\\
  && \|u\|_{H^{s,h}_{g}(Q)}=\|h^{\frac{d}{2}-\mathrm{deg}}(1+H_{0,\frac{1}{h^{2}}g})^{s/2}u\|_{L^{2}_{\frac{1}{h^{2}}g}(Q)}=\|h^{\frac{d}{2}-\mathrm{deg}}(1+H_{0,g^{h}})^{s/2}\phi_{h,*}u\|_{L^{2}_{g^{h}}(Q^{h})}\,.
\end{eqnarray*}
If instead of the quadratic partition of unity $\sum_{j=1}^{J}\theta_{j}^{2}(q)\equiv 1$ on $Q$ one takes an $h$-dependent partition of unity $\sum_{j=1}^{J_{h}}\theta_{j,h}^{2}(q)\equiv 1$ subordinate to an atlas $\cup_{j=1}^{J_{h}}\Omega_{j,h}=Q$ with $\mathrm{diam}_{g}(\Omega_{j,h})\leq Ch$ or ($\mathrm{diam}_{\frac{1}{h^{2}}g}(\Omega_{j,h})\leq C$) and an intersection number uniformly bounded with respect to $h$\,, one sees that $g^{h}$ in a coordinates system in $\phi_{h}\Omega_{j,h}$ satisfies
$$
\|\partial_{q}^{\alpha}g^{h}\|_{L^{\infty}(\phi_{h}\Omega_{j,h})}+
\|\partial_{q}^{\alpha}(g^{h})^{-1}\|_{L^{\infty}(\phi_{h}\Omega_{j,h})}\leq C_{\alpha}\\,.
$$
Although the volume of $Q^{h}$\,, $\mathrm{Vol}(Q^{h})=h^{-d}\mathrm{Vol}(Q)$ increases as $h\to 0$\,, the above quantity $\|(1+H_{0,g^{h}})^{s}v\|_{L^{2}_{g^{h}}(Q^{h})}$ correspond to the standard Sobolev space norm on $Q^{h}$ with a uniform control of the local variations of the metric while $\nabla_{g^{h}}V^{h}$  is uniformly bounded as well as its covariant derivatives with respect to vector fields with a bounded $g^{h}$-norms. We will use the short notation $H^s(Q^h;\Lambda T^{*}Q^{h}\otimes F_{\pm})$ for $H^{s,1}_{g^h}(Q^h;\Lambda T^{*}Q^{h}\otimes F_{\pm})$\,.
\\
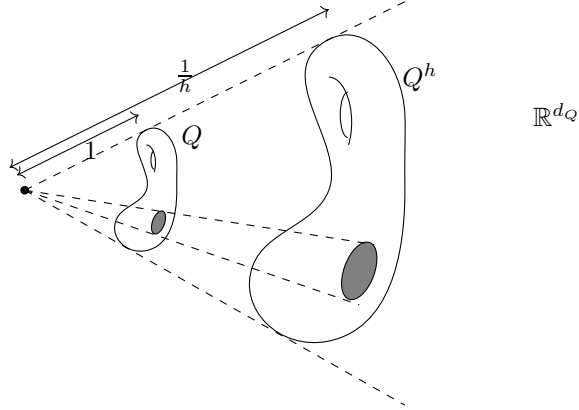
\begin{figure}[h]
  \centering
\begin{tikzpicture}
  \draw (5+0,0) to[closed, curve through =
  { (5-1,-2) (5-2,-1) (5-1,0) (5-1,2) (5+0,1) }] (5+0,0.5);
  \draw (4,1.5) arc[start angle = 90, end angle =-30, x radius=0.3,y radius=0.6];
  \draw (4.25,1.3) arc[start angle = 90, end angle =270), x radius=0.1,y radius=0.3];
  \draw[rotate=-20,fill=gray] (4.5,0.5)  ellipse[x radius=0.2, y radius=0.4];
    \draw (5.2,1.5) node{$Q^{h}$};
  \draw (2.2,0.7) node{$Q$};
   \draw (1.6,0.6) arc[start angle = 90, end angle =-30, x radius=0.12,y radius=0.24];
   \draw (1.7,0.52) arc[start angle = 90, end angle =270), x radius=0.04,y radius=0.12];
  \draw(2+0,0) to[closed, curve through =
  { (2-0.4,-0.8) (2-0.8,-0.4) (2-0.4,0) (2-0.4,0.8) (2+0,0.4) }] (2+0.,0.2);
  \draw[rotate=-20,fill=gray] (1.8,0.2) ellipse[x radius=0.08, y radius=0.16];
  \draw[dashed] (0,0)--(5,2.5);
  \draw[dashed] (0,0)--(5,-2.85);
  \draw[dashed] (0,0)--(4.6,-0.72);
  \draw[dashed] (0,0)--(4.4,-1.52);
  \draw[fill=black] (0,0) circle[radius=0.05];
  \draw[<->] (-0.1,0.2)--(1.5,1);
  \draw[<->] (-0.2,0.3)--(4.,2.4);
  \draw[rotate=20] (1.,0.2) node{$1$};
  \draw[rotate=20] (2.5,0.7) node{$\frac{1}{h}$};
  \draw (7,1) node{$\mathbb{R}^{d_{Q}}$};
\end{tikzpicture}\\
\caption{ The grey areas represent on $Q$ a ball of radius $1$ (resp. $h$) for the metric $\frac{1}{h^{2}}g$ (resp. $g$)
and on $Q^{h}$ the isometric ball of radius $1$ for the metric $g^{h}$\,.}
\end{figure}\\

\noindent\textbf{Bismut hypoelliptic Laplacian:} We do the same kind of scalings as above for the Bismut hypoelliptic Laplacian. Actually we will start from the expression  \eqref{eq:WeitzbL2}\eqref{eq:alpha}\eqref{eq:betaprelim}\eqref{eq:gamma} of the operator $B_{\pm,b,\frac{V}{h}}$ which is actually determined by the data $(Q,g,\frac{V}{h},b)$ where $(Q,g)$ is the base riemannian manifold, $V\in \mathcal{C}^{\infty}(Q;\mathbb{R})$ is the potential function and $b,h>0$ are the two parameters:
$$
B_{\pm,b,\frac{V}{h}}=B_{\pm,(Q,g,\frac{V}{h},b)}=\frac{1}{b^{2}}\alpha_{\pm,(Q,g)}+\frac{1}{b}\beta_{\pm,(Q,g,\frac{V}{h})}+\gamma_{\pm,(Q,g,\frac{V}{h})}\,.
$$
By mimicking what we observed for the Witten Laplacian, we firstly want to establish a simple relation between $B_{\pm,(Q,g,\frac{V}{h},*)}$ and $B_{\pm,(Q,\frac{1}{h^{2}}g,\frac{V}{h},*)}$\,. We notice
$$
(\frac{1}{h^{2}}g)\oplus (\frac{1}{h^{2}}g)^{-1}=(\frac{1}{h^{2}}g)\oplus (h^{2}g^{-1})
$$
while the Christoffel symbols $\Gamma_{ij}^{k}(q)$ are the same for the metric $g$ and the rescaled metric $\frac{1}{h^{2}}g$ and the Levi-Civita connection is not changed.
If local orthonormal frames given by \eqref{eq:fihfi1} and \eqref{eq:fihfi2}
are denoted by $(f_{i,g})_{1\leq i\leq d}$\,,$(\hat{f}^{i}_{g})_{1\leq i\leq d}$ $(f^{i}_{g})_{1\leq i\leq d}$ and $(\hat{f}_{i,g})_{1\leq i\leq d}$ for the metric $g$\,, orthonormal frames for the metric $\frac{1}{h^{2}}g$ are given by
\begin{eqnarray*}
  && f_{i,\frac{1}{h^{2}}g}=hf_{i,g}\quad,\quad \hat{f}^{i}_{\frac{1}{h^{2}}g}=\frac{1}{h}\hat{f}^{i}_{g}\quad\,,\\
  \text{and}&&
    f^{i}_{\frac{1}{h^{2}}g}=\frac{1}{h}f^{i}_{g}\quad,\quad \hat{f}_{i,\frac{1}{h^{2}}g}=h\hat{f}_{i,g}\,.
\end{eqnarray*}
Other simple relations for the kinetic energy and the hamiltonian vector field $\mathcal{Y}$ are:
$$
\frac{|p|_{q,g}^{2}}{2}=\frac{g^{ij}(q)p_{i}p_{j}}{2}=\frac{1}{2h^{2}}(\frac{1}{h^2}g)^{ij}p_{i}p_{j}=\frac{|p|^{2}_{q,\frac{1}{h^{2}}g}}{2h^{2}}\quad
\text{and}\quad \mathcal{Y}_{g}=\frac{1}{h^{2}}\mathcal{Y}_{\frac{1}{h^{2}}g}
$$
while the symplectic form $\sigma=dp\wedge dq$ on $T^{*}Q$ is not changed.\\
Although $B_{\pm,(Q,g,\frac{V}{h},b)}$ preserves the total degree $|I|+|J|$\,, it mixes the horizontal degree $|I|$ and vertical degree $|J|$ for sections $s_{I}^{J}(x)f^{I}\hat{f}_{J}$\,. The different homogeneities in the conformal change of metric from $g$ to $\frac{1}{h^{2}}g$ must be considered carefully as well.\\
Because $X=T^{*}Q$ is a vector bundle on $Q$\,, while $\mathcal{E}_{\pm}=\pi^{*}(\Lambda T^{*}Q\otimes \Lambda TQ\otimes F_{\pm})$\,, we define the mapping $\Psi_{h}:\mathcal{S}(X;\mathcal{E}_{\pm})\to \mathcal{S}(X;\mathcal{E}_{\pm})$ by
$$
\Psi_{h}(s_{I}(q,p)^{J}f^{I}_{\frac{1}{h^{2}}g}\hat{f}_{J,\frac{1}{h^2}g})= h^{-\frac{d}{2}+|I|-|J|}s_{I}^{J}(q,h^{-1} p)f^{I}_{\frac{1}{h^{2}}g}\hat{f}_{J,\frac{1}{h^{2}}g}=h^{-\frac{d}{2}}s_{I}^{J}(q,h^{-1} p)f^{I}_{g}\hat{f}_{J,g}\,.
$$
We obtain
\begin{align*}
   \Psi_{h}^{-1}(\alpha_{\pm,(Q,g)})\Psi_{h}&=\frac{1}{2}\left(-h^{-2}\Delta_{g}^{V}+h^{2}|p|^{2}_{q,g}\pm (2\hat{f}_{i,g}\mathbf{i}_{\hat{f}^{i}_{g}}-d)\right)
    =
                                              \alpha_{\pm, (Q,\frac{1}{h^{2}}g)}\\
\text{and}\quad  \Psi_{h}^{-1}  (\beta_{\pm,(Q,g,\frac{V}{h})})\Psi_{h}&=-(\pm h\nabla^{\mathcal{E}_{\pm}}_{\mathcal{Y}_{g}}-\frac{1}{h^{2}}(f_{i,g}V)\nabla^{\mathcal{E}_{\pm}}_{\hat{f}^{i}_{g}})=
  -\frac{1}{h}(\pm \nabla^{\mathcal{E}_{\pm}}_{\mathcal{Y}_{\frac{1}{h^{2}}g}}-(f_{i,\frac{1}{h^{2}}g}\frac{V}{h})\nabla^{\mathcal{E}_{\pm}}_{\hat{f}^{i}_{\frac{1}{h^{2}}g}})=\frac{1}{h}\beta_{\pm,(Q,\frac{1}{h^{2}}g,\frac{V}{h})} \,. 
\end{align*}
For $\Psi_{h}^{-1}\gamma_{\pm,Q,g}\Psi_{h}$ firstly notice that the Riemann curvature tensors are compared according to
$$
R^{TQ}_{g}=h^{2}R^{TQ}_{\frac{1}{h^{2}}g}
$$
while the coefficient $\tilde{\Gamma}_{ij}^{k}(q)=f^{k}(\nabla^{TQ}_{f_{i}}f_{j})$ satisfies
$$
\tilde{\Gamma}_{ij,g}^{k}(q)=\frac{1}{h}\tilde{\Gamma}^{k}_{ij,\frac{1}{h^{2}}g}\,.
$$
The definition of the mapping $\Psi_{h}$ ensures the identity of the tensorial operations
$$
\Psi_{h}^{-1}
\begin{pmatrix}
  f^{i}_{g}\wedge\\
  \hat{f}_{i,g}\wedge
\end{pmatrix}
\Psi_{h}=\begin{pmatrix}
  f^{i}_{\frac{1}{h^{2}}g}\wedge\\
  \hat{f}_{i,\frac{1}{h^{2}}g}\wedge
         \end{pmatrix}\quad\text{and}\quad
\Psi_{h}^{-1}
\begin{pmatrix}
  \mathbf{i}_{f_{i,g}}\\
  \mathbf{i}_{\hat{f}^{i}_{g}}
\end{pmatrix}
\Psi_{h}=
\begin{pmatrix}
  \mathbf{i}_{f_{i,\frac{1}{h^{2}}g}}\\
  \mathbf{i}_{\hat{f}^{i}_{\frac{1}{h^{2}}g}}
\end{pmatrix}\,.
$$
We deduce
$$
\Psi_{h}^{-1}\gamma_{\pm,(Q,g,\frac{V}{h})}\Psi_{h}=\frac{1}{h^{2}}\gamma_{\pm,(Q,\frac{1}{h^{2}}g,\frac{V}{h})}\,.
$$
We have proved
$$
\Psi_{h}^{-1}B_{\pm,(Q,g,\frac{V}{h},b)}\Psi_{h}=\frac{1}{h^{2}}\left[\frac{h^{2}}{b^{2}}\alpha_{\pm,(Q,\frac{1}{h^{2}}g)}+\frac{h}{b}\beta_{\pm,(Q,\frac{1}{h^{2}g},\frac{V}{h})}
+\gamma_{\pm,(Q,\frac{1}{h^{2}}g,\frac{V}{h})}\right]=\frac{1}{h^{2}}B_{\pm,(Q,\frac{1}{h^{2}}g,\frac{V}{h},\frac{b}{h})}
$$
Let us consider now what happens on the functional spaces.\\
The linear map $\Psi_{h}$ is actually a unitary transform from $L^{2}_{\frac{1}{h^{2}}g}(X,dqdp;\mathcal{E}_{\pm})$ to  $L^{2}_{g}(X,dqdp;\mathcal{E}_{\pm})$\,.\\
The $h$-dependent norms for the $\tilde{W}^{s_{1},s_{2}}(X;\mathcal{E}_{\pm})_{g}$ were not studied in \cite{NSW} but we follow the dilatation trick presented for the Witten Laplacian in order to reduce the problem to uniform estimates in  the case $h=1$\,.
\begin{definition}
  \label{def:Ws1s2h}
  On the cotangent space $X=T^{*}Q$ where $Q$ is endowed with the riemannian metric $g$\,, the $h$-dependent norm, $h\in ]0,1]$\,, of $\tilde{\mathcal{W}}^{s_{1},s_{2}}(X;\mathcal{E}_{\pm})$ is given by
  $$
\|u\|_{\tilde{\mathcal{W}}^{s_{1},s_{2}}_{h}}=\|\mathcal{O}^{s_{1}/2}(W_{\theta,h}^{2})^{s_{2}/2}u\|_{L^{2}(X;\mathcal{E}_{\pm})}
$$
where
$$
W_{\theta,h}^{2}=\sum_{j=1}^{J}\theta_{j}(q)(C_{g}-h^{2}\Delta_{H}+C_{g}\mathcal{O}^{2})\theta_{j}(q)\,.
$$
\end{definition}
Remember that the operator $W_{\theta,h}^{2}$ is an elliptic self-adjoint operator for any fixed $(Q,g,h)$ with $h\in]0,1]$ when $C_{g}\geq 1$ is chosen large enough. However the uniformity of the subelliptic estimates for $B_{\pm,(Q,g,\frac{V}{h},b)}$ with these $h$-dependent norms requires some explanation.\\
We keep track of the change of riemannian metrics with subscripts like before and write $\Delta_{H}=\Delta_{H,g}$ and $W^{2}_{\theta,h}=W^{2}_{\theta,h,g}$\,.
Actually the definition of $\Delta_{H,\frac{1}{h^{2}}g}$ gives $\Delta_{H,\frac{1}{h^{2}}g}=h^{2}\Delta_{H,g}$ and
$$
\Psi_{h}^{-1}W^{2}_{\theta,h,g}\Psi_{h}=(C_{g}-\Delta_{H,\frac{1}{h^{2}}g}+C_{g}(\mathcal{O}_{\frac{1}{h^{2}}g})^{2})=W^{2}_{\theta,1,\frac{1}{h^{2}}g}\,.
$$
After the riemannian embedding $Q\to \mathbb{R}^{d_{Q}}$ and the identfication of $X=T^{*}Q$ as a subbundle of $(T^{*}\mathbb{R}^{d_{q}})\big|_{Q}$ the dilatation $\Phi_{h}:q\mapsto \frac{q}{h}$ in $\mathbb{R}^{d_{Q}}$ with $\Phi_{h,*}:X=T^{*}Q\to X^h=T^{*}Q^h$ is a symplectic map, and an isometry from $(X,(\frac{1}{h^{2}}g)\oplus^{\perp}(h^{2}g^{-1}))$ to $(X^{h}, g^{h}\oplus^{\perp}(g^{h})^{-1})$\,.
We obtain
$$
\Phi_{h,*}\Psi_{h}^{-1}(W^{2}_{\theta,h,g})\Psi_{h}\Phi_{h}^{*}=W^{2}_{\theta,1,g^{h}}\,,
$$
and
\begin{align*}
 \|u\|_{\tilde{\mathcal{W}}^{s_{1},s_{2}}_{h,g}(X;\mathcal{E}_{\pm})}&=\|\mathcal{O}_{g}^{s_{1}/2}(W^{2}_{\theta,h,g})^{s_{2}/2}u\|_{L^{2}_{g}(X;\mathcal{E}_{\pm})}=
                                                                       \|\mathcal{O}_{g^{h}}^{s_{1}/2}(W^{2}_{\theta,1,g^{h}})^{s_{2}/2}\Phi_{h,*}\Psi_{h}^{-1}u\|_{L^{2}_{g^{h}}(X_{h};\Phi_{h,*}\mathcal{E}_{\pm})}\\
  &=
\|\Phi_{h,*}\Psi_{h}^{-1}u\|_{\tilde{\mathcal{W}}^{s_{1},s_{2}}_{1,g^{h}}(X_{h};\Phi_{h,*}\mathcal{E}_{\pm})}\,.
\end{align*}
The above discussion can be summarized by the following statement.
\begin{proposition}
  \label{pr:scaling}
  With the above notation $\Phi_{h,*}\Psi_{h}^{-1}$  is a unitary map from $\tilde{\cal{W}}^{s_{1},s_{2}}_{h,g}(X,\mathcal{E}_{\pm})_{g}$ to \\$\tilde{\mathcal{W}}^{s_{1},s_{2}}_{1,g^{h}}(X^{h};\Phi_{h,*}\mathcal{E}_{\pm})$ for all $s_{1},s_{2}\in \mathbb{R}$ with
  \begin{eqnarray*}
    && \Phi_{h,*}\Psi_{h}^{-1} B_{\pm,(Q,g,\frac{V}{h},b)}\Psi_{h}\Phi_{h}^{*}=\frac{1}{h^{2}}B_{\pm,(Q^{h},g^{h},V^{h},\frac{b}{h})}\\
    &&
       \Phi_{h,*}\Psi_{h}^{-1} \alpha_{\pm,(Q,g)}\Psi_{h}\Phi_{h}^{*}=\alpha_{\pm,(Q^{h},g^{h})}\\
    && \Phi_{h,*}\Psi_{h}^{-1} \beta_{\pm,(Q,g,\frac{V}{h})}\Psi_{h}\Phi_{h}^{*}=\frac{1}{h}\beta_{\pm,(Q^{h},g^{h},V^{h})}\quad,\quad
       \Phi_{h,*}\Psi_{h}^{-1} \gamma_{\pm,(Q,g,\frac{V}{h})}\Psi_{h}\Phi_{h}^{*}=\frac{1}{h^{2}}\gamma_{\pm,(Q^{h},g^{h},V^{h})}\\
    \text{and}&&
                  \Phi_{h,*}\Psi_{h}^{-1} W^{2}_{\theta,g,h}\Psi_{h}\Phi_{h}^{*}=W^{2}_{\theta(h.),g^{h},1}\,,
  \end{eqnarray*}
  with $V^{h}(q)=\frac{1}{h}V(hq)$\,, $\theta_{j}(h.)(q)=\theta_{j}(hq)$ for $q\in Q^{h}$\,. Additionally $\nabla_{g^{h}}V^{h}$ and $g^{h}$ and $(g^{h})^{-1}$\,, expressed in the coordinates associated with the atlas $Q^{h}=\mathop{\cup}_{j=1}^{J}\frac{1}{h}\Omega_{j}$ (or $\frac{1}{h}$ times the coordinates associated with  $Q=\mathop{\cup}_{j=1}^{J}\Omega_{j}$) have uniformly bounded derivatives.
\end{proposition}
This result and what we recalled just above for  the semiclassical Witten Laplacian, allow to eliminate the parameter $h\in ]0,1]$ in the analysis. Actually it suffices to make the analysis for $(Q^{h},g^{h},V^{h})$ and $(Q^{h},g^{h},V^{h},\frac{b}{h})$ where the parameter $b$ is replaced by $\frac{b}{h}$ and to use the uniform control of all the norm estimates with respect to $h\in ]0,1]$ on the dilated manifolds $Q^{h}$ and $X^{h}=T^{*}Q^{h}$\,. 
\subsection{The Hypoelliptic Laplacian as a perturbed Geometric Kramers-Fokker-Planck operator}
\label{sec:hyplapSubEll}

Although we are ultimately interested in Bismut's hypoelliptic Laplacian $B_{\pm,b,\frac{V}{h}}=B_{\pm,(Q,g,\frac{V}{h},b)}$\,, Proposition~\ref{pr:scaling} with
$$
\Phi_{h,*}\Psi_{h}^{-1} B_{\pm,(Q,g,\frac{V}{h},b)}\Psi_{h}\Phi_{h}^{*}=\frac{1}{h^{2}}B_{\pm,(Q^{h},g^{h},V^{h},\frac{b}{h})}
$$
allows to reduce the analysis to $B_{\pm,\frac{b}{h},V^{h}}=B_{\pm,(Q^{h},g^{h},V^{h},\frac{b}{h})}$ with uniform controls of the local derivatives (in coordinate charts) of $\nabla_{g^{h}}V^{h}$\,, $g^{h}$ and $(g^{h})^{-1}$\,.
For the sake of simplicity we replace $\frac{b}{h}$ by $b>0$ and  we write
$B_{\pm,b,V^{h}}=B_{\pm,(Q^{h},g^{h},V^{h},b)}$\,.\\
We will use the short notations $\mathcal{E}_{\pm}^{h}=\Phi_{h,*}\mathcal{E}_{\pm}$ for the vector bundle above $X^{h}=T^{*}Q^{h}$ and the connection $\nabla^{\mathcal{E}_{\pm},h}$ will be the connection on $\mathcal{E}_{\pm}^{h}$ associated with the metric $g^{h}$ on $Q^{h}$\,.\\

With these modifications,  Bismut's hypoelliptic Laplacian can be written as
\begin{equation} \label{eq:formOfTheOperator}
B_{\pm,b,V^{h}} = P_{\pm,b} + R_{0,h} + R_{2,h} + \frac{1}{b}R_{1,\bot,h},
\end{equation} 
where the principal part is
\begin{equation}
\label{eq:defPpmb}
P_{\pm,b} = \frac{1}{b^2}\alpha_{\pm,g^{h}} \mp \frac{1}{b}\nabla_{\mathcal{Y}_{g^{h}}}^{{\cal E}_{\pm},h}
\end{equation}
and the three lower order corrections are
\begin{eqnarray*}
R_{0,h}&=&-\frac{1}{4}\langle R_{g^{h}}^{TQ}(f_{i,g^{h}},f_{j,g^{h}})f_{k,g^{h}},f_{\ell,g^{h}} \rangle (f^i_{g^{h}}-\hat{f}_{i,g^{h}})(f^j_{g^{h}}-\hat{f}_{j,g^{h}})
           \mathbf{i}_{f_{k,g^{h}} + \hat{f}^k_{g^{h}}}\mathbf{i}_{f_{\ell,g^{h}} + \hat{f}^{\ell}_{g^{h}} } + ( f_{i,g^{h}}(f_{j,g^{h}}V^{h}) \\
       &&\hspace{4cm}+ \tilde{\Gamma}_{ij,g^{h}}^k f_{k,g^{h}}V^{h} )(f^i_{g^{h}} - \hat{f}_{i,g^{h}}) \mathbf{i}_{f_{j,g^{h}} + \hat{f}^{j}_{g^{h}}}\\
R_{1,\bot,h}&=&R_{1}= (f_{i,g^{h}} V^{h}) \nabla_{\hat{f}^i_{g^{h}}}^{\mathcal{E}_{\pm},h}\\
  R_{2,h} &=&  \mp \langle R^{TQ}_{g^{h}}( p, f_{i,g^{h}})p,f_{j,g^{h}} \rangle (f^i_{g^{h}} - \hat{f}_{i,g^{h}})\mathbf{i}_{f_{j,g^{h}} + \hat{f}^j_{g^{h}}}
\end{eqnarray*}
where we have neglected the $\pm$ sign in the notations $R_{0,h},R_{1,\bot,h}, R_{2,h}$\,.\\
These notations are motivated by the following conditions, for a given differential operator $A$, indexed by $i=0,1,2$~:
\begin{eqnarray}
  \label{eq:hypothesisForPerturbationRi}
&\hspace{-1cm}\forall s\in \mathbb{R}, \exists C_{s,i}>0\,, \|A\|_{\mathcal{L}(\tilde{\mathcal{W}}^{i,0};L^2)} +\|A\|_{\mathcal{L}(L^2;\tilde{\mathcal{W}}^{-i,0})} + \|(W^2_{\theta})^{s/2} A (W^2_{\theta})^{-s/2} - A \|_{\mathcal{L}(L^2 ; L^2)} \leq C_{s,i}\,,&\\
  \label{eq:extraHypothesisForPerturbationRiPerp}
 & \pi_{0,\pm} A \pi_{0,\pm} = 0 \,.&
\end{eqnarray}
The collection of operators $R_{0,h},R_{1,h}=R_{1,\bot,h},R_{2,h}$ are differential operators in the class $\mathrm{OpS}_{\Psi}^{1}(Q^{h};\mathrm{End}(\mathcal{E}_{\pm}^{h}))$ introduced in \cite{NSW} while $(W^{2}_{\theta})^{s/2}\in \mathrm{OpS}^{s}_{\Psi}(Q^{h};\mathrm{End}({\cal E}_{\pm}^{h}))$ has a scalar principal symbol. We recall that actually $W^{2}_{\theta}=W^{2}_{\theta,g^{h}}$ and that, because of the uniform bounds of Proposition~\ref{pr:scaling}, all of the local seminorms of symbols are uniformly controlled with respect to $h\in]0,1]$\,.  Therefore $R_{0,h}$, $R_{1,\bot,h}$, $R_{2,h}$ all satisfy, uniformly with respect to $h\in ]0,1]$\,,
$\|(W^2_{\theta})^{s/2} R_{i,h}(W^{2}_{\theta})^{-s/2} - R_{i,h} \|_{\mathcal{L}(L^2 ; L^2)} \leq C_{{s,i}}$\,, while the inequality  $\|R_{i,h}\|_{{\cal L}(\tilde{{\cal W}}^{i,0};L^{2})}+\|R_{i,h}\|_{{\cal L}(L^2;\tilde{{\cal W}}^{-i,0})}\leq C_{i}$ according to the index $i=0,1,2$ is straigthforward.\\
Finally the index $\bot$ in $R_{1,\bot,h}$ recalls that $R_{1,h}=R_{1,\bot,h}$ satisfies the condition
\eqref{eq:extraHypothesisForPerturbationRiPerp}.\\
The following result allows to reduce the analysis of Bismut's hypoelliptic Laplacians in any $\tilde{\mathcal{W}}^{0,s}(X^{h};\mathcal{E}_{\pm}^{h})$ space to the case $s=0$\,.
\begin{proposition}
  \label{pr:perturbconjug}
  The conditions \eqref{eq:hypothesisForPerturbationRi} and \eqref{eq:extraHypothesisForPerturbationRiPerp} are left invariant by a conjugation by $(W^{2}_{\theta})^{s'/2}$ for any $s'\in \mathbb{R}$\,, or by taking the formal adjoint for the $L^{2}$-scalar product. Namely if $A$ satisfy condition \eqref{eq:hypothesisForPerturbationRi} (resp.\eqref{eq:extraHypothesisForPerturbationRiPerp} ) then $(W^{2}_{\theta})^{s'/2} A (W^{2}_{\theta})^{-s'/2}$ and the formal $L^2$-adjoint $A'$ also satisfy \eqref{eq:hypothesisForPerturbationRi} (resp. \eqref{eq:extraHypothesisForPerturbationRiPerp}).

  The conjugation of Bismut's hypoelliptic Laplacian by $(W^{2}_{\theta})^{s'/2}$\,, $s'\in \mathbb{R}$\,, equals
  \begin{eqnarray*}
&&(W^{2}_{\theta})^{s'/2}B_{\pm,b,V^{h}}(W^{2}_{\theta})^{-s'/2}=P_{\pm,b}+ R_{0,h}^{s'} + R_{2,h}^{s'} + \frac{1}{b}R_{1,\bot,h}^{s'}
    \\
    \text{with}&& R_{0,h}^{s'}=(W^{2}_{\theta})^{s'/2}R_{0,h}(W^{2}_{\theta})^{-s'/2}\quad,\quad  R_{2,h}^{s'}=(W^{2}_{\theta})^{s'/2}R_{2,h}(W^{2}_{\theta})^{-s'/2}\\
    \text{and}&& R^{s'}_{1,\bot,h}=R_{1,h}^{s'}=(W^{2}_{\theta})^{s'/2}R_{1,h}(W^{2}_{\theta})^{-s'/2}
                 \mp \left[(W^{2}_{\theta})^{s'/2}\nabla^{\mathcal{E}_{\pm},h}_{\mathcal{Y}_{g^{h}}}(W^{2}_{\theta})^{-s'/2}-\nabla^{\mathcal{E}_{\pm},h}_{\mathcal{Y}_{g^{h}}}\right]
  \end{eqnarray*}
  where $R_{0,h}^{s'},R_{1,h}^{s'},R_{2,h}^{s'}$ satisfy the condition \eqref{eq:hypothesisForPerturbationRi}, uniformly with respect to $h\in ]0,1]$\,, for the respective values of  $i=0,1,2$ and $R_{1,\bot,h}^{s}=R_{1,h}^{s'}$ satisfies the condition \eqref{eq:extraHypothesisForPerturbationRiPerp}. Additionally $(R_{0,h}=0~\text{and}~R_{2,h}=0)\Rightarrow (R_{0,h}^{s'}=0~\text{and}~R_{2,h}^{s'}=0)$\,.\\
  Finally the formal adjoint $B_{\pm,b,V^{h}}^{\prime,s'}$ for the $\tilde{\mathcal{W}}^{0,s'}(X^{h};\mathcal{E}_{\pm}^{h})$ scalar product, according to Definition~\ref{de:adjoints} satisfies
  \begin{eqnarray}
\label{eq:formadj1}
&&
  (W^{2}_{\theta})^{s'/2}B_{\pm,b,V^{h}}(W^{2}_{\theta})^{-s'/2}=(P_{\pm,b})^{\prime,0}+ (R_{0,h}^{s'})^{\prime} + (R_{2,h}^{s'})^{\prime} + \frac{1}{b}(R_{1,\bot,h}^{s'})^{\prime}
    \\
\label{eq:formadj2}
    \text{with}&&
                  (P_{\pm,b})^{\prime}=\frac{1}{b^{2}}\alpha_{\pm,g^{h}}\pm \frac{1}{b}\nabla^{\mathcal{E}_{\pm},h}_{\mathcal{Y}_{g^{h}}}\quad:\quad L^{2}(X^{h},dqdp;\mathcal{E}_{\pm}^h)\longrightarrow \mathcal{S}'(X^{h};\mathcal{E}_{\pm}^h)\,.
  \end{eqnarray}
\end{proposition}
\begin{proof}
  The invariance of \eqref{eq:hypothesisForPerturbationRi} actually comes from the continuous imbeddings $\tilde{\mathcal{W}}^{i,0}\subset L^{2}\subset \tilde{\mathcal{W}}^{-i,0}$ for $i=0,1,2$\,.\\
The invariance of \eqref{eq:extraHypothesisForPerturbationRiPerp} is due to the commutation of $(W^{2}_{\theta})^{s'/2}$ with $\pi_{0,\pm}=1_{\left\{0\right\}}(\alpha_{\pm,g^{h}})$: Actually $W^{2}_{\theta}$ strongly commutes with $\mathcal{O}_{g^{h}}$ and preserves the vertical degree. It therefore commutes with $\alpha_{\pm,g^{h}}$ and with any functions of $\alpha_{\pm,g^{h}}$\,.\\
For the last property it suffices to check that $A_{\mathcal{Y}}=\left[(W^{2}_{\theta})^{s'/2}\nabla^{\mathcal{E}_{\pm},h}_{\mathcal{Y}_{g^{h}}}(W^{2}_{\theta})^{-s'/2}-\nabla^{\mathcal{E}_{\pm},h}_{\mathcal{Y}_{g^{h}}}\right]$ satisfies the two conditions \eqref{eq:hypothesisForPerturbationRi} for $i=1$ and \eqref{eq:extraHypothesisForPerturbationRiPerp}.\\
The estimate
\begin{equation}
\label{eq:WYWmY}
\|\left[(W^{2}_{\theta})^{s'/2}\nabla^{\mathcal{E}_{\pm},h}_{\mathcal{Y}_{g^{h}}}(W^{2}_{\theta})^{-s'/2}-\nabla^{\mathcal{E}_{\pm},h}_{\mathcal{Y}_{g^{h}}}\right]\|_{\mathcal{L}(\tilde{\mathcal{W}}^{1,0};L^{2})}\leq C_{s'}
\end{equation}
was proved in \cite{NSW}-Proposition~3.8, where the uniform constant $C_s'$  with respect to $h\in]0,1]$ is made possible by the uniform control of the derivatives of $g^h$ and $(g^h)^{-1}$ recalled in Proposition~\ref{pr:scaling}. By duality and because $\mathcal{A}_{\mathcal{Y}}^*+\mathcal{A}_{\mathcal{Y}}\in\mathcal{L}(L^2)$ we deduce as well
$$
\|\left[(W^{2}_{\theta})^{s'/2}\nabla^{\mathcal{E}_{\pm},h}_{\mathcal{Y}_{g^{h}}}(W^{2}_{\theta})^{-s'/2}-\nabla^{\mathcal{E}_{\pm},h}_{\mathcal{Y}_{g^{h}}}\right]\|_{\mathcal{L}(L^{2};\tilde{\mathcal{W}}^{-1,0})}\leq C_{s'}\,.
$$
Because $\nabla_{\mathcal{Y}_{g^{h}}}^{\mathcal{E}_{\pm},h}\in \mathrm{OpS}^{3/2}_{\Psi}(Q^h;\mathcal{E}_{\pm}^h)$\,, while $(W^{2}_{\theta})^{s/2}\in \mathrm{OpS}^{s}_{\Psi}(Q^h;\mathcal{E}_{\pm}^h)$ with a scalar principal symbol, we deduce that $A_{\mathcal{Y}}\in \mathrm{OpS}^{1/2}_{\Psi}(Q^h;\mathcal{E}_{\pm}^h)$\,, with local seminorm of symbols uniformly bounded with respect to $h\in]0,1]$\,.  Therefore $(W^{2}_{\theta})^{s}A_{\mathcal{Y}_{g^{h}}}(W^{2}_{\theta})^{-s}-A_{\mathcal{Y}_{g^{h}}}\in \mathrm{OpS}^{-1/2}_{\Psi}(Q^{h};\mathrm{End}(\mathcal{E}_{\pm}^{h}))$ is a bounded operator in $L^{2}(X^{h},dqdp;\mathcal{E}_{\pm}^{h})$ with norm uniformly bounded with respect to $h\in ]0,1]$\,.\\
The condition \eqref{eq:extraHypothesisForPerturbationRiPerp} is due to the identity $\pi_{0,\pm}\nabla^{\mathcal{E}_{\pm},h}_{\mathcal{Y}_{g^{h}}}\pi_{0,\pm}=0$ as continuous operator on $\mathcal{S}'(X^{h};\mathcal{E}_{\pm}^{h})$\,.\\
For the formal adjoint $B_{\pm,b,V^{h}}^{\prime,s'}$\,, it suffices to apply \eqref{eq:formadj1}\eqref{eq:formadj2} after noticing that $B_{\pm,b,V^{h}}$ is continuous as an operator  $\mathcal{S}(X^{h};\mathcal{E}_{\pm}^{h})\to \mathcal{S}(X^{h};\mathcal{E}_{\pm}^{h})$ and $\mathcal{S}'(X^{h};\mathcal{E}_{\pm}^{h})\to \mathcal{S}'(X^{h};\mathcal{E}_{\pm}^{h})$\,. The expression of $P_{\pm,b}'$ comes from the fact that $\alpha_{\pm,g^{h}}$ is self-adjoint and $\nabla_{\mathcal{Y}_{g^{h}}}^{\mathcal{E}_{\pm},h}$ is anti-adjoint because the connection $\nabla^{\mathcal{E}_{\pm},h}$ is unitary.
\end{proof}

\begin{proposition}
  \label{pr:main1sttext} 
There exists a constant  $C_{g}\geq 1$ determined by the metric $g$ and, for any $s\in \mathbb{R}$\,, a constant $C_{g,V,s}\geq 1$
  determined by $s\in\mathbb{R}$, the metric  $g$ and the potential function $V\in \mathcal{C}^{\infty}(Q;\mathbb{R})$ such that the following properties hold when 
  $0<b<\frac{1}{C_{g}}$ and $\kappa_{s}\geq C_{g,V,s}$\,:\\
  The operator $\frac{\kappa_{s}}{b^{2}}+B_{\pm,b,V^{h}}$\,, as an unbounded operator in  $\tilde{\mathcal{W}}^{0,s}(X^{h};\mathcal{E}_\pm^{h})$\,,  is essentially maximal accretive on $\mathcal{C}^{\infty}_{0}(X^{h};\mathcal{E}_{\pm}^{h})$ (or on $\mathcal{S}(X^{h};\mathcal{E}_\pm^{h})$).\\
  If $\overline{B_{\pm,b,V^{h}}}^{s}$ denotes its closure according to Definition~\ref{de:adjoints}, the inequalities
  \begin{equation}
    \label{eq:IPPineqTh}
\mathbb{R}\mathrm{e}\langle u\,,\, (\frac{\kappa_{s}}{b^{2}}+\overline{B_{\pm,b,V^{h}}}^{s})u\rangle_{\tilde{\mathcal{W}}^{0,s}}\geq \frac{1}{16 b^{2}}\left[\|u\|_{\tilde{\mathcal{W}}^{1,s}}^{2}+\kappa_{s}\|u\|_{\tilde{\mathcal{W}}^{0,s}}^{2}\right]\,,
\end{equation}
and
  \begin{align} \nonumber
  \left\| \left(\overline{B_{\pm,b,V^{h}}}^{s} - \frac{i\lambda}{b}\right)u\right\|_{\tilde{\mathcal{W}}^{0,s}}+\frac{2    \kappa_s}{b^2}\|u\|_{\tilde{\mathcal{W}}^{0,s}} 
   \geq &  \frac{1}{C_{g}}
    \Bigg(\left\| \frac{\mathcal{O}_{g^h}}{b^{2}}u \right\|_{\tilde{\mathcal{W}}^{0,s}} + 
    \left\|
      \frac{1}{b}\left( \nabla^{\mathcal{E}_{\pm},h}_{\mathcal{Y}_{g^{h}}} - i\lambda \right)u
    \right\|_{\tilde{\mathcal{W}}^{0,s}}
 \\
    \label{eq:principalInequality}
&+ \frac{1}{b^{4/3}}\left[||u||_{\tilde{\cal W}^{0,s+\frac{2}{3}}} + \|\left(\frac{|\lambda|}{\langle p \rangle_q} \right)^{2/3} u\|_{\tilde{\mathcal{W}}^{0,s}}+
                                                                                            \right]  +                                                                                           \left(\frac{|\lambda|^{1/2}}{b^{3/2}}\right)\|u\|_{\tilde{\mathcal{W}}^{0,s}}
    \Bigg)
\end{align}
hold for  every 
$u\in D(\overline{B_{\pm,b,V^{h}}}^{s})$ 
and every $\lambda\in \mathbb{R}$\,.\\
The formal adjoint $B_{\pm,b,V^{h}}^{\prime,s}$ and adjoint $B_{\pm,b,V^{h}}^{*,s}$ of Definition~\ref{de:adjoints} satisfy $\overline{(B_{\pm,b,V^{h}}^{\prime,s}\big|_{\mathcal{S}(X^h;\mathcal{E}_{\pm}^h)})}^{s}=B_{\pm,b,h}^{*,s}$ while the formal adjoint $B_{\pm,b,V^{h}}^{\prime,s}=(W^{2}_{\theta})^{-s}B_{\pm,b,V^{h}}^{\prime}(W^{2}_{\theta})^{s}$ satisfies \eqref{eq:formadj1}\eqref{eq:formadj2}.
\end{proposition}
\begin{remark}
  \label{re:accretBb}
  It will be checked after Proposition~\ref{pr:subellipticEstimateA} that $C_{g,V,s}+\overline{B_{\pm,b,V^{h}}}^{s}$ is maximal accretive with
  $$
\forall u\in D(\overline{B_{\pm,b,V^{h}}}^{s})\,,\quad \mathbb{R}\mathrm{e}~\langle u\,,\, B_{\pm,b,V^{h}}u\rangle_{\tilde{\mathcal{W}}^{0,s}}\geq 0\,.
$$
\end{remark}
\begin{proof}
  By Proposition~\ref{pr:perturbconjug} the problem is reduced to the case $s=0$ for the operator
$$
P_{\pm,b}+ R_{0,h}^{s} + R_{2,h}^{s} + \frac{1}{b}R_{1,\bot,h}^{s}=\frac{1}{b^{2}}\mathcal{O}_{g^{h}}\mp \nabla^{\mathcal{E}_{\pm},h}_{\mathcal{Y}_{g^{h}}} +M_{0,s}(b,h)+M_{1,s}(b,h)+R_{2,h}
$$
with
\begin{eqnarray*}
  && M_{0,s}(b,h)=\frac{\pm(2\hat{f}_{i,g^{h}} \mathbf{i}_{\hat{f}^{i}_{g^{h}}} - d) + d}{2b^2}+R_{0,h}^{s} + R_{2,h}^s-R_{2,h}\quad,\quad \|M_{0,s}(b,h)\|_{\mathcal{L}(L^{2};L^{2})}\leq \frac{\nu_{0,s}}{b^{2}}\\
  && M_{1,s}(b,h)=  \frac{1}{b}R_{1,\bot,h}^{s}\quad,\quad \|M_{1}(b,h)\|_{\mathcal{L}(\tilde{\mathcal{W}}^{1,0};L^{2})}\leq \frac{\nu_{1,s}}{bh}\leq \frac{C_{g}+8\nu_{0,s}}{16 b^{2}}(1+b^{2})\,,\\
  && R_{2,h}= \mp \langle R^{TQ}_{g^{h}}( p, f_{i,g^{h}})p,f_{j,g^{h}} \rangle (f^i_{g^{h}} - \hat{f}_{i,g^{h}})\mathbf{i}_{f_{j,g^{h}} + \hat{f}^{j,g^{h}}}\quad,\quad \|R_{2}\|_{\mathcal{L}(\tilde{\mathcal{W}}^{2,0};L^{2})}\leq \nu_{g}\,.
\end{eqnarray*}
Actually 
$$
R_{2,h}^s-R_{2,h}=(W^2_\theta)^{s/2}R_2(W^2_\theta)^{-s/2}-R_{2,h} \in \mathrm{OpS}^{1-1}_{\Psi}(Q^{h};\mathrm{End}(\mathcal{E}_{\pm}^{h}))\subset \mathcal{L}(L^2;L^2)
$$
and the above inequalities hold true for suitably well chosen $s$-dependent values of  $\nu_{0,s}>0$ and $\nu_{1,s}>0$ when $0<b\leq 1$\,, uniformly with respect to $h\in ]0,1]$\,.
The last result concerned with the equality of the minimal and maximal extension of the formal adjoint results from the essential maximal accretivity, as it is recalled after Definition~\ref{de:adjoints}.\\
When the final term $R_{2,h}$ is replaced by $0$\,, the result is actually given by Proposition~7.2 in \cite{NSW} with the following changes:
\begin{itemize}
\item the lower bound $\frac{1}{8b^{2}}\left[\|u\|_{\tilde{\mathcal{W}}^{1,0}}^{2}+\kappa_{s}\|u\|^{2}_{L^{2}}\right]$ in \eqref{eq:IPPineqTh}\;
\item the coefficient $\frac{1}{4C_g(1+b)^7}$ in the right-hand side of \eqref{eq:principalInequality}, under the sufficient condition $\kappa_{s}\geq (C_g+16\nu_{0,s})(1+b^2)$\,;
\item the term $\left(\frac{|\lambda|^{1/2}}{b^{3/2}}\right)\|u\|_{L^{2}}$ in the right-hand side of  \eqref{eq:principalInequality}
which is not written in \cite{NSW}.
\end{itemize}
For the last term of $\left(\frac{|\lambda|^{1/2}}{b^{3/2}}\right)\|u\|_{L^{2}}$ in  \eqref{eq:principalInequality}\,, it suffices to notice the interpolation inequality
\begin{align*}
\left(\frac{|\lambda|^{1/2}}{b^{3/2}}\right)\|u\|_{L^{2}}&\leq 3\Big[\frac{1}{b^{4/3}}\|\frac{|\lambda|}{\langle p\rangle_{q}}u\|_{L^{2}}+
\frac{1}{b^{2}}\| \langle p\rangle_{q}^{2}u\|_{L^{2}}\Big]
  \\
  &\leq 12\Big[\frac{1}{b^{4/3}}\|\frac{|\lambda|}{\langle p\rangle_{q}}u\|_{L^{2}}+
\frac{1}{b^{2}}\|\mathcal{O}_{g^{h}}u\|_{L^{2}}\Big]\,.
\end{align*}
 Because $0<b\leq 1$\,, it suffices to replace the constant $C_{g,old}$ of Proposition~7.2 in \cite{NSW} by $C_g=C_{g,new}2^9\times 13 C_{g,old}$ and then to choose $C_{g,s}=2(C_{g,new}+16\nu_{0,s})$\,.\\
Let us consider now the case with the final term $R_{2,h}= \mp \langle R^{TQ}_{g^{h}}( p, f_{i,g^{h}})p,f_{j,g^{h}} \rangle (f^i_{g^{h}} - \hat{f}_{i,g^{h}})\mathbf{i}_{f_{j,g^{h}} + \hat{f}^j_{g^{h}}}$\,.
We set $A_s(b,h)=P_{\pm,b}+M_{0,s}(b,h)+M_{1,s}(b,h)$ and we now consider $A_{s}(b,h)+R_{2,h}$ by perturbative arguments.
The accretivity of $A_s(b,h)+R_{2,h}$ is due to 
$$
\left|\mathbb{R}\mathrm{e}\langle u\,,\, R_{2,h}u\rangle_{L^{2}}\right|\leq C_g'\|u\|_{\tilde{\mathcal{W}}^{1,0}}^{2}
$$
while we know
$$
\mathbb{R}\mathrm{e}\langle u\,,\, A_s(b,h)u\rangle_{L^{2}}\geq \frac{1}{8b^2}\left[\|u\|_{\tilde{\mathcal{W}}^{1,0}}^{2}+\kappa_{s}\|u\|^{2}_{L^{2}}\right]\,.
$$
It thus suffices to assume $0<b\leq \frac{1}{4\sqrt{C_g'}}$\,. 
The second inequality \eqref{eq:principalInequality} for $A_s(b,h)$ implies
$$
\forall u\in D(\overline{A_s(b,h)})\,,\quad \|\overline{A_s(b,h)}u\|_{L^2}+\frac{2\kappa_s}{b^2}\|u\|_{L^2}\geq \frac{1}{C_g b^2}\|\mathcal{O}_{g^h}u\|_{L^2}
\geq \frac{1}{C_g \nu_g b^2}\|R_{2,h} u\|_{L^2}\,.
$$
Therefore $R_{2,h}$ is a relatively bounded perturbation of $\overline{A_s(b,h)}$ with relative bound $C_g\nu_g b^2\leq 1/4<1$ provided that $0<b\leq \frac{1}{2\sqrt{C_g\nu_g}}$\,. By \cite{ReSi}-Theorem~X.50, $\overline{A_s(b,h)+R_{2,h}}$ is maximal accretive with the same domain as $\overline{A_s(b,h)}$\,. This relative boundedness also implies
$$
\|(\overline{B_{\pm,b,V^{h}}}-i\lambda)u\|_{L^2}+\frac{2\kappa_s}{b^2}\|u\|_{L^2}\geq \frac{3}{4}
\left[\|(\overline{A_s(b,h)}-i\lambda)u\|_{L^2}+\frac{2\kappa_s}{b^2}\|u\|_{L^2}\right]
$$
and the subelliptic estimate \eqref{eq:principalInequality}, with the coefficient $\frac{3}{4 C_g}$ in the right-hand side follows.\\
We end the proof by adjusting a new value of $C_g$ according to $C_{g,new}=\max(4/3 C_g,\sqrt{4C_g'},2\sqrt{C_g \nu_g})$\,.\\

\end{proof}
Let us recall a few consequences of Proposition~\ref{pr:main1sttext}:
\begin{enumerate}
\item For any $s\in \mathbb{R}$ and $z\in \mathbb{C}$\,, the compact imbedding: $\tilde{\mathcal{W}}^{0,s+2/3}(X^{h};\mathcal{E}_{\pm}^{h})\subset \tilde{\mathcal{W}}^{0,s}(X^{h};\mathcal{E}_{\pm}^{h})$ implies  $\overline{B_{\pm,b,V^{h}}}^{s}-z:D(\overline{B_{\pm,b,V^{h}}}^{s})\to \tilde{\mathcal{W}}^{s}(X^{h};\mathcal{E}_{\pm}^{h})$ is a Fredholm operator with index $0$\,. Therefore the spectrum of $\overline{B_{\pm,b,V^{h}}}^{s}$ is discrete.
\item By a bootstrap argument when $z\not\in \mathrm{Spec}(\overline{B_{\pm,b,V^{h}}}^{s})$ the resolvent $(\overline{B_{\pm,b,V^{h}}}^{s}-z)^{-1}$ sends continuously $\mathcal{S}(X^{h};\mathcal{E}_{\pm}^{h})$ to $\mathcal{S}(X^{h};\mathcal{E}_{\pm}^{h})$ and the same holds for $(B_{\pm,b,V^{h}}^{*,s}-z)^{-1}$\,. Hence for two different $s,s'\in \mathbb{R}$ the resolvent $(\overline{B_{\pm,b,V^{h}}}^{s}-z)^{-1}$ and $(\overline{B_{\pm,b,V^{h}}}^{s'}-z)^{-1}$ coincide as $\mathcal{L}(\mathcal{S}(X^{h};\mathcal{E}_{\pm}^{h});\mathcal{S}'(X^{h};\mathcal{E}_{\pm}^{h}))$-valued meromophic functions and
  $\mathrm{Spec}(\overline{B_{\pm,b,V^{h}}}^{s})$ does not depend on $s\in\mathbb{R}$ as well.
\item The subelliptic estimate \eqref{eq:principalInequality} ensures that $\overline{B_{\pm,b,V^{h}}}^{s}$ is cuspidal according to the terminology of \cite{Nie} (see also \cite{HerNi}\cite{HeNi}\cite{EcHa}\cite{BiLe}) and the integral representation
  $$
e^{-t\overline{B_{\pm,b,V^{h}}}^{s}}=\frac{1}{2i\pi}\int_{\Gamma_{b}}e^{-tz}(z-\overline{B_{\pm,b,h}}^{s})^{-1}~dz
$$
is a convergent integral for $t>0$ when
$$
\Gamma_{b}=\left\{z\in\mathbb{C}\,, \mathbb{R}\mathrm{e}\,z  = \frac{1}{C_{b}}\langle \mathrm{Im}\,z\rangle^{1/2}-C_{b}\right\}
$$
and $e^{-t\overline{B_{\pm,b,V^{h}}}^{s}}:\mathcal{S}'(X^{h};\mathcal{E}_{\pm}^{h})\mapsto \mathcal{S}(X^{h};\mathcal{E}_{\pm}^{h})$\,.\\
This implies that the poles of the resolvent $(z-B_{\pm,b,V^{h}})^{-1}$ are continuous finite rank operators from $\mathcal{S}'(X^{h};\mathcal{E}_{\pm}^{h})$ to $\mathcal{S}(X^{h};\mathcal{E}_{\pm}^{h})$\,.\\
\item Changing the contour $\Gamma_{b}$ above  allows to isolate the main contribution to $e^{-t\overline{B_{\pm,b,V^{h}}}^{s}}$ associated with eigenvalues with small real part from the others with exponentially smaller remainder as $t\to +\infty$\,.
\item With the scaling and Proposition~\ref{pr:scaling} all these functional properties can be transferred to the operator $B_{\pm,b,\frac{V}{h}}$ associated with $(Q,g,\frac{V}{h},b)$ after replacing the condition $0<b\leq \frac{1}{C_{g}}$ by $0<\frac{b}{h}\leq\frac{1}{C_{g}}$\,, the spaces $\tilde{\mathcal{W}}^{s_{1},s_{2}}(X^{h};\mathcal{E}_{\pm}^{h})$ by the spaces $\tilde{\mathcal{W}}^{s_{1},s_{2}}_{h}(X;\mathcal{E}_{\pm})$ according to Definition~\ref{def:Ws1s2h} and by multiplying the spectral parameter by $\frac{1}{h^{2}}$ or the time by $h^{2}$\,.
\end{enumerate}

\section{Improved lower bounds for modified operators} \label{sec:modifiedOperators}
In this whole section we work with the rescaled Bismut Laplacian $B_{\pm,b,V^h}$ associated with the scaled data $(Q^h,g^h,V^h,b)$ and the Sobolev spaces $\tilde{\cal{W}}^{s_1,s_2}(X^h;\mathcal{E}_{\pm}^h)=\tilde{\cal{W}}^{s_1,s_2}_{1}(X^h;\mathcal{E}_{\pm}^h)$\,. Although the connection, the vector field $ \cal{Y}$\,, the terms $\alpha_{\pm}$\,, $\beta_{\pm}$\,, $\gamma_{\pm}$\,, and some other related quantities depend on $h$ or the metric $g^h$\,, we will drop the corresponding subscript notations for the sake of simplicity.  This is especially relevant owing to the uniform estimates Proposition~\ref{pr:scaling} and  of Proposition~\ref{pr:main1sttext}. For further comparisons,  we keep the memory of the $h$-parameter only via the notations $V^h$, $Q^h$, $X^h$\,, $\mathcal{E}_{\pm}^h$ and $\nabla^{\mathcal{E}_{\pm},h}$\,.\\
For the accurate spectral asymptotic analysis we need subelliptic estimates for the operator $\overline{B_{\pm,b,V^h}}^s$ itself without adding the constant $\frac{\kappa_{s}}{b^{2}}$ in order to study the spectrum around $0$\,. Because $\alpha_{\pm}$ and possibly $\overline{B_{\pm,b,V^h}}^{s}$ have a non trivial kernel, resolvent estimates must be given for operators modified in such a way that the singularity of the resolvent at $z=0$ is removed with a good control as the parameter $b$ tends to $0$ (uniform with respect to $h\in]0,1]$)\,. The first modification consists in adding $A^{2}\pi_{0,\pm}$ with $A=A(b)$ suitably chosen according to $b$\,, the second modification consists in looking at $\pi_{\perp,\pm}\overline{B_{\pm,b,V^h}}^s\pi_{\perp,\pm}$ with $\pi_{\perp,\pm}=1-\pi_{0,\pm}$\,. Finally the third one consists in adding $A^{2}\pi_{0,\pm}\chi\big(\frac{2 W^{2}_{\theta}}{(LA)^{2}})\pi_{0,\pm}$ with $\chi\in\mathcal{C}^{\infty}_{0}(\mathbb{R};[0,1])$ instead of $A^{2}\pi_{0,\pm}$\,.

\subsection{The first modified operator $B_{\pm,b,V^h}+A^{2}\pi_{0,\pm}$}

The main result of this paragraph is about a subelliptic estimate for  $B_{\pm,b,V^h}+A^{2}\pi_{0,\pm}$ without adding a remainder term $\frac{\kappa_{b,h}}{b^{2}}$ and where the lower bound has coefficients which can be fixed large, independently of $b\to 0^+$. With this aim, the maximal subelliptic exponent $2/3$ is replaced by the lower value $2/9$ as a result of interpolation.
\begin{proposition}\label{pr:subellipticEstimateA}
     There exist two constants $C, C_s\geq 1$, which are respectively uniform and $s$-dependent, $s\in \mathbb{R}$, such that the condition $C_{s}\max(Ab,b,A^{-1})\leq 1 $  implies that $\overline{B_{\pm,b,V^{h}}+A^{2}\pi_{0,\pm}}^{s}$ with $D(\overline{B_{\pm,b,V^{h}}+A^{2}\pi_{0,\pm}-\frac{A^{2}}{2}}^{s})=D(\overline{B_{\pm,b,V^{h}}}^{s})\subset \tilde{\mathcal{W}}^{0,s}(X^{h};\mathcal{E}_{\pm}^{h})$ is maximal accretive with 
     \begin{eqnarray} \label{eq:subellipticEstimateWith2/5ForBismutLaplacian}
   C \left\| (B_{\pm,b,V^h} +A^2 \pi_{0,\pm}-z)u \right\|_{\tilde{\mathcal{W}}^{0,s}}  & \geq & A^{2}\left\| u \right\|_{\tilde{\mathcal{W}}^{0,s}} + A^{2} \left\| \mathcal{O} u \right\|_{\tilde{\mathcal{W}}^{0,s}} + A^{2}b\left\| (\nabla_{\mathcal{Y}}^{\mathcal{E}_{\pm},h}-i\mathrm{Im}~z)u \right\|_{\tilde{\mathcal{W}}^{0,s}} \nonumber\\ 
                                                                                       & &+ bA^{2}|\mathrm{Im}~z|^{1/2}\|u\|_{\tilde{\mathcal{W}}^{0,s}}+ A^{2}b^{\frac{2}{3}} \left\| u \right\|_{\tilde{\mathcal{W}}^{0,s+\frac{2}{3}}} + \frac{A^{\frac{8}{5}}}{\langle b |\mathrm{Im}~z|^{1/2}\rangle^{\frac{4}{5}}} \|u\|_{\tilde{\mathcal{W}}^{0,s+\frac{2}{5}}} \nonumber \\
       & & \quad + A^{\frac{16}{9}}\|u\|_{\tilde{\cal W}^{0,s+\frac{2}{9}}}
  \end{eqnarray} 
  for all $u\in \mathcal{S}(X^h;\mathcal{E}_{\pm}^h)$ and all $z\in\mathbb{C}$ such that $\real z\leq \frac{A^{2}}{2}$\,, and  where we recall that the operators $\mathcal{O},\mathcal{Y} $ and the Sobolev spaces $\tilde{W}^{s_1,s_2}$ depend on the metric $g^h$.
\end{proposition}

\begin{remark}
The constants $C,C_{s}\geq 1$ in Proposition~\ref{pr:subellipticEstimateA} are  obtained after several steps, and at every step the values of the constants $C,C_{s}$ are suitably tuned. We will often conclude such an intermediate  analysis at step $n$ with the sentence ``Choose $(C_{new},C_{R,s,new})=$ Expression of  $(C_{old},C_{R,s,old})$'', where $_{old}$ refers to the values obtained at step $n-1$ and $_{new}$ to the conclusion for the step $n$. 
\end{remark}
Before starting a proof let us verify the maximal accretivity announced in Remark~\ref{re:accretBb}.
\begin{corollary}
  \label{cor:maxaccr}
  For all $s\in \mathbb{R}$ there exists $C_{s}\geq 1$ such that $C_{s}+\overline{B_{\pm,b,V^{h}}}^{s}$ is maximal accretive when $C_{s}b\leq 1$ and $h\in ]0,1]$:
$$
\forall u\in D(\overline{B_{\pm,b,V^{h}}}^{s})\,,\quad \mathbb{R}\mathrm{e}~\langle u\,,\, (C_{s}+B_{\pm,b,V^{h}})u\rangle_{\tilde{\mathcal{W}}^{0,s}}\geq 0\,.
$$
\end{corollary}
\begin{proof}
  It suffices to notice
  $$
  \mathbb{R}\mathrm{e}~\langle u\,,\, (A^{2}+B_{\pm,b,V^{h}})u\rangle_{\tilde{\mathcal{W}}^{0,s}}\geq
    \mathbb{R}\mathrm{e}~\langle u\,,\, (B_{\pm,b,V^{h}}+A^{2}\pi_{0,\pm})u\rangle_{\tilde{\mathcal{W}}^{0,s}}\geq 0
    $$
    when $C_{s,old}\max(Ab,b,\frac{1}{A})\leq 1$\,, to choose $A=C_{s,old}$\,, $b\leq \frac{1}{C_{s,old}^{2}}$ and to take $C_{s,new}=C_{s,old}^{2}$\,.
\end{proof}
We start the proof of Proposition~\ref{pr:subellipticEstimateA} with the simpler operator 
\begin{equation}\label{eq:simplifiedOperator}
    P_{\pm,b}+\frac{1}{b}R_{1,\bot,h}+A^2\pi_{0,\pm}
\end{equation}
where $A$ is a positive number and $P_{\pm,b}$ and $R_{1,\bot,h}$ are defined as in \eqref{eq:formOfTheOperator}. 
Remember that the conjugated operator  $(W^2_{\theta})^{\frac{s}{2}}[P_{\pm,b}+\frac{1}{b}R_{1,\bot,h}+A^2\pi_{0,\pm}](W^2_{\theta})^{-\frac{s}{2}}$ with  $(W^2_{\theta})^{\frac{s}{2}} = (W^2_{\theta,g^h})^{\frac{s}{2}}$\,, takes the same form
  $P_{\pm,b}+\frac{1}{b}\tilde{R}_{1,s,\bot,h}+A^2\pi_{0,\pm}$ with a new $s$-dependent remainder term $\frac{1}{b}R_{1,\bot,h}$ with the same uniform estimates.
After this we will consider 
$$ 
B_{\pm,b,V^h}=[ P_{\pm,b} + \frac{1}{b}R_{1,\bot,h} ]+ R_{0,h} + R_{2,h} 
$$
by a simple perturbative argument. 

We use the notations    $u_0= \pi_{0,\pm}(u)$ and $u_{\bot} = \pi_{\bot,\pm}u=u-u_{0}$ for $u\in \mathcal{S}'(X^h;\mathcal{E}_{\pm}^h)$. The following properties are obvious
\begin{itemize}
    \item The equality $\mathcal{O}u_0 =\mathcal{O}_{g^h}u_0 = \frac{d}{2}u_0$ holds and therefore
    $$ \|u_0\|_{\tilde{\mathcal{W}}^{1,0}}^2 = \frac{d}{2} \|u_0\|^2_{L^2} \geq \frac{1}{2} \|u_0\|^2_{L^2}.$$
    \item With $\alpha_{\pm}= \mathcal{O}\pm (N_v-d/2)$ we  have  $\mathcal{O}+d/2\geq \alpha_\pm \geq \mathcal{O}-d/2$ and
\begin{eqnarray}
\label{eq:ubotW1}
    &&\|u_\bot\|^2_{\tilde{\mathcal{W}}^{1,0}}+\frac{d}{2}\|u_\bot\|_{L^2}^2\geq  \langle u_{\bot}\,,\, \alpha_\pm u_{\bot}\rangle\geq \|u_\bot\|^2_{\tilde{\mathcal{W}}^{1,0}}-\frac{d}{2}\|u_\bot\|_{L^2}^2
    \\
\label{eq:ubotL2}
    \text{while we know}&&
    \langle u_{\bot}\,,\, \alpha_\pm u_{\bot}\rangle\geq \|u_\bot\|_{L^2}^2\,.
\end{eqnarray}
\end{itemize}
We begin with the following integration by parts.

\begin{proposition}\label{pr:IppWithRealPart}
  For all $A,b \in \mathbb{R}_{+}^{*}$, the inequality
  \begin{equation}\label{eq:IppInequality}
     \real \langle (P_{\pm,b} + A^2 \pi_{0,\pm})u,u \rangle_{L^2} \geq \frac{2}{ (d+2)b^2} \| u_{\bot}\|_{\tilde{\mathcal{W}}^{1,0}}^2 +A^2 \|u_0\|_{L^2}
  \end{equation}
holds for all $u\in \mathcal{S}(X^h;\mathcal{E}_{\pm}^h)$.
\end{proposition}

\begin{proof}
 Just use $\real\langle P_{\pm,b} u\,,\, u\rangle = \frac{1}{b^2}\langle u_\bot\,,\, \alpha_\pm u_\bot\rangle$ and \eqref{eq:ubotW1}\eqref{eq:ubotL2}.
\end{proof}

\begin{proposition}\label{pr:controlOfRByIpp}
  There is a positive constant $c_{R}>0$\,, such that for all $\varepsilon>0$, the inequality 
  $$ c_{R}|\real\left\langle R_{1,\bot,h}u , u \right\rangle | \leq \varepsilon \left\| u_{0} \right\|_{L^{2}}^{2} + (1+\frac{1}{\varepsilon})\left\| u_{\bot} \right\|_{\tilde{\mathcal{W}}^{1,0}}^{2} $$
  holds for all $u\in \mathcal{S}(X^h;\mathcal{E}_{\pm}^h)$.
\end{proposition}
\begin{proof}
From conditions \eqref{eq:hypothesisForPerturbationRi}, \eqref{eq:extraHypothesisForPerturbationRiPerp}  we deduce 
$$ R_{1,\bot,h}= \pi_{0,\pm}R_{1,h}'\pi_{\bot,\pm} + \pi_{\bot,\pm}R_{1,h}''\pi_{0,\pm} + \pi_{\bot,\pm}R_{1,h}'''\pi_{\bot,\pm} ,$$
with $R_{1,h}',R_{1,h}'',R_{1,h}'''\in\mathcal{L}(\tilde{\mathcal{W}}^{1,0};L^2)$. 
The triangular and Cauchy-Schwarz inequalities yield
  \begin{eqnarray*}
    |\real \left\langle R_{1,\bot,h}u , u \right\rangle|  &\leq &|\left\langle R_{1,h}' u_{\bot} , u_{0} \right\rangle| + |\left\langle R_{1,h}''u_{0} , u_{\bot} \right\rangle| + |\left\langle R_{1,h}'''u_{\bot} , u_{\bot} \right\rangle| \\
   & \leq & C_{R}\left( \left\| u_{\bot} \right\|_{\tilde{\mathcal{W}}^{1,0}}\left\| u_{0} \right\|_{L^{2}} + \left\| u_{0} \right\|_{L^2}\left\| u_{\bot} \right\|_{L^{2}} + \left\| u_{\bot} \right\|_{\tilde{\mathcal{W}}^{1,0}}\left\| u_{\bot} \right\|_{L^{2}} \right) \,,
  \end{eqnarray*}
  where $0<\frac{1}{2c_{R}}=C_{R}= \sup_{h\in ]0,1]} \max (\| R_{1,h}' \|_{\mathcal{L}(\tilde{\mathcal{W}}^{1,0};L^2)} ,\sqrt{\frac{d}{2}}\|R_{1,h}''\|_{\mathcal{L}(\tilde{\mathcal{W}}^{1,0};L^2)},\|R_{1,h}'''\|_{\mathcal{L}(\tilde{\mathcal{W}}^{1,0};L^2)})<\infty$ by our hypothesis on $R_{1,\bot,h}$. 
  \begin{eqnarray*}
    c_{R}|\real\left\langle R_{1,\bot,h}u , u \right\rangle| \leq  \left\| u_{\bot} \right\|_{\tilde{\mathcal{W}}^{1,0}}\left\| u_{0} \right\|_{L^{2}} + \left\| u_{\bot} \right\|_{\tilde{\mathcal{W}}^{1,0}}^{2}.
  \end{eqnarray*}
  The result follows when we apply the inequality 
  $$ \forall a,b,\varepsilon\in \mathbb{R}_{+}^{*}\,, \quad 2ab\leq \varepsilon a^{2} + \frac{1}{\varepsilon}b^{2},$$
  with $a=\|u_0\|_{L^2}$ and $b=\|u_{\bot}\|_{\tilde{\mathcal{W}}^{1,0}}$.
\end{proof}

The following proposition is a consequence of Proposition~\ref{pr:IppWithRealPart} and Proposition~\ref{pr:controlOfRByIpp}.

\begin{proposition}\label{pr:firstEstimate}
  There is a constant $C_{R,s}\geq 1$, which depends $s\in\mathbb{R}$, such that the condition $\max(Ab,b,\frac{1}{A})\leq \frac{1}{C_{R,s}}$ implies the inequalities
  \begin{eqnarray}
    \label{eq:IppBis}
    & \real\left\langle (P_{\pm,b}+\frac{1}{b}R_{1,\bot,h}+A^2\pi_{0,\pm}) u , u \right\rangle_{\tilde{\mathcal{W}}^{0,s}} & \geq \frac{1}{(d+2)b^{2}} \left\| u_{\bot} \right\|_{\tilde{\mathcal{W}}^{1,s}}^{2} + \frac{3A^{2}}{4} \left\| u_{0} \right\|_{\tilde{\mathcal{W}}^{0,s}}^{2}\\
    &\left\| (P_{\pm,b} + \frac{1}{b}R_{1,\bot,h} +A^2\pi_{0,\pm}-i\lambda) u \right\|_{\tilde{\mathcal{W}}^{0,s}} &  \geq \frac{3A^{2}}{4} \left\| u \right\|_{\tilde{\mathcal{W}}^{0,s}}  \label{eq:L2L2ResolvantEstimate} \,, \\
    \textrm{and} &\left\| (P_{\pm,b} +\frac{1}{b}R_{1,\bot,h} +A^2\pi_{0,\pm} -i\lambda) u \right\|_{\tilde{\mathcal{W}}^{0,s}}^{2} & \geq  \frac{3A^{2}}{4(d+2)b^{2}} \left\| u_{\bot} \right\|_{\tilde{\mathcal{W}}^{1,s}}^{2} + \frac{9A^{4}}{16} \left\| u_{0} \right\|_{\tilde{\mathcal{W}}^{0,s}}^{2} \label{eq:Ipp}
  \end{eqnarray}
for all $u\in \mathcal{S}(X^h;\mathcal{E}_{\pm}^h)$ and all $\lambda\in\mathbb{R}$\,. Moreover under the above condtion, $(P_{\pm,b}+\frac{1}{b}R_{1,\bot,h}+A^2\pi_{0,\pm})$ is essentially maximal accretive on $\mathcal{S}(X^h;\mathcal{E}_{\pm}^h)$ in $\tilde{\mathcal{W}}^{0,s}(X^h;\mathcal{E}_{\pm}^h)$.\\
\end{proposition}

\begin{proof}
  We begin with the case $s=0$, Proposition~\ref{pr:controlOfRByIpp} gives
   $$ \real \langle (P_{\pm,b} +A^2\pi_{0,\pm} + \frac{1}{b}R_{1,\bot,h}) u , u \rangle_{L^2} \geq \big(\frac{2}{(d+2)b^2} - \frac{1}{bc_{R}}(1+\frac{1}{\varepsilon})\big) \|u_{\bot}\|_{\tilde{\mathcal{W}}^{1,0}}^2 + \big(A^2-\frac{\varepsilon}{bc_R} \big) \| u_0 \|_{L^2}^2 \,, $$
  for all $\varepsilon>0$. Choosing $\varepsilon =Ab\sqrt{d+2}$ 
and the sufficient conditions
$$
 b \leq \frac{c_R}{2(d+2)} \quad \text{and} \quad       \frac{4\sqrt{d+2}}{c_R}\leq A 
$$
imply
$$
\big(\frac{2}{(d+2)b^2} - \underbrace{\frac{1}{bc_{R}}}_{\leq \frac{1}{2(d+2)b^2}} - \underbrace{\frac{1}{bc_R \varepsilon}}_{\leq \frac{1}{2(d+2)b^2} }\big)\geq \frac{1}{(d+2)b^2} \quad\text{and}\quad \big(A^2-\frac{\varepsilon}{bc_R} \big)\geq \frac{3A^2}{4}\,.
$$
This proves \eqref{eq:IppBis} under the condition  $ \max (Ab , b ,\frac{1}{A}) \leq \frac{1}{C_{R,0}}$, with $C_{R,0}=\max(\frac{2(d+2)}{c_R},\sqrt{d+2},\frac{4\sqrt{d+2}}{c_R})$\,.\\

With $\|u_\bot\|_{\tilde{\mathcal{W}}^{1,0}}^2\geq \frac{d}{2}\|u_\bot\|_{L^2}^2$ and
$$
\real \langle (P_{\pm,b} +A^2\pi_{0,\pm} + \frac{1}{b}R_{1,\bot,h}) u , u \rangle_{L^2}=
 \real \langle (P_{\pm,b} +A^2\pi_{0,\pm} + \frac{1}{b}R_{1,\bot,h}-i\lambda) u , u \rangle_{L^2} 
$$
the Cauchy-Schwarz inequality combined with \eqref{eq:IppBis} gives
\begin{equation}
\label{eq:CS1}
\|(P_{\pm,b}+\frac{1}{b}R_{1,\bot,h}+A^2\pi_{0,\pm}-i\lambda) u\|_{L^2}\|u\|_{L^2}\geq 
\frac{d}{2(d+2)b^{2}} \left\| u_{\bot} \right\|_{L^2}^{2} + \frac{3A^{2}}{4} \left\| u_{0} \right\|_{L^2}^{2}
\geq \frac{3A^2}{4}\|u\|_{L^2}^2
\end{equation}
as soon as $\frac{d}{(d+2)b^2}\geq A^2$\,, which is implied by $\sqrt{d+2}Ab \leq C_{R,0} Ab\leq 1 $. This yields \eqref{eq:L2L2ResolvantEstimate}.\\
The inequality~\eqref{eq:Ipp} is a consequence of \eqref{eq:L2L2ResolvantEstimate} and \eqref{eq:CS1}. \\

For the maximal accretivity property, the decomposition
  $$
  P_{\pm,b}+\frac{1}{b}R_{1,\bot,h}+A^2\pi_{0,\pm}=\left[\frac{C'}{b^2}+\frac{1}{b^2}\mathcal{O} \mp \frac{1}{b}\nabla^{\mathcal{E}_{\pm},h}_{\mathcal Y}+\frac{1}{b}R_{1,\bot,h}\right]+\left[A^2\pi_0 -\frac{C'}{b^2} \pm \frac{1}{b^2}(N_V-d/2)\right]
  $$
  shows that $(P_{\pm,b}+\frac{1}{b}R_{1,\bot,h}+A^2\pi_{0,\pm})$ is a bounded perturbation of 
  $$
  \frac{C'}{b^2}+P_{\pm,b,M}=\frac{C'}{b^2}+\frac{1}{b^2}\mathcal{O} \mp \frac{1}{b}\nabla^{\mathcal{E}_{\pm},h}_{\mathcal Y}+M_1
  $$
  where $M_1=\frac{1}{b}R_{1,\bot,h}$ fulfills the assumptions of Proposition~7.2 in \cite{NSW} when $C_{R,0}b\leq 1$\,. Then  Proposition~7.2 in \cite{NSW} says that $\frac{C'}{b^2}+P_{\pm,b,M}$ is essentially maximal accretive on $\mathcal{S}(X^h;\mathcal{E}_{\pm}^h)$ for $C'>0$ chosen large enough.\\

  Finally, the case with  a general  $s\in\mathbb{R}$ amounts to the case $s=0$ owing to
  $$ (W_{\theta}^{2})^{\frac{s}{2}}  \big(P_{\pm,b}+\frac{1}{b}R_{1,\bot,h} +A^2\pi_{0,\pm} \big) (W_{\theta}^{2})^{-\frac{s}{2}} = P_{\pm,b}+\frac{1}{b}R_{1,\bot,h}^s +A^2\pi_{0,\pm}. $$
\end{proof}

Below we give a first global subelliptic estimate without remainder for $P_{\pm,b}+\frac{1}{b}R_{1,\bot,h}+A^2\pi_0$.
\begin{proposition} \label{pr:subellipticestimateWithbDependence}
  There exist two constants $C, C_{R,s}\geq 1$, which are respectively uniform and $s$-dependent, $s\in \mathbb{R}$, such that the inequality
  \begin{multline}\label{eq:MaximalSubEllipticEstimateWithDependanceOnb}
    C\left\| (P_{\pm,b} + \frac{1}{b}R_{1,\bot,h} + A^2 \pi_{0,\pm}-z) u \right\|_{\tilde{\mathcal{W}}^{0,s}}\geq A^{2}\left\| u \right\|_{\tilde{\mathcal{W}}^{0,s}} + A^{2}\left\| \mathcal{O} u \right\|_{\tilde{\mathcal{W}}^{0,s}} + 
    A^{2}b\left\| (\nabla_{\mathcal{Y}}^{\mathcal{E}_{\pm},h}-i\mathrm{Im}~z)u \right\|_{\tilde{\mathcal{W}}^{0,s}} \\
    + A^{2}b^{\frac{2}{3}} \left\| u \right\|_{\tilde{\mathcal{W}}^{0,s+\frac{2}{3}}}+A^{2}b|\mathrm{Im}\,z|^{1/2}\|u\|_{\tilde{\mathcal{W}}^{0,s}}
  \end{multline}
holds  for all $u\in \mathcal{S}(X^h;\mathcal{E}_\pm^h)$ and all $z\in\mathbb{C}$\,, such that $\real z\leq \frac{A^{2}}{2}$ as soon as $C_{R,s} \max(Ab,b,A^{-1})\leq 1$\,. 
\end{proposition}

\begin{proof}
  Owing to the  accretivity of Proposition~\ref{pr:firstEstimate}, the case of a general $z\in \mathbb{C}$\,, $\real z\leq \frac{A^{2}}{2}$\,, is reduced to the case $z=i\lambda$\,, $\lambda\in \mathbb{R}$\,. Actually the accretivity of $P_{\pm,b}+\frac{1}{b}R_{1,\bot,h}+A^{2}\pi_{0,\pm}-\frac{A^{2}}{2}$ implies
  $$
\|(P_{\pm,b}+\frac{1}{b}R_{1,\bot,h}+A^{2}\pi_{0,\pm}-z)u\|_{\tilde{\mathcal{W}}^{0,s}}  \geq \|(P_{\pm,b}+\frac{1}{b}R_{1,\bot,h}+A^{2}\pi_{0,\pm}-\frac{A^{2}}{2}-i\mathrm{Im}\,z)u\|_{\tilde{\mathcal{W}}^{0,s}}
$$
when $\real z\leq\frac{A^{2}}{2}$\,. But the inequality \eqref{eq:L2L2ResolvantEstimate} also says
\begin{align*}
\|(P_{\pm,b}+\frac{1}{b}R_{1,\bot,h}+A^{2}\pi_{0,\pm}-\frac{A^{2}}{2}-i\lambda)u\|_{\tilde{\mathcal{W}}^{0,s}}&\geq
\|(P_{\pm,b}+\frac{1}{b}R_{1,\bot,h}+A^{2}\pi_{0,\pm}-i\lambda)u\|_{\tilde{\mathcal{W}}^{0,s}}-\frac{A^{2}}{2}\|u\|_{\tilde{\mathcal{W}}^{0,s}}\\
&\geq
\frac{1}{3}\|(P_{\pm,b}+\frac{1}{b}R_{1,\bot,h}+A^{2}\pi_{0,\pm}-i\lambda)u\|_{\tilde{\mathcal{W}}^{0,s}}\,.
\end{align*}
So we focus on the case $z=i\lambda$\,, $\lambda\in \mathbb{R}$\,.\\
In the case $s=0$  we refer again to Proposition~7.2 in \cite{NSW}. Actually with $R_{1,\bot,h}=0$\,, we set $M_0=A^2\pi_{0,\pm} \pm \frac{1}{b^2}(N_V-d/2)$ and
$$
P_{\pm,b}+A^2\pi_{0,\pm}= \frac{1}{b^2}\mathcal{O}+\frac{1}{b}\nabla_{\mathcal{Y}}^{\mathcal{E}_{\pm},h}+\left[A^2\pi_{0,\pm} \pm \frac{1}{b^2}(N_V-d/2)\right]=P_{\pm,b,M_0}
$$
where the right-hand side refers to the notation introduced in \cite{NSW}. 
The operator  $M_0=A^2\pi_{0,\pm} \pm \frac{1}{b^2}(N_V-d/2)$ fulfills the assumptions of Proposition~7.2 in \cite{NSW} with $\nu_1=0$ and a uniform $\nu_0>0$\,.
It provides us the subelliptic estimate
\begin{equation*}
  C'_{1}\big( \left\|(P_{\pm,b,M_0}-i\frac{\lambda}{b}) u \right\|_{L^{2}} + \frac{C'_0}{b^{2}}\|u\|_{L^2} \big) \geq \frac{C'_0}{b^{2}} \left\| u \right\|_{L^{2}} + \frac{1}{b^{2}} \left\| \mathcal{O} u \right\|_{L^{2}} + \frac{1}{b}\left\| (\nabla_{\mathcal{Y}}^{\mathcal{E}_{\pm},h}-i\lambda)u \right\|_{L^{2}} + \frac{1}{b^{\frac{4}{3}}}\left\| u \right\|_{\tilde{\mathcal{W}}^{0,\frac{2}{3}}}+\frac{|\lambda|^{1/2}}{b^{3/2}}\|u\|_{L^{2}}\,,
\end{equation*}
or after replacing $\frac{\lambda}{b}$ by $\lambda$\,,
\begin{equation}
  \label{eq:MaximalSubEllipticEstimatePb0}
  C'_{1}\big( \left\|(P_{\pm,b,M_0}-i\lambda u) \right\|_{L^{2}} + \frac{C'_0}{b^{2}}\|u\|_{L^2} \big) \geq \frac{C'_0}{b^{2}} \left\| u \right\|_{L^{2}} + \frac{1}{b^{2}} \left\| \mathcal{O} u \right\|_{L^{2}} + \left\| (\frac{1}{b}\nabla_{\mathcal{Y}}^{\mathcal{E}_{\pm},h}-i\lambda)u \right\|_{L^{2}} + \frac{1}{b^{\frac{4}{3}}}\left\| u \right\|_{\tilde{\mathcal{W}}^{0,\frac{2}{3}}}+\frac{|\lambda|^{1/2}}{b}\|u\|_{L^{2}}\,,
\end{equation}
for fixed uniform constants $C'_1\geq 1$ and $C_0'\geq 1$ when $b\leq 1$. Interpolation  or the functional calculus tells us 
\begin{equation}
\label{eq:interpWks}
\|u\|_{\tilde{\mathcal{W}}^{1,s}}\leq \|u\|_{\tilde{\mathcal{W}}^{0,s}}^{1/2}\|u\|_{\tilde{\mathcal{W}}^{2,s}}^{1/2}= \|u\|_{\tilde{\mathcal{W}}^{0,s}}^{1/2}\|\mathcal{O}u\|_{\tilde{\mathcal{W}}^{0,s}}^{1/2}
\leq \delta \|u\|_{\tilde{\mathcal{W}}^{0,s}}+ \delta^{-1}\|\mathcal{O}u\|_{\tilde{\mathcal{W}}^{0,s}}\,.
\end{equation}
Applied here with $s=0$ and $\delta=\sqrt{C'_{0}}$, this implies
$$
\| \frac{1}{b}R_{1,\bot,h}u\|_{L^2}\leq \frac{C_R}{b}\|u\|_{\tilde{\mathcal{W}}^{1,0}}
\leq \frac{C_R}{b}\times \frac{C'_1 b^2}{\sqrt{C'_0}}\big( \left\| P_{\pm,b,M_0} u \right\|_{L^{2}} + \frac{C'_0}{b^{2}}\|u\|_{L^2} \big)
$$
and
$$
 C'_{1}\big( \left\|  (P_{\pm,b} + \frac{1}{b}R_{1,\bot,h} + A^2 \pi_{0,\pm}-i\lambda) u \right\|_{L^{2}} + \frac{C'_0}{b^{2}}\|u\|_{L^2} \big) \geq   C'_{1}\left(1-
 \frac{C_R C'_1 b}{\sqrt{C_0'}}\right)\big( \left\| (P_{\pm,b,M_0}-i\lambda) u \right\|_{L^{2}} + \frac{C'_0}{b^{2}}\|u\|_{L^2}\big).
$$
By assuming $C_{R,0}\max(Ab,b,\frac{1}{A})\leq 1$ the inequality \eqref{eq:L2L2ResolvantEstimate} implies
$$
 C'_{1}\left( 1+ \frac{2C_0'}{b^2A^2}\right)\left\|  (P_{\pm,b} + \frac{1}{b}R_{1,\bot,h} + A^2 \pi_{0,\pm}-i\lambda) u \right\|_{L^{2}} \geq   C'_{1}\left(1-
 \frac{C_R C'_1 b}{\sqrt{C_0'}}\right)\big( \left\| (P_{\pm,b,M_0}-i\lambda) u \right\|_{L^{2}} + \frac{C'_0}{b^{2}}\|u\|_{L^2}\big).
$$
With $b\leq \frac{\sqrt{C_0'}}{2C_R C'_1} $ after a multiplication by $2A^2b^2$ we obtain
\begin{align*}
  C'_1\left(2A^2b^2+4C_0'\right)&\left\|  (P_{\pm,b} + \frac{1}{b}R_{1,\bot,h} + A^2 \pi_{0,\pm}-i\lambda) u \right\|_{L^{2}}
  \geq   A^2b^2C'_{1}\big( \left\| (P_{\pm,b,M_0}-i\lambda) u \right\|_{L^{2}} + \frac{C'_0}{b^{2}}\|u\|_{L^2}\big)
 \\
  &\geq C_0'A^2\left\| u \right\|_{L^{2}} + A^2 \left\| \mathcal{O} u \right\|_{L^{2}}
    + A^2 b\left\| (\nabla_{\mathcal{Y}}^{\mathcal{E}_{\pm},h}-ib\lambda)u \right\|_{L^{2}} + A^2b^{2/3}\left\| u \right\|_{\tilde{\mathcal{W}}^{0,\frac{2}{3}}}
    +bA^{2}|\lambda|^{1/2}\|u\|_{L^{2}}\,.                                                     
\end{align*}
Because $A^2b^2\leq \frac{1}{(C_{R,0})^2}\leq C_0'$ we deduce
\begin{align*}
6C'_1C_0'\left\|  (P_{\pm,b} + \frac{1}{b}R_{1,\bot,h} + A^2 \pi_{0,\pm}-i\lambda) u \right\|_{L^{2}}
  &\geq A^2\left\| u \right\|_{L^{2}} + A^2 \left\| \mathcal{O} u \right\|_{L^{2}} + A^2 b\left\| (\nabla_{\mathcal{Y}}^{\mathcal{E}_{\pm},h}-ib\lambda)u \right\|_{L^{2}}\\
  &\hspace{2cm}
    + A^2b^{2/3}\left\| u \right\|_{\tilde{\mathcal{W}}^{0,\frac{2}{3}}} +bA^{2}|\lambda|^{1/2}\|u\|_{L^{2}}.
\end{align*}
We have proved the result for $s=0$ if we take $C=6 C'_1C_0'$\,, after replacing the initial value of $C_{R,0}=C_{R,0,old}$ by $C_{R,0,new}=\max(C_{R,0,old},  \frac{2C'_1 C_R}{\sqrt{C'_0}})$.\\
Let us now consider the case of a general $s\in\mathbb{R}$.
We apply the inequality \eqref{eq:MaximalSubEllipticEstimatePb0} in the case $s=0$ to the operator $ P_{\pm,b}+A^2 \pi_{0,\pm} + \frac{1}{b}R_{1,\bot,h}^s = (W_{\theta}^2)^{\frac{s}{2}} (P_{\pm,b}+A^2 \pi_{0,\pm} + \frac{1}{b}R_{1,\bot,h}) (W_{\theta}^2)^{-\frac{s}{2}}  $ and the function $v=(W_{\theta}^{2})^{\frac{s}{2}}u$, with the constant $C=6C'_1C'_0$. 
Because $R_{1,\bot,h}^s$ satisfies the same estimates uniform with respect to $h\in ]0,1]$ as $R_{1,\bot,h}$, there exists a constant $C_{R,s}\geq 1$ such that when $\max(Ab,b,\frac{1}{A})\leq \frac{1}{C_{R,s}}$ we have 
\begin{align}
  \nonumber
  6C_1'C_0' \left\| ( P_{\pm,b}+A^2 \pi_{0,\pm} + \frac{1}{b}R_{1,\bot,h}-i\lambda) u  \right\|_{\tilde{\mathcal{W}}^{0,s}} \geq&
   A^2 \|u\|_{\tilde{\mathcal{W}}^{0,s}} + A^2 \| \mathcal{O} u \|_{\tilde{\mathcal{W}}^{0,s}} + A^2 b \| (\nabla_{\mathcal{Y}}^{\mathcal{E}_{\pm},h}-ib\lambda )v  \|_{L^2}\\
  \label{eq:intermediateSubellipticEstimate000}
               & + A^2b^{2/3} \| u \|_{\tilde{\mathcal{W}}^{0,s+\frac{2}{3}}}
  +A^{2}b|\lambda|^{1/2}\|u\|_{\tilde{\mathcal{W}}^{0,s}}\,.
 \end{align}
 We use again \eqref{eq:WYWmY} which gives the uniform bound  $\| [\nabla_{\mathcal{Y}}^{\mathcal{E}_{\pm},h}, (W_{\theta}^2)^{\frac{s}{2}}](W_{\theta}^2)^{-\frac{s}{2}} \|_{\mathcal{L}(\tilde{\mathcal{W}}^{1,0}; L^2)}< C_{s} $. Thus the decomposition $ \nabla_{\mathcal{Y}}^{\mathcal{E}_{\pm},h}v = (W_{\theta}^{2})^{\frac{s}{2}}\nabla_{\mathcal{Y}}^{\mathcal{E}_{\pm},h} u + [\nabla_{\mathcal{Y}}^{\mathcal{E}_{\pm},h} , (W_{\theta}^2)^{\frac{s}{2}}] (W_{\theta}^2)^{-\frac{s}{2}}v  $ entails 
$$ 
\| (\nabla_{\mathcal{Y}}^{\mathcal{E}_{\pm},h}-ib\lambda) u \|_{\tilde{\mathcal{W}}^{0,s}} \leq \| (\nabla_{\mathcal{Y}}^{\mathcal{E}_{\pm},h}-ib\lambda) v \|_{L^2} + C_{s}\|u\|_{\tilde{\mathcal{W}}^{1,s}}\,.
$$
The interpolation inequality \eqref{eq:interpWks} used with $\delta=1$ tells us $ \|u\|_{\tilde{\mathcal{W}}^{1,s}} \leq   \| u \|_{\tilde{\mathcal{W}}^{0,s}} + \|\mathcal{O} u\|_{\tilde{\mathcal{W}}^{0,s}}$ while  \eqref{eq:intermediateSubellipticEstimate000} gives
$ \|u\|_{\tilde{\mathcal{W}}^{1,s}} \leq \frac{12C_0'C_1'}{A^2} \| ( P_{\pm,b} + A^2\pi_{0,\pm} + \frac{1}{b}R_{1,\bot,h} ) u \|_{\tilde{\mathcal{W}}^{0,s}} $.
\\ We finally obtain
\begin{align*}
  6C_0'C_1'(1+2C_{s}b)\left\| ( P_{\pm,b}+A^2 \pi_{0,\pm} + \frac{1}{b}R_{1,\bot,h}-i\lambda) u  \right\|_{\tilde{\mathcal{W}}^{0,s}} \geq &A^2 \|u\|_{\tilde{\mathcal{W}}^{0,s}} + A^2 \| \mathcal{O} u \|_{\tilde{\mathcal{W}}^{0,s}}
  \\
  &+ A^2 b \| (\nabla_{\mathcal{Y}}^{\mathcal{E}_{\pm},h}-ib\lambda) u  \| _{\tilde{\mathcal{W}}^{0,s}}
  \\&  + A^2b^{2/3} \| u \|_{\tilde{\mathcal{W}}^{0,s+\frac{2}{3}}} +A^{2}b|\lambda|^{1/2}\|u\|_{\tilde{\mathcal{W}}^{0,s}}\,.
\end{align*}
It now suffices to take $C=18 C_0'C_1'$\,, while $C_{R,s,new} = \max(C_{R,s,old},2C_{s})$ for the result concerned with $z=i\lambda$\,, and to take $C=3*18 C_{0}'C_{1}'$ for a general $z\in \mathbb{C}$ such that $\real z\leq \frac{A^{2}}{2}$\,.
\end{proof}

The subelliptic estimate \eqref{eq:MaximalSubEllipticEstimateWithDependanceOnb} is not yet satisfactory because the norm $\|u\|_{\tilde{\mathcal{W}}^{0,s+2/3}}$  appears in the right-hand side with the factor $A^2 b^{2/3}$ which is too small as $b\to 0$.
By possibly reducing the $2/3$-gain of regularity, we seek a factor of the form $A^{\alpha}$, $\alpha >0$\,. 
In order to do this we write for $u\in \mathcal{S}(X^h;\mathcal{E}_{\pm}^h)$ 
$$
\big(P_{\pm,b} + \frac{1}{b}R_{1,\bot,h} +A^2\pi_{0,\pm}-i\lambda\big)u=f
$$
where we focus again on the case $z=i\lambda$ and decompose $u$ and the right-hand side $f$ according to
$$
u= u_0+u_{\bot}=\pi_{0,\pm} u+\pi_{\bot,\pm}u\quad,\quad f=f_0+f_{\bot}=\pi_{0,\pm} f+\pi_{\bot,\pm}f\,.
$$
 \begin{lemma}
 \label{le:intermediateProposition002}
  There is a constant $C_{R,s}\geq 1$, which depends on $s\in\mathbb{R}$\,, such that the inequality 
    \begin{equation}
      \label{eq:estimateOnUOrthogonal}
    \left\| u_{\bot} \right\|_{\tilde{\mathcal{W}}^{1,s}} \leq \frac{2(d+2)b^{2}}{\varepsilon} \left\| f \right\|_{\tilde{\mathcal{W}}^{0,s}} + \varepsilon \left\| u_{0} \right\|_{\tilde{\mathcal{W}}^{0,s}}\,,
    \end{equation}
    holds true for all $u\in \mathcal{S}(X^h;\mathcal{E}_\pm^h)$ and all $\varepsilon \in ]0,1]$ as soon as $C_{R,s}\max( Ab,b,A^{-1})\leq 1$.
  \end{lemma}
  \begin{proof}
    In this case inequality~\eqref{eq:IppBis} and orthogonality give 
    $$ \frac{1}{(d+2)b^{2}} \left\| u_{\bot} \right\|_{\tilde{\mathcal{W}}^{1,s}}^{2} \leq |\left\langle f_{0} \, , \, u_{0} \right\rangle_{\tilde{\mathcal{W}}^{0,s}} + \left\langle f_{\bot} \, , \, u_{\bot} \right\rangle_{\tilde{\mathcal{W}}^{0,s}}|
 \leq \left\| f_{0} \right\|_{\tilde{\mathcal{W}}^{0,s}} \left\| u_{0} \right\|_{\tilde{\mathcal{W}}^{0,s}} + \left\| f_{\bot} \right\|_{\tilde{\mathcal{W}}^{0,s}}\left\| u_{\bot} \right\|_{\tilde{\mathcal{W}}^{0,s}}.$$ 
 By using $\alpha \beta\leq \frac{\alpha^{2} + \beta^{2}}{2}$  with $(\alpha ,\beta) = (\frac{b}{\varepsilon_0}\|f_0\|_{\tilde{\cal W}^{0,s}}, \frac{\varepsilon_{0}}{b}\|u_{0}\|_{\tilde{\cal W}^{0,s}} )$ and
 $(\alpha,\beta)= ( \sqrt{\frac{2d+4}{d}}b \| f_{\bot} \|_{\tilde{\cal W}^{0,s}} , \sqrt{\frac{d}{2d+4}}\frac{1}{b} \|u_{\bot}\|_{\tilde{\cal W}^{0,s}} ) $ we obtain
$$ \frac{1}{(d+2)b^{2}}\left\| u_{\bot} \right\|_{\tilde{\mathcal{W}}^{1,s}}^{2} \leq \frac{b^{2}}{2\varepsilon_{0}^{2}}\left\| f_{0} \right\|_{\tilde{\mathcal{W}}^{0,s}}^{2} + \frac{\varepsilon_{0}^{2}}{2b^{2}}\left\| u_{0} \right\|_{\tilde{\mathcal{W}}^{0,s}}^{2} + \frac{(d+2)b^{2}}{d}\left\| f_{\bot} \right\|_{\tilde{\mathcal{W}}^{0,s}}^{2} + \frac{d}{4(d+2)b^{2}} \underbrace{\left\| u_{\bot} \right\|_{\tilde{\mathcal{W}}^{0,s}}^{2}}_{\leq  \frac{2}{d} \|u_{\bot}\|_{\tilde{\cal W}^{1,s}}} .$$
Multiplied by $2(d+2)b^2$, it becomes
$$ \left\| u_{\bot} \right\|_{\tilde{\mathcal{W}}^{1,s}}^2 \leq \frac{(d+2)b^4}{\varepsilon_0^2}  \left\| f_0 \right\|_{\tilde{\mathcal{W}}^{0,s}}^{2} + \varepsilon_0^{2}(d+2)\left\| u_{0} \right\|_{\tilde{\mathcal{W}}^{0,s}}^{2} + \frac{2(d+2)^2b^4}{d}\|f_{\bot}\|_{\tilde{\cal W}^{0,s}}^2. $$
    By choosing $\varepsilon_0=\sqrt{\frac{1}{2(d+2)}}\varepsilon$, $\varepsilon\in ]0,1]\subset ]0,\sqrt{d}]$, we obtain
    $$
     \left\| u_{\bot} \right\|_{\tilde{\mathcal{W}}^{1,s}}^2\leq \frac{2(d+2)^2b^4}{\varepsilon^2}\left\| f \right\|_{\tilde{\mathcal{W}}^{0,s}}^{2} +\varepsilon^2\left\| u_{0} \right\|_{\tilde{\mathcal{W}}^{0,s}}^{2} \leq
     \left(\frac{2(d+2)b^2}{\varepsilon}\left\| f \right\|_{\tilde{\mathcal{W}}^{0,s}} +\varepsilon\left\| u_{0} \right\|_{\tilde{\mathcal{W}}^{0,s}}\right)^2.
    $$
  \end{proof}

  \begin{lemma}\label{le:IntermediateProposition001}
  There exist two constants $\tilde{C}, C_{R,s}\geq 1$, which are respectively uniform and $s$-dependent, $s\in \mathbb{R}$, such that $C_{R,s}\max(Ab,b,A^{-1})\leq 1 $ implies
  \begin{equation}
    \label{eq:estimateOnU0}
    \frac{ 1}{\tilde{C}}\left\| u_{0} \right\|_{\tilde{\mathcal{W}}^{0,s+1}} \leq \langle b |\lambda|^{1/2}\rangle^2( b + \frac{1}{A} ) \left\| f \right\|_{\tilde{\mathcal{W}}^{0,s}} + b^{2}\left\| f \right\|_{\tilde{\mathcal{W}}^{0,s+1}}\,,
  \end{equation}
  for all $u\in \mathcal{S}(X^h,\mathcal{E}_{\pm}^h)$.
\end{lemma}

\begin{proof}
  We start with the proof in the $s=0$ case. \\
  Let $\pi_{i,\pm}$ denote the spectral projection on the eigenspace of the operator $\alpha_{\pm}$ associated with the eigenvalue $i\in \{0,1,2\}$. By projecting the equation $ f = \big(P_{\pm,b} + \frac{1}{b}R_{1,\bot,h} +A^2\pi_{0,\pm}-i\lambda\big)u $ on $\mathrm{Ran}(\pi_{1,\pm})$, we obtain
  $$ (\frac{1}{b^{2}}-i\lambda) u_{1} \mp \frac{1}{b} \pi_{1}\nabla^{\mathcal{E}_{\pm},h}_{\mathcal{Y}}(u_{0} + u_{2}) + \frac{1}{b}\pi_{1}R_{1,\bot,h}u = f_{1} ,$$
  where $u_i = \pi_{i,\pm}u$ and $f_i=\pi_{i,\pm}f$ for $i\in \{0,1,2\}$.
  By isolating $\pi_{1}\nabla^{\mathcal{E}_{\pm},h}_{\mathcal{Y}}u_{0} = \nabla^{\mathcal{E}_{\pm},h}_{\mathcal{Y}}u_{0}$, it gives:
  $$  \nabla^{\mathcal{E}_{\pm},h}_{\mathcal{Y}}u_{0}= bf_{1} - (\frac{1}{b}-ib\lambda)u_{1} - \pi_{1}\nabla^{\mathcal{E}_{\pm},h}_{\mathcal{Y}}u_{2} - \pi_{1}R_{1,\bot}u . $$
  An upper bound of $\left\| \nabla^{\mathcal{E}_{\pm},h}_{\mathcal{Y}}u_{0} \right\|_{L^{2}}$ is thus given by
  \begin{equation} \label{eq:intermediateInquality003}
      \left\| \nabla^{\mathcal{E}_{\pm},h}_{\mathcal{Y}}u_{0} \right\|_{L^{2}} \leq b\left\| f_{1} \right\|_{L^{2}}  + \frac{1}{b}\langle b|\lambda|^{1/2}\rangle^2 \left\| u_{1} \right\|_{L^{2}} + \left\| u_{2} \right\|_{\mathcal{W}^{0,1}} + \left\| R_{1,\bot}u \right\|_{L^{2}}.
  \end{equation}
  We now use the inequality~\eqref{eq:estimateOnUOrthogonal} with the two regularity exponents $s=0$ and $s=1$  and with different values of $\varepsilon\in ]0,1]$\,. It makes sense under the following constraints 
  $$\max(C_{R,0},C_{R,1})\max(Ab,b,A^{-1})\leq 1 \,. $$
 We obtain
  \begin{itemize}
  \item $\left\| u_{1} \right\|_{L^{2}} \leq \frac{2}{d} \|\mathcal{O}^{1/2}u_1\|_{L^2}\leq 2 \left\| u_{\bot} \right\|_{\tilde{\mathcal{W}}^{1,0}} \leq \frac{4(d+2)b^{2}}{\varepsilon_{0}} \left\| f \right\|_{L^{2}} + 2\varepsilon_{0}\left\| u_{0} \right\|_{L^{2}}  $ \,;
  \item $\left\| u_{2} \right\|_{\tilde{\mathcal{W}}^{0,1}} \leq 2 \left\| u_{\bot} \right\|_{\tilde{\mathcal{W}}^{1,1}} \leq  \frac{4(d+2)b^{2}}{\varepsilon_{1}} \left\| f \right\|_{\tilde{\mathcal{W}}^{0,1}} + 2\varepsilon_{1} \left\| u_{0} \right\|_{\tilde{\mathcal{W}}^{0,1}}$ \, ;
  \item $\left\| R_{1,\bot}u \right\|_{L^{2}} \leq C_{R,0}' \left\| u \right\|_{\tilde{\mathcal{W}}^{1,0}} \leq  C_{R,0}' (\left\| u_{\bot} \right\|_{\tilde{\mathcal{W}}^{1,0}} + \sqrt{\frac{d}{2}} \left\| u_{0} \right\|_{L^{2}})\leq C_{R,0}'(\frac{2(d+2)b^{2}}{\varepsilon_{2}} \left\| f \right\|_{L^{2}} + (\sqrt{\frac{d}{2}}+\varepsilon_{2})\left\| u_{0} \right\|_{L^{2}}^{2})$ \, ;
  \end{itemize}
  where $C_{R,0}' = \sup_{h\in ]0,1]}\|R_{1,\bot,h}\|_{\mathcal{L}(\tilde{\cal W}^{1,0};L^2)}<\infty$. The upper bound \eqref{eq:intermediateInquality003} becomes
  \begin{align*}
      \left\| \nabla_{\mathcal{Y}}^{\mathcal{E}_{\pm},h}u_{0} \right\|_{L^{2}}
      &\leq \left(b + \frac{4(d+2)b}{\varepsilon_{0}} \langle b |\lambda|^{1/2}\rangle^2 + \frac{C_{R,0}' 2(d+2) b^{2}}{\varepsilon_{2}} \right)\left\| f \right\|_{L^{2}} 
      \\
      &\qquad + \left( \frac{2\varepsilon_{0}}{b}\langle b |\lambda|^{1/2}\rangle^2 + C_{R,0}'(\sqrt{\frac{d}{2}}+\varepsilon_{2}) \right)\left\| u_{0} \right\|_{L^{2}} + \frac{4(d+2) b^{2}}{\varepsilon_{1}} \left\| f \right\|_{\tilde{\mathcal{W}}^{0,1}} + 2\varepsilon_{1} \left\| u_{0} \right\|_{\tilde{\mathcal{W}}^{0,1}}\,. 
      \end{align*}
  On $\mathcal{S}(X^h;\mathcal{E}_{\pm}^h)\cap \mathrm{Ran}\,\pi_{0,\pm}=\mathcal{S}(X^h;\mathcal{E}_{\pm}^h)\cap \ker{\alpha}_{\pm}$, Lemma~\ref{le:nabYpi0} provides the equivalence
  $$ \frac{1}{\tilde{C_{0}}} \left\| u_{0} \right\|_{\tilde{\mathcal{W}}^{0,1}} \leq  \left\| \nabla_{\mathcal{Y}}^{\mathcal{E}_{\pm},h}u_{0} \right\|_{L^{2}} + \left\| u_{0} \right\|_{L^{2}} \leq \tilde{C}_0 \left\| u_{0} \right\|_{\tilde{\mathcal{W}}^{0,1}} $$
  for some uniform $\tilde{C_0}\geq 1$ so that
  \begin{align*}
      \left(\frac{1}{\tilde{C_{0}}}-2\varepsilon_1\right) \left\| u_{0} \right\|_{\tilde{\mathcal{W}}^{0,1}} 
      &\leq \left(b + \frac{4(d+2)b}{\varepsilon_{0}}\langle b |\lambda|^{1/2}\rangle^2 + \frac{C_{R,0}' 2(d+2) b^{2}}{\varepsilon_{2}} \right)\left\| f \right\|_{L^{2}} 
      \\
      &\qquad + \left(1+ \frac{2\varepsilon_{0}}{b}\langle b |\lambda|^{1/2}\rangle^2 + C_{R,0}'(\sqrt{\frac{d}{2}}+\varepsilon_{2}) \right)\left\| u_{0} \right\|_{L^{2}} + \frac{4(d+2) b^{2}}{\varepsilon_{1}} \left\| f \right\|_{\tilde{\mathcal{W}}^{0,1}} \,. 
  \end{align*}
  The integration by parts inequality~\eqref{eq:L2L2ResolvantEstimate} of Proposition~\ref{pr:firstEstimate} says
  $$
  \left\| u_{0} \right\|_{L^{2}}\leq \left\| u \right\|_{L^{2}}\leq\frac{2}{A^2}\left\|f\right\|_{L^{2}}\,.
  $$
    We have proved that $ (\frac{1}{\tilde{C_0}} - 2 \varepsilon_1) \|u_{0}\|_{\tilde{\cal W}^{0,1}}$ is less than
  \begin{multline*}
    \left(b + \frac{4(d+2)b}{\varepsilon_0}\langle b |\lambda|^{1/2}\rangle^2 + \frac{2(d+2)C_{R,0}'b^2}{\varepsilon_{2}} + \frac{2}{A^2}\big( 1 + \frac{2\varepsilon_0}{b}\langle b |\lambda|^{1/2}\rangle^2 + C_{R,0}'(\sqrt{\frac{d}{2}}+\varepsilon_2)\big)\right) \|f\|_{L^2}
    \\+ \frac{4(d+2)b^2}{\varepsilon_1}\|f\|_{\tilde{\cal W}^{0,1}}\,,
  \end{multline*}
and we choose 
$$
\varepsilon_{0} = \varepsilon_2=\sqrt{d+2}\,b A \leq C_{R,0}Ab \leq 1\quad,\quad \varepsilon_{1}=\frac{1}{4\tilde{C_{0}}}\leq 1\,.
$$
This implies that $\frac{1}{2 \tilde{C_0}} \| u_0 \|_{\tilde{\cal W}^{0,1}}$ is bounded by
  \begin{multline*}
    \left( b + \frac{2\sqrt{d+2}}{A}\langle b |\lambda|^{1/2}\rangle^2 + \frac{2\sqrt{d+2}C_{R,0}'b}{A} + \frac{2}{A^2} + \frac{2\sqrt{d+2}}{A}\langle b |\lambda|^{1/2}\rangle^2+\frac{\sqrt{2d} C_{R,0}'}{A^2 }  + \frac{2\sqrt{d+2}C_{R,0}'b}{A}\right) \|f\|_{L^2}
    \\  + 16\tilde{C_0}(d+2)b^2\|f\|_{\tilde{\cal W}^{0,1}} \,
  \end{multline*}
  and
$$\frac{1}{2 \tilde{C_0}} \| u_0 \|_{\tilde{\cal W}^{0,1}} \leq
  \left( b + \frac{4\sqrt{d+2}}{A}\langle b |\lambda|^{1/2}\rangle^2 + \frac{4\sqrt{d+2}C_{R,0}'b}{A} + \frac{2}{A^2} + \frac{\sqrt{2d} C_{R,0}'}{A^2 }\right) \|f\|_{L^2} + 16\tilde{C_0}(d+2)b^2\|f\|_{\tilde{\cal W}^{0,1}} \,.
$$
  The condition $ C'_{R,0}\max(  b ,A^{-1})\leq 1 $ ensures 
  $$ \frac{1}{2\tilde{C_0}} \|u_{0}\|_{\tilde{\cal W}^{0,1}} \leq \big(b+ \frac{1}{A}( 6\sqrt{d+2} +\frac{2}{A} + \sqrt{2d} ) \big) \langle b |\lambda|^{1/2}\rangle^2\|f\|_{L^2} + 16\tilde{C_0}(d+2)b^2  \|f\|_{\tilde{\cal W}^{0,1}}\,.$$
  We conclude the proof of the case $s=0$ by choosing new values of the constants $C_{R,0}$ and $\tilde{C}$ according to 
  \begin{eqnarray*}
  C_{R,0,new} &=& \max( C_{R,0}' , C_{R,0,old} , C_{R,1} )\,,\\
  \sqrt{\tilde{C}} &=& \max( 2\tilde{C_0} , 6\sqrt{d+2} + \sqrt{2d} + 2 , 16 \tilde{C}_0(d+2) ) \geq 1\,.
 \end{eqnarray*}
  For general $s\in \mathbb{R}$, writing $ (W^2_\theta)^{s/2}f = (W^2_\theta)^{s/2}\big(P_{\pm,b} + \frac{1}{b}R_{1,\bot,h} +A^2\pi_{0,\pm}\big)(W^2_\theta)^{-s/2}(W^2_\theta)^{s/2}u$ in the form
  $$
   (W^2_\theta)^{s/2}f=\big(P_{\pm,b} + \frac{1}{b}R_{1,\bot,h}^s +A^2\pi_{0,\pm}\big)(W^2_\theta)^{s/2}u
  $$
  reduces the problem to the case $s=0$ with $R_{1,\bot,h}$ replaced again by 
  $$
  R_{1,\bot,h}^s= (W^2_\theta)^{s/2}R_{1,\bot,h}(W^2_\theta)^{-s/2}+(W^2_\theta)^{s/2}\nabla_{\mathcal{Y}}^{\mathcal{E}_{\pm},h}(W^2_\theta)^{-s/2}-\nabla_{\mathcal{Y}}^{\mathcal{E}_{\pm},h}\,.
  $$
\end{proof}

\begin{lemma}\label{le:intermediateCorollary000}
 There exist two constants $\tilde{C}, C_{R,s}\geq 1$, which are respectively uniform and $s$-dependent, $s\in \mathbb{R}$, such that $C_{R,s}\max(Ab,b,A^{-1})\leq 1$ implies
  \begin{equation}
    \label{eq:estimateOnUOrhotgonalInW11}
    \frac{1}{\tilde{C}}\left\| u_{\bot} \right\|_{\tilde{\mathcal{W}}^{1,s+1}} \leq  \langle b |\lambda|^{1/2}\rangle^2
    \left(b+\frac{1}{A}\right)\left\| f \right\|_{\tilde{\mathcal{W}}^{0,s}} + b^2 \left\| f \right\|_{\tilde{\mathcal{W}}^{0,s+1}}
  \end{equation}
 for all $u\in \mathcal{S}(X^h;\mathcal{E}_{\pm}^h)$.
\end{lemma}

\begin{proof}
  Apply the inequality~\eqref{eq:estimateOnUOrthogonal} with $s$ replaced by $s+1$ and $\varepsilon=1$:
  $$ \left\| u_{\bot} \right\|_{\tilde{\mathcal{W}}^{1,s+1}} \leq 2(d+2)b^{2}\left\| f \right\|_{\tilde{\mathcal{W}}^{0,s+1}} + \left\| u_{0} \right\|_{\tilde{\mathcal{W}}^{0,s+1}} \,.$$
  With~\eqref{eq:estimateOnU0} we deduce
  {
  $$ \left\| u_{\bot} \right\|_{\tilde{\mathcal{W}}^{1,s+1}} \leq b^{2}\left(2(d+2) + \tilde{C} \right) \left\| f \right\|_{\tilde{\mathcal{W}}^{0,s+1}} + \tilde{C}\left(b+\frac{1}{A}\right)\langle b |\lambda|^{1/2}\rangle^2 \left\| f \right\|_{\tilde{\mathcal{W}}^{0,s}}\,. $$
}
  Finally choose $\tilde{C}_{new} = 2(d+2) + \tilde{C}_{old}$ and $C_{R,s,new}\geq\max(C_{R,s, Lemma~\ref{le:intermediateProposition002}}, C_{R,s, Lemma~\ref{le:IntermediateProposition001}})$\,.
\end{proof}

We now decompose the equation
$$
f=(P_{\pm,b} + \frac{1}{b}R_{1,\bot,h}+A^2 \pi_{0,\pm}-i\lambda)u
$$
into terms which adapt the low and high-frequency analysis of \cite{ReTa}: The frequency truncations are actually replaced  by spectral truncations associated with $(W^2_{\theta})$\,.\\
Let us set:
\begin{itemize}
\item $f_{L} = \mathds{1}_{\{W_{\theta}^{2} \leq \frac{1}{b^{2}}\}}f$ the orthogonal projection of $f$  on the low lying spectral part of $(W^2_\theta)$\,;
\item $f_{H} = \mathds{1}_{\{W_{\theta}^{2}> \frac{1}{b^{2}}\}}f$ the spectral projection of $f$ corresponding to high energies of  $(W^2_\theta)$\,;
\item $u_{L}=(P_{\pm,b} + \frac{1}{b}R_{1,\bot,h}+ A^2 \pi_{0,\pm}-i\lambda)^{-1}f_{L}$ the preimage of $f_{L}$ by $P_{\pm,b} + \frac{1}{b}R_{1,\bot,h}+ A^2 \pi_{0,\pm}-i \lambda$\,;
  \item $u_{H} = (P_{\pm,b} + \frac{1}{b}R_{1,\bot,h}+ A^2 \pi_{0,\pm}-i\lambda)^{-1}f_{H}$ the preimage of $f_{H}$ by $P_{\pm,b} + \frac{1}{b}R_{1,\bot,h}+ A^2 \pi_{0,\pm}-i\lambda$\,.
\end{itemize}

\begin{lemma}
\label{le:IntermediateProposition000}
Under the condition $C_{R,s}\max(Ab,b,A^{-1})\leq 1 $ with constants $\tilde{C}, C_{R,s}\geq 1$, which are respectively uniform and $s$-dependent, $s\in \mathbb{R}$, the inequalities
  \begin{eqnarray}
    \left\| u_{H} \right\|_{\tilde{\mathcal{W}}^{0,s+1}} & \leq &  \frac{\tilde{C}b\langle b |\lambda|^{1/2}\rangle^2}{A} \left\| f_H \right\|_{\tilde{\mathcal{W}}^{0,s+1}},\label{eq:IntermediateEstimate001} \\
              \left\| u_{L} \right\|_{\tilde{\mathcal{W}}^{0,s+1}} & \leq & \frac{\tilde{C}\langle b |\lambda|^{1/2}\rangle^2}{A} \left\| f_{L} \right\|_{\tilde{\mathcal{W}}^{0,s}} \label{eq:IntermediateEstimate002}.
  \end{eqnarray}
  hold
  for all $u\in \mathcal{S}(X^h;\mathcal{E}_{\pm}^h)$.
  \end{lemma}
\begin{proof}
  By the triangular inequality and the continuity of the inclusion $\mathcal{W}^{1,s+1} \hookrightarrow \mathcal{W}^{0,s+1}$\,, we know
  $$\left\| u \right\|_{\tilde{\mathcal{W}}^{0,s+1}} \leq \left\| u_{0} \right\|_{\tilde{\mathcal{W}}^{0,s+1}} +\sqrt{\frac{2}{d}} \left\| u_{\bot} \right\|_{\tilde{\mathcal{W}}^{1,s+1}}.$$
  Let us fix
  $$\tilde{C}\geq \tilde{C}_{Lemma~\ref{le:IntermediateProposition001}} + \sqrt{\frac{2}{d}} \tilde{C}_{Lemma~\ref{le:intermediateCorollary000}} \quad , \quad C_{R,s} \geq \max(C_{R,s+1, Lemma~\ref{le:IntermediateProposition001}} , C_{R,s+1, Lemma~\ref{le:intermediateCorollary000}}) \,.$$
  We can  now apply inequality~\eqref{eq:estimateOnU0} to $\left\| u_{0} \right\|_{\tilde{\mathcal{W}}^{0,s+1}} $ and inequality~\eqref{eq:estimateOnUOrhotgonalInW11}  to $\left\| u_{\bot} \right\|_{\tilde{\mathcal{W}}^{1,s+1}}$ with 
  \begin{equation}
    \label{eq:IntermediateEstimate000}
    \frac{1}{\tilde{C}}\left\| u \right\|_{\tilde{\mathcal{W}}^{0,s+1}} \leq \langle b |\lambda|^{1/2}\rangle^2(b+\frac{1}{A})\left\| f \right\|_{\tilde{\mathcal{W}}^{0,s}} + b^{2} \left\| f \right\|_{\tilde{\mathcal{W}}^{0,s+1}}.
  \end{equation}
 The functional calculus gives
  $$
  \left\lbrace\begin{array}{ccl}
    \left\| f_{L} \right\|_{\tilde{\mathcal{W}}^{0,s+1}} & \leq & \frac{1}{b}\left\| f_{L} \right\|_{\tilde{\mathcal{W}}^{0,s}}, \\
    \left\| f_{H} \right\|_{\tilde{\mathcal{W}}^{0,s}} & \leq & b \left\| f_{H} \right\|_{\tilde{\mathcal{W}}^{0,s+1}}.
  \end{array}\right.
  $$
  By combining \eqref{eq:IntermediateEstimate000} for $u=u_{H}$ and $u =u_{L}$  respectively, we deduce 
  $$\left\lbrace
    \begin{array}{ccl}
      \left\| u_{L} \right\|_{\tilde{\mathcal{W}}^{0,s+1}} & \leq & 2\tilde{C}\langle b |\lambda|^{1/2}\rangle^2(b+\frac{1}{A})\left\| f_{L} \right\|_{\tilde{\mathcal{W}}^{0,s}},\\
      \left\| u_{H} \right\|_{\tilde{\mathcal{W}}^{0,s+1}}& \leq & 2\tilde{C}b\langle b |\lambda|^{1/2}\rangle^2( b +\frac{1}{A} ) \left\| f_{H} \right\|_{\tilde{\mathcal{W}}^{0,s+1}}.
    \end{array}
  \right.$$
 Since we assumed  $ C_{R,s}Ab\leq 1$ it suffices to take $ \tilde{C}_{new} = 4\tilde{C}_{old}$\,.
\end{proof}

\begin{lemma} \label{le:intermediateProposition000}
Under the condition $C_{R,s}\max(Ab,b,A^{-1}) \leq 1 $ with constants $\tilde{C},  C_{R,s}\geq 1$, which are respectively uniform and $s$-dependent, $s\in \mathbb{R}$, the inequalities  
  $$\left\lbrace
    \begin{array}{ccl}
      \left\| u_{L} \right\|_{\tilde{\mathcal{W}}^{0,s+\frac{2}{5}}} & \leq & \frac{\tilde{C}\langle b |\lambda|^{1/2}\rangle^{\frac{4}{5}}}{A^{\frac{8}{5}}} \left\| f_{L} \right\|_{\tilde{\mathcal{W}}^{0,s}}, \\
      \left\| u_{H} \right\|_{\tilde{\mathcal{W}}^{0,s+\frac{2}{5}}} & \leq & \frac{\tilde{C}\langle b |\lambda|^{1/2}\rangle^{4/5}}{A^{\frac{8}{5}}} \left\| f_H \right\|_{\tilde{\mathcal{W}}^{0,s}}
    \end{array}
  \right.$$
  hold for all $u\in \mathcal{S}(X^h;\mathcal{E}_{\pm}^h)$.
\end{lemma}
\begin{proof}
  When we interpolate between integration by parts inequality~\eqref{eq:L2L2ResolvantEstimate} 
  $$ \left\| u_{L} \right\|_{\tilde{\mathcal{W}}^{0,s}}  \leq  \frac{2}{A^{2}}\left\| f_{L} \right\|_{\tilde{\mathcal{W}}^{0,s}}\,, $$
  and \eqref{eq:IntermediateEstimate002}
  $$
          \| u_{L} \|_{\tilde{\mathcal{W}}^{0,s+1}}  \leq  \frac{\tilde{C}\langle b |\lambda|^{1/2}\rangle^2}{A}\|f_{L}\|_{\tilde{\mathcal{W}}^{0,s}} \,.
          $$
  We obtain
  $$
   \left\| u_{L} \right\|_{\tilde{\mathcal{W}}^{0,s+\frac{2}{5}}}  \leq  \frac{2^{\frac{3}{5}} \tilde{C}^{\frac{2}{5}}\langle b |\lambda|^{1/2}\rangle^{\frac{4}{5}}}{A^{\frac{8}{5}}}\left\| f_{L} \right\|_{\tilde{\mathcal{W}}^{0,s}}.$$
  By doing the same with the subelliptic estimate \eqref{eq:MaximalSubEllipticEstimateWithDependanceOnb}, we get
  $$ \left\| u_{H} \right\|_{\tilde{\mathcal{W}}^{0,s+\frac{2}{3}}}  \leq  \frac{C}{A^{2}b^{\frac{2}{3}}}\left\| f_{H} \right\|_{\tilde{\mathcal{W}}^{0,s}}\,, $$
  and \eqref{eq:IntermediateEstimate001}
  $$
     \left\| u_{H} \right\|_{\tilde{\mathcal{W}}^{0,s}}  \leq  \frac{ \tilde{C} b \langle b |\lambda|^{1/2}\rangle^2}{A}\left\| f_H \right\|_{\tilde{\mathcal{W}}^{0,s}}\,,
  $$
and thus
  $$ \left\| u_{H} \right\| _{\tilde{\mathcal{W}}^{0,s+\frac{2}{5}}} \leq C^{\frac{3}{5}} \tilde{C}^{ \frac{2}{5}}\frac{\langle b |\lambda|^{1/2}\rangle^{\frac{4}{5}}}{A^{\frac{8}{5}}} \left\| f_H \right\|_{\tilde{\mathcal{W}}^{0,s}}.$$
  Take $\tilde{C}_{new} = \max( 2^{\frac{3}{5}} , C^{\frac{3}{5}} )\tilde{C}^{ \frac{2}{5}}$. The result follows for some large enough constants $C_s,C_{R,s}\geq 1$\,.
\end{proof}
\begin{proposition}\label{pr:subellipticEstimateForPerturbedLaplacianA}
   There exist two constants $C, C_{R,s}\geq 1$, which are respectively uniform and $s$-dependent, $s\in \mathbb{R}$, such that $C_{R,s} \max(Ab,b,A^{-1})\leq 1 $ implies
  \begin{eqnarray} \label{eq:subellipticEstimateWith2/5}
    C \left\| (P_{\pm,b} + \frac{1}{b}R_{1,\bot,h}+A^2 \pi_{0,\pm}-z)u \right\|_{\tilde{\mathcal{W}}^{0,s}}  \hspace{-0.5cm}& \geq & A^{2}\left\| u \right\|_{\tilde{\mathcal{W}}^{0,s}} + A^{2} \left\| \mathcal{O} u \right\|_{\tilde{\mathcal{W}}^{0,s}} + A^{2}b\left\| (\nabla_{\mathcal{Y}}^{\mathcal{E}_{\pm},h}-ib\lambda)u \right\|_{\tilde{\mathcal{W}}^{0,s}} \nonumber\\
                                                                                                                            &&+A^{2}b^{\frac{2}{3}} \left\| u \right\|_{\tilde{\mathcal{W}}^{0,s+\frac{2}{3}}} + \frac{A^{\frac{8}{5}}}{\langle b |\lambda|^{1/2}\rangle^{\frac{4}{5}}} \|u\|_{\tilde{\mathcal{W}}^{0,s+\frac{2}{5}}}\nonumber
    \\ && +A^{2}b|\lambda|^{1/2}\|u\|_{\tilde{\mathcal{W}}^{0,s}} + A^{\frac{16}{9}}\|u\|_{\tilde{\cal W}^{0,s+\frac{2}{9}}}
  \end{eqnarray}
  for all $u\in \mathcal{S}(X^h;\mathcal{E}_{\pm}^h)$ and all $z\in \mathbb{C}$ such that $\real z\leq\frac{A^{2}}{2}$
\end{proposition}
  
  \begin{proof}
   When $z=i\lambda$ the result follows from Proposition~\ref{pr:subellipticestimateWithbDependence} and the inequality
   $$ \|u\|_{\tilde{\mathcal{W}}^{0,s+\frac{2}{5}}} \leq \| u_{L} \|_{\tilde{\mathcal{W}}^{0,s+\frac{2}{5}}} + \|u_H\|_{\tilde{\mathcal{W}}^{0,s+\frac{2}{5}}} \leq \frac{\tilde{C}\langle b |\lambda|^{1/2}\rangle^{\frac{4}{5}}}{A^{\frac{8}{5}}} (\|f_L\|_{\tilde{\mathcal{W}}^{0,s}} + \|f_H\|_{\tilde{\mathcal{W}}^{0,s}} ) \leq \frac{\sqrt{2}\tilde{C}\langle b |\lambda|^{1/2}\rangle^{\frac{4}{5}}}{A^{\frac{8}{5}}} \|f\|_{\tilde{\mathcal{W}}^{0,s}}. $$\\
   For a general $z\in \mathbb{C}$\,, $\real z\leq \frac{A^{2}}{2}$\,, we use again the inequality 
   $$
   \left\| (P_{\pm,b} + \frac{1}{b}R_{1,\bot,h}+A^2 \pi_{0,\pm}-z)u \right\|_{\tilde{\mathcal{W}}^{0,s}}\geq \frac{1}{3}\left\| (P_{\pm,b} + \frac{1}{b}R_{1,\bot,h}+A^2 \pi_{0,\pm}-i\mathrm{Im}z)u \right\|_{\tilde{\mathcal{W}}^{0,s}}\,,
   $$
   which is a consequence of Proposition~\ref{pr:firstEstimate}.\\
   Taking $C=3\sqrt{2}(C_{old}+1)\tilde{C}$\,, the inequality
          \begin{eqnarray}\label{eq:subelliptic2/5}
   C \left\| (P_{\pm,b} + \frac{1}{b}R_{1,\bot,h}+A^2 \pi_{0,\pm}-z)u \right\|_{\tilde{\mathcal{W}}^{0,s}}  \hspace{-0.5cm}& \geq & A^{2}\left\| u \right\|_{\tilde{\mathcal{W}}^{0,s}} + A^{2} \left\| \mathcal{O} u \right\|_{\tilde{\mathcal{W}}^{0,s}} + A^{2}b\left\| (\nabla_{\mathcal{Y}}^{\mathcal{E}_{\pm},h}-ib\lambda)u \right\|_{\tilde{\mathcal{W}}^{0,s}}   \nonumber\\ 
       & &+A^{2}b^{\frac{2}{3}} \left\| u \right\|_{\tilde{\mathcal{W}}^{0,s+\frac{2}{3}}}+ \frac{A^{\frac{8}{5}}}{\langle b |\lambda|^{1/2}\rangle^{\frac{4}{5}}} \|u\|_{\tilde{\mathcal{W}}^{0,s+\frac{2}{5}}}\nonumber\\
       &&+A^{2}b|\lambda|^{1/2}\|u\|_{\tilde{\mathcal{W}}^{0,s}}
  \end{eqnarray}
is proved. The Inequality~\eqref{eq:subelliptic2/5} above implies
   
      \begin{eqnarray*}
         &&C\left\| (P_{\pm,b} + \frac{1}{b}R_{1,\bot,h}+A^2 \pi_{0,\pm}-z)u \right\|_{\tilde{\mathcal{W}}^{0,s}}\geq \frac{A^{\frac{8}{5}}}{\langle b |\lambda|^{1/2}\rangle^{\frac{4}{5}}}\|u\|_{\tilde{\cal W}^{0,s+\frac{2}{5}}} \\
       \text{and} && C\left\| (P_{\pm,b} + \frac{1}{b}R_{1,\bot,h}+A^2 \pi_{0,\pm}-z)u \right\|_{\tilde{\mathcal{W}}^{0,s}}\geq A^2\langle b |\lambda|^{1/2}\rangle \|u\|_{\tilde{\cal W}^{0,s}}\,.
       \end{eqnarray*}
     Interpolation gives
     \begin{eqnarray*}
       C\left\| (P_{\pm,b} + \frac{1}{b}R_{1,\bot,h}+A^2 \pi_{0,\pm}-z)u \right\|_{\tilde{\mathcal{W}}^{0,s}} \geq A^{\frac{16}{9}}\|u\|_{\tilde{\mathcal{W}}^{0,s+\frac{2}{9}}}.
     \end{eqnarray*}
     Finally, taking $C_{new} = 2 C$, the inequality~\eqref{eq:subellipticEstimateWith2/5} is satisfied.
  \end{proof}
  
\begin{proof}[End of the proof of Proposition~\ref{pr:subellipticEstimateA}]
We use
$$ 
B_{\pm,b,V^h}=[ P_{\pm,b} + \frac{1}{b}R_{1,\bot,h} ]+ R_{0,h} + R_{2,h}\,,
$$
where Proposition~\ref{pr:firstEstimate} (resp. Proposition~\ref{pr:subellipticEstimateForPerturbedLaplacianA}) provides the result about the maximal accretivity (resp. the desired subelliptic estimate) with some constants $C,C_{R,s}\geq 1$ when $R_{0,h}=0$ and $R_{2,h}=0$\,.\\
The property \eqref{eq:hypothesisForPerturbationRi} of $R_{0,h}$ and $R_{2,h}$
implies the uniform inequality
$$
\|R_{0,h}\|_{\mathcal{L}(\tilde{\mathcal{W}}^{0,s};\tilde{\mathcal{W}}^{0,s})}
+\|R_{2,h}\|_{\mathcal{L}(\tilde{\mathcal{W}}^{2,s};\tilde{\mathcal{W}}^{0,s})}
+\|R_{2,h}\|_{\mathcal{L}(\tilde{\mathcal{W}}^{1,s};\tilde{\mathcal{W}}^{-1,s})}
\leq C_{s}^{(1)}$$
For the accretivity of $\overline{B_{b,\pm,V^{h}}+A^{2}\pi_{0,\pm}-\frac{A^{2}}{2}}^{s}$ we write
\begin{eqnarray*}
\mathbb{R}\mathrm{e}~\langle u\,,\, (B_{\pm,b,V^{h}}+A^{2}\pi_{0}-\frac{A^{^{2}}}{2})u\rangle_{\tilde{\mathcal{W}}^{0,s}}&\geq &\frac{1}{(d+2)b^{2}}\|u_{\perp}\|_{\tilde{\mathcal{W}}^{1,s}}^{2}+
                                                                                                              \frac{3A^{2}}{4}\|u_{0}\|^{2}_{\tilde{\mathcal{W}}^{0,s}}-2C_{s}^{(1)}\|u\|^{2}_{\tilde{\mathcal{W}}^{1,s}}-\frac{A^{2}}{2}\|u\|_{\tilde{\mathcal{W}}^{0,s}}^{2}\\
                                                                                                      &&\geq  [\frac{1}{(d+2)b^{2}}-2C_{s}^{(1)}-\frac{A^{2}}{4}]\|u_{\perp}\|_{\tilde{\mathcal{W}}^{1,s}}^{2}
+[\frac{A^{2}}{4}-dC_{s}^{(1)}]\|u_{0}\|_{\tilde{\mathcal{W}}^{0,s}}\geq 0
\end{eqnarray*}
 where the inequality $A^{2}\geq 4dC_{s}^{(1)}$ holds when $C_{s}\geq 1$ is large enough\,, and where the inequality $\frac{1}{(d+2)b^{2}}-2C_{s}^{(1)}\geq \frac{A^{2}}{4}$ holds as soon as $C_{s}\max(Ab,b,\frac{1}{A})\leq 1$\,.\\
For the subelliptic estimate,  \eqref{eq:subellipticEstimateWith2/5} implies
    \begin{align*}
        \| (R_{0,h} + R_{2,h}) u \|_{\tilde{\mathcal{W}}^{0,s}} \leq C^{(1)}_{s} (\|u\|_{\tilde{\mathcal{W}}^{0,s}} + \|u\|_{\tilde{\mathcal{W}}^{2,s}}) &
        \leq \frac{CC^{(1)}_{s}}{A^2} \| (P_{\pm,b} + \frac{1}{b}R_{1,\bot,h}+A^2 \pi_{0,\pm}-i\lambda)u \|_{\tilde{\mathcal{W}}^{0,s}} 
        \\
        &\leq \frac{2 C C^{(1)}_{s}}{A^2}\| (P_{\pm,b} + \frac{1}{b}R_{1,\bot,h}+A^2 \pi_{0,\pm}-i\lambda)u \|_{\tilde{\mathcal{W}}^{0,s}}\,,
    \end{align*}
    when
    $$
    C_{R,s}\max(Ab,b,A^{-1})\leq 1\,.
    $$
    We conclude by choosing $A$ such that $\frac{2 C C^{(1)}_{s}}{A^2}\leq \frac{1}{2}\leq \frac{C}{2}$, which is ensured by the new value
    $$
    C_{s}=\max( C_{R,s}, \sqrt{4CC^{(1)}_{s}})\,,
    $$
    and provides the result of Proposition~\ref{pr:subellipticEstimateA} with the new value $C_{new}=2C$\,.
  \end{proof}

  \subsection{Second modified operator $\pi_{\perp,\pm}B_{\pm,b,V^h}\pi_{\perp,\pm}$} 
  \label{sec:firstModifiedOperator}

Another lower bound without the term $\frac{\kappa_b}{b^2}$ can be obtained after considering the block diagonal restriction of $B_{b,\pm,V^h}$ to $\mathrm{Ran}\,\pi_{\bot,\pm}=\ker(\alpha_{\pm})$\,. Actually it is not a subelliptic estimate because there is no gain of regularity. But the strong accretivity inequality with a lower bound $\frac{1}{b^{2}}$ will be strong enough for our applications.
Because $W^{2}_{\theta}$ commutes with $\pi_{0,\pm}$ and $\pi_{\bot,\pm}$ the same notions of closure, formal adjoint and adjoint, as in Definition~\ref{de:adjoints} can be considered with $\tilde{\mathcal{W}}^{0,s}(X^h;\mathcal{E}_{\pm}^h)$\,,\, $\mathcal{S}(X^h;\mathcal{E}_{\pm}^h)$ and $\mathcal{S}'(X^h;\mathcal{E}_{\pm}^h)$ replaced respectively by
$$
\pi_{\bot,\pm}\tilde{\mathcal{W}}^{0,s}(X^h;\mathcal{E}_{\pm}^h)\quad,\quad \pi_{\bot,\pm}\mathcal{S}(X^h;\mathcal{E}_{\pm}^h)\quad\text{and}\quad \pi_{\bot,\pm}\mathcal{S}'(X^h;\mathcal{E}_{\pm}^h)\,.
$$
All the discussion after Definition~\ref{de:adjoints} and the properties of Proposition~\ref{pr:perturbconjug} can be adapted with
\begin{eqnarray*}
  &&B_{\pm,b,V^h,\bot}=\pi_{\bot,\pm}B_{\pm,b,V^h}\pi_{\bot,\pm}\,,\\
  &&   (W^{2}_{\theta})^{s/2}B_{\pm,b,V^h,\bot}(W^{2}_{\theta})^{-s/2}=
     \pi_{\bot,\pm}\left[
     P_{\pm,b}+ R_{0,h}^{s}+\frac{1}{b}R_{1,\bot,h}^{s}+R_{2,h}^{s}\right]\pi_{\bot,\pm}\\
  &&B_{\pm,b,V^h,\bot}^{\prime, s}=\pi_{\bot,\pm}(B_{\pm,b,V^h}^{\prime,s})\pi_{\bot,\pm}=
      \pi_{\bot,\pm}\left[
     P_{\pm,b}^{\prime}+(R_{0,h}^{s})^{\prime}+\frac{1}{b}(R_{1,\bot,h}^{s})^{\prime}+(R_{2,h}^{s})^{\prime}\right]\pi_{\bot,\pm}\,.
\end{eqnarray*}

\begin{proposition}\label{pr:maximalAccretivityOfPiBPi}
There exists a constant $C_{s}\geq 1$ which depends on $s\in \mathbb{R}$ such that the condition $C_sb\leq1$ guaranties the following results.\\
  The densely defined operator $B_{\pm,b,V^h,\bot}=\pi_{\bot,\pm}B_{\pm,b,V^h}\pi_{\bot,\pm}$ in $ \pi_{\bot,\pm}\tilde{\cal W}^{0,s}(X^h;\mathcal{E}_{\pm}^h)$ with domain $\pi_{\bot,\pm}\mathcal{S}(X^h;\mathcal{E}_{\pm}^h)$ is essentially maximal accretive
  with the inequality
  \begin{equation}
 \forall u\in   \pi_{\bot,\pm}\mathcal{S}(X^h;\mathcal{E}_{\pm}^h)\,,\quad   \real \langle u\,,\,[\pi_{\bot,\pm} B_{\pm,b,V^h} \pi_{\bot,\pm}] u \rangle_{\tilde{\cal W}^{0,s}} \geq \frac{1}{12b^2} \| u \|_{\tilde{\cal W}^{0,s}}^2\,.
\end{equation}
Addtionally the resolvent of its closure $\overline{B_{\pm,b,V^h,\bot}}^{s}=\overline{[\pi_{\bot,\pm} B_{\pm,b,V^h} \pi_{\bot,\pm}]}^{\pi_{\bot,\pm}\tilde{\mathcal{W}^{0,s}}}$ satisfies
$$
\forall z\in \mathbb{C}, \mathbb{R}\mathrm{e}\,z\leq \frac{1}{24 b^2}\,,\quad \|(z -\overline{B_{\pm,b,V^h,\bot}}^s)^{-1}\|_{\mathcal{L}(\tilde{\mathcal{W}}^{0,s};\tilde{\cal W}^{0,s})} \leq 24 b^2\,.
$$
Finally the closure of the $\pi_{\bot,\pm}\tilde{\mathcal{W}}^{0,s}(X^h;\mathcal{E}_{\pm}^h)$ formal adjoint $B_{\pm,b,V^h,\bot}^{\prime,s}$ satisfies
$$
\overline{B_{\pm,b,V^h,\bot}^{\prime,s}\big|_{\mathcal{S}(X^h;\mathcal{E}_{\pm}^h)}}^{s}=B_{\pm,b,V^h,\bot}^{\prime, *}\,.
$$
\end{proposition}
For the proof, we  check firstly the accretivity in Proposition \ref{pr:accretivityOfPiBPi} and secondly the injectivity of the adjoint $B_{\pm,b,V^{h},\bot}^{*,s}=(\overline{B_{\pm,b,V^{h},\bot}})^{*,s}$ in Proposition \ref{pr:injectivityOfTheAdjointOfPiBPi}. The final statement above is just a consequence of the essential maximal accretivity.

\begin{proposition} \label{pr:accretivityOfPiBPi}
 There exists a constant $C_{s}\geq 1$ which depends on $s\in \mathbb{R}$ such that, under the assumption $C_sb\leq 1$, the  operator $B_{\pm,b,V^h,\bot}=\pi_{\bot,\pm}B_{\pm,b,V^h}\pi_{\bot,\pm}$ is accretive in $\pi_{\bot,\pm}\tilde{\mathcal{W}}^{0,s}(X^h;\mathcal{E}_{\pm}^h)$, with
  \begin{equation}\label{eq:intermediateInequality101}
      \forall u\in \pi_{\bot,\pm}\mathcal{S}(X^h;\mathcal{E}_{\pm}^h)\,,\quad \real \langle u\,,\, B_{\pm,b,V^h,\bot} u \rangle_{\tilde{\cal W}^{0,s}} \geq \frac{1}{2(d+2)b^2} \| u \|_{\tilde{\cal W}^{1,s}}^2 \geq \frac{d}{4(d+2)b^2} \| u \|_{\tilde{\cal W}^{0,s}}^2\geq \frac{1}{12 b^2} \| u \|_{\tilde{\cal W}^{0,s}}^2\,.
    \end{equation}
Hence its closure $\overline{B_{\pm,b,V^h,\bot}}^{s}$ is one to one with a closed range.
\end{proposition}

\begin{proof}
Take $u=u_{\bot}$, since $\pi_{\bot,\pm}\pi_{0,\pm} \pi_{\bot,\pm} = 0$ the inequality \eqref{eq:IppBis} holds at least with an arbitrary $A>0$ that satisfies the required condition. Integration by parts gives
\begin{align*}
    \real \langle B_{\pm,b,V^h} u \, , \, u \rangle_{\tilde{\cal W}^{0,s}} & = \frac{1}{b^2} \real \langle (P_{\pm,b}+\frac{1}{b}R_{1,\bot,h}) u \, , \, u \rangle_{\tilde{\cal W}^{0,s}} + \real \langle (R_{0,h} + R_{2,h}) u \, , \, u \rangle_{\tilde{\cal W}^{0,s}} \\
    & \geq \frac{1}{(d+2)b^2} \|u\|_{\tilde{\cal W }^{1,s}}^2 - \|R_{0,h}\|_{\mathcal{L}(\tilde{\cal W}^{0,s};\tilde{\cal W}^{0,s})} \|u\|_{\tilde{\cal W}^{0,s}}^2 - \|R_{2,h}\|_{\mathcal{L}(\tilde{\cal W}^{1,s}; \tilde{\cal W}^{-1,s})} \| u \|_{\tilde{\cal W}^{1,s}}^2 \,.
\end{align*}
The condition \eqref{eq:hypothesisForPerturbationRi} ensures $R_{0,h}\in \mathcal{L}(\tilde{\cal W}^{0,s};\tilde{\cal W}^{0,s})$ and by interpolation $R_{2,h} \in \mathcal{L}(\tilde{\cal W}^{1,s} ; \tilde{\cal W}^{-1,s})$. Remember that $ \frac{d}{2} \|u\|_{\tilde{\cal W}^{0,s}}^2 \leq \| u \|_{\tilde{\cal W}^{1,s}}^2 $. Finally
\begin{eqnarray*}
  \real \langle B_{\pm,b,V^h} u \, , \, u \rangle_{\tilde{\cal W}^{0,s}} & \geq & \frac{1}{b^2}\left( \frac{1}{(d+2)} - \frac{2\|R_{0,h} \|_{\mathcal{L}(\tilde{\cal W}^{0,s};\tilde{\cal W}^{0,s})}b^2}{d} - \|R_{2,h}\|_{\mathcal{L}(\tilde{\cal W}^{1,s};\tilde{\cal W}^{-1,s})}b^2 \right) \|u\|_{\tilde{\cal W}^{1,s}}^2 \\
                                                                         & \geq & \frac{1}{b^2}\left( \frac{1}{d+2} - (\frac{2}{d}+1)\tilde{C}_R b^2 \right) \| u \|_{\tilde{\cal W}^{1,s}}
\end{eqnarray*}
with $ \tilde{C}_{R} = \sup_{h\in ]0,1]} \left[\| R_{0,h} \|_{\mathcal{L}(\tilde{\cal W}^{0,s};\tilde{\cal W}^{0,s})} + \|R_{2,h}\|_{\mathcal{L}(\tilde{\cal W}^{1,s};\tilde{\cal W}^{-1,s})}\right] $.\\
Taking $C_{s} > \frac{\sqrt{2\tilde{C}_R}}{d}(d+2)$  and $C_{R,s}$ as in Proposition \ref{pr:firstEstimate}\,, we have
$$ 
\frac{1}{d+2} - \frac{ 2 \|R_{0,h} \|_{\mathcal{L}(\tilde{\cal W}^{0,s};\tilde{\cal W}^{0,s})}b^2}{d} - \|R_{2,h}\|_{\mathcal{L}(\tilde{\cal W}^{1,s};\tilde{\cal W}^{-1,s})}b^2  \geq \frac{1}{2(d+2)}
$$
as soon as $ C_sb \leq 1$.
\end{proof}

\begin{proposition}\label{pr:injectivityOfTheAdjointOfPiBPi}
  There is a constant $C_{s}\geq 1$ which depends on  $s\in \mathbb{R}$, such that  the adjoint $B_{\pm,b,V^h,\bot}^{*,s}=(\overline{B_{\pm,b,V^h,\bot}}^{s})^{*,s}$ is one to one and therefore $\overline{B_{\pm,b,V^h,\bot}}^s$ is maximal accretive, as soon as $C_sb\leq 1$\,.
\end{proposition}

\begin{proof}
Assume $w\in \ker(B_{\pm,b,V^h,\bot}^{*,s})$ and let us prove $w=0$\,.
By setting $w=(W^2_\theta)^{-s/2}v$, the assumption is equivalent to $(W^2_\theta)^{s/2}B_{\pm,b,V^h,\bot}^{*,s}(W^2_\theta)^{-s/2}v=(W^2_\theta)^{-s/2}B_{\pm,b,V^h,\bot}^{*}(W^2_\theta)^{s/2}v=0$ and $v\in L^2(X^h,dqdp;\mathcal{E}_{\pm}^h)$.
The problem is reduced to
$$
\left.
\begin{array}{rcl}
\pi_{\bot,\pm}\left[P_{\pm,b}^* + (R_{0,h}^{s})^{\prime} + \frac{1}{b}(R_{1,\bot,h}^{s})^{\prime} + (R_{2,h}^{s})^{\prime}\right]\pi_{\bot,\pm}v=0 &\text{in}&\mathcal{S}'(X^h;\mathcal{E}^h_{\pm})\\
v\in L^2(X^h,dqdp;\mathcal{E}^h_{\pm})
\end{array}
\right\}
\Rightarrow v=0\,.
$$
We set $\tilde{B}_{\pm,b,V^h,\bot,s}=P_{\pm,b}^{\prime} + (R_{0,h}^{s})^{\prime} + \frac{1}{b}(R_{1,\bot,h}^{s})^{\prime} + (R_{2,h}^{s})^{\prime}$ which has the same form as $B_{\pm,b,V^h}$ with a changed sign before $\nabla_{\mathcal{Y}}^{\mathcal{E}_{\pm},h}$ and remainder terms
$R_{0,h},R_{1,\bot,h},R_{2,h}$ respectively replaced by the $s$-dependent versions $(R_{0,h}^{s})^{\prime},(R_{1,\bot,h}^{s})^{\prime},(R_{2,h}^{s})^{\prime}$\,.
In  particular $\pi_{\bot,\pm}\tilde{B}_{\pm,b,V^h,s}\pi_{\bot,\pm}$ is accretive on $\pi_{\bot,\pm}\mathcal{S}(X^h;\mathcal{E}_{\pm}^h)$ for the $L^{2}(X^h,dqdp;\mathcal{E}_{\pm}^h)$ scalar product, with the same lower bounds as in \eqref{eq:intermediateInequality101}, as soon as $C_s^1b\leq 1$ for some $C_{s}^{1}\geq 1$\,.\\
We take  two cut-off functions $\chi,\tilde{\chi}\in C_{0}^{\infty}(\mathbb{R},[0,1])$ such that
  \begin{itemize}
      \item $\chi$ and $\tilde{\chi}$ is equal to 1 near 0 \,,
      \item $\mathrm{supp}(\tilde{\chi}) \subset \chi^{-1} (\{1\})$ \,,
\end{itemize}and we recall that $f(W^2_\theta)$ is continuous from $\mathcal{S}'(X^h;\mathcal{E}_{\pm}^h )$ to $\mathcal{S}(X^h;\mathcal{E}_{\pm}^h)$ and commutes with $\pi_{\bot,\pm}$ and more generally with any function of $\alpha_{\pm}$ 
for all $f\in \mathcal{C}^{\infty}_0(\mathbb{R};\mathbb{C})$\,.
For $\varepsilon>0$ we set $v_{\varepsilon} = \chi(\varepsilon W^2_{\theta})v\in\mathcal{S}(X^h;\mathcal{E}_{\pm}^h)$. A straightforward computation shows
\begin{align}
\pi_{\bot,\pm} \tilde{B}_{\pm,b,V^h,s}\pi_{\bot,\pm} v_{\varepsilon} &  =  \pi_{\bot,\pm} [\tilde{B}_{\pm,b,V^h,s} ,\chi(\varepsilon W^2_{\theta})] \pi_{\bot,\pm}v \nonumber\\
 & =  \pi_{\bot,\pm} [\underbrace{ \pm \frac{1}{b} \nabla_{\mathcal{Y}}^{\mathcal{E}_{\pm},h} + (R_{0,h}^{s})^{\prime} + \frac{1}{b} (R_{1,\bot,h}^{s})^{\prime}+(R_{2,h}^{s})^{\prime}}_{= D }, \chi(\varepsilon W^2_{\theta}) ]\pi_{\bot,\pm}v, \label{eq:intermediateEquality100}
\end{align}
where $D$ is a differential operator in $\mathrm{OpS}^{3/2}_{\Psi}(Q^h;\mathrm{End}\,\mathcal{E}_{\pm}^h)$\,. The
Helffer-Sj\"ostrand formula for the commutator gives

\begin{equation*}
   [D,\chi(\varepsilon W_{\theta}^2)] = [D ,  W_{\theta}^2]W_{\theta}^{-2} \chi'(\varepsilon W_{\theta}^2) \varepsilon W_{\theta}^2 +r_{\varepsilon},
\end{equation*}
where 
$$
r_{\varepsilon} = -\frac{1}{2i\pi}\int_{\mathbb{C}} \partial_{\bar{z}}\tilde{\chi}(z)\varepsilon^2 (z-\varepsilon W_{\theta}^2)^{-1}[[D,W_{\theta}^2],W^{2}_{\theta}](z-\varepsilon W_{\theta}^2)^{-2}\, dz\wedge d\bar{z}\,.
$$
By pseudo differential calculus in $\mathrm{OpS}^{*}_{\Psi}(Q^h;\mathrm{End}\,\mathcal{E}_{\pm}^h)$, the commutator $[[D,W_{\theta}^2],W_{\theta}^2]$ is a pseudo differential operator of order $ \frac{3}{2}+(2-1)+(2-1) = \frac{7}{2}$. This implies that 
$ [[D,W_{\theta}^2],W_{\theta}^2](W_{\theta}^2)^{-\frac{7}{4}} $ is bounded. The inequality
$$
\|\varepsilon^{7/4}(W_{\theta}^2)^{-\frac{7}{4}}(z-\varepsilon W_{\theta}^2)^{-2} 
\|
\leq \sup_{\lambda>0} \left|\frac{\lambda^{7/8}}{|z-\lambda|}\right|^2\leq C\left[\sup_{\lambda>0}\frac{1+\lambda}{|z-\lambda|}\right]^2\leq C'\frac{\langle z\rangle^2}{|\mathrm{Im}\,z|^2}
$$
yields
\begin{equation}\label{eq:estimationOfError100}
\|r_{\varepsilon}\|_{\mathcal{L}(L^2;L^2)}\leq C_{R,s}\varepsilon^{1/4}\,,
\end{equation}
where again the constant $C_{R,s}$ is uniform with respect to $h\in ]0,1]$\,.\\
The obvious equality $ \chi'(t) = (1-\tilde{\chi}(t)) \chi'(t) $ for all $t \in \mathbb{R}$ implies 
\begin{equation} \label{eq:intermediateEquality101}
[D,\chi(\varepsilon W_{\theta}^2)] = [D , W_{\theta}^2]W_{\theta}^{-2} \chi'(\varepsilon W_{\theta}^2) \varepsilon W_{\theta}^2(1-\tilde{\chi}(\varepsilon W_{\theta}^2)) +r_{\varepsilon} .
\end{equation}
Let us consider the following scalar product by using \eqref{eq:intermediateEquality100} and \eqref{eq:intermediateEquality101}
\begin{eqnarray*}
    \real \langle  v_{\varepsilon} , \pi_{\bot,\pm} B_{\pm,b,V^h,s}  \pi_{\bot,\pm} v_{\varepsilon} \rangle_{L^2} & = & \real \langle v_{\varepsilon} \, , \, [D , W_{\theta}^2]W_{\theta}^{-2} \chi'(\varepsilon W_{\theta}^2) \varepsilon W_{\theta}^2(1-\tilde{\chi}(\varepsilon W_{\theta}^2)) v \rangle_{L^2} + \real\langle v_{\varepsilon} \, , \, r_{\varepsilon}v \rangle_{L^2} \\
    & = & \real  \langle \varepsilon W_{\theta}^2 \chi'(\varepsilon W_{\theta}^2) W_{\theta}^{-2}[ W_{\theta}^2,D'] v_{\varepsilon} \, , \, (1-\tilde{\chi}(\varepsilon W_{\theta}^2)) v \rangle_{L^2} + \real\langle v_{\varepsilon} \, , \, r_{\varepsilon}v \rangle_{L^2}.
\end{eqnarray*}
Applying the Cauchy-Schwarz inequality to the right-hand side of this equality yields
\begin{equation}\label{eq:intermediateInequality103}
    \real \langle  v_{\varepsilon} , \pi_{\bot,\pm} \tilde{B}^*_{\pm,b,V^h,s}  \pi_{\bot,\pm} v_{\varepsilon} \rangle_{L^2} \leq \| \varepsilon W_{\theta}^2 \chi'(\varepsilon W_{\theta}^2)  W_{\theta}^{-2}[W_{\theta}^2,D'] v_{\varepsilon} \|_{L^2} \| (1-\tilde{\chi}(\varepsilon W_{\theta}^2)) v  \|_{L^2} + \| v_{\varepsilon} \|_{L^2} \|r_{\varepsilon}v\|_{L^2}.
\end{equation}
By functional calculus the operator $ \varepsilon W_{\theta}^2 \chi'(\varepsilon W_{\theta}^2) $ is a bounded operator 
$$ \|\varepsilon W_{\theta}^2 \chi'(\varepsilon W_{\theta}^2)\|_{\mathcal{L}(L^2 ; L^2 )} \leq M=\sup_{t\in [0,+\infty[} |t \chi'(t)| < \infty \,,$$
while \eqref{eq:WYWmY} and \eqref{eq:hypothesisForPerturbationRi} ensure 
$$  W_{\theta}^{-2}[ W_{\theta}^2,D' ] = \mp\underbrace{W_{\theta}^{-2}[ W_{\theta}^2, \frac{1}{b} \nabla_{\cal Y}^{\mathcal{E}_{\pm},h} ]}_{\in \mathcal{L}(\tilde{\cal W}^{1,0};L^2)} + \underbrace{W_{\theta}^{-2}[ W_{\theta}^2, R_{0,h}^{s}+\frac{1}{b}R_{1,\bot,h}^{s}+R_{2,h}^{s} ]}_{\in \mathcal{L}(L^2;L^2)} \,.$$
The left-hand side of  \eqref{eq:intermediateInequality103} is thus bounded by
\begin{alignat*}{4}
&\real \langle  v_{\varepsilon} , \pi_{\bot,\pm} \tilde{B}_{\pm,b,V^h,s} \pi_{\bot,\pm} v_{\varepsilon} \rangle_{L^2} &\leq &\quad MC_{R,\mathcal{Y},s} \| v_{\varepsilon} \|_{\tilde{\cal W}^{1,0}} \| (1-\tilde{\chi}(\varepsilon W_{\theta}^2)) v \|_{L^2} + C_{R,s}\varepsilon^{1/4}\|v \|_{L^2}^2  \\
&&\leq &\quad \frac{1}{2}\varepsilon' MC_{R,\mathcal{Y},s,b} \| v_{\varepsilon} \|_{\tilde{\cal W}^{1,0}}^2+ \frac{MC_{R,\mathcal{Y},s,b}}{2\varepsilon'}\| (1-\tilde{\chi}(\varepsilon W_{\theta}^2)) v \|_{L^2}^2+ C_{R,s}\varepsilon^{1/4}\|v \|_{L^2}^2
\end{alignat*}
where
$$ C_{R,\mathcal{Y},s,b} =\frac{1}{b}\sup_{h\in]0,1]}\max( \| W_{\theta}^{-2}[ W_{\theta}^2, \nabla_{\cal Y}^{\mathcal{E}_{\pm},h} ] \|_{\mathcal{L}(\tilde{\cal W}^{1,0};L^2)} ,  W_{\theta}^{-2}[ W_{\theta}^2, R_{i=0,1,2;h}^{s} ]  \|_{\mathcal{L}(L^2;L^2)})\,.
$$
By Proposition~\ref{pr:accretivityOfPiBPi} and \eqref{eq:intermediateInequality101} applied now to $\pi_{\bot,\pm} \tilde{B}_{\pm,b,V^h,s}\pi_{\bot,\pm}$\,, the left hand-side is bounded from below by
$\frac{1}{2(d+2)b^2}\|v_\varepsilon\|_{\tilde{\mathcal{W}}^{1,0}}^2$ when $C_sb\leq 1$\,.\\
By choosing $\varepsilon'=\frac{1}{2(d+2)b^2 MC_{R,\mathcal{Y},s,b}}$\,, we obtain
$$ \frac{d}{8(d+2)b^2} \| v_{\varepsilon}\|_{L^2}^2 \leq \frac{1}{4(d+2)b^2} \| v_{\varepsilon}\|_{\tilde{\cal W}^{1,0}}^2 \leq 2(d+2)b^2M^2C_{R,\mathcal{Y},s,b}^2\|(1-\tilde{\chi}(\varepsilon W_{\theta}^2)) v\|_{L^2}^2 + C_{R,s}\varepsilon^{1/4} \| v \|_{L^2}^2 . $$
When $\varepsilon $ goes to $0$, the spectral theorem and the dominated convergence Theorem imply
$$
\lim_{\varepsilon\to 0} \| v_{\varepsilon}\|_{L^2}^2= \| v\|_{L^2}^2\quad\text{and}\quad \lim_{\varepsilon\to 0} \|1-\tilde{\chi}(\varepsilon W_{\theta}^2)) v\|_{L^2}^2= 0.
$$
We have proved $v=0$ and $B^{*,s}_{\pm,b,V^h,\bot}$ is one to one.

\end{proof}
\subsection{Final modifications with a frequency or spectral truncation}
\label{sec:QALV}
For $ \chi \in \mathcal{C}^{\infty}_0(\mathbb{R},[0,1]) $ such that 
$$ 
\mathrm{Supp}(\chi)\subset [-2,2] \quad \mathrm{and} \quad \chi(t)=1 \quad \mathrm{for} \quad t\in [-1,1]\,.$$
let $ Q_{A,L} = A^2 \pi_{0,\pm} \chi \big( \frac{2 W_{\theta}^2}{ (LA)^2} \big) \pi_{0,\pm} $ for $L,A\geq 1$\,, where we drop subscript $_{h}$ although it depends on $h\in]0,1]$\,,
and consider the operator 
$$ B_{\pm,b,V^{h}}+Q_{A,L} \,.
$$
Remember that the operator $U_{\pm,\theta}=U_{\pm,\theta,h}:L^2(Q^{h},d\mathrm{Vol}_{g^{h}}; \Lambda T^*Q^{h}\otimes F_\pm)\to \ker(\alpha_{\pm})=\ker(\alpha_{\pm,g^{h}})$\,, with $F_+ =Q\times\mathbb{C}$
and $F_{-}=Q\times\mathbb{C}\times\mathbf{or}_Q$\,, introduced in \eqref{eq:U+theta}\eqref{eq:U-theta}\eqref{eq:U+thetaexpl}\eqref{eq:U-thetaexp}\,, is unitary and satisfies
$$
2\pi_{0,\pm} (W^2_\theta)\pi_{0,\pm} =U_{\pm,\theta} (\underbrace{2C+C\frac{d^2}{2}}_{C_d}+H_0) U_{\pm,\theta}^{-1}
$$
where $H_0=H_{0,h}$ is the non negative elliptic, Laplace type operator, $H_0=-\sum_{j=1}^J \theta_j(hq)(\Delta_{Q,g^{h}})_{sc} \theta_j(hq)$ and  where $\Delta_{Q,g^{h}}$ is the Laplace Beltrami operator,  with a scalar realization in the orthonormal frame $(u^{I}_{j,g^{h}})_{I\subset \{1,\ldots,d\}}$ for every $j\in\{1,\ldots,J\}$\,.\\
Owing to the uniform estimates of $g^{h}$ and $V^{h}$ stated in Proposition~\ref{pr:scaling}
the operators   $(d^{Q^{h}}+d^{Q^{h},*^{g^{h}}})^2$ and 
the Witten Laplacian $\Delta_{V^{h},1}=(d^{Q^{h}}+d^{Q^{h},*^{g^{h}}})^2+|\nabla V_{h}(q)|^2 +(\mathcal{L}_{\nabla V^{h}}+\mathcal{L}_{\nabla V^{h}}^*)$ are elliptic operators in the classical space $\mathrm{OpS}^{2}(Q^{h};\mathcal{E}^{h}_{\pm})$ with the same scalar principal symbol as $H_{0}$ and with uniformly controlled lower order corrections.\\
By choosing the above constant $C\geq C_{g,V}\geq 1$ large enough and by setting again $C_{d}=2C+C\frac{d^{2}}{2}$\,, we deduce for every $s\in \mathbb{R}$
the equivalence of norms
$$
 \forall u\in H^{s}(Q^{h};\Lambda T^{*}Q^{h}\otimes F_{\pm})\,,\quad\left(\frac{\|u\|_{H^{s}}}{\|(C_{d}+\Delta_{V^{h},1})^{s/2}u\|_{L^{2}}}\right)^{\pm 1}=\left(\frac{\|(C_{d}+H_{0})^{s/2}u\|_{L^{2}}}{\|(C_{d}+\Delta_{V^{h},1})^{s/2}u\|_{L^{2}}}\right)^{\pm 1}\leq C_{s}\,.
 $$
 With the  operator  $U_{\pm,\theta}$ we also have
\begin{align*}
    \|U_{\pm,\theta}u\|_{\tilde{\mathcal{W}}^{s_{1},s_{2}}}&=\|(\mathcal{O})^{s_{1}/2}(W_{\theta})^{s/2}U_{\pm,\theta}u\|_{L^{2}}\\
    &=\left(\frac{d}{2}\right)^{s_{1}}
\|(C_{d}+H_{0})^{s_{2}/2}u\|_{L^{2}}\\
&=\left(\frac{d}{2}\right)^{s_{1}}\|u\|_{H^{s_{2}}}\asymp \|(C_{d}+\Delta_{V^{h},1})^{s_{2}/2}u\|_{L^{2}}\,.
\end{align*}
Because we aim at clarifying the relations between $\mathrm{Spec}(B_{\pm,b,V^{h}})$ and $\mathrm{Spec}\,(\Delta_{V^{h},1})=\mathrm{Spec}\,(\Delta_{(Q^{h},g^{h},V^{h}),1})=\mathrm{Spec}~(\Delta_{(Q,g,V,h)})$ (see Subsection~\ref{sec:scalings})\,, we consider   the two perturbations of $B_{\pm,b,V^{h}}$
\begin{align}
&Q_{A,L} =  A^2 \pi_{0,\pm}\circ \chi \big(\frac{2 W_{\theta}^2}{ (LA)^2} \big)\circ \pi_{0,\pm}= A^2 U_{\pm,\theta}\circ \chi\big(\frac{C_d +H_0}{(LA)^2}\big)\circ U_{\pm,\theta}^{-1} \label{eq:defQAL}\\
\text{and}&\qquad 
            Q_{A,L,V^{h}} =   A^2 U_{\pm,\theta}\circ \chi\big(\frac{C_d+\Delta_{V^{h},1}}{(LA)^2}\big)\circ U_{\pm,\theta}^{-1} \label{eq:defQALV}\,.
\end{align}
The comparison of $Q_{A,L}$ and $Q_{A,L,V^{h}}$ is easier to understant while staying on the base manifold $Q^{h}$ and we also use the following notations.
\begin{definition}
\label{de:tQALVh}~
The operators $\tilde{Q}_{A,L} : \mathcal{E}'(Q^h;\Lambda T^*Q^h\otimes F_{\pm})\to \mathcal{C}^{\infty}(Q^h;\Lambda T^*Q^h\otimes F_\pm)$ and $\tilde{Q}_{A,L,V^{h}}: \mathcal{E}'(Q^h;\Lambda T^*Q^h\otimes F_{\pm})\to \mathcal{C}^{\infty}(Q^h;\Lambda T^*Q^h\otimes F_\pm)$ are defined by
$$
\tilde{Q}_{A,L} = A^2\chi(\frac{C_d+H_0}{(LA)^2})\quad\text{and}\quad \tilde{Q}_{A,L,V^{h}} = A^2\chi(\frac{C_d+\Delta_{V^h,1}}{(LA)^2})\,.
$$
\end{definition}

These two operators are bounded as well as $A^{2}\pi_{0,\pm}-Q_{A,L}$ and $A^{2}\pi_{0,\pm}-Q_{A,L,V^{h}}$\,.
We will prove the following result.
\begin{proposition}\label{pr:subEllipticQALVh}
  There are constants $L\geq 1$\,, $C_{s}\geq 1$ and $C_{\chi,s}\geq 1$
  respectively uniform, $s\in \mathbb{R}$-dependent and $(\chi,s)$-dependent,  such that inequality
  $$ \frac{1}{4} \|\overline{(B_{\pm,b,V^{h}}+A^{2}\pi_{0,\pm}-z)}^{s}u\|_{\tilde{\cal W}^{0,s}} \leq \|\overline{(B_{\pm,b,V^{h}} + Q_{A,L,V^{h}}-z)}^{s} u \|_{\tilde{\cal W}^{0,s}} \leq \frac{9}{4} \|\overline{(B_{\pm,b,V^{h}}+A^{2}\pi_{0,\pm}-z)}^{s}u\|_{\tilde{\cal W}^{0,s}} $$
  holds for all $u\in D(\overline{B_{\pm,b,V^{h}}}^{s})$ and all $z\in\mathbb{C}$\,, such that $\real z\leq \frac{A^{2}}{2}$\,,  as soon as
  \begin{equation}
    \label{eq:condQALV}
    C_s\max(b,Ab,\frac{1}{A})\leq 1\,.
  \end{equation}
  Therefore the subelliptic estimate \eqref{eq:subellipticEstimateWith2/5ForBismutLaplacian} holds true with $B_{\pm,b,V^{h}}+A^{2}\pi_{0,\pm}$ replaced by $B_{\pm,b,V^{h}}+Q_{A,L,V^{h}}$ and the constant $C\geq 1$ replaced by $4C$\,.
\end{proposition}
\begin{proof}
  \textbf{1)} We start with the simpler perturbation $Q_{A,L}$ instead of $Q_{A,L,V^{h}}$ and we write:
  \begin{eqnarray*}
 B_{\pm,b,V^{h}}+Q_{A,L}-z & = & B_{\pm,b,V^{h}} + A^2 \pi_{0,\pm} -z+ ( A^2 \pi_{0,\pm} (1- \chi)\big( \frac{2W_{\theta}^2}{ (LA)^2} \big) \pi_{0,\pm} ) \\
 & = & (1 + A^2 \pi_{0,\pm} (1-\chi)(\frac{2 W_{\theta}^2}{(LA)^2}) \pi_{0,\pm} (B_{\pm,b,V^{h}} + A^2 \pi_{0,\pm}-z)^{-1}) (B_{\pm,b,V^{h}} + A^2 \pi_{0,\pm}-z).
\end{eqnarray*}
According to Proposition \ref{pr:subellipticEstimateA} and  under the condition~\ref{eq:condQALV}, the resolvent $(B_{\pm,b,V^{h}}+A^2\pi_{0,\pm}-z)^{-1} $ is continuous from $\tilde{\cal W}^{0,s}$ to
$\tilde{\cal W}^{0,s+\frac{2}{9}}$ with norm less than $ \frac{C}{A^{\frac{16}{9}}}$. Because $\chi\equiv 1$ on $[-1,1]$\,, the operator $(1-\chi)(\frac{2W_{\theta}^2}{(LA)^2}) $ is a bounded operator from $\tilde{\cal W}^{0,s+\frac{2}{9}} $ to $\tilde{\cal W}^{0,s}$ with norm less than  $ \frac{1}{(LA)^{\frac{2}{9}}} $.  When $ L\geq (2C)^{\frac{9}{2}} $ we obtain
    $$ \| A^2 \pi_{0,\pm} (1-\chi)(\frac{2W_{\theta}^2}{(LA)^2}) \pi_{0,\pm} (B_{\pm,b,V^{h}}+A^2\pi_{0,\pm}-z)^{-1}  \|_{\mathcal{L}(\tilde{\cal W}^{0,s}; \tilde{\cal W}^{0,s})}\leq \frac{C}{L^{\frac{2}{9}}}\leq \frac{1}{2}. $$
Therefore the operator 
  $$ 1+ A^2 \pi_{0,\pm}(1-\chi)(\frac{2W_{\theta}^2}{(LA)^2})\pi_{0,\pm}(B_{\pm,b,V^{h}}+A^2\pi_{0,\pm}-z)^{-1} $$
  is invertible by Neumann series and the norm of its inverse is less than $2$\,, while its norm is bounded by $3/2$\,.
  We have proved
  \begin{equation}
    \label{eq:BQAL}
\frac{1}{2} \|\overline{(B_{\pm,b,V^{h}}+A^{2}\pi_{0,\pm}-z)}^{s}u\|_{\tilde{\cal W}^{0,s}} \leq \|\overline{(B_{\pm,b,V^{h}} + Q_{A,L}-z)}^{s} u \|_{\tilde{\cal W}^{0,s}} \leq \frac{3}{2} \|\overline{(B_{\pm,b,V^{h}}+A^{2}\pi_{0,\pm}-z)}^{s}u\|_{\tilde{\cal W}^{0,s}}\,,
\end{equation}
for all $u\in D(\overline{B_{\pm,b,V^{h}}}^{s})$ and all $z\in \mathbb{C}$ such that $\real z\leq \frac{A^{2}}{2}$\,.\\

\noindent\textbf{2)} We now use the similar perturbative argument for
$$ B_{\pm,b,V^{h}} + Q_{A,L,V^{h}}-z= B_{\pm,b,V^{h}}+ Q_{A,L}-z - (Q_{A,L}-Q_{A,L,V^{h}}). $$
The inequality \eqref{eq:QALV} of Lemma~\ref{le:QALV} below
gives
\begin{equation}
    \label{eq:QALQALVH}
    \|Q_{A,L}-Q_{A,L,V^{h}}\|_{\mathcal{L}(\tilde{\mathcal{W}}^{0,s};\tilde{\mathcal{W}}^{0,s})}
=\|\tilde{Q}_{A,L}-\tilde{Q}_{A,L,V^{h}}\|_{\mathcal{L}(H^{s};H^{s})}\leq \frac{C_{\chi,s}A^{2}}{(LA)^{2}}
\end{equation}
when $L,A\geq 1$\,, uniformly with respect to $h\in]0,1]$\,.\\
The subelliptic inequality \eqref{eq:subellipticEstimateWith2/5ForBismutLaplacian} combined with \eqref{eq:BQAL}, leads to
\begin{align*}
  \big| \| (B_{\pm,b,V^{h}} + Q_{A,L,V^{h}}-z)u\|_{\tilde{\cal W}^{0,s}} - \|(B_{\pm,b,V^{h}}+ Q_{A,L}-z)u\|_{\tilde{\cal W}^{0,s}} \big| &\leq   C_{\chi,s} \frac{A^2}{(LA)^2}\|u\|_{\tilde{\cal W}^{0,s}}
  \\
                          &\leq \frac{2CC_{\chi,s}}{(LA)^2} \|(B_{\pm,b,V^{h}}+ Q_{A,L})u \|_{\tilde{\cal W}^{0,s}}\,.
\end{align*}
By taking $C_{s,new}\geq 1$ such that $C_{s,new}\max(b,Ab,\frac{1}{A})\leq 1$ implies $A \geq \sqrt{4CC_{\chi, s}}$ and $(LA)^{2}\geq 4CC_{\chi,s}$\,, the right-hand side is less than $\frac{1}{2}\|(B_{\pm,b,V^{h}}+Q_{A,L} )u\|_{\tilde{\cal W}^{0,s}} $\,.\\
We conclude by stating the result with $C_{s}=C_{s,new} = \max(\sqrt{4CC_{\chi,s}},C_{s,old}) $.
\end{proof}

\begin{lemma}
  \label{le:QALV}
  For all $s,s'\in \mathbb{R}$ there exists $C_{\chi,s,s'}\geq 1$ such that
  \begin{eqnarray}
  \label{eq:fss2}
&&  \|\chi(\frac{C_{d}+\Delta_{V^{h},1}}{(LA)^{2}})\|_{\mathcal{L}(H^{s};H^{s'})}
    +
     \|\chi(\frac{C_{d}+H_{0}}{(LA)^{2}})\|_{\mathcal{L}(H^{s};H^{s'})}\leq C_{\chi,s,s'}(LA)^{(s'-s)_{+}}\,,
     \\
  \label{eq:resHs}
 \forall z\in\mathbb{C}\,, \mathbb{R}\mathrm{e}~z\leq \frac{A^2}{2}\,,\hspace{-0.5cm}&&  \|\frac{1}{2}(\Delta_{V^{h},1}+\tilde{Q}_{A,L,V^{h}}-z)^{-1}\|_{\mathcal{L}(H^{s};H^{s})}\leq \frac{2}{A^2-\mathbb{R}\mathrm{e}~z +|\mathrm{Im}~z|}\leq \frac{4}{A^2+2|\mathrm{Im}\,z|}
\end{eqnarray}
   and
   \begin{equation}
  \label{eq:QALV}
 \|\chi(\frac{C_{d}+H_{0}}{(LA)^{2}})-\chi(\frac{C_{d}+\Delta_{V^{h},1}}{(LA)^{2}})\|_{\mathcal{L}(H^{s};H^{s})}\leq \frac{C_{\chi,s,s}}{(LA)^2}
\end{equation}
hold as soon as $\frac{A}{C_{s}},L\geq 1$ and for all $h\in]0,1]$\,.
\end{lemma}
\begin{proof}
  The two first inequalities \eqref{eq:fss2} and \eqref{eq:resHs} are straightforward applications of the functional calculus, because the $H^{s}$-norm is equivalently evaluated with $\|(C_{d}+H_{0})^{s/2}u\|_{L^{2}}$ or with $\|(C_{d}+\Delta_{V^{h},1})^{s/2}u\|_{L^{2}}$\,.\\
  For \eqref{eq:QALV}, the difference
\begin{equation}
  \label{eq:defRALh}
R_{A,L,h}= (C_{d}+\Delta_{V^{h},1}) - (C_{d}+H_{0})
\end{equation}
satisfies  $\|R_{A,L,h}\|_{\mathcal{L}(H^{s};H^{s})}\leq C^{(1)}_{\chi,s}$ uniformly with respect to $h\in ]0,1]$\,.\\
The Helffer-Sj{\"o}strand formula gives

$$
\chi(\frac{C_{d}+\Delta_{V^{h},1}}{(LA)^{2}})-\chi(\frac{C_{d}+H_{0}}{(LA)^{2}})
=\frac{1}{2i\pi}
\int_{\mathbb{C}}\frac{1}{(LA)^{2}}(\partial_{\bar z}\tilde{\chi})(\frac{z}{(LA)^{2}}) (z-C_{d}-\Delta_{V^{h},1})^{-1}R_{A,L,h}(z-C_{d}-H_{0})^{-1}~dz\wedge d\bar z\,,
$$
where $\tilde{\chi}\in \mathcal{C}^{\infty}_{0}(\mathbb{C};\mathbb{C})$ is an almost analytic extension of $\chi$ with
$$
|\partial_{\overline{z}}\tilde{\chi}(z)|\leq C_{\chi,N}|\mathrm{Im}\,,z|^{N}
$$
while $\partial_{\overline{z}}\tilde{\chi}\equiv 0$ in a neighborhood of $0$\,.
The $\mathcal{L}(H^{s};H^{s})$-norm of this difference in given by the $\mathcal{L}(L^{2};L^{2})$-norm of
$$
(C_{d}+\Delta_{V^{h},1})^{-s/2}\left[\chi(\frac{C_{d}+\Delta_{V^{h},1}}{(LA)^{2}})-\chi(\frac{C_{d}+H_{0}}{(LA)^{2}})\right](C_{d}+H_{0})^{s/2}
$$
or, by setting $\tilde{R}_{A,L,h}=(C_{d}+\Delta_{V^{h},1})^{-s/2}R_{A,L,h}(C_{d}+H_{0})^{s/2}$\,, of
$$
\frac{1}{2i\pi}
\int_{\mathbb{C}}\frac{1}{(LA)^{2}}(\partial_{\bar z}\tilde{\chi})(\frac{z}{(LA)^{2}}) (z-C_{d}-\Delta_{V^{h},1})^{-1}\tilde{R}_{A,L,h}(z-C_{d}-H_{0})^{-1}~dz\wedge d\bar z\,.
$$
With
$$
\|\tilde{R}_{A,L,h}\|_{\mathcal{L}(L^{2};L^{2})}\leq C_{\chi,s}^{(2)}\|R_{A,L,h}\|_{\mathcal{L}(H^{s};H^{s})}\,,
$$
and the inequalities
\begin{eqnarray*}
  &&
\|(z-C_{d}-\Delta_{V^{h},1})^{-1}\|_{\mathcal{L}(L^{2};L^{2})}
\|(z-C_{d}-H_{0})^{-1}\|_{\mathcal{L}(L^{2};L^{2})}
     \leq \frac{1}{|\mathrm{Im}\,z|^{2}}\,,\\
  \text{and}&& \left|\partial_{\overline{z}}\tilde{\chi}(\frac{z}{(LA)^{2}})\right|\leq C_{\chi,2}\frac{|\mathrm{Im}\,z|^{2}}{(LA)^{4}}\,,
\end{eqnarray*}
a simple integration yields the result.
\end{proof}
\section{The Grushin problem}
\label{sec:Grushin}
\subsection{Functional analysis of the Grushin problem}
\label{sec:funcGru}

We consider the operators
\begin{equation}
  \label{eq:Grushin}
  \mathcal{P}_{z}=
  \begin{pmatrix}
    B_{\pm,b,V^{h}}+Q_{A,L,V^{h}}-z& U_{\pm,\theta}\\
    U_{\pm,\theta}^{-1}\pi_{0}& 0
  \end{pmatrix}
  \quad\text{and}\quad
  \mathcal{P}_{z}'= \begin{pmatrix}
    B_{\pm,b,V^{h}}'+Q_{A,L,V^{h}}-\bar{z}& U_{\pm,\theta}\\
    U_{\pm,\theta}^{-1}\pi_{0}& 0
  \end{pmatrix}\,.
\end{equation}
 Remember that $B_{\pm,b,V^{h}}\in \mathrm{OpS}^{3/2}_{\Psi}(Q^{h},\mathrm{End}(\mathcal{E}_{\pm}^{h}))$ while 
 $Q_{A,L,V^{h}}:\mathcal{S'}(X^{h};\mathcal{E}_{\pm}^{h})\to \mathcal{S}(X^{h};\mathcal{E}_{\pm}^{h})$ and $U_{\pm,\theta}$ is an isomorphism from $H^s(Q^{h};\Lambda T^*Q^{h}\otimes F_\pm)$ to $\tilde{\mathcal{W}}^{0,s}(X^{h};\mathcal{E}_{\pm}^{h})\cap\ker(\alpha_{\pm})$ for all $s\in \mathbb{R}$\,, with $F_{+}=Q\times \mathbb{C}$ and $F_{-}=(Q\times \mathbb{C})\otimes \mathbf{or}_{Q^{h}}$\,. In particular, the operators $\mathcal{P}_{z}$, $\mathcal{P}_{z}'$ are bounded
 \begin{equation}
   \label{eq:contGrushin}
   \mathcal{P}_{z},\mathcal{P}_{z}':
   \tilde{\mathcal{W}}^{0,s}(X^{h};\mathcal{E}_{\pm}^{h})\oplus H^{s}(Q^{h};\Lambda T^{*}Q^{h}\otimes F_{\pm})\longrightarrow \tilde{\mathcal{W}}^{0,s-3/2}(X^{h};\mathcal{E}_{\pm}^{h})\oplus H^{s}(Q^{h};\Lambda T^{*}Q^{h}\otimes F_{\pm})
 \end{equation}
for all $s\in \mathbb{R}$\,. With
 \begin{eqnarray}
 \label{eq:SE}
 \mathop{\cap}_{s\in \mathbb{R}}\tilde{\mathcal{W}}^{0,s}(X^{h};\mathcal{E}_{\pm}^{h})\oplus H^{s}(Q^{h};\Lambda T^{*}Q^{h}\otimes F_{\pm})&=&\mathcal{S}(X^{h};\mathcal{E}_{\pm}^{h})\oplus \mathcal{C}^{\infty}(Q^{h};\Lambda T^{*}Q^{h}\otimes F_{\pm})\,,\\
   \label{eq:SpEp}
   \mathop{\cup}_{s\in \mathbb{R}}\tilde{\mathcal{W}}^{0,s}(X^{h};\mathcal{E}_{\pm}^{h})\oplus H^{s}(Q^{h};\Lambda T^{*}Q^{h}\otimes F_{\pm})&=&\mathcal{S}'(X^{h};\mathcal{E}_{\pm}^{h})\oplus \mathcal{E}'(Q^{h};\Lambda T^{*}Q^{h}\otimes F_{\pm})
 \end{eqnarray}
 the continuity also holds with these spaces endowed with their usual topology.\\
 We will use the following abbreviations
 \begin{eqnarray*}
   \tilde{\mathcal{W}}^{0,s}\oplus H^{s'}&=&\tilde{\mathcal{W}}^{0,s}(X^{h};\mathcal{E}_{\pm}^{h})\oplus H^{s'}(Q^{h};\Lambda T^{*}Q^{h}\otimes F_{\pm})\,,\\
   \mathcal{S}\oplus \mathcal{C}^{\infty}&=&\mathcal{S}(X^{h};\mathcal{E}_{\pm}^{h})\oplus \mathcal{C}^{\infty}(Q^{h};\Lambda T^{*}Q^{h}\otimes F_{\pm})\,,\\
 B_{Q,z}&=&B_{\pm,b,V^{h}}+Q_{A,L,V^{h}}-z\quad\text{and}\quad B_{Q,z}'=B_{\pm,b,V^{h}}'+Q_{A,L,V^{h}}-\bar{z}\,.
 \end{eqnarray*}
 We recall
 \begin{eqnarray*}
   && B_{\pm,b,V^{h}}=\frac{1}{b^{2}}\alpha_{\pm}+\frac{1}{b}\beta_{\pm}+\gamma_{\pm}\\
      \text{with}&&
                   \pi_{0,\pm}B_{Q,z}\pi_{\bot,\pm}=\pi_{0,\pm}(\frac{1}{b}\beta_{\pm}+\gamma_{\pm})\pi_{\bot,\pm}\,,\\
   &&  \pi_{\bot,\pm}B_{Q,z}\pi_{0,\pm}=\pi_{\bot,\pm}(\frac{1}{b}\beta_{\pm}+\gamma_{\pm})\pi_{0,\pm}\,,\\
\text{and}&&  \pi_{0,\pm}B_{Q,z}\pi_{0,\pm}=\pi_{0,\pm}(\gamma_{\pm}+Q_{A,L,V^{h}}-z)\pi_{0,\pm}\,.
 \end{eqnarray*}
 We check that $\mathcal{P}_{z}$ and $\mathcal{P}_{z}'$ are invertible in a weak sense under suitable conditions and then we deduce via the Schur complement formula an explicit expression of $\overline{(B_{\pm,b,V^{h}}+Q_{A,L,V^{h}}}^s-z)^{-1}$ as an operator from $\tilde{\mathcal{W}}^{0,s}(X^{h};\mathcal{E}_{\pm}^{h})\to \tilde{\mathcal{W}}^{0,s-3}(X^{h};\mathcal{E}_{\pm}^{h})$\,. Because the condition $b\leq \frac{1}{C_{s}}$\,, which ensures the invertibility of $\overline{\pi_{\bot,\pm}(B_{\pm,b,V^{h}}-z)\pi_{\bot,\pm}}^{s}$ in Proposition~\ref{pr:maximalAccretivityOfPiBPi}, depends on $s$\,, there is no choice of parameters which guarantees the meaning of some formulas simultaneously for all $s\in \mathbb{R}$\,, but only for all $s$ in a fixed interval $[s_{\min},s_{\max}]\subset\mathbb{R}$\,. Therefore not all the compositions of operators in what follows make sense with the topologies of \eqref{eq:SE} and \eqref{eq:SpEp} and products or inverses must be handled carefully.
In particular, although we make product of continuous operators between different spaces, they do not necessarily have a closed range and we must distinguish clearly left-inverses and right-inverses. Alternatively it is better handled  by the separate studies of the uniqueness and of the existence of solutions for linear systems.
\begin{proposition}
  \label{pr:Grushin}
  Assume that the condition $b\leq\frac{1}{C_{s}}$ of Proposition~\ref{pr:maximalAccretivityOfPiBPi} holds true for all $s\in[-s_{max},s_{max}]$\,, for some $s_{max}\in ]3,+\infty[$\,.
  \begin{description}
  \item[1)] When $|s|\leq s_{max}$ and $\mathbb{R}\mathrm{e}z\leq \frac{1}{24b^{2}}$\,, the operator $\mathcal{P}_{z}:\mathcal{S}\oplus \mathcal{C}^{\infty}\to \tilde{\mathcal{W}}^{\pm s}\oplus H^{\pm s+1}$ admits a left-inverse
$$
\mathcal{G}_{z}=
\begin{pmatrix}
  E&E_{+}\\
  E_{-}&E_{-+}
\end{pmatrix}\quad \in \mathcal{L}(\tilde{\mathcal{W}}^{0,\pm s}\oplus H^{\pm s+1};\tilde{\mathcal{W}}^{0,\pm s}\oplus H^{\pm s-1})\,.
$$
The same result holds of $\mathcal{P}_{z}':\mathcal{S}\oplus \mathcal{C}^{\infty}\to \tilde{\mathcal{W}}^{\mp s}\oplus H^{\mp s+1}$ with the left inverse
$$
\mathcal{G}'_{z}=
\begin{pmatrix}
  E'&E'_{+}\\
  E'_{-}&E'_{-+}
\end{pmatrix}
\in \mathcal{L}(\tilde{\mathcal{W}}^{0,\mp s}\oplus H^{\mp s+1};\tilde{\mathcal{W}}^{0,\mp s}\oplus H^{\mp s-1})\,.
$$
\item[2)] For $|s|\leq s_{max}-3/2$ and $\mathbb{R}\mathrm{e}z\leq \frac{1}{24b^{2}}$\,, the relations
  \begin{eqnarray*}
    &&\mathcal{G}_{z}\circ \mathcal{P}_{z}=i_{\tilde{\mathcal{W}}^{0,\pm s}\oplus H^{\pm s}\to \tilde{\mathcal{W}}^{0,\pm s-3/2}\oplus H^{\pm s-5/2}}\\
    \text{and}&&
      \mathcal{G}'_{z}\circ \mathcal{P}'_{z}=i_{\tilde{\mathcal{W}}^{0,\mp s}\oplus H^{\mp s}\to \tilde{\mathcal{W}}^{0,\mp s-3/2}\oplus H^{\mp s-5/2}}
  \end{eqnarray*}
make sense as the products  $A\circ B$\,, with $A\in \mathcal{L}(\tilde{\mathcal{W}}^{0,\pm s-3/2}\oplus H^{\pm s-1/2};\tilde{\mathcal{W}}^{0,\pm s-3/2}\oplus H^{\pm s-5/2})$ and  $B\in \mathcal{L}(\tilde{\mathcal{W}}^{0,\pm s}\oplus H^{\pm s};\tilde{\mathcal{W}}^{0,\pm s-3/2}\oplus H^{\pm s})\subset \mathcal{L}(\tilde{\mathcal{W}}^{0,\pm s}\oplus H^{\pm s};\tilde{\mathcal{W}}^{0,\pm s-3/2}\oplus H^{\pm s-1/2})$\,.
\item[3)] For $|s|\leq s_{max}-3/2$\,, $z\in\mathbb{C}\setminus \sigma(B_{\pm,b,V^{h}}+Q_{A,L,V^{h}})$ and $\mathbb{R}\mathrm{e}\, z<\frac{1}{24b^{2}}$\,, the operator $E_{-+}\in \mathcal{L}(H^{s-1/2}; H^{s-5/2})\subset \mathcal{L}(\mathcal{C}^{\infty};\mathcal{\mathcal{E}}')$ admits a right-inverse $(E_{-+})^{-1}_{r}\in \mathcal{L}(\mathcal{C}^{\infty};\mathcal{C}^{\infty})$ and a left-inverse $(E_{-+})^{-1}_{\ell}\in \mathcal{L}(\mathcal{E}';\mathcal{E}')$ with $(E_{-+})^{-1}_{\ell}\in \mathcal{L}(H^{s'};H^{s'+2/3})$ for all $s'\in \mathbb{R}$ and $(E_{-+})^{-1}_{\ell}\big|_{\mathcal{C}^{\infty}}=(E_{-+})^{-1}_{r}$\,.
\item[4)] For $|s|\leq s_{max}-4$ the equality

  \begin{equation}
    \label{eq:inverseOfBQ}
    (\overline{B_{\pm,b,V^{h}}+Q_{A,L,V^{h}}}^{s}-z)^{-1}=E-E_{+}(E_{-+})^{-1}_{\ell}E_{-} 
  \end{equation}
  
holds in the sense of $\mathcal{L}(\tilde{\mathcal{W}}^{0,s};\tilde{\mathcal{W}}^{0,s-3})$-valued meromorphic functions in $\left\{z\in \mathbb{C}\,, \mathbb{R}e\, z< \frac{1}{24b^{2}}\right\}$\,.
\end{description}
\end{proposition}
\begin{proof}
  \textbf{1)} The range of $\mathcal{P}_{z}\big|_{\mathcal{S}\oplus \mathcal{C}^{\infty}}$ is included in $\mathcal{S}\oplus \mathcal{C}^{\infty}$\,. Let us check that for $
  \begin{pmatrix}
    f\\f_{+}
  \end{pmatrix}
  \in \mathcal{S}\oplus\mathcal{C}^{\infty}$ the equation $\mathcal{P}_{z}
  \begin{pmatrix}
    u\\u_{-}
  \end{pmatrix}
  =
  \begin{pmatrix}
    f\\f_{+}
  \end{pmatrix}
  $ admits at most one solution in  $\tilde{\mathcal{W}}^{0,s}\oplus H^{s-1}$ when $|s|\leq s_{max}$\,.\\
By applying $\pi_{\bot,\pm}$ to the first line of the system
$$
\left\{
  \begin{array}[c]{ll}
 B_{Q,z}u+U_{\pm,\theta}u_{-}&=f\\
U^{-1}_{\pm,\theta}\pi_{0,\pm}u&=f_{+}
  \end{array}
\right.
$$
we must have
$$
\pi_{\bot,\pm}B_{Q,z}\pi_{\bot,\pm}u+\pi_{\bot,\pm}(\frac{1}{b}\beta_{\pm}+\gamma_{\pm})U_{\pm,\theta}f_{+}=\pi_{\bot,\pm}f\,.
$$
When $\pi_{\bot,\pm}B_{Q,z}\pi_{\bot,\pm}=\pi_{\bot,\pm}(B_{\pm,b,V^{h}}-z)\pi_{\bot,\pm}$ is invertible
$$
u=\pi_{\bot,\pm}u+\pi_{0,\pm}u=[\pi_{\perp}(\mathcal{B}_{\pm,b,V^{h}}-z)\pi_{\perp}]^{-1}(\pi_{\bot,\pm}f-\pi_{\bot,\pm}(\frac{1}{b}\beta_{\pm}+\gamma_{\pm})U_{\pm,\theta}f_{+}) + U_{\pm,\theta}f_{+}
$$
With the conditions $b\leq \frac{1}{C_{s}}$ and $\mathbb{R}\mathrm{e}\, z\leq \frac{1}{24 b^{2}}$ and by noticing $\pi_{\bot,\pm}(\frac{1}{b}\beta_{\pm}+\gamma_{\pm})U_{\pm,\theta}\in \mathcal{L}(H^{s+1};\tilde{\mathcal{W}}^{0,s})$\,, Proposition~\ref{pr:maximalAccretivityOfPiBPi} actually says
\begin{eqnarray}
  && u=E f +E_{+}f_{+} \nonumber \\
  \text{with}&&
                E=(\overline{B_{\pm,b,V^{h},\bot}}^{s}-z)^{-1}\pi_{\bot,\pm}\in \mathcal{L}(\tilde{\mathcal{W}}^{0,s};\tilde{\mathcal{W}}^{0,s})\,, \label{eq:defE} \\
\text{and}&&
  E_{+}=U_{\pm,\theta}-(\overline{B_{\pm,b,V^{h},\bot}}^{s}-z)^{-1}\pi_{\bot,\pm}(\frac{1}{b}\beta_{\pm}+\gamma_{\pm})U_{\pm,\theta}\in \mathcal{L}(H^{s+1};\tilde{\mathcal{W}}^{0,s}) \label{eq:defE+}
\end{eqnarray}
By applying the projection $\pi_{0,\pm}$ on the first line of the system, we must have
$$
\pi_{0,\pm}f=
\pi_{0,\pm}(\frac{1}{b}\beta_{\pm}+\gamma_{\pm}+Q_{A,L,V^{h}}-z)u+U_{\pm,\theta}u_{-}=
\pi_{0,\pm}(\frac{1}{b}\beta_{\pm}+\gamma_{\pm})\pi_{\bot,\pm}u+\pi_{0,\pm}(\gamma_{\pm}+Q_{A,L,V^{h}}-z)U_{\pm,\theta}f_{+}+U_{\pm,\theta}u_{-}
$$
and $\pi_{0,\pm}(\frac{1}{b}\beta_{\pm}+\gamma_{\pm}+Q_{A,L,V^{h}}-z)\in \mathcal{L}(\tilde{\mathcal{W}}^{0,s};H^{s-1})$ now gives
\begin{eqnarray}
  && u_{-}=E_{-}f+E_{-+}f_{+}\,, \nonumber\\
  \text{with}&&
                E_{-}=U_{\pm,\theta}^{-1}\pi_{0,\pm}-U_{\pm,\theta}^{-1}\pi_{0,\pm}(\frac{1}{b}\beta_{\pm}+\gamma_{\pm})(\overline{B_{\pm,b,V^{h},\bot}}^{s}-z)^{-1}\pi_{\bot,\pm}\in \mathcal{L}(\tilde{\mathcal{W}}^{0,s};H^{s-1})\,, \label{eq:defE-}\\
&&
               E_{-+}=U_{\pm,\theta}^{-1}\pi_{0}(z-Q_{A,L,V^{h}}-\gamma_{\pm})U_{\pm,\theta} + U_{\pm,\theta}^{-1}\pi_{0}(\frac{1}{b}\beta_{\pm}+\gamma_{\pm})(\overline{B_{\pm,b,V^{h},\bot}}^{s}-z)^{-1}\pi_{\bot,\pm}(\frac{1}{b}\beta_{\pm}+\gamma_{\pm})U_{\pm,\theta}\,, \nonumber \\
\text{and}&& E_{-+}\in \mathcal{L}(H^{s+1};H^{s-1})\,. \label{eq:defE-+}
\end{eqnarray}
The result for $\mathcal{P}_{z}'$ is straightforward because
$$
\mathcal{P}_{z}'=
\begin{pmatrix}
  B_{\pm,b,V^{h}}'+Q_{A,L,V^{h}}-\overline{z}& U_{\pm,\theta}\\
  U_{\pm,\theta}^{-1}\pi_{0}&0
\end{pmatrix}
$$
and $(\overline{B_{\pm,b,V^{h},\bot}'}^{~-s}-\overline{z})=(W^{2}_{\theta})^{s}(\overline{B_{\pm,b,V^{h},\bot}}^{s}-z)^{*,s}(W^{2}_{\theta})^{-s}$\,.\\
\textbf{2)} By \textbf{1)} and $|s|\leq s_{max}-3/2$\,, we know that $\mathcal{G}_{z}\in \mathcal{L}(\tilde{\mathcal{W}}^{0,s-3/2}\oplus H^{s-1/2};\tilde{\mathcal{W}}^{0,s-3/2}\oplus H^{s-5/2})$ is a left-inverse on $\mathcal{S}\oplus \mathcal{C}^{\infty}$ and we deduce
$$
\forall
\begin{pmatrix}
  u\\u_{-}
\end{pmatrix}
\in \mathcal{S}\oplus \mathcal{C}^{\infty}\,,\quad \mathcal{G}_{z}\circ \mathcal{P}_{z}
\begin{pmatrix}
  u\\u_{-}
\end{pmatrix}
=
\begin{pmatrix}
  u\\u_{-}
\end{pmatrix}\,.
$$
The density of $\mathcal{S}\oplus \mathcal{C}^{\infty}$ in $\tilde{\mathcal{W}}^{0,s}\oplus H^{s}$\,, combined with $\mathcal{P}_{z}\in \mathcal{L}(\mathcal{S}\oplus \mathcal{C}^{\infty};\mathcal{S}\oplus \mathcal{C}^{\infty})\cap \mathcal{L}(\tilde{\mathcal{W}}^{0,s}\oplus H^{s};\tilde{\mathcal{W}}^{0,s-3/2}\oplus H^{s-1/2})$ and  $\mathcal{G}_{z}\in \mathcal{L}(\tilde{\mathcal{W}}^{0,s-3/2}\oplus H^{s-1/2})$\,, yields the result.\\
\textbf{3)} Let us prove firstly that for $u_{-}\in \mathcal{C}^{\infty}\subset H^{s-5/2}$ the equation $E_{-+}f_{+}=u_{-}$ admits at least one solution in $f_{+}\in \mathcal{C}^{\infty}\subset H^{s-1/2}$ when $|s|\leq s_{\max}-3/2$\,, $z\not\in \sigma(B_{\pm,b,V^{h}}+Q_{A,L,V^{h}})$ and $\mathbb{R}\mathrm{e}\,z\leq \frac{1}{b^{2}}$\,. Take $w=(\overline{B_{\pm,b,V^{h}}}^{s'}+Q_{A,L,V^{h}}-z)^{-1}U_{\pm,\theta}u_{-}$ which belongs to $D(\overline{B_{\pm,b,V^{h}}}^{s'})\subset \tilde{\mathcal{W}}^{0,s'+2/3}$ according to Proposition~\ref{pr:main1sttext} where the condition $0<b\leq h\leq\frac{1}{C_{g}}$ does not depend on $s'\in \mathbb{R}$\,. Set $u=\pi_{0,\pm}w\in \pi_{0,\pm}\left(\mathop{\cap}_{s'\in\mathbb{R}}\tilde{\mathcal{W}}^{0,s'}\right)=U_{\pm,\theta}\mathcal{C}^{\infty}$ and $v=\pi_{\bot,\pm}w\in \pi_{\perp,\pm}\left(\mathop{\cap}_{s'\in \mathbb{R}}\tilde{\mathcal{W}}^{0,s'}\right)=\pi_{\bot,\pm}\mathcal{S}$\,.
By projecting the equation $(B_{\pm,b,V^{h}}+Q_{A,L,V^{h}}-z)w=U_{\pm,\theta}u_{-}$ written in $\mathcal{S}$ with the projections $\pi_{0,\pm}$ and $\pi_{\bot,\pm}$\,, we obtain:
\begin{eqnarray*}
  && \pi_{0,\pm}(Q_{A,L,V^{h}}+\gamma_{\pm}-z)u+ \pi_{0}(\frac{1}{b}\beta_{\pm}+\gamma_{\pm})v=U_{\pm,\theta}u_{-}\,,\\
  \text{and}&&
         \pi_{\perp,\pm}(\frac{1}{b}\beta_{\pm}+\gamma_{\pm})u+ \pi_{\perp,\pm}(B_{\pm,b,V^{h}}-z)v=0\,.
\end{eqnarray*}
By taking $f_{+}=-U_{\pm,\theta}^{-1}u$ and by noticing $(\frac{1}{b}\beta_{\pm}+\gamma_{\pm})u\in \tilde{\mathcal{W}}^{0,s-3/2}$\,, with $|s|\leq s_{max}-3/2$\,, we deduce
$$
\pi_{0,\pm}(Q_{A,L,V^{h}}+\gamma_{\pm}-z)u-\pi_{0,\pm}(\frac{1}{b}\beta_{\pm}+\gamma_{\pm})(\overline{B_{\pm,b,V^{h},\bot}}^{s-3/2}-z)^{-1}(\frac{1}{b}\beta_{\pm}+\gamma_{\pm})u=U_{\pm,\theta}u_{-}\,.
$$
This means exactly that $f_{+}=U_{\pm,\theta}^{-1}u\in \mathcal{C}^{\infty}\subset H^{s-1/2}$ satisfies $E_{-+}f_{+}=u_{-}$ in $H^{s-5/2}$\,.
But the formula $f_{+}=(E_{-+})^{-1}_{r}u_{-}=U_{\pm,\theta}^{-1}\pi_{0}(\overline{B_{\pm,b,V^{h}}+Q_{A,L,V^{h}}}^{s'}- z)^{-1}U_{\pm,\theta}u_{-}$ for all $s'\in \mathbb{R}$\,, proves that this right-inverse $(E_{-+})^{-1}_{r}$ actually belongs to $\mathcal{L}(\mathcal{C}^{\infty};\mathcal{C}^{\infty})$\,.\\
With the dual statements of \textbf{1)} and \textbf{2)} this means also that the formal adjoint $E_{-+}'\in \mathcal{L}(H^{-s-1/2};H^{-s-5/2})\subset \mathcal{L}(\mathcal{C}^{\infty};\mathcal{E}')$ admits the right-inverse $(E_{-+}')^{-1}_{r}u_{-}=U_{\pm,\theta}\pi_{0}(\overline{B_{\pm,b,V^{h}}'+Q_{A,L,V^{h}}}^{-s'}-\bar{z})^{-1}U_{\pm,\theta}$ for all $s'\in \mathbb{R}$ and $(E'_{-+})^{-1}_{r}\in \mathcal{L}(\mathcal{C}^{\infty};\mathcal{C}^{\infty})$\,.
Duality implies that
$(E_{-+})^{-1}_{\ell}=[(E^{'}_{-+})^{-1}_{r}]'\in \mathcal{L}(\mathcal{E}';\mathcal{E}')$ is a left-inverse of $E_{-+}\in \mathcal{L}(\mathcal{C}^{\infty};\mathcal{D}')$\,. With the formula 
\begin{equation}
    \label{eq:expressionOfRightInverse}
    [(E^{'}_{-+})^{-1}_{r}]'=\left[U_{\pm,\theta}^{-1}(\overline{B'_{\pm,b,V^{h}}+Q_{A,L,V^{h}}}^{s'}-\bar z)^{-1}U_{\pm,\theta}\right]'=U_{\pm,\theta}^{-1}(\overline{B_{\pm,b,V^{h}}+Q_{A,L,V^{h}}}^{-s'}-z)^{-1}U_{\pm,\theta}\,,
\end{equation}
the regularizing property of $(E_{-+})^{-1}_{\ell}\in \mathcal{L}(H^{s'};H^{s'+2/3})$ and the identification $(E_{-+})^{-1}_{\ell}\big|_{\mathcal{C}^{\infty}}=(E_{\pm})^{-1}_{r}$ are straightforward.\\
\noindent\textbf{4)} When $|s|\leq s_{max}-3/2$ and $\mathbb{R}\mathrm{e}\,z\leq\frac{1}{24b^{2}}$\,, the first formula of \textbf{2)} implies
\begin{eqnarray}
&&\label{eq:intermediate400}
   (EB_{Q,z}+E_{+}U_{\pm,\theta}^{-1}\pi_{0})u=u \quad\text{in}~\tilde{\mathcal{W}}^{0,s-3/2}\,,\\
&&\label{eq:intermediate401}
   (E_{-}B_{Q,z}+E_{-+}U_{\pm,\theta}^{-1}\pi_{0})u=0\quad\text{in}~H^{s-5/2}\,,
\end{eqnarray}
for all $u\in \tilde{\mathcal{W}}^{0,s}$\,.
With $|s-5/2|\leq |s|+5/2\leq s_{max}-3/2$\,, and \textbf{3)}, the equality~\eqref{eq:intermediate401} becomes
$$
U_{\pm,\theta}^{-1}\pi_{0}u=-(E_{-+})_{\ell}^{-1}E_{-}B_{Q,z}u\quad\text{in}~H^{s-5/2+2/3}\,.
$$
for $z\not\in \sigma(B_{\pm,b,V^{h}}+Q_{A,L,V^{h}})$\,.\\
Put in the equality~\eqref{eq:intermediate400} we obtain
$$
(E-E_{+}(E_{-+})^{-1}_{\ell}E_{-})(B_{Q,z})u=u \quad
\text{in}~\tilde{\mathcal{W}}^{0,s-5/2+2/3-1}\,,
$$
when $z\not\in \sigma(B_{\pm,b,V^{h}}+Q_{A,L,V^{h}})$ and $u\in \tilde{\mathcal{W}}^{0,s}$\,. Applied to $u=(\overline{B_{\pm,b,V^{h}}+Q_{A,L,V^{h}}}^{s-2/3}-z)^{-1}v\in \tilde{\mathcal{W}}^{0,s}$ for $v\in \tilde{\mathcal{W}}^{0,s-2/3}$\,, we obtain
$$
(E-E_{+}(E_{-+})^{-1}_{\ell}E_{-})=
(\overline{B_{\pm,b,V^{h}}+Q_{A,L,V^{h}}}^{s-2/3}-z)^{-1}\text{in}~\mathcal{L}(\tilde{\mathcal{W}}^{0,s-2/3};\tilde{\mathcal{W}}^{s-5/2-1/3})\,.
$$
and we know that the right-hand side is a meromorphic function of $z$ in $\mathbb{C}$\,.\\
The condition $|s|\leq s_{max}-4 \leq s_{max}-5/2-2/3$ allows to replace $s-2/3$ by $s$\,, by noticing  $s+2/3-5/2-1/3\geq s-3$\,, with
$$
(E-E_{+}(E_{-+})^{-1}E_{-})=(\overline{B_{\pm,b,V^{h}}+Q_{A,L,V^{h}}}^{s}-z)^{-1}
$$
as an $\mathcal{L}(\tilde{\mathcal{W}}^{0,s};\tilde{\mathcal{W}}^{0,s-3})$-valued meromorphic function in $\left\{z\in \mathbb{C}\,,\, \mathbb{R}\mathrm{e}\,z<\frac{1}{24b^{2}}\right\}$\,.
\end{proof}

\subsection{Quantitative comparisons of truncated resolvents}
\label{sec:quantGrushin}
After setting
\begin{equation}
  \label{eq:defdeltaBDz}
\delta_{B,\Delta,z}=(B_{\pm,b,V^h}+Q_{A,L,V^h}-z)^{-1}-U_{\pm,\theta}(\frac{1}{2}\Delta_{V^h,1}+\tilde{Q}_{A,L,V^h}-z)^{-1}U_{\pm,\theta}^{-1}
\end{equation}
we consider the finite rank operators
$$
\delta_{B,\Delta,z}\circ Q_{A,L',V^h}\quad\text{and}\quad Q_{A,L',V^h}\circ \delta_{B,\Delta,z}\,,
$$
where $L'\geq L\geq 1$ will be fixed later.
\begin{proposition}
\label{pr:restronc}
Let $L\geq 1$ be the uniform constant of Proposition~\ref{pr:subEllipticQALVh} and fix $L'\geq 1$\,. For all $s\in\mathbb{R}$\,, there exist constants $C_{s}\geq 1$\,, determined by $s$ such that the condition $ C_s\max(b,Ab,A^{-1})\leq 1 $ implies the inequalities
\begin{eqnarray}
  \label{eq:restoncQ}
  \|\delta_{B,\Delta,z}\circ Q_{A,L',V^{h}}
  \|_{\mathcal{L}(\tilde{\mathcal{W}}^{0,s};\tilde{\mathcal{W}}^{0,s})} &\leq & C_s\frac{bA^{\frac{3}{2}}}{1+A^{-1}\sqrt{|\mathrm{Im}z|}} \\
  \label{eq:Qrestronc}
  \|Q_{A,L',V^h}\circ\delta_{B,\Delta,z}
  \|_{\mathcal{L}(\tilde{\mathcal{W}}^{0,s};\tilde{\mathcal{W}}^{0,s})} &\leq & C_s\frac{bA^{\frac{3}{2}}}{1+A^{-1}\sqrt{|\mathrm{Im}z|}}
\end{eqnarray}
for all $z\in\mathbb{C}$ such that $|\real z|\leq \frac{A^2}{2}$\,.
\end{proposition}

\begin{proof}
  For a given $s\in\mathbb{R}$ we fix $s_{\max}\geq |s|+10$
so that the estimates of Proposition~\ref{pr:Grushin} and the expressions of $E,E_{-}, E_{+}$ and $E_{-+}$ given in the proof make sense for the Sobolev exponent $s$  replaced by $s_2\in [s,s+6]$\,. Actually 
 for $|s_2|\leq s_{\max}-4$ and $\real z\leq \frac{A^2}{2} \leq \min(\frac{A^2}{2},\frac{1}{24b^2})$\,, the equality \eqref{eq:inverseOfBQ} gives
 $$
 \delta_{B,\Delta,z}\circ Q_{A,L',V^h}  =  ( E-E_+(E_{-+})_{\ell}^{-1}E_-) Q_{A,L',V^h}-U_{\pm,\theta}(\frac{1}{2}\Delta_{V^h,1}+\tilde{Q}_{A,L,V^h}-z)^{-1}U_{\pm,\theta}^{-1}Q_{A,L',V^h}\,,
 $$ 
 as an equality of $\mathcal{L}(\tilde{\mathcal{W}}^{0,s_2};\tilde{\mathcal{W}}^{0,s_2-3})$-valued meromorphic functions. 
 Actually because $Q_{A,L',V^h}\in \mathcal{L}(\mathcal{S}';\mathcal{S})$ while $\delta_{B,\Delta,z}\in \mathcal{L}(\mathcal{S},\mathcal{S})$ when $\real z\leq \frac{A^2}{2}$\,, 
 $z\not \in \mathrm{Spec}(\overline{B_{\pm,b,V^h}+Q_{A,L,V^h}}^{s_2})$\,, the left-hand side as well as the final term belong to 
 $\mathcal{L}(\mathcal{S}';\mathcal{S})$\,. 
 The above equality can therefore be extended to an equality of $\mathcal{L}({\tilde{\cal W}}^{0,s_2};\tilde{\mathcal{W}}^{0,s_2})$-valued meromorphic function. Actually with $\real z\leq A^2/2$, $z$ is not in $\mathrm{Spec}(\Delta_{V^h,1}+\tilde{Q}_{A,L,V^h})$ and the final term is holomorphic.
 Owing to $EQ_{A,L,V^h}=E \pi_{0,\pm}=0$ and with the expressions \eqref{eq:defE+} of $E_{+}$ and \eqref{eq:defE-} of $E_{-}$, we obtain
  \begin{eqnarray*}
   \delta_{B,\Delta,z}\circ Q_{A,L',V^h}         & = & - E_{+}(E_{-+})_{\ell}^{-1}E_- Q_{A,L',V^h} - U_{\pm,\theta} (\frac{1}{2}\Delta_{V^h,1}+\tilde{Q}_{A,L,V^h}-z)^{-1}U_{\pm,\theta}^{-1}Q_{A,L',V^h} \\
    & = & (I) + (II) \,,
  \end{eqnarray*}
  where
  \begin{eqnarray*}
    (I) & = & - U_{\pm,\theta} [(E_{-+})_{\ell}^{-1} + (\frac{1}{2}\Delta_{V^h,1}+\tilde{Q}_{A,L,V^h}-z)^{-1}  ]U_{\pm,\theta}^{-1} Q_{A,L',V^h}  \\
      & = & -U_{\pm,\theta} (E_{-+})_{\ell}^{-1}[ \frac{1}{2}\Delta_{V^h,1}+\tilde{Q}_{A,L,V^h}-z + E_{-+}  ] (\frac{1}{2}\Delta_{V^h,1}+\tilde{Q}_{A,L,V^h}-z)^{-1}U_{\pm,\theta}^{-1} Q_{A,L',V^h} \,, \\
    (II) & = & (\overline{B_{\pm,b,V^h,\bot}}^{s_2'}-z)^{-1}\pi_{\bot,\pm}(\frac{1}{b}\beta_{\pm}+\gamma_{\pm})U_{\pm,\theta}(E_{-+})_{\ell}^{-1}U_{\pm,\theta}^{-1} Q_{A,L',V^h}.
  \end{eqnarray*}
  and $s_2'$ is any other exponent such that $|s_2'|\leq s_{\max}$\,.\\
  By combining Bismut's formula
  \eqref{eq:algebraicIdentityForWittenLaplacianUtheta}\,, recalled here,
  $$ U_{\pm,\theta}^{-1}[\pi_{0,\pm}(\gamma_{\pm} - \beta_{\pm} \alpha_{\pm}^{-1} \beta_{\pm} )\pi_{0,\pm}] U_{\pm,\theta}= \frac{1}{2} \Delta_{V^h,1}\,,$$
   with the expression  \eqref{eq:defE-+} of $E_{-+}$\,,
  $(I)$ becomes
  $$ (I) =  U_{\pm,\theta}(E_{-+})_{\ell}^{-1} U_{\pm,\theta}^{-1} \pi_{0,\pm} (III) U_{\pm,\theta} (\frac{1}{2}\Delta_{V^h,1}+\tilde{Q}_{A,L,V^h}-z)^{-1}U_{\pm,\theta}^{-1} Q_{A,L',V^h}\,, $$
where
  $$ (III) = \beta_{\pm}\alpha_{\pm}^{-1}\beta_{\pm}-(\frac{1}{b}\beta_{\pm}+\gamma_{\pm})(\overline{B_{\pm,b,V^h,\bot}}^{s''_2}-z)^{-1}\pi_{\bot,\pm}(\frac{1}{b}\beta_{\pm}+\gamma_{\pm})\,,$$
  and $|s_2''|\leq s_{\max}$\,.
  The above operator  can be rewritten
  \begin{eqnarray*}
    (III) & = & (\frac{1}{b}\beta_{\pm}+\gamma_{\pm}) \pi_{\bot,\pm} [b^2 \alpha_{\pm}^{-1} - (\overline{B_{\pm,b,V^h,\bot}}^{s''_2}-z)^{-1}] \pi_{\bot,\pm} (\frac{1}{b}\beta_{\pm}+\gamma_{\pm}) \\
    & & \hspace{8cm} - b(\beta_{\pm}+ b\gamma_{\pm}) \alpha_{\pm}^{-1}\pi_{\bot,\pm}\gamma_{\pm}- b\gamma_{\pm}\alpha^{-1}_{\pm}\beta_{\pm} \\
        & = & (\frac{1}{b}\beta_{\pm}+\gamma_{\pm}) \alpha_{\pm}^{-1} \pi_{\bot,\pm} [b^2(\overline{B_{\pm,b,V^h,\bot}}^{s''_2}-z)  - \alpha_{\pm}](\overline{B_{\pm,b,V^h,\bot}}^{s''_2}-z)^{-1}\pi_{\bot,\pm} (\frac{1}{b}\beta_{\pm}+\gamma_{\pm}) \\
    & & \hspace{8cm} - b(\beta_{\pm} + b\gamma_{\pm}) \alpha_{\pm}^{-1}\pi_{\bot,\pm}\gamma_{\pm} - b\gamma_{\pm}\alpha^{-1}_{\pm}\beta_{\pm} \\
        & = &\underbrace{(\frac{1}{b}\beta_{\pm}+\gamma_{\pm}) \alpha_{\pm}^{-1} \pi_{\bot,\pm} [b\beta_{\pm}+b^2(\gamma_{\pm}-z)](\overline{B_{\pm,b,V^h,\bot}}^{s_2''}-z)^{-1}\pi_{\bot,\pm} (\frac{1}{b}\beta_{\pm}+\gamma_{\pm})}_{(III')} \\
    & & \hspace{8cm} \underbrace{-b(\beta_{\pm} + b\gamma_{\pm}) \alpha_{\pm}^{-1}\pi_{\bot,\pm}\gamma_{\pm} - b\gamma_{\pm}\alpha^{-1}_{\pm}\beta_{\pm}}_{-(III.3)}\,.
  \end{eqnarray*}
  Let's rewrite $(III') $ as the sum of $(III.1) $ and $(III.2)$ where
  \begin{eqnarray*}
    (III.1)& = &(\frac{1}{b}\beta_{\pm}+\gamma_{\pm}) \alpha_{\pm}^{-1} \pi_{\bot,\pm} [ b \beta_{\pm}+b^2\gamma_{\pm}](\overline{B_{\pm,b,V^h,\bot}}^{s_2''}-z)^{-1}\pi_{\bot,\pm} (\frac{1}{b}\beta_{\pm}+\gamma_{\pm}) \\
    (III.2)& = & - z b^2(\frac{1}{b}\beta_{\pm}+\gamma_{\pm}) \alpha_{\pm}^{-1} (\overline{B_{\pm,b,V^h,\bot}}^{s_2''}-z)^{-1}\pi_{\bot,\pm} (\frac{1}{b}\beta_{\pm}+\gamma_{\pm})\,.
  \end{eqnarray*}
  The operator $(III.1) $ can be depicted by
  $$\begin{tikzcd}
{\tilde{\cal W}^{0,s}} & {\tilde{\cal W}^{2,s+1}} \arrow[l, "\frac{1}{b}\beta_{\pm}+\gamma_{\pm}"] &  & {\tilde{\cal W}^{0,s+1}} \arrow[ll, "{\alpha_{\pm}^{-1}\pi_{\bot,\pm}}"] & {\tilde{\cal W}^{0,s+\frac{5}{2}}} \arrow[l, " b \beta_{\pm} + b^2\gamma_{\pm}"] &  & {\tilde{\cal W}^{0,s+\frac{5}{2}}} \arrow[ll, "{(B_{\pm,b,V^h,\bot}-z)^{-1}\pi_{\bot,\pm}}"] & {\tilde{\cal W}^{2,s+\frac{7}{2}}} \arrow[l, "\frac{1}{b}\beta_{\pm}+\gamma_{\pm}"] 
\end{tikzcd}
$$

and $``-(III.2)/zb^2" $ is depicted by
$$\begin{tikzcd}
{\tilde{\cal W}^{0,s}} &  & {\tilde{\cal W}^{2,s+1}} \arrow[ll, "\frac{1}{b}\beta_{\pm}+\gamma_{\pm}"] &  & {\tilde{\cal W}^{0,s+1}} \arrow[ll, "{\alpha_{\pm}^{-1}\pi_{\bot,\pm}}"] &  & {\tilde{\cal W}^{0,s+1}} \arrow[ll, "{(B_{\pm,b,V^h,\bot}-z)^{-1}\pi_{\bot,\pm}}"] &  & {\tilde{\cal W}^{2,s+2}} \arrow[ll, "\frac{1}{b}\beta_{\pm}+\gamma_{\pm}"]\,.
\end{tikzcd}$$
Combining the previous decomposition with, on one side
  $$\forall s_1,s_2\in \mathbb{R}\,, \quad \beta_{\pm}+b\gamma_{\pm}\in \mathrm{OpS}_{\Psi}^{\frac{3}{2}}(X^h,\mathcal{E}_{\pm}^h) \cap \mathcal{L}(\tilde{\cal W}^{s_1+2,s_2+1};\tilde{\cal W}^{s_1,s_2})$$
 with semi-norms uniformly bounded with respect to $b\in (0,1]$, and on the other side
  $$\forall s\in \mathbb{R}\,, \quad \| (B_{\pm,b,V^h,\bot}-z)^{-1}\pi_{\bot,\pm}\|_{\mathcal{L}(\tilde{\cal W}^{0,s};\tilde{\cal W}^{0,s})}\leq 24b^2$$
due to Proposition~\ref{pr:maximalAccretivityOfPiBPi} as soon as $\real z \leq \frac{1}{24b^2}$, implies
\begin{eqnarray*}
  \| (III.1) \|_{\mathcal{L}(\tilde{\cal W}^{2,s+\frac{7}{2}};\tilde{\cal W}^{0,s})} \leq b C_s \quad  \text{and} \quad \| (III.2) \|_{\mathcal{L}(\tilde{\cal W}^{2,s+2};\tilde{\cal W}^{0,s})} \leq C_s|z|b^2.
\end{eqnarray*}
We claim that
$$\|(III.3) \|_{\mathcal{L}(\tilde{\cal W}^{2,s+1};\tilde{\cal W}^{0,s})} \leq C_s b  .$$
Now we decompose $(I)$ as
$$(I) = (I.1)+(I.2)-(I.3) $$
With $(I.*) $ depicted by
\begin{equation}
  \label{eq:composition1}
      \begin{tikzcd}
{\tilde{\cal W}^{0,s}} &  &  &  & {\tilde{\cal W}^{0,s}} \arrow[llll, "{U_{\pm,\theta}(E_{-+})_{\ell}^{-1}U_{\pm,\theta}^{-1}\pi_{0,\pm}}"] &  & {\tilde{\cal W}^{2,s'}} \arrow[ll, "(III.*)"] &  &  &  &  &  & {\tilde{\cal W}^{0,s}} \arrow[llllll, "{U_{\pm,\theta} (\Delta_{V^h,1}+ \tilde{Q}_{A,L,V^{h}}-z)^{-1}U_{\pm,\theta}^{-1}Q_{A,L',V^{h}}}"] \,,
\end{tikzcd}
\end{equation}
where the choice of $s'$ will depend on the  cases for indexed by $*$\,.\\
$$
\|U_{\pm,\theta} (\frac{1}{2}\Delta_{V^h,1}+ \tilde{Q}_{A,L,V^h}-z)^{-1}U_{\pm,\theta}^{-1}Q_{A,L',V^h}\|_{\mathcal{L}(\tilde{\mathcal{W}}^{0,s};\tilde{\mathcal{W}}^{2,s'})}
=\frac{d}{2}
\|(\frac{1}{2}\Delta_{V^h,1}+ \tilde{Q}_{A,L,V^h}-z)^{-1}\tilde{Q}_{A,L',V^h}\|_{\mathcal{L}(H^{s};H^{s'})}\,.
$$
With $\tilde{Q}_{A,L',V^h}=A^{2}\chi(\frac{C_{d}+\Delta_{V^h,1}}{(L'A)^{2}})$\,, the inequality \eqref{eq:fss2} gives
$$
\|\tilde{Q}_{A,L',V^h}u\|_{H^{s'}}\leq C_{L',s,s'}A^{2}A^{(s'-s)_{+}}\|u\|_{H^{s}};
$$
The inequality~\ref{eq:resHs} gives
$$
\|(\frac{1}{2}\Delta_{V^h,1}+\tilde{Q}_{A,L,V^h}-z)^{-1}\|_{\mathcal{L}(H^{s_2};H^{s_2})}\leq \frac{4}{A^2+2|\mathrm{Im}z|} 
$$
when $\frac{1}{A}\leq\frac{1}{C_{s'}}$ and $C_{s'}\geq 1$ large enough, and
$$
\|(\frac{1}{2}\Delta_{V^h,1}+\tilde{Q}_{A,L,V^h}-z)^{-1}\tilde{Q}_{A,L',V^h}u\|_{\mathcal{L}(H^s;H^{s'})}\leq C_{s,s'}\frac{A^{(s'-s)_{+}}}{1+2|\mathrm{Im}z|A^{-2}}\,.
$$
We conclude that
$$
\|U_{\pm,\theta} (\frac{1}{2}\Delta_{V^h}+ \tilde{Q}_{A,L,V^h}-z)^{-1}U_{\pm,\theta}^{-1}Q_{A,L',V^h}\|_{\mathcal{L}(\tilde{\mathcal{W}}^{0,s};\tilde{\mathcal{W}}^{2,s'})}
\leq C_{s,s'}\frac{A^{(s'-s)_{+}}}{1+2|\mathrm{Im}z|A^{-2}} \,.
$$
For the left arrow,  the expression \eqref{eq:expressionOfRightInverse} of $E_{-+}^{-1}$ combined with the subelliptic estimate of Proposition~\ref{pr:subEllipticQALVh} (and Proposition~\ref{pr:subellipticEstimateA}) gives
$$ \| (E_{-+})_{\ell}^{-1} \|_{\mathcal{L}(H^s;H^s)} \leq \frac{C_{s}}{A^2(1+b\sqrt{|\mathrm{Im}~z|})} \,.$$
We can now conclude for the norm estimate of $I.*$ decomposed as \eqref{eq:composition1}:
\begin{itemize}
\item For $(I.*)=(I.1)$, we take $s'=s+\frac{7}{2}$ and we obtain
\begin{eqnarray*}
  \|(I.1)\|_{\mathcal{L}(\tilde{\mathcal{W}}^{0,s};\tilde{\mathcal{W}}^{0,s})} &\leq & C_s\frac{A^{-2}}{1+b\sqrt{|\mathrm{Im}z|}} \|(III.1)\|_{\mathcal{L}(\tilde{\mathcal{W}}^{0,s};\tilde{\mathcal{W}}^{0,s+\frac{7}{2}})}\frac{A^{\frac{7}{2}}}{1+2|\mathrm{Im}z|A^{-2}} \\
  &\leq & \frac{C_s}{(1+b\sqrt{|\mathrm{Im}z|})(1+2|\mathrm{Im}z|A^{-2})} bA^{\frac{3}{2}}\,.
 \end{eqnarray*}
\item For $(I,*) = (I.2) $, we take $s'=s+2$ and we obtain
  $$ \|(I.2)\|_{\mathcal{L}(\tilde{\mathcal{W}}^{0,s};\tilde{\mathcal{W}}^{0,s})} \leq C_s\frac{|z|b^2}{(1+b\sqrt{|\mathrm{Im}z|})(1+2|\mathrm{Im}z|A^{-2})}\,. $$
\item For $(I,*)= (I.3) $, we take $s'=s+1 $ and we obtain
  $$ \|(I.3)\|_{\mathcal{L}(\tilde{\mathcal{W}}^{0,s};\tilde{\mathcal{W}}^{0,s})} \leq C_sA^{-2} \|(III.3)\|_{\mathcal{L}(\tilde{\mathcal{W}}^{0,s};\tilde{\mathcal{W}}^{0,s+1})}A \leq C_s \frac{b}{A(1+b\sqrt{|\mathrm{Im}z|})(1+2|\mathrm{Im}z|A^{-2})}\,. $$
\end{itemize}
we proved
$$\|(I)\|_{\mathcal{L}(\tilde{\mathcal{W}}^{0,s};\tilde{\mathcal{W}}^{0,s})} \leq C_s \frac{bA^{\frac{3}{2}} +b^2|z| +\frac{b}{A}}{(1+b\sqrt{|\mathrm{Im}z|})(1+2|\mathrm{Im}z|A^{-2})} \,.$$
The operator $(II) $ is depicted by
\begin{center}
\begin{tikzcd}
{\tilde{\cal W}^{0,s}} &  &  & {\tilde{\cal W}^{0,s}} \arrow[lll, "{(\overline{B_{\pm,b,V^h,\bot}}^{s}-z)^{-1}\pi_{\bot,\pm}}"'] &  &  & {\tilde{\cal W}^{2,s+1}} \arrow[lll, "\frac{1}{b}\beta_{\pm}+\gamma_{\pm}"'] &  &  & {\tilde{\cal W}^{2,s+1}}\arrow[lll, "{U_{\pm,\theta}(E_{-+})_{\ell}^{-1}U_{\pm,\theta}^{-1}}"'] &  & {\tilde{\cal W}^{0,s}} \arrow[ll, "{Q_{A,L',V^h}}"']\,.
\end{tikzcd}
\end{center}
The same arguments as above lead to
$$ \|(II)\|_{\mathcal{L}(\tilde{\mathcal{W}}^{0,s};\tilde{\mathcal{W}}^{0,s})} \leq C_s\frac{Ab}{(1+b\sqrt{|\mathrm{Im}z|})(1+2|\mathrm{Im}z|A^{-2})} .$$
The largest upper bound is $ C_s \frac{bA^{\frac{3}{2}} +b^2|z|}{(1+b\sqrt{|\mathrm{Im}z|})(1+2|\mathrm{Im}z|A^{-2})}) $ obtained for the term $(I)$\,.\\
Under the condition $|\real z| \leq A^2$ we have the inequaity 
$$ \frac{|z|b^2}{(1+|\mathrm{Im}z|A^{-2})(1+b\sqrt{|\mathrm{Im}z|})} \leq A^2b^2 + \frac{\sqrt{|\mathrm{Im}z|}b}{1+2|\mathrm{Im}z|A^{-2}} \,. $$
This leads to 
$$ \| \delta_{B,\Delta,z}\circ Q_{A,L',V^h} \|_{\mathcal{L}(\tilde{\cal W}^{0,s};\tilde{\cal W}^{0,s})}\leq C_s\frac{bA^{\frac{3}{2}}+b\sqrt{\mathrm{Im}(z)}}{1+|\mathrm{Im}z|A^{-2}} \leq C_s'\frac{bA^{\frac{3}{2}}}{1+A^{-1}\sqrt{|\mathrm{Im}z|}}\,.
$$
 Finally the estimate \eqref{eq:Qrestronc} for  $Q_{A,L',V^h}\delta_{B,\Delta,z}$ is obtained by taking the adjoints with $z$ replaced by $\overline{z}$ and $B_{\pm,b,V^h}$ replaced by its formal adjoint $B_{\pm,b,V^h}'$ which has the same properties as $B_{\pm,b,V^h}$\,.
\end{proof}

For all $z\not \in \mathrm{Spec}(\frac{1}{2}\Delta_{V^h,1} + \tilde{Q}_{A,L,V^h} )$\,, we define the intermediate operators
\begin{align*}
   M_{B,z}&= I_{\tilde{\cal W}^{0,s}}-(B_{\pm,b,V^h}+Q_{A,L,V^h}-z)^{-1}Q_{A,L,V^h} \\
  M_{\Delta,z} & =   I_{\tilde{\cal W}^{0,s}} - U_{\pm,\theta}(\frac{1}{2}\Delta_{V^h,1} + \tilde{Q}_{A,L,V^h}-z)^{-1}U_{\pm,\theta}^{-1}Q_{A,L,V^h}\\
  \tilde{M}_{\Delta,z} &=  I_{H^s} - (\frac{1}{2}\Delta_{V^h,1}+\tilde{Q}_{A,L,V^h}-z)^{-1}\tilde{Q}_{A,L,V^h} \\
\end{align*}
while the other ordered products are recovered by taking the formal adjoints
\begin{align*}
   M_{B',\overline{z}}'&= I_{\tilde{\cal W}^{0,s}}-Q_{A,L,V^h} (B_{\pm,b,V^h}+Q_{A,L,V^h}-z)^{-1}\\
  M_{\Delta,\overline{z}}' & =   I_{\tilde{\cal W}^{0,s}} - Q_{A,L,V^h} U_{\pm,\theta}(\frac{1}{2}\Delta_{V^h,1} + \tilde{Q}_{A,L,V^h}-z)^{-1}U_{\pm,\theta}^{-1} \\
  \tilde{M}_{\Delta,\overline{z}}' &=  I_{H^s} - \tilde{Q}_{A,L,V^h}(\frac{1}{2}\Delta_{V^h,1}+\tilde{Q}_{A,L,V^h}-z)^{-1}\,.
\end{align*}
\begin{lemma}For $s\in \mathbb{R}$, there is a constant $C_s\geq 1$ depending on $s$ and for all $z\in \mathbb{C}$ the inequality holds
  \begin{eqnarray}
     \|\tilde{M}_{\Delta,z}^{-1}\|_{\mathcal{L}(H^s;H^s)} &\leq & C_s(1+\frac{A^2}{\mathrm{dist}(z,\mathrm{Spec}(\frac{1}{2}\Delta_{V^h,1}))})  \nonumber \\ 
    \label{eq:MzInv}
    \|M_{\Delta,z}^{-1}\|_{\mathcal{L}(\tilde{\cal W}^{0,s};\tilde{\cal W}^{0,s})} & \leq & C_s\left(1+\frac{A^2}{\mathrm{dist}(z,\mathrm{Spec}(\Delta_{V^h,1}))}\right)                                                        
  \end{eqnarray}
  for all $z\not\in \mathrm{Spec}(\frac{1}{2}\Delta_{V^h,1})\cap \mathrm{Spec}(\frac{1}{2}\Delta_{V^h,1}+\tilde{Q}_{A,L,V^h})$\,. \\
 The more accurate conditions
 \begin{eqnarray*}
&&  |\real z|\leq \frac{A^2}{2}\quad,\quad bA^4\mathrm{dist}~(z,\frac{1}{2}\mathrm{Spec}~(\Delta_{V^h,1})\leq \frac{1}{C_s} \,,
\\
&& C_s \max(b,Ab,A^{-1})\leq 1\,,
\end{eqnarray*}
  suffice for the uniform estimates
  \begin{eqnarray}
  \label{eq:MdBzQ}
   &&\| M_{\Delta,z}^{-1}\delta_{B,\Delta,z}Q_{A,L,V^h} \|_{\mathcal{L}(\tilde{\cal W}^{0,s};\tilde{\cal W}^{0,s})} 
  + \| M_{B,z}^{-1}\delta_{B,\Delta,z}Q_{A,L,V^h} \|_{\mathcal{L}(\tilde{\cal W}^{0,s};\tilde{\cal W}^{0,s})}
 \leq \frac{1}{2}\,\\
    \label{eq:IBQinv}
    &&\|M_{B,z}^{-1}\|_{\mathcal{L}(\tilde{\cal W}^{0,s};\tilde{\cal W}^{0,s})} \leq  C_s\left(1+\frac{A^2}{\mathrm{dist}(z,\mathrm{Spec}(\Delta_{V^h,1}))}\right) \\
&&
\label{eq:estimation300}
\|(M_{B',\bar z}')^{-1}\|_{\mathcal{L}(\tilde{\cal W}^{0,s};\tilde{\cal W}^{0,s})} \leq  C_s\left(1+\frac{A^2}{\mathrm{dist}(z,\mathrm{Spec}(\Delta_{V^h,1}))}\right)\,.
  \end{eqnarray}
\end{lemma}

\begin{proof}
  A straightforward computation gives
  $$ \tilde{M}_{\Delta,z} =  (\frac{1}{2}\Delta_{V^h,1}+\tilde{Q}_{A,L,V^h}-z)^{-1} (\frac{1}{2}\Delta_{V^h,1}-z) \,.$$
 We deduce that the operator $\tilde{M}_{\Delta,z} $ is invertible when $z\not \in \mathrm{Spec}(\frac{1}{2}\Delta_{V^h,1}) $ and the inverse is given by
  \begin{eqnarray*}
    \tilde{M}_{\Delta,z}^{-1} & = & (\frac{1}{2}\Delta_{V^h,1}-z)^{-1}(\frac{1}{2}\Delta_{V^h,1}+\tilde{Q}_{A,L,V^h}-z)  \\
    & = & I+\tilde{Q}_{A,L,V^h}(\frac{1}{2}\Delta_{V^h,1}-z)^{-1}\,.
  \end{eqnarray*}
  The functional calculus for the self adjoint operator $ \frac{1}{2}\Delta_{V^h,1} $ yields
  $$\|\tilde{M}_{\Delta,z}^{-1}\|_{\mathcal{L}(H^s;H^s)}\leq C_s \left(1+ \frac{A^2}{\mathrm{dist}(z,\mathrm{Spec}(\frac{1}{2}\Delta_{V^h,1}))}\right)\,.$$
    The operator $M_{\Delta,z}$ is actually invertible when $z\not\in \mathrm{Spec}(\frac{1}{2}\Delta_{V^h,1})$ for the following reason:
  \begin{align*}
    M_{\Delta,z}
    &=\pi_{\bot,\pm}+\pi_{0,\pm}\big[I_{\ker(\alpha_{\pm})} - U_{\pm,\theta}(\frac{1}{2}\Delta_{V^h,1} + \tilde{Q}_{A,L,V^h}-z)^{-1}U_{\pm,\theta}^{-1}Q_{A,L,V^h}\big]\pi_{0,\pm}\\
       &=\pi_{\bot,\pm}+U_{\pm,\theta} \tilde{M}_{\Delta,z} U_{\pm,\theta}^{-1}
  \end{align*}
  and its inverse equals
 \begin{align*}
   M_{\Delta,z}^{-1}
   &=\pi_{\bot,\pm}+U_{\pm,\theta}\tilde{M}_{\Delta,z}^{-1}U_{\pm,\theta}^{-1} \\
   & = I_{\tilde{\mathcal{W}}^{0,s}} +U_{\pm,\theta} \big[ \tilde{Q}_{A,L,V^h}(\frac{1}{2}\Delta_{V^h,1}-z)^{-1} \big] U_{\pm,\theta}^{-1}.
 \end{align*}
 The estimate \eqref{eq:MzInv} of $M_{\Delta,z}^{-1}$ follows.\\
 Another computation gives 
 $$ M_{B,z} = M_{\Delta,z}(I-M_{\Delta,z}^{-1}\delta_{B,\Delta,z}Q_{A,L,V^h}) $$
 we deduce that it is invertible as soon as $\| M_{\Delta,z}^{-1}\delta_{B,\Delta,z}Q_{A,L,V^h} \|_{\mathcal{L}(\tilde{\cal W}^{0,s};\tilde{\cal W}^{0,s})}< 1$ and the norm of its inverse is given by
 $$\|M_{B,z}^{-1}\|_{\mathcal{L}(\tilde{\cal W}^{0,s};\tilde{\cal W}^{0,s})} \leq \frac{\| M_{\Delta,z}^{-1} \|_{\mathcal{L}(\tilde{\cal W}^{0,s};\tilde{\cal W}^{0,s})}}{1-\| M_{\Delta,z}^{-1}\delta_{B,\Delta,z}Q_{A,L,V^h} \|_{\mathcal{L}(\tilde{\cal W}^{0,s};\tilde{\cal W}^{0,s})}}\,. $$
 Applying inequality~\eqref{eq:restoncQ} yields
\begin{eqnarray*}
   \| M_{\Delta,z}^{-1}\delta_{B,\Delta,z}Q_{A,L,V^h} \|_{\mathcal{L}(\tilde{\cal W}^{0,s};\tilde{\cal W}^{0,s})} & \leq &
 C_s\left(1+ \frac{A^2}{\mathrm{dist}(z,\mathrm{Spec}(\frac{1}{2}\Delta_{V^h,1}))}\right)
 \left(\frac{bA^{\frac{3}{2}}}
 {1+A^{-1}\sqrt{|\mathrm{Im}z|}} \right) \\
 & \leq & \left\lbrace \begin{array}{lc}
     C_s \frac{bA^{\frac{7}{2}}}{\mathrm{dist}(z,\mathrm{Spec}(\frac{1}{2}\Delta_{V^h,1}))} & \text{if } |\mathrm{Im}z| \leq A^2 \\
    C_s \frac{bA^{\frac{5}{2}}}{\sqrt{|\mathrm{Im}z|}}  & \text{if } |\mathrm{Im}z| \geq A^2 
 \end{array}\right.
\end{eqnarray*}
Conditions $|\mathrm{Im}z|\geq A^2 $ and $|\real z| \leq \frac{A^2}{2} $ ensure 
\begin{eqnarray}\label{eq:equivImzAndDist}
  \mathrm{dist}(z,\mathrm{Spec}(\frac{1}{2}\Delta_{V^h,1}))\geq |\mathrm{Im}z|\geq \frac{1}{2}|z|\geq \frac{1}{2} \mathrm{dist}(z,\mathrm{Spec}(\frac{1}{2}\Delta_{V^h,1})).
\end{eqnarray}
The condition $C_s \max(Ab,b,A^{-1})\leq 1 $ and \ref{eq:equivImzAndDist} allow us to give a sufficient condition for
$$
\|M_{\Delta,z}^{-1}\delta_{B,\Delta,z}Q_{A,L,V^h}\|_{\mathcal{L}(\tilde{\cal W}^{0,s};\tilde{\cal W}^{0,s})} < \frac{1}{2}
$$
which is $b\frac{A^4}{\mathrm{dist}(z,\mathrm{Spec}(\frac{1}{2}\Delta_{V^h,1}))} \leq \frac{1}{4C_s}$\,.
 Finally the estimate $M_{B',\bar{z}}'$ is obtained by taking the adjoints with $z$ replaced by $\overline{z}$ and $B_{\pm,b,V^h}$ replaced by its formal adjoint $B_{\pm,b,V^h}'$, which has the same properties as $B_{\pm,b,V^h}$\,.
\end{proof}

\subsection{Quantitative comparison of resolvents}\label{sec:quandres}

When $\real z\leq\frac{A^{2}}{2}$\,, Proposition~\ref{pr:subEllipticQALVh} and the 
subelliptic estimate \eqref{eq:subellipticEstimateWith2/5ForBismutLaplacian} for $B_{\pm,b,V^h}+A^{2}\pi_{0,\pm}-z$ say that $z$ belongs to the resolvent set of $\overline{B_{\pm,b,V^h} +Q_{A,L,V^h}}^{s}$
and
   \begin{eqnarray}
     B_{\pm,b,V^h}-z  &=&  (I -Q_{A,L,V^h}(B_{\pm,b,V^h}+Q_{A,L,V^h}-z)^{-1})(B_{\pm,b,V^h}+Q_{A,L,V^h}-z) \nonumber \\
     \label{eq:MBZb}&  = & M_{B',\bar{z}}'(B_{\pm,b,V^h}+Q_{A,L,V^h}-z) \hspace{2cm} \text{in}~\mathcal{L}(D(\overline{B_{\pm,b,V^h}}^{s});\tilde{\mathcal{W}}^{0,s} )\\
     \text{and} \quad 
     B_{\pm,b,V^h}-z &= &(B_{\pm,b,V^h}+Q_{A,L,V^h}-z)(I -(B_{\pm,b,V^h}+Q_{A,L,V^h}-z)^{-1}Q_{A,L,V^h}) \nonumber \\
     \label{eq:MBZb1}& = & (B_{\pm,b,V^h}+Q_{A,L,V^h}-z) M_{B,z} \hspace{2cm} \text{in}~\mathcal{L}(\tilde{\cal W}^{0,s};D((B_{\pm,b,V^h}^{*,s})'))\,.
   \end{eqnarray}
 Because Corollary~\ref{cor:maxaccr} says that $\overline{B_{\pm,b,V^{h}}+C_{s}}^{s}$ is maximal accretive, we focus on the case $\left|\real z\right|\leq\frac{A^{2}}{2}$ under the condition $A\geq C_{s}\geq 1$\,.
   
\begin{proposition}\label{pr:estimateForTheConvergence}
   For $s\in \mathbb{R}$ there exists $C_{s}\geq 1$ such that the conditions $C_{s}\max(Ab,b,A^{-1})\leq 1$
     and
     $$
     |\real z|\leq \frac{A^2}{2} \quad,\quad \mathrm{dist}(z,\mathrm{Spec}(\frac{1}{2}\Delta_{V^{h},1}))\geq C_s bA^4\,,
     $$
     imply the inequalities
     \begin{multline}\label{eq:comRes}
\|\big((B_{\pm,b,V^h}-z)^{-1} - U_{\pm,\theta} (\frac{1}{2}\Delta_{V^h,1}-z)^{-1}U_{\pm,\theta}^{-1}\big)\|_{\mathcal{L}(\tilde{\mathcal{W}}^{0,s};\tilde{\mathcal{W}}^{0,s})}
\\ 
\leq \left[ \left(1+ \frac{A^2}{\mathrm{dist}(z,\mathrm{Spec}(\frac{1}{2}\Delta_{V^h,1}))}\right)^2 bA^{-\frac{1}{2}}+A^{-2}\right]\frac{C_{s}}{1+b\sqrt{|\mathrm{Im}z|}}\,.
     \end{multline}
   \end{proposition}

   \begin{proof}
    Let us decompose $u\in \tilde{\cal W}^{0,s}$ as $u=u_{low}+u_{high} $ with
$$
\left\lbrace\begin{array}{ccl}
  u_{low} & = & U_{\pm,\theta}\chi(\frac{C_d + \Delta_{V^h,1}}{(2LA)^2})U_{\pm,\theta}^{-1}u  = \frac{1}{A^2} Q_{A,L',V^h}u \quad \textrm{with} \quad L'=2L \\
  u_{high} & = & u-u_{low}
\end{array}\right.\,.
$$
The upper bound on the norm of the operator $\mathfrak{D}=\big((B_{\pm,b,V^h}-z)^{-1} - U_{\pm,\theta} (\frac{1}{2}\Delta_{V^h,1}-z)^{-1}U_{\pm,\theta}^{-1}\big) $ will be obtained by considering separately its action on the two pieces $u_{high}$ and $u_{low}$.\\
For $u_{high}$ we have
\begin{eqnarray*}
     \| \mathfrak{D} u_{high} \|_{\tilde{\cal W}^{0,s}}\leq \| (B_{\pm,b,V^h}-z)^{-1}u_{high} \|_{\tilde{\cal W}^{0,s}} + \|U_{\pm,\theta}( \frac{1}{2}\Delta_{V^h,1}-z )^{-1}U_{\pm,\theta}^{-1} u_{high}\|_{\tilde{\cal W}^{0,s}}\,.
\end{eqnarray*}
The second term of the right-hand side is bounded by
\begin{eqnarray*}
  \|U_{\pm,\theta}( \frac{1}{2}\Delta_{V^h,1}-z )^{-1}U_{\pm,\theta}^{-1}u_{high}\|_{\tilde{\cal W}^{0,s}}  & = & \|U_{\pm,\theta} (\frac{1}{2}\Delta_{V^h,1}-z)^{-1}(1-\chi(\frac{C_d+\Delta_{V^h,1}}{(2LA)^2})) U_{\pm,\theta}^{-1}u\|_{\tilde{\cal W}^{0,s}}\\
  & \leq & \| (\frac{1}{2}\Delta_{V^h}-z)^{-1}(1-\chi(\frac{C_d+\Delta_{V^h,1}}{(2LA)^2})) \|_{\mathcal{L}(H^s;H^s)} \|u\|_{\tilde{\cal W}^{0,s}}\,.
\end{eqnarray*}
The choice of the cut-off  function $\chi$ and $\tilde{Q}_{A,L,V^h}=A^2\chi\left(\frac{C_d +\Delta_{V^h,1}}{(LA)^2}\right)$ ensure the equality
$$
(\frac{1}{2}\Delta_{V^h,1}-z)^{-1}(1-\chi(\frac{C_d+\Delta_{V^h,1}}{(2LA)^2})) = (\frac{1}{2}\Delta_{V^h,1}+\tilde{Q}_{A,L,V^h}-z)^{-1}\circ (1-\chi(\frac{C_d+\Delta_{V^h,1}}{(2LA)^2})) 
$$
where the two factors satisfy
\begin{align*}
    &\|(\frac{1}{2}\Delta_{V^h,1}+\tilde{Q}_{A,L,V^h}-z)^{-1}\|_{\mathcal{L}(H^{s};H^{s})}\leq \frac{4}{A^2+2|\mathrm{Im}z|}\\
    \text{and} \quad & \|(1-\chi(\frac{C_d+\Delta_{V^h,1}}{(2LA)^2})) \|_{\mathcal{L}(H^{s};H^{s})}\leq C_{s}
\end{align*}
respectively according to inequality~\eqref{eq:fss2} and inequality~\eqref{eq:resHs}. We have proved
     \begin{equation*}
  \|  (\frac{1}{2}\Delta_{V^h}-z)^{-1}(1-\chi(\frac{C_d+\Delta_{V^h,1}}{(2LA)^2}))\|_{\mathcal{L}(H^s;H^s)} 
    \leq  \frac{C_s}{A^2+2|\mathrm{Im}z|} \leq \frac{C_s}{A^2} \frac{1}{1+b\sqrt{|\mathrm{Im}z|}}\,.
  \end{equation*}
According to the formula \eqref{eq:MBZb}, the inverse of $B_{\pm,b,V^h}-z $ equals
\begin{eqnarray*}
  (B_{\pm,b,V^h}-z)^{-1} & = & (B_{\pm,b,V^h}+Q_{A,L,V^h}-z)^{-1}M_{B',\bar{z}}^{\prime -1}\\
                       & = & (B_{\pm,b,V^h}+Q_{A,L,V^h}-z)^{-1}\left( M_{\bar{z}}^{\prime -1} + [M_{B',\bar{z}}^{\prime -1} - M_{\bar{z}}^{\prime -1}] \right) \\
  & = & (B_{\pm,b,V^h}+Q_{A,L,V^h}-z)^{-1}\left( I + M_{B',\bar{z}}^{\prime -1}\left[Q_{A,L,V^h}\delta_{B,\Delta,z} \right] \right)M_{\bar{z}}^{\prime -1}\,.
\end{eqnarray*}
Therefore, the subelliptic estimate given in Proposition \ref{pr:subEllipticQALVh} for $B_{\pm,b,V^h}+Q_{A,L,V^h}-z $ when $|\real z|\leq \frac{A^2}{2}$,
$$
\| (B_{\pm,b,V^h}+Q_{A,L,V^h}-z)^{-1} \|_{\mathcal{L}(\tilde{\cal W}^{0,s};\tilde{\cal W}^{0,s})} \leq \frac{C_s}{A^2(1+b\sqrt{|\mathrm{Im}z|})}\,,
$$
and the inequality \eqref{eq:MdBzQ} imply

\begin{eqnarray*}
  \| (B_{\pm,b,V^{h}}-z)^{-1}u_{high} \|_{\tilde{\cal W}^{0,s}} & \leq &  C_s\frac{1}{A^2(1+b\sqrt{|\mathrm{Im}z|})}\|M_{\bar{z}}^{\prime -1}u_{high}\|_{\tilde{\cal W}^{0,s}} \\
  & = & C_s\frac{1}{A^2(1+b\sqrt{|\mathrm{Im}z|})} \|u_{high}\|_{\tilde{\cal W}^{0,s}}\,.
\end{eqnarray*}
We have proved 
$$ \| \mathfrak{D}u_{high} \|_{\tilde{\cal W}^{0,s}}\leq \frac{C_s}{A^2}\frac{1}{1+b\sqrt{|\mathrm{Im}z|}} \,.$$
The last inequality comes from the fact that $Q_{A,L,V^h}u_{high} = U_{\pm,\theta}\chi(\frac{C_d+\Delta_{V^h,1}}{(LA)^2}) (1- \chi(\frac{C_d+\Delta_{V^h,1}}{(2LA)^2})) U_{\pm,\theta}^{-1}u = 0  $ implies $M_{\bar{z}}^{\prime -1} u_{high} = (I-Q_{A,L,V^h}U_{\pm,\theta}(\frac{1}{2}\Delta_{V^h,1}-z+Q_{A,L,V^h})^{-1}U_{\pm,\theta}^{-1})^{-1} u_{high} = u_{high} $.\\
For the $u_{low}$-component, $u_{low}=\frac{1}{A^2} Q_{A,L',V^{h}}u$\,, the formula \eqref{eq:MBZb1} for $B_{\pm,b,V^h}$ and the analogous one for $\frac{1}{2}\Delta_{V^h,1} $ give
\begin{eqnarray}
  \nonumber \mathfrak{D} & = & M_{B,z}^{-1}(B_{\pm,b,V^h}+Q_{A,L,V^h}-z)^{-1}  - U_{\pm,\theta}\tilde{M}_{\Delta,z}^{-1}(\frac{1}{2}\Delta_{V^h,1}+\tilde{Q}_{A,L,V^h}-z)^{-1}U_{\pm,\theta}^{-1} \\
      & = & (M_{B,z})^{-1}\delta_{B,\Delta,z} + \bigg[(M_{B,z})^{-1} - U_{\pm,\theta} \tilde{M}_{\Delta,z}^{-1} U_{\pm,\theta}^{-1}\bigg]U_{\pm,\theta}(\frac{1}{2}\Delta_{V^h,1}+\tilde{Q}_{A,L,V^h} -z)^{-1}U_{\pm,\theta}^{-1}
\label{eq:MBzMz}
\end{eqnarray}
Owing to
\begin{eqnarray*}
 (M_{B,z})^{-1}\delta_{B,\Delta,z} u_{low} = \frac{1}{A^2}(M_{B,z})^{-1}\delta_{B,\Delta,z}Q_{A,L',V^h}u\,,
\end{eqnarray*}
the inequality \eqref{eq:restoncQ} combined with the inequality \eqref{eq:estimation300} imply that the first term of \eqref{eq:MBzMz} applied to $u_{low}$ is bounded by
$$ \|(M_{B,z})^{-1}\delta_{B,\Delta,z} u_{low}\|_{\tilde{\cal W}^{0,s}}\leq \frac{C_s}{A^2}\left(1+ \frac{A^2}{\mathrm{dist}(z,\mathrm{Spec}(\frac{1}{2}\Delta_{V^h,1}))}\right)
 \left(\frac{bA^{\frac{3}{2}}}
   {1+A^{-1}\sqrt{|\mathrm{Im}z|}} \right)\|u\|_{\tilde{\cal W}^{0,s}}  \,.$$
On $\ker(\alpha_{\pm})=\mathrm{Ran}(U_{\pm,\theta}) $\,, we know
$$
U_{\pm,\theta}\tilde{M}_{\Delta,z}^{-1}U_{\pm,\theta}^{-1} = M_{\Delta,z}^{-1}\,.
$$
The second term of \eqref{eq:MBzMz} thus equals

\begin{multline*}
  \bigg[(M_{B,z})^{-1} - U_{\pm,\theta} \tilde{M}_{\Delta,z}^{-1} U_{\pm,\theta}^{-1}\bigg]U_{\pm,\theta}(\frac{1}{2}\Delta_{V^h,1}+\tilde{Q}_{A,L,V^h} -z)^{-1}U_{\pm,\theta}^{-1} \\
  = \bigg[(M_{B,z})^{-1} - M_{\Delta,z}^{-1}\bigg]U_{\pm,\theta}(\frac{1}{2}\Delta_{V^h,1}+\tilde{Q}_{A,L,V^h} -z)^{-1}U_{\pm,\theta}^{-1}\,.
\end{multline*}
Further computations show that the right-hand side is
\begin{eqnarray*}
  &&\hspace{-3cm}
  \bigg[(M_{B,z})^{-1} - M_{\Delta,z}^{-1} \bigg]U_{\pm,\theta}(\frac{1}{2}\Delta_{V^h,1}+\tilde{Q}_{A,L,V^h} -z)^{-1}U_{\pm,\theta}^{-1} \\
  & = & M_{B,z}^{-1}\bigg[M_{\Delta,z}-M_{B,z}\bigg]M_{\Delta,z}^{-1}U_{\pm,\theta}(\frac{1}{2}\Delta_{V^h,1}+\tilde{Q}_{A,L,V^h} -z)^{-1}U_{\pm,\theta}^{-1}  \\
   & = & M_{B,z}^{-1}\bigg[M_{\Delta,z}-M_{B,z}\bigg]U_{\pm,\theta}\tilde{M}_{\Delta,z}^{-1}(\frac{1}{2}\Delta_{V^h,1}+\tilde{Q}_{A,L,V^h} -z)^{-1}U_{\pm,\theta}^{-1}\\
  & = & M_{B,z}^{-1} \bigg[\delta_{B,\Delta,z}Q_{A,L,V^h}\bigg] \bigg(U_{\pm,\theta} (\frac{1}{2}\Delta_{V^h,1}-z)^{-1}  U_{\pm,\theta}^{-1}\bigg)\,.
\end{eqnarray*}
By using again inequality \eqref{eq:IBQinv} and Proposition \ref{pr:restronc} for the right-hand side, the above operator is estimated by
\begin{multline*}
  \bigg\|\bigg[(M_{B,z})^{-1} - U_{\pm,\theta} \tilde{M}_{\Delta,z}^{-1} U_{\pm,\theta}^{-1}\bigg]U_{\pm,\theta}(\frac{1}{2}\Delta_{V^h,1}+\tilde{Q}_{A,L,V^h} -z)^{-1}U_{\pm,\theta}^{-1}\bigg\|_{\mathcal{L}(\tilde{\cal W}^{0,s};\tilde{\cal W}^{0,s})}  \\
  \leq C_s\left(1+ \frac{A^2}{\mathrm{dist}(z,\mathrm{Spec}(\frac{1}{2}\Delta_{V^h,1}))}\right)
 \left(\frac{bA^{\frac{3}{2}}}
   {1+A^{-1}\sqrt{|\mathrm{Im}z|}} \right) \frac{1}{\mathrm{dist}(z,\mathrm{Spec}(\frac{1}{2}\Delta_{V^h,1}))} \,.
\end{multline*}
By adding the two terms we obtain 
$$\|\mathfrak{D}u_{low}\|_{\tilde{\cal W}^{0,s}}\leq C_s\left(1+ \frac{A^2}{\mathrm{dist}(z,\mathrm{Spec}(\frac{1}{2}\Delta_{V^h,1}))}\right)^2
 \left(\frac{bA^{-\frac{1}{2}}}
   {1+A^{-1}\sqrt{|\mathrm{Im}z|}} \right) \,.$$
 By summing the two upper bounds for  $\|\mathfrak{D}u_{high}\|_{\tilde{\mathcal{W}}^{0,s}}$ and $\|\mathfrak{D}u_{low}\|_{\tilde{\mathcal{W}}^{0,s}}$\,, we get
$$\|\mathfrak{D}\|_{\mathcal{L}(\tilde{\cal W}^{0,s};\tilde{\cal W}^{0,s})}\leq C_s\left[ \left(1+ \frac{A^2}{\mathrm{dist}(z,\mathrm{Spec}(\frac{1}{2}\Delta_{V^h,1}))}\right)^2
 \left(\frac{bA^{-\frac{1}{2}}}
   {1+A^{-1}\sqrt{|\mathrm{Im}z|}} \right) + \frac{1}{A^2} \frac{1}{1+b\sqrt{|\mathrm{Im}z|}}  \right] \,,$$
which can be simplifed into \eqref{eq:comRes} owing to $b\leq A^{-1}$\,.
\end{proof}

\section{Spectral consequences}
\label{sec:specconseq}

This final section actually ends the proof of Theorem~\ref{th:main} and its various statements are picked from Proposition~\ref{pr:spectrum}, Proposition~\ref{pr:semiGroup} and Proposition~\ref{pr:accSpec}. A simple translation is obtained after recalling the unitary equivalences
\begin{eqnarray*}
   B_{\pm,b,\frac{V}{h}}=B_{\pm,(Q,g,\frac{V}{h},b)} &\longleftrightarrow & \frac{1}{h^2}B_{\pm,(Q^h,g^h,V^h,b')}=\frac{1}{h^2}B_{\pm,b',V^h}\quad b'=\frac{b}{h}\,, \\
   \Delta_{V,h}=\Delta_{(Q,g,V,h)}&\longleftrightarrow &\Delta_{(Q^h,g^h,V^h,1)}=\Delta_{V^h,1}\,\\
   \frac{z}{h^2}\in \mathrm{Spec}(B_{\pm,b,\frac{V}{h}})&\longleftrightarrow & z\in \mathrm{Spec}(B_{\pm,b',V^h})\quad b'=\frac{b}{h}\,,\\
   \tilde{\mathcal{W}}^{s_1,s_2}_h(X,\mathcal{E}_\pm)&\longleftrightarrow& \tilde{\mathcal{W}}^{s_1,s_2}_{h=1}(X^h,\mathcal{E}_\pm^h)\,.
\end{eqnarray*}
Once this is fixed the first statement \textbf{a)} of Theorem~\ref{th:main} is a corollary of Proposition~\ref{pr:spectrum}.  The second statement \textbf{b)} of Theorem~\ref{th:main} is a rewriting of Proposition~\ref{pr:accSpec} of which the proof strongly relies on the Hodge structure 
of restricted operator $B_{\pm,b,V^h}\big|_{E_{\pm,b,V^h}}$ and where the hermitian form $\langle~,~\rangle_r\Big|_{(E_{\pm,b,V^h})^2}$  is positive definite by the PT-symmetry argument checked in the proof of Proposition~\ref{pr:spectrum}.
Finally the third statement \textbf{c)} of Theorem~\ref{th:main} about the semigroup expansion is a transcription of Proposition~\ref{pr:semiGroup}.
\subsection{Rough estimates}
\label{sec:roughspec}
The data of our problem are the spectrum of the semiclassical Witten Laplacian $\mathrm{Spec}(\Delta_{V,h})=\mathrm{Spec}(\Delta_{V^{h},1})$ and the parameters $b,h\in ]0,1]$\,. We introduced the additional parameter $A\geq 1$ and recall the condition
$$
C_{s}\max (Ab,b,A^{-1})\leq 1\,.
$$
We recall, according to Definition~\ref{def:rhoh},  that the parameter $\varrho_{h}\in]0,1]$  measures a spectral gap for $\Delta_{V,h}$ according to
  \begin{eqnarray}
    &&\mathrm{Spec}(\frac{1}{2}\Delta_{V,h})\cap [0,\varrho_{h}]\subset [0,e^{-\frac{c}{h}}] \subset [0,\frac{\varrho_{h}}{2}] \label{eq:spectralGapLower}\\
       \text{and}&& \mathrm{Spec}(\frac{1}{2}\Delta_{V,h})\cap ]\varrho_{h},+\infty[\subset [4\varrho_{h},+\infty[ \label{eq:spectralGapUpper}
  \end{eqnarray}
  for all $h\in ]0,1]$\,. Remember as well the notations 
  $$
  \mathcal{N}_{\pm}^{(p)}(V)=\mathrm{rank}~1_{[0,\varrho_{h}]}(\frac{1}{2}\Delta_{V,h}^{(p)})
  \quad\text{and}\quad \mathcal{N}_{\pm}(V)=\sum_{p=0}^{d}\mathcal{N}_{\pm}^{(p)}(V)\,,
  $$
 where the $\pm$ sign refers to the choice of the line bundle $F_{+}=Q\times\mathbb{C}$ or $F_{-}=(Q\times \mathbb{C})\otimes\mathbf{or}_{Q}$\,.\\
Making an accurate use of the quantitative comparison of the resolvents in Proposition~\ref{pr:estimateForTheConvergence} requires the identification of different areas in the complex plane, presented in the picture below, and of which the accurate definitions are given just after.
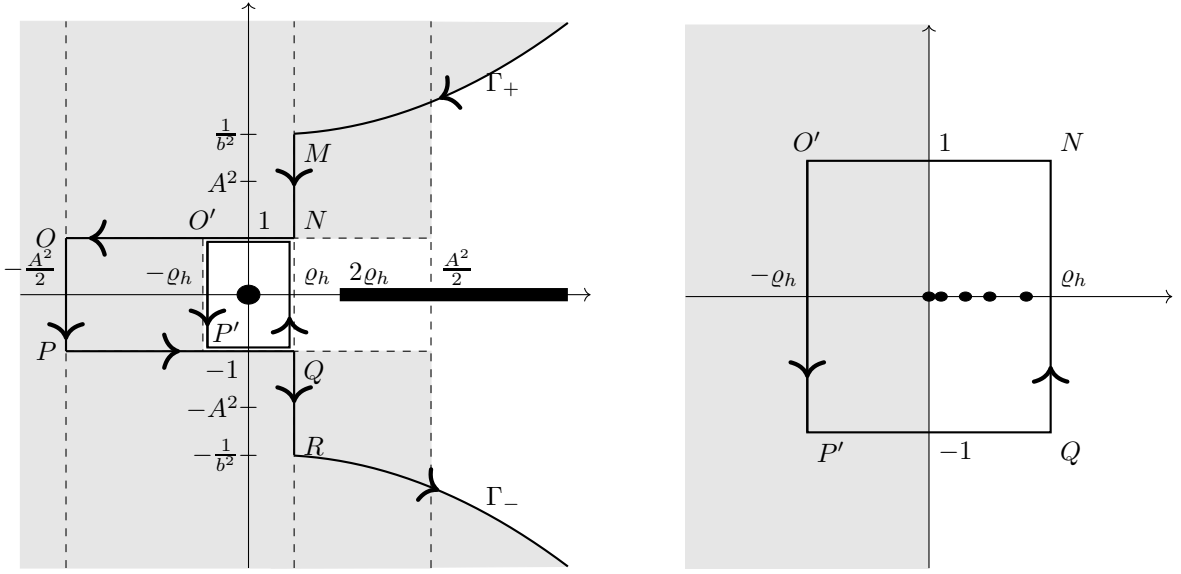
\begin{figure}[h]
  \centering
  \begin{tikzpicture}[xscale=0.6,yscale=0.5]
    \pgfdeclarelayer{nodelayer}
    \pgfdeclarelayer{edgelayer}
    \pgfsetlayers{nodelayer,edgelayer}
    \tikzset{none/.style={coordinate}}
\tikzset{fleche/.style args={#1:#2}{ postaction = decorate,decoration={name=markings,mark=at position #1 with {\arrow[#2,scale=2]{>}}}}}
 	\begin{pgfonlayer}{nodelayer}
		\node [style=none] (0) at (-5, 0) {};
		\node [style=none] (1) at (7.5, 0) {};
		\node [style=none] (2) at (0, -7.25) {};
		\node [style=none] (3) at (0, 7.75) {};
		\node [style=none] (4) at (4., 0) {};
		\node [style=none] (5) at (4., -7.25) {};
		\node [style=none] (6) at (4., 7.25) {};
		\node [style=none] (8) at (-4., -7.25) {};
		\node [style=none] (9) at (-4., 7.25) {};
		\node [style=none] (10) at (-4., 0) {};
		\node [style=none] (11) at (0, 0) {};
		\node [style=none] (12) at (0, 1.5) {};
		\node [style=none] (13) at (0, 3) {$A^{2}$};
		\node [style=none] (14) at (0, 4.25) {};
		\node [style=none] (15) at (0, -1.5) {};
		\node [style=none] (16) at (0, -3) {$-A^{2}$};
		\node [style=none] (17) at (0, -4.25) {};
		\node [style=none] (18) at (4., 1.5) {};
		\node [style=none] (19) at (-4., 1.5) {};
		\node [style=none] (20) at (-4., -1.5) {};
		\node [style=none] (21) at (4., -1.5) {};
		\node [style=none] (22) at (2, 0) {};
		\node [style=none] (23) at (1, 0) {};
		\node [style=none] (24) at (1, -1.5) {};
		\node [style=none] (25) at (1, 1.5) {};
		\node [style=none] (26) at (-1, 1.5) {};
		\node [style=none] (27) at (-1, -1.5) {};
		\node [style=none] (28) at (1, 7.25) {};
		\node [style=none] (29) at (1, -7.25) {};
		\node [style=none] (30) at (1, 4.25) {};
		\node [style=none] (31) at (1, -4) {};
                \node [style=none] (32) at (-5,7.25) {};
                \node [style=none] (33) at (-5,4.25) {};
                \node [style=none] (34) at (-5,-7.25) {};
                \node [style=none] (35) at (-5,-4.25) {};
                \node [style=none] (36) at (1,-4.25) {};
		\node [style=none] (37) at (0.9, 1.4) {};
		\node [style=none] (38) at (-0.9, -1.4) {};
                \node [style=none] (39) at (-0.9, 1.4) {};
                \node [style=none] (40) at (0.9, -1.4) {};
                \node [style=none] (41) at (-1,0){};
                \node [style=none] (42) at (7,0){};
                \node [style=none] (43) at (5,5){};
                \node [style=none] (44) at (5,-6){};
	\end{pgfonlayer}
	\begin{pgfonlayer}{edgelayer}
                \newcommand{\parab}{0.07*(\t-0.5)*(\t-0.5)+4.25};
                \draw[fill=gray!20,draw=gray!20,domain=1:7] (30.center)-- plot [variable=\t] (\t,{\parab})--(28.center)--cycle;
                \draw[fill=gray!20,draw=gray!20,domain=1:7] (0.5,-4.25)-- plot [variable=\t] (\t,{-(\parab)})--(0.5,-7.25)--cycle;
		\draw[fill=gray!20,draw=gray!20]   (26.center)--(18.center)--(6.center)--(32.center)--(34.center)--(5.center)--(21.center)--(27.center)--cycle;
		\draw[->] (0.center) to (1.center);
		\draw[->] (2.center) to (3.center);
		\draw[dashed] (5.center) to (6.center);
		\draw[dashed] (8.center) to (9.center);
		\draw[dashed] (19.center) to (18.center);
		\draw[dashed] (20.center) to (21.center);
		\draw[dashed] (29.center) to (28.center);
		\draw[dashed] (26.center) to (27.center);
                \draw[thick,domain=7:1, fleche=0.5:black] plot [variable=\t] (\t,{\parab});
                \draw[thick,domain=1:7,fleche=0.5:black] plot [variable=\t] (\t,{-(\parab)});
                \draw[thick,fleche=0.5:black] (30.center)--(25.center);
                \draw[thick,fleche=0.9:black] (25.center)--(19.center);
                \draw[thick,fleche=0.9:black] (19.center)--(20.center);
                \draw[thick,fleche=0.5:black] (20.center)--(24.center);
                \draw[thick,fleche=0.5:black] (24.center)--(36.center);
                \draw[line width=5pt] (22.center)--(42.center);
                \draw[fill=black] (11.center) circle(0.25);
                \draw[thick,fleche=0.8:black] (37.center)--(39.center)--(38.center)--(40.center)--cycle;
		\draw[thick,fleche=0.8:black] (39.center)--(38.center);
                \draw (23.center) node[above right] {$\varrho_{h}$};
                \draw (22.center) node[above right] {$2\varrho_{h}$};
                \draw (41.center) node[above left] {$-\varrho_{h}$};
                \draw (12.center) node[above right] {$1$};
                \draw (15.center) node[below left] {$-1$};
               	\draw (16.center) node[left] {$-A^{2}$};
                \draw (16.center) node {$-$};
                \draw (13.center) node[left] {$A^{2}$};
                \draw (13.center) node {$-$};
                \draw (17.center) node[left] {$-\frac{1}{b^{2}}$};
                \draw (17.center) node {$-$};
                \draw (14.center) node[left] {$\frac{1}{b^{2}}$};
                \draw (14.center) node {$-$};
                 \draw (30.center) node[below right] {$M$};
                \draw (25.center) node[above right] {$N$};
                \draw (19.center) node[left] {$O$};
                \draw (20.center) node[left] {$P$};
                \draw (24.center) node[below right] {$Q$};
                \draw (31.center) node[right] {$R$};
                \draw (26.center) node[above] {$O'$};
                \draw (27.center) node[above right] {$P'$};
                \draw (43.center) node[above right] {$\Gamma_{+}$};
                \draw (44.center) node[above right] {$\Gamma_{-}$};
                \draw (10.center) node[above left] {$-\frac{A^{2}}{2}$};
                \draw (4.center) node[above right] {$\frac{A^{2}}{2}$};
              \end{pgfonlayer}
            \end{tikzpicture}
   \hspace{1cm}
  \begin{tikzpicture}[xscale=1.6,yscale=1.2]
    \pgfdeclarelayer{nodelayer}
    \pgfdeclarelayer{edgelayer}
    \pgfsetlayers{nodelayer,edgelayer}
    \tikzset{none/.style={coordinate}}
\tikzset{fleche/.style args={#1:#2}{ postaction = decorate,decoration={name=markings,mark=at position #1 with {\arrow[#2,scale=2]{>}}}}}
 	\begin{pgfonlayer}{nodelayer}
		\node [style=none] (0) at (-2., 0) {};
		\node [style=none] (1) at (2., 0) {};
		\node [style=none] (2) at (0, -3.) {};
		\node [style=none] (3) at (0, 3.) {};
		\node [style=none] (12) at (0, 1.5) {};
		\node [style=none] (15) at (0, -1.5) {};
		\node [style=none] (23) at (1, 0) {};
		\node [style=none] (24) at (1, -1.5) {};
		\node [style=none] (25) at (1, 1.5) {};
		\node [style=none] (26) at (-1, 1.5) {};
		\node [style=none] (27) at (-1, -1.5) {};
		\node [style=none] (37) at (1, 1.5) {};
		\node [style=none] (38) at (-1, -1.5) {};
                \node [style=none] (39) at (-1, 1.5) {};
                \node [style=none] (40) at (1, -1.5) {};
                \node [style=none] (41) at (-1,0){};
                \node [style=none] (46) at (-2, 3) {};
              \end{pgfonlayer}
              \begin{pgfonlayer}{edgelayer}
                \draw[fill=gray!20,draw=gray!20] (46.center) rectangle (2.center); 
		\draw[->] (0.center) to (1.center);
		\draw[->] (2.center) to (3.center);
                \draw[thick,fleche=0.8:black] (37.center)--(39.center)--(38.center)--(40.center)--cycle;
		\draw[thick,fleche=0.8:black] (39.center)--(38.center);
                \draw (23.center) node[above right] {$\varrho_{h}$};
                \draw (41.center) node[above left] {$-\varrho_{h}$};
                \draw (12.center) node[above right] {$1$};
                \draw (15.center) node[below right] {$-1$};
                \draw (25.center) node[above right] {$N$};
                \draw (24.center) node[below right] {$Q$};
                \draw (26.center) node[above] {$O'$};
                \draw (27.center) node[below right] {$P'$};
                \draw[fill=black] (0,0) circle (0.05);
                \draw[fill=black] (0.1,0) circle (0.05);
		\draw[fill=black] (0.3,0) circle (0.05);
                \draw[fill=black] (0.5,0) circle (0.05);
                \draw[fill=black] (0.8,0) circle (0.05);                
              \end{pgfonlayer}
            \end{tikzpicture}
            \caption{The left-hand side summarizes the spectral localization and the shape of contours deduced from the analysis in $|\real z|\leq\frac{A^{2}}{2}$\,. In this picture $\mathrm{Spec}(\mathrm{\frac{1}{2}\Delta_{V^{h},1}})=\mathrm{Spec}(\frac{1}{2}\Delta_{V,h})$ is represented by the black circle around $0$ and the thick real half-line $[2\varrho_{h},+\infty[$.\\
              The picture on the right-hand side is zoomed into the region $|\real z|\leq \varrho_{h}$ and $|\mathrm{Im}~z|\leq 1$. The small circles represent the eignvalues of $B_{\pm,b,V^{h}}$ in the region $|\real z|\leq \varrho_{h}$\,.\\
            In both pictures the gray area represents a part of the resolvent set of $B_{\pm,b,V^{h}}$\,.}
        \label{fig2}
          \end{figure}
           \pagebreak[4]

          \vspace{0.5cm}
         \noindent The curves $\Gamma_{\pm}$ are defined by
          \begin{equation}
            \label{eq:gammapm}
            \Gamma_{\pm}=\left\{z\in \mathbb{C}\,,\quad \pm [1+ (\real z-\varrho_{h})^{2}]= b^{2}\mathrm{Im}~z\right\}\,.
          \end{equation}
          The points $M,N,Q,R$ are on the line $\real z=\varrho_{h}$ with the respective imaginary parts $\frac{1}{b^{2}},1,-1,-\frac{1}{b^{2}}$\,. The points $O,O'$ (resp. $P,P'$) have an imaginary part equal to $1$ (resp. $-1$) and real part equal to $-\frac{A^{2}}{2}$ and $-\varrho_{h}$\,.\\
          We will use the oriented contours $\Gamma=\Gamma_{+}+\Gamma_{0}+\Gamma_{-}$ with $\Gamma_{0}=\left\{z\in\mathbb{C}\,, |z-\varrho_{h}|=\frac{1}{b^{2}}\,, \real z\leq \varrho_{h}\right\}$\,, $\Gamma_{+}+[MNOPQR]+\Gamma_{-}$\,, $\Gamma_{+}+[MNO'P'QR]+\Gamma_{-}$\,, $\Gamma_{+}+[MR]+\Gamma_{-}$ and $[NO'P'Q]$\,.\\
         
         The main results of this section are gathered in the two following propositions.
          \begin{proposition}
            \label{pr:spectrum}
            There exists $C_{0}\geq 1$ such that $A=C_{0}$ and $2C_{0}^{5}b\leq \varrho_{h}$ implies the following properties.
            \begin{description}
            \item[a)]   The sets $\left\{\real z\leq \frac{A^{2}}{2}~\text{and}~|\mathrm{Im}~z|\geq 1\right\}$ \,, $\left\{\real z\leq -\varrho_{h}\right\}$\,, $\left[N,Q\right]=\left\{\real z=\varrho_{h}, \left|\mathrm{Im}~z\right|\leq 1\right\}$ and  the one partly delimited by $\Gamma_{\pm}$\,,
            $\left\{\real z\geq \varrho_{h}~\text{and}~1+(\real z-\varrho_{h})^{2}\leq b^{2}|\mathrm{Im}~z|\right\}$\,,  are all contained in the resolvent set of $B_{\pm,b,V^{h}}$\,.  The union of these sets (the gray area in the left-hand side picture of Figure~\ref{fig2}) contains all the oriented contours listed above.
           \item[b)]  If $E^{(p)}_{\pm,b,V^h}$ is the characteristic space, which is the range of
            $$
            \pi_{E^{(p)}_{\pm,b,V^h}}=\frac{1}{2i\pi}\int_{NOPQ}(z-B_{\pm,b,V^{h}}^{(p)})^{-1}dz=\frac{1}{2i\pi}\int_{NO'P'Q}(z-B_{\pm,b,V^{h}}^{(p)})^{-1}~dz
            $$
            then $\mathcal{N}_{+}^{(p)}:=\dim E^{(p)}_{+, b,V^h}=\mathcal{N}_{+}^{(p)}(V)$  and $\mathcal{N}_{-}^{(p)}:=\dim E_{-,b,V^h}^{(p)}=\mathcal{N}_{-}^{(p-d)}(V)$\,.
          \item[c)] When $r:X^{h}\to X^{h}$ is the involution defined by $r(q,p)=(q,-p)$ and $r^{*}$ denotes its action on $\mathcal{S}(X^{h};\mathcal{E}_{\pm}^{h})$\,, according to Definition~\ref{de:Bismutnot}, then $r^{*}$ is a unitary involution of $L^{2}(X^{h};\mathcal{E}_{\pm}^{h})$ such that $r^{*}B_{\pm,b,V^{h}}(r^{*})^{-1}=r^* B_{\pm,b,V^h}r^*=B_{\pm,b,V^{h}}'$\,.
          \item[d)] The hermitian form $(u,v)\mapsto \langle u\,,\,v \rangle_{r}=\langle u\,,\,r^{*}v \rangle_{L^{2}}$, of Definition~\ref{de:Bismutnot} is positive definite on $E_{\pm,b,V^h}$ and $B_{\pm,b,V^{h}}\big|_{E_{\pm,b,V^h}}$ is self-adjoint and positive for the scalar product $\langle~,~\rangle_{r}$\,. The eigenvalues of $B_{\pm,b,V^{h}}$ in the disc $\{z\in\mathbb{C},|z|\leq \varrho_h\}$ actually belong to  $[0,\varrho_{h}[$\,. Additionally on $E_{\pm,b,V^h}$ we have the equivalence of norms 
          \begin{equation}
            \label{eq:Epmequi}
          \forall u\in E_{\pm,b,V^h}\,,\quad \sqrt{1-C_0b^2}\|u\|_{L^2}\leq \|u\|_r\leq \|u\|_{L^2}\,.
          \end{equation}
          \end{description}
        \end{proposition}
       
        \begin{definition}
          The eigenvalues of $B_{\pm,b,V^{h}}^{(p)}$ lying in $[0,\varrho(h)]$ will be denoted by $(\lambda_{\pm,j}^{(p)})_{1\leq j\leq \mathcal{N}_{\pm}^{(p)}}$ with the associated $\langle ~, ~\rangle_{r} $-orthonormal basis of eigenvectors $(u_{\pm,j}^{(p)})_{1\leq j\leq \mathcal{N}_{\pm}^{(p)}}$\,.
      \end{definition}
      \begin{proposition}\label{pr:semiGroup}
        For every $s\in\mathbb{R}$ there exists $C_{s}\geq 1$ such that taking $A=C_{s}$ with the condition $\varrho_{h}\geq bC_{s}^{5}$ implies that the semigroup $(e^{-t B_{\pm,V^{h}}})_{t>0}$ admits for every $t>0$ the following convergent integral representation
        \begin{eqnarray*}
          e^{-tB_{\pm,b,V^{h}}^{(p)}}=\frac{1}{2i\pi}\int_{\Gamma}\frac{e^{-tz}}{z-B_{\pm,b,V^{h}}^{(p)}}~dz
          =\underbrace{\frac{1}{2i\pi}\int_{NO'P'Q}\frac{e^{-tz}}{z-B_{\pm,b,V^{h}}^{(p)}}~dz}_{(I)}+
          \underbrace{\frac{1}{2i\pi}\int_{\Gamma_{+}+[MR]+\Gamma_{-}}\frac{e^{-tz}}{z-B_{\pm,b,V^{h}}^{(p)}}~dz}_{(II)}\,.
        \end{eqnarray*}
        In the above formula the first term equals
\begin{eqnarray*}
  && (I)=\sum_{i=1}^{\mathcal{N}_{\pm}^{(p)}}e^{-t\lambda_{\pm,j}^{(p)}}|u_{\pm,j}^{(p)}\rangle\langle r^*u_{\pm,j}^{(p)}| \\
\text{with}  && \|u_{\pm,j}^{(p)}\|_{\tilde{\cal W}^{0,s}} = \|r^*u_{\pm,j}^{(p)}\|_{\tilde{\cal W}^{0,s}} \leq C_s \,,
\end{eqnarray*}
and the second term satisfies
$$
\|(II)\|_{\mathcal{L}(\tilde{\mathcal{W}}^{s};\tilde{\mathcal{W}}^{s})}\leq \frac{1+t^{-1}}{b^{2}}e^{-t\varrho_{h}}\,.
$$
      \end{proposition}
      The proofs of Proposition~\ref{pr:spectrum} and Proposition~\ref{pr:semiGroup} actually rely on the two following lemmas.
      The first one is an application of Proposition~\ref{pr:estimateForTheConvergence} with the specific geometric partition of $\mathbb{C}$ represented in Figure~\ref{fig2}.
   \begin{lemma}
   \label{le:tableres}
   For any $s\in \mathbb{R}$\,, there exists $C_s\geq 1$ such that $A\geq C_s$ and $\frac{\varrho_h}{2}\geq C_s b A^4$ imply that the norms of 
        $$ 
        (B_{\pm,b,V^h}-z)^{-1}\quad,\quad
        \mathfrak{D}_z = (B_{\pm,b,V^h}-z)^{-1} - U_{\pm,\theta} (\frac{1}{2}\Delta_{V^h,1}-z)^{-1} U_{\pm,\theta}^{-1}\,,
        $$
        have upper bounds given by the following table.  Because $B_{\pm,b,V^h}$ and $U_{\pm,\theta}\frac{1}{2}\Delta_{V^h,1}U_{\pm,\theta}^{-1}$ preserve the total degree $p\in\{0,\ldots,2d\}$\,,  $(B_{\pm,b,V^h}-z)^{-1}$ and $\mathcal{D}_{z}$ can be respectively replaced by
          $$ 
        (B_{\pm,b,V^h}^{(p)}-z)^{-1}\quad,\quad
        \mathfrak{D}_z^{(p)} = (B_{\pm,b,V^h}^{(p)}-z)^{-1} - U_{\pm,\theta} (\frac{1}{2}\Delta_{V^h,1}^{(p-\frac{d}{2}\pm\frac{d}{2})}-z)^{-1} U_{\pm,\theta}^{-1}\,.
        $$     
        \begin{table}[h]
            \centering
        \begin{tabular}{|c|c|c|c|}
        \hline
        Sets & $\| \mathfrak{D}_z \|_{\mathcal{L}(\tilde{\cal W}^{0,s};\tilde{\cal W}^{0,s})}$ & $\| (B_{\pm,b,V^h}-z)^{-1} \|_{\mathcal{L}(\tilde{\cal W}^{0,s};\tilde{\cal W}^{0,s})}$ & Label \\
        \hline
            $\{ \real z \leq -\frac{A^2}{2} \}$ & & $\frac{4}{A^2}~\text{for}~s=0$ & \textbf{(1)} 
            \\
            \hline
            $\{ |\real z| \leq \frac{A^2}{2} \text{ and } |\mathrm{Im}~z|\geq 1\}$ & $C_s(bA^{\frac{7}{2}}+A^{-2})$ & $C_s(\frac{1}{|\mathrm{Im}~z|} + bA^{\frac{7}{2}}+A^{-2})$& 
            \textbf{(2)}
            \\
            \hline
            $ \left\{ \left(\begin{array}{c}
                 \real z \leq -\varrho_h  \\
                 \text{or}~\real z=\varrho_h
            \end{array}\right) \text{ and } |\mathrm{Im}~z| \leq 1 \right\}$ & $C_s\left(A^{-2} + \frac{bA^{\frac{7}{2}}}{\varrho_{h}^2/4 + |\mathrm{Im}~z|^2}\right)$ & $C_s\left( A^{-2} + \frac{1+bA^{\frac{7}{2}}}{\varrho_{h}^2/4 + |\mathrm{Im}~z|^2}\right)$ & \textbf{(3)}
            \\
            \hline
            $\left\{\real ~ z= 0~\text{and}~|\mathrm{Im}~z|\geq\frac{1}{b^2}\right\}$ & $C_s \frac{A^{-2}+bA^{-\frac{1}{2}}}{b\sqrt{|\mathrm{Im}~z|}}\leq\frac{1}{8b\sqrt{|\mathrm{Im}~z|}}$  & 
            $\frac{1}{|\mathrm{Im}~z|}+\frac{1}{8b\sqrt{|\mathrm{Im}~z|}}\leq \frac{1}{4b\sqrt{|\mathrm{Im}~z|}}$ & \textbf{(4)} \\
            \hline
            \end{tabular}
            \caption{Resolvent norm estimates}
             \label{tab:resodiff}
        \end{table}
   \end{lemma}     
\begin{proof} 
The conditions $A\geq C_s$ and $\frac{\varrho_h}{2}\geq C_s b A^4$ with $\varrho_h\leq 1$  ensure the validity of the hypotheses of Proposition~\ref{pr:estimateForTheConvergence}:
$$
C_s\max(Ab,b,A^{-1})\leq 1 \quad\text{and}\quad \mathrm{dist}(z,\mathrm{Spec}(\frac{1}{2}\Delta_{V^{h},1}))\geq C_s bA^4
$$
as soon as $\mathrm{dist}(z,\mathrm{Spec}(\frac{1}{2}\Delta_{V^{h},1}))\geq\frac{\varrho_h}{2}$\,. 
        \begin{description}
        \item[(1)] This line is only concerned with the case $s=0$. By Corollary~\ref{cor:maxaccr} there is constant $C_0\geq 1$ such that 
        $$ C_0+B_{\pm,b,V^h} $$
        is accretive as soon as $ C_0b\leq 1 $ and $h\in ]0,1]$. Take $z\in \{ \real z\leq -\frac{A^2}{2}\}$, 
        \begin{eqnarray*}
        & & \hspace{-1.5cm}\| (B_{\pm,b,V^h} -z)u\|_{L^2}^2 \\
        &=& \|(B_{\pm,b,V^h}-i\mathrm{Im}~z)u \|_{L^2}^2 + |\real ~ z|^2\|u\|_{L^2}^2 + 2 (-\real z) ~ \real \left\langle B_{\pm,b,V^h}u , u \right\rangle \\
        & = & \|(B_{\pm,b,V^h}-i\mathrm{Im}~z)u \|_{L^2}^2 + \big[|\real ~ z|^2-2C_0|\real z|\big]\|u\|_{L^2}^2 + 2 |\real z|~\underbrace{\real \left\langle (C_0+B_{\pm,b,V^h})u , u \right\rangle}_{\geq 0} \\
        &\geq & \big[|\real ~ z|^2-2C_0|\real z|\big]\|u\|_{L^2}^2 \\
        & \geq & \frac{A^2}{2}(\frac{A^2}{2}-2C_0)\|u\|_{L^2}^2
        \end{eqnarray*}
       When $ \frac{A^2}{8}\geq C_0 $\,, we obtain
        $$ \| (B_{\pm,b,V^h}-z)u \|_{L^2} \geq \frac{A^2}{4} \| u \|_{L^2}\,. $$
        \item[(2)] For $z\in \{ |\real z|\leq \frac{A^2}{2} \text{ and } |\mathrm{Im}~z|\geq 1 \} $ we know
        $$
        \mathrm{dist}(z,\mathrm{Spec}(\frac{1}{2}\Delta_{V^h,1}))\geq |\mathrm{Im}z|\geq 1\geq\frac{\varrho_h}{2}
        $$
        and
        \begin{equation}\label{eq:lowerBoundDistToSpec2}
        \| U_{\pm,\theta} (\frac{1}{2}\Delta_{V^h,1}-z)^{-1}U_{\pm,\theta}^{-1} \|_{\mathcal{L}(\tilde{\cal W}^{0,s};\tilde{\cal W }^{0,s})}
            \leq \frac{1}{ \mathrm{dist}(z,\mathrm{Spec}(\frac{1}{2}\Delta_{V^h,1}))}\leq \frac{1}{|\mathrm{Im}~z|}\leq 1\,.
        \end{equation}
        Proposition~\ref{pr:estimateForTheConvergence} gives for a proper choice of $C_s\geq 1$ the inequality
        \begin{eqnarray}
        \|\mathfrak{D}_z\|_{\mathcal{L}(\tilde{\cal W}^{0,s};\tilde{\cal W}^{0,s})} & \leq & \frac{C_s}{1+b\sqrt{|\mathrm{Im}~z}|} \left[ A^{-2} + (1+\frac{A^2}{\mathrm{dist}(z,\mathrm{Spec}(\frac{1}{2}\Delta_{V^h,1}))})^2bA^{-\frac{1}{2}} \right] \nonumber\\
        & \leq & 4C_s [A^{-2} + bA^{7/2} ] \label{eq:UpperBoundForFrakD2}\,.
                \end{eqnarray}
                The upper bound for $ (B_{\pm,b,V^h}-z)^{-1} $ is deduced at once from \eqref{eq:lowerBoundDistToSpec2} and  \eqref{eq:UpperBoundForFrakD2}.
        \item[(3)] The following inequality holds for $z\in \{ (\real z \leq -\varrho_h~\text{or}~\real z=\varrho_h)\text{ and } |\mathrm{Im}~z| \leq 1 \}$
        $$ \mathrm{dist}(z,\mathrm{Spec}(\frac{1}{2}\Delta_{V^h,1}))^2\geq \frac{\varrho_h^2}{4} + |\mathrm{Im}~z|^2\,,$$
        with the detailed cases:
        \begin{itemize}
            \item if $\real z\leq -\varrho_h$ then  $\mathrm{dist}(z,\mathrm{Spec}(\frac{1}{2}\Delta_{V^h,1}))^2\geq |z|^2\geq  \varrho_h^2+|\mathrm{Im}z|^2 \geq \frac{\varrho_h^2}{4}+|\mathrm{Im}z|^2$\,;
            \item if $\real z = \varrho_h$ then the hypothesis \eqref{eq:spectralGapLower} ensure that $\mathrm{dist}(z,\mathrm{Spec}(\frac{1}{2}\Delta_{V^h,1}))^2 \geq (\frac{\varrho_h}{2})^2 + |\mathrm{Im}~z|^2$\,.
        \end{itemize}
        Applying Proposition~\ref{pr:estimateForTheConvergence} gives
        $$ \| \mathfrak{D}_z \|_{\mathcal{L}(\tilde{\cal W}^{0,s};\tilde{\cal W}^{0,s})}\leq C_s\left[A^{-2}+\frac{bA^{7/2}}{\varrho_h^2/4 + |\mathrm{Im}~z|^2}\right] \,.$$
        \item[(4)] If $z\in \{ \real z = \text{ and } |\mathrm{Im}~z|\geq \frac{1}{b^2} \}$ the distance to the spectrum is bounded by 
        $$ \mathrm{dist}(z,\frac{1}{2}\Delta_{V^h,1})\geq |\mathrm{Im}~z|\geq \frac{1}{b^2} \,. $$
        Proposition~\ref{pr:estimateForTheConvergence}, with $Ab\leq 1$, implies
        \begin{eqnarray*}
        \| \mathfrak{D}_z \|_{\mathcal{L}(\tilde{\cal W}^{0,s};\tilde{\cal W}^{0,s})}\leq C_s \frac{A^{-2}+bA^{-\frac{1}{2}}}{b\sqrt{|\mathrm{Im}~z|}} \leq \frac{1}{8b\sqrt{|\mathrm{Im}~z|}}\,,
        \end{eqnarray*}
        \end{description}
        by choosing again $A$ and $\frac{1}{b}$ large enough. Finally
        $$
         \| (B_{\pm,b,V^h}-z)^{-1} \|_{\mathcal{L}(\tilde{\cal W}^{0,s};\tilde{\cal W}^{0,s})}\leq \frac{1}{|\mathrm{Im}\,z|}+\frac{1}{8b\sqrt{|\mathrm{Im}~z|}}\leq \frac{b^2+\frac{1}{8}}{b\sqrt{|\mathrm{Im}~z|}}
        \leq \frac{1}{4b\sqrt{|\mathrm{Im}~z|}}\,.$$
\end{proof}
 
The second lemma is a variation of \cite{Nie14}-Proposition~2.15, which was itself inspired from the article of H\'erau-Hitrik-Sj\"ostrand \cite{HHS}.
      \begin{lemma}\label{lem:PTsymmetry}
Let $(B,D(B))$ be a closed densely defined operator in a separable
Hilbert space
$\mathcal{H}$, 
such that $(1+B)^{-1}$ is compact, so that $\text{Spec}\,(B)$ is discrete, and $D(B)=D(B^{*})$\,.
Assume that there exists  a unitary involution $U^{*}=U^{-1}=U$  such that $U^{*}BU=B^{*}$\,. Then the spectrum $\textrm{Spec}\,(B)$ is invariant by
complex conjugation.\\
If additionally there exist $\gamma>0$\,, $\varepsilon\in ]0,\frac{1}{4}[$\,, an orthogonal projection $\Pi_{0}=\Pi_{0}^{*}$ and a bounded contour $\Gamma$\,, symmetric w.r.t $z\to \overline{z}$\,, such that:
\begin{itemize}
\item $\Pi_{0}U=U\Pi_{0}=\Pi_{0}$\,;
\item the real part $\real B=\frac{B+B^{*}}2$ is non negative and
  $\real B \geq \gamma(1-\Pi_{0})-\varepsilon\gamma$\,;
\item with $\Pi_{\Gamma}=\frac{1}{2\pi i}\int_{\Gamma}(z-B)^{-1}~dz$\,,
  \begin{equation}
    \label{eq.condtrace}
\real \mathrm{Tr}\,\left[\frac{1}{2\pi
    i}\int_{\Gamma}B(z-B)^{-1}~dz\right]=\mathrm{Tr}\,\left[B\Pi_{\Gamma}\right]
\leq \varepsilon \gamma\,;
\end{equation}
\end{itemize}
then the following properties hold:
\begin{itemize}
\item The form $\langle u,v\rangle_{U}=\langle u, U v\rangle$ is a hermitian
  positive definite form on $E_{\Gamma}={\rm Ran\,} \Pi_{\Gamma}$\,.
\item The norms $\|u\|_{U}=\sqrt{\langle u\,,\, Uu\rangle}$ and $\|u\|$ are equivalent
\begin{equation}
\label{eq:PTequi}
    \forall u\in E_{\Gamma}\,,\quad \sqrt{1-4\varepsilon}\|u\|\leq \|u\|_{U}\leq \|u\|\,.
\end{equation}
\item The restricted operator $B\big|_{E_{\Gamma}}=B\Pi_{\Gamma}\big|_{E_{\Gamma}}=\Pi_{\Gamma}B\big|_{E_{\Gamma}}$ is self-adjoint and non negative for the
  scalar product $\langle~,~\rangle_{U}$\,.
\item The vector space $E_{\Gamma}$ admits a basis of eigenvectors of $B$,
  $(e_{1},\ldots, e_{N})$, orthonormal for the scalar product
  $\langle~,~\rangle_{U}$\,.
\item For all
  $z\in\mathbb{C}$ inside the contour $\Gamma$\,, the inequality
$$
\|(z-B)^{-1}\big|_{E_{\Gamma}}\|_{\mathcal{L}(E_{\Gamma};E_{\Gamma})}\leq \frac{1}{\sqrt{1-4\varepsilon}\, \mathrm{dist}(z,
  \{\lambda_{1}, \ldots, \lambda_{N}\})}\,,
$$
holds with the initial norm $\|~\|$ on $E_{\Gamma}$\,.
\item The ``distance'' $\vec{d}(E_{\Gamma},\mathrm{Ran}\,\Pi_{0})=\|(1-\Pi_{0})\Pi_{\Gamma}\|_{\mathcal{L}(\mathcal{H})}$ is  bounded by $\sqrt{2\varepsilon}$ and $\Pi_0$ is an isomorphism from $E_{\Gamma}$ to $\Pi_0 E_{\Gamma}$\,.
\end{itemize}
\end{lemma}
\begin{proof}
  The $PT$-symmetry property $B^{*}=U^{*}BU$ implies
$$
(z-B^{*})^{-1}=(z-U^{*}BU)^{-1}=U^{*}(z-B)^{-1}U
$$
whenever one of the resolvent exists, so that
$$
\overline{\text{Spec}\,(B)}=\text{Spec}\,(B^{*})=\text{Spec}\,(B)
$$
and  the spectrum $\text{Spec}\,(B)$ is symmetric with respect to the real
axis. The accretivity of $B$  gives $\text{Spec}\,(B)\subset
\left\{z\in\mathbb{C},~\real z\geq 0\right\}$\,.
Another consequence is that if the integration contour $\Gamma$
is symmetric w.r.t the real axis and if $z\to
f(z)$ is a holomorphic function satisfying $\overline{f(\overline{z})}=f(z)$ then
$$
f_{\Gamma}(B)^{*}=\frac{-1}{2i\pi}\int_{\Gamma}\frac{f(\bar{z})}{(\bar
  z-B^{*})}~d\bar z
=U^{*}\left[\frac{1}{2i\pi}\int\frac{f(z)}{(z-B)}~dz\right] U=U^{*}f_{\Gamma}(B)U\,,
$$
and in particular 
$$
\mathrm{Tr}\,\left[f_{\Gamma}(B)\right]=\frac{1}{2}\left(\mathrm{Tr}\,\left[f_{\Gamma}(B)\right]+\mathrm{Tr}\,\left[U^{*}f_{\Gamma}(B)U\right]\right)
=\real \mathrm{Tr}\,\left[f_{\Gamma}(B)\right]\in\mathbb{R}\,.
$$
Therefore the condition \eqref{eq.condtrace} makes sense (take
$f(z)=z$) when there are
eigenvalues of $B$ with small real parts and multiplicities that are not too large.
On  the space $E_{\Gamma}={\rm Ran\,}\Pi_{\Gamma}$\,,  the form $(u,v)\mapsto \langle u\,,\,v\rangle_{U}=\langle
u\,,\, Uv\rangle$ is a hermitian form and it is a scalar product when
$\langle u\,,\, u\rangle_{U}>0$ for any nonzero $u\in E_{\Gamma}$\,.
\\
When $u\in E_{\Gamma}$ with $\|u\|=1$\,, it can be completed into an
orthonormal basis $(e_{1}=u,e_{2},\ldots, e_{N})$ of $E_{\Gamma}$ for
the scalar product $\langle~\,,~\rangle$\,. We have
\begin{eqnarray*}
\varepsilon\gamma\geq \mathrm{Tr}\,\left(B\Pi_{\Gamma}\right)
&=&\real
\left(\sum_{j=1}^{N}\langle e_{j}\,,\, Be_{j}\rangle\right)
\\
&=&\sum_{j=1}^{N}\langle e_{j}\,,\, \real B e_{j}\rangle\geq \langle
u\,,\,\real Bu\rangle\geq \gamma\|(1-\Pi_{0})u\|^{2}-\varepsilon\gamma\,,
\end{eqnarray*}
and
\begin{equation}
  \label{eq:estim1mPi0}
\|(1-\Pi_{0})u\|^{2}< 2\varepsilon=2\varepsilon\|u\|^2\,.
\end{equation}
Now compute
\begin{eqnarray*}
\langle u\,,\, U u\rangle
&=&\langle u\,,\Pi_{0} U u\rangle+\langle
u\,,\,(1-\Pi_{0})U u\rangle\\
&=&\|\Pi_{0}u\|^{2}+\langle(1-\Pi_{0})u\,,\, U(1-\Pi_{0})u\rangle
\\
&\geq&
\|\Pi_{0}u\|^{2}-\|(1-\Pi_{0})u\|^{2}=\|u\|^{2}-2\|(1-\Pi_{0})u\|^{2}\geq (1-4\varepsilon)\|u\|^2>0\,.
\end{eqnarray*}
This proves that the hermitian form $\langle~,~\rangle_U$ is positive definite on $E_{\Gamma}$ and the equivalence of norms comes at once.\\
Let $B_{\Gamma}$ be the restriction of $B$ to $E_{\Gamma}$\,, which is
a finite dimensional Hilbert space with the scalar product
$\langle~\,~\rangle_{U}$ (and $\langle~\,,\,~\rangle$)\,. For $u,v\in
E_{\Gamma}$\,, the series of equalities
\begin{eqnarray*}
\langle u\,,\, B_{\Gamma}v\rangle_{U}=\langle u\,,\,
U(B\Pi_{\Gamma})v\rangle
&=&\langle
u\,,(B\Pi_{\Gamma})^{*}U v\rangle=\langle(B\Pi_{\Gamma})u\,,\,U v\rangle
\\&=&\langle B_{\Gamma}u\,,\, v\rangle_{U}\,.
\end{eqnarray*}
says that $B_{\Gamma}$ is self-adjoint on
$(E_{\Gamma},\langle~,~\rangle_{U})$\,.
The two statements follow for the scalar product
$\langle~,~\rangle_{U}$ and the norm $\|~\|_{U}$\,, are consequences.\\
The equivalence of the norms $\|~\|_{U}$ and
$\|~\|$ gives the upper bound on $\|(z-B)^{-1}\big|_{E_{\Gamma}}\|_{\mathcal{L}(E_{\Gamma};E_{\Gamma})}$\,.\\
Finally the estimate on $\vec{d}(E_{\Gamma};\mathrm{Ran}\,\Pi_{0})$ is due to \eqref{eq:estim1mPi0}
\end{proof}

\begin{proof}[Proof of Proposition~\ref{pr:spectrum}] This is concerned essentially with the localization of the spectrum, which does not depend on $s\in\mathbb{R}$ for the closed realization $\overline{B_{\pm,b,V^h}^s}$ in $\tilde{\mathcal{W}}^{0,s}(X;\mathcal{E}_\pm)$\,. Therefore we focus on the case $s=0$ in particular while applying Lemma~\ref{le:tableres}.\\
The conditions $A=C_0$ and $\varrho_h\geq 2 C_0^{5}b$\,, for $C_0\geq 1$\,, large enough, imply  the conditions
$$
A\geq C_0\quad\text{and}\quad \frac{\varrho_h}{2}\geq C_0 b A^4
$$
of Lemma~\ref{le:tableres}.\\
        \textbf{a)}  
        The first three complex domains considered in \textbf{a)} are covered by unions of $(1),(2)$ and $(3)$ in Lemma~\ref{le:tableres}.
        For the last domain $\left\{\real z\geq \varrho_{h}~\text{and}~1+(\real z-\varrho_{h})^{2}\leq b^{2}|\mathrm{Im}~z|\right\}$ we start with
        \begin{eqnarray*}
        \| (B_{\pm,b,V^h}-z)u \|_{\mathcal{L}(L^2;L^2)} & \geq & \| (B_{\pm,b,V^h}-i\mathrm{Im}~z) u \|_{L^2} - |\real z|\|u\|_{L^2} \\
        & \geq & \big(4b\sqrt{|\mathrm{Im}~z|}-|\real z|\big)\|u\|_{L^2}\,.
        \end{eqnarray*}
        We conclude by noticing that $b\sqrt{|\mathrm{Im}~z|}\geq \sqrt{1+(\real z-\varrho_h)^2}$ implies
        \begin{eqnarray*}
        && b\sqrt{|\mathrm{Im}~z|}\geq |\real z-\varrho_h|\geq  |\real z|-1\\
        \text{and}&&  |\real z|\leq 1+b\sqrt{|\mathrm{Im}~z|}\leq 2b\sqrt{|\mathrm{Im}~z|}\,.
        \end{eqnarray*}
        We have actually proved $\|(B_{\pm,b,V^h}-z)^{-1}\|_{\mathcal{L}(L^2;L^2)}\leq \frac{1}{2b\sqrt{|\mathrm{Im}~z|}}$ in this domain.\\
        \textbf{b)} Let us consider now the operator
        $$  \mathfrak{R}^{(p)}=\frac{1}{2i\pi} \int_{NO'P'Q} \mathfrak{D}_z^{(p)} \, dz $$
        which is the difference between the projection $ \pi_{E_{\pm,b,V^h}^{(p)}} $ and the orthogonal projection $\pi_h^{(p)}=U_{\pm,\theta} \tilde{\pi}_h^{(p)} U_{\pm,\theta}^{-1} $ with
        $$\tilde{\pi}_h^{(p)} = \frac{1}{2i\pi} \int_{NO'P'Q} (z-\frac{1}{2} \Delta_{V^h,1}^{(p-\frac{d}{2} \pm \frac{d}{2})} )^{-1}\,dz \,. $$ 
        The following bounds are consequences of Table~\ref{tab:resodiff} of Lemma~\ref{le:tableres} for $s=0$: 
        \begin{eqnarray}
          \|\int_{NO' \cup P'Q} \mathfrak{D}_z^{(p)} \,dz\|_{\mathcal{L}(L^2;L^2)} &\leq& 8C_0\varrho_h (A^{-2}+bA^{7/2})\leq 8C_0A^{-2}+ 4 A^{-1/2}  \label{eq:estimpi1}  \\                     
                                      &  \leq &            \frac{8}{C_0}+\frac{4}{\sqrt{C_0}}\,, \nonumber
                                                                                          \\
\label{eq:estimpi2}        \|\int_{O'P' \cup QN} \mathfrak{D}_z^{(p)} \,dz\|_{\mathcal{L}(L^2;L^2)} &\leq& 2 C_0 A^{-2} + 2 C_0 bA^{7/2} \int_{-1}^1 \frac{1}{(\varrho_h/2)^2 + t^2 } \, dt   \\
        &\leq & 2C_0(A^{-2}+ 2A^{7/2}\varrho_h^{-1} \pi) \leq  2C_0 A^{-2}+A^{-1/2}2\pi \nonumber \\
        & \leq & \frac{2}{C_0}+\frac{2\pi}{\sqrt{C_0}}\,, \nonumber
        \end{eqnarray}
        where we used $C_0 b A^4\leq \frac{\varrho_h}{2}\leq\frac{1}{2}$\,, and finally $A=C_0$ for the last upper bounds\,. With $C_0\geq 1$ large enough  we have proved
        $$
        \|\pi_{E_{\pm,b,V^h}^{(p)}}-\pi_h^{(p)}\|_{\mathcal{L}(L^2;L^2)}<1\,.
        $$
        With $\|\pi_h^{(p)}\|_{\mathcal{L}(L^2;L^2)}\leq 1$\,,we obtain 
        $$
        \|(1-\pi_{E_{\pm,b,V^h}^{(p)}})\pi_h^{(p)}\|_{\mathcal{L}(L^2;L^2)}=\|(\pi_{E_{\pm,b,V^h}^{(p)}}-\pi_h^{(p)})\pi_h^{(p)}\|_{\mathcal{L}(L^2;L^2)}<1\,,
        $$
        and  $\pi_{E_{\pm,b,V^h}^{(p)}}:\mathrm{Ran}\,\pi_h^{(p)}\to {E_{\pm,b,V^h}^{(p)}}$ is one to one and $\mathcal{N}_{\pm}^{(p)}=\dim E_{\pm,b,V^h}^{(p)}\geq \mathcal{N}_\pm^{(p-\frac{d}{2}\pm\frac{d}{2})}(V)$\,.\\
        With $\|(1-\pi_h^{(p)})\|_{\mathcal{L}(L^2;L^2)}\leq 1$\,, we obtain
         $$
        \|(1-\pi_h^{(p)})\pi_{E_{\pm,b,V^h}^{(p)}}\|_{\mathcal{L}(L^2;L^2)}=\|(1-\pi_h^{(p)})\pi_{E_{\pm,b,V^h}^{(p)}}(\pi_{E_{\pm,b,V^h}^{(p)}}-\pi_h^{(p)})\|_{\mathcal{L}(L^2;L^2)}<1\,,
        $$
and  $\pi_h^{(p)}: {E_{\pm,b,V^h}^{(p)}} \to \mathrm{Ran}\,\pi_h^{(p)}$ is one to one and $\mathcal{N}_{\pm}^{(p)}=\dim E_{\pm,b,V^h}^{(p)}\leq  \mathcal{N}_\pm^{(p-\frac{d}{2}\pm\frac{d}{2})}(V)$\,, which is finite and independent of $h\in ]0,h_0]$\,.\\
\textbf{c)} It comes from Bismut identification of $B_{\pm,b,V^h}$ as a Hodge type operator for the $\langle~,~\rangle_r$ hermitian form, recalled in \eqref{eq:BHodge}: 
           $$ B_{\pm,b,V^h} = 2(\delta_{\pm,b,V^h} + \delta_{\pm,b,V^h}^{*,r})^2: \mathcal{S}(X^h;\mathcal{E}_\pm^h)\to \mathcal{S}'(X^h;\mathcal{E}_{\pm}^h)\,, $$
           which implies $ r^* B_{\pm,b,V^h}(r^*)^{-1} = r^* B_{\pm,b,V^h} r^* = B_{\pm,b,V^h}' $\,. \\
\textbf{d)} We apply Lemma~\ref{lem:PTsymmetry} with $B=C_0+B_{\pm,b,V^h} , U=r^* , \Pi_0 = \pi_{0,\pm}$ and the translated contour $\Gamma = C_0+NO'P'Q$ .
        Let us check the assumption of Lemma~\ref{lem:PTsymmetry} while specifying the values of $\gamma>0$ and $\varepsilon\in ]0,\frac{1}{4}[$:
        \begin{itemize}
            \item The equality $\Pi_0 U = U\Pi_0 = \Pi_0$ comes from the fact that $r^* \alpha_{\pm}(r^{*})^{-1} = \alpha_{\pm}$.
            \item The real part $\real B = C_0 + \frac{1}{b^2}\alpha_{\pm} + \real \gamma_{\pm,V^h}$ is non negative owing to the accretivity of $C_0+B_{\pm,b,V^h}$ in Corollary~\ref{cor:maxaccr}. 
            The inequality $\real B\geq \gamma(1-\Pi_0)-\varepsilon\gamma$ is obtained by the same integration by parts computations as in Proposition~\ref{pr:IppWithRealPart} and for the accretivity of Proposition~\ref{pr:subellipticEstimateA}. Let us compute
            \begin{eqnarray*}
                    \langle u , \real B u  \rangle_{L^2} & \geq & C_0 \|u\|_{L^2}^2 +\frac{1}{b^2}\langle u , \alpha_{\pm} u \rangle - \| \gamma_{\pm,V^h} \|_{\mathcal{L}(\tilde{\cal W}^{1,0};\tilde{\cal W}^{-1,0})}\| u \|_{\tilde{\cal W}^{1,0}}^2 \\
                    & \geq &  C_0 \|u\|_{L^2}^2 + \frac{1}{b^2}\langle u , \alpha_{\pm }u \rangle - \| \gamma_{\pm,V^h}\|_{\mathcal{L}(\tilde{\cal W}^{1,0};\tilde{\cal W}^{-1,0})} (\frac{d}{2}\|\Pi_0 u\|_{L^2}^2 + \| (1-\Pi_0)u\|_{\tilde{\cal W}^{1,0}}^2) \\
                    & \geq & (C_0 - d\| \gamma_{\pm,V^h}\|_{\mathcal{L}(\tilde{\cal W}^{1,0} ;\tilde{\cal W}^{-1,0})} ) \|u\|_{L^2}^2 + (\frac{1}{b^2}-\| \gamma_{\pm,V^h}\|_{\mathcal{L}(\tilde{\cal W}^{1,0} ;\tilde{\cal W}^{-1,0})}) \| (1-\Pi_{0})u \|_{L^2}^2 \\
                    &\geq & \frac{C_0}{2} \|u\|_{L^2}^2 + \frac{1}{2b^2} \|(1-\Pi_0)u\|_{L^2}^2\,,
            \end{eqnarray*}
            by fixing $C_0$ larger than $ 2d\| \gamma_{\pm,V^h}\|_{\mathcal{L}(\tilde{\cal W}^{1,0} ;\tilde{\cal W}^{-1,0})} $  and by using 
            $$\frac{1}{b^2}\geq \frac{1}{b}\geq 2C_0 A^4\geq  2C_0\geq 2\| \gamma_{\pm,V^h}\|_{\mathcal{L}(\tilde{\cal W}^{1,0} ;\tilde{\cal W}^{-1,0})}\,.
            $$
           With this constraint on $C_0\geq 1$, we have proved
            $$
           \langle u , \real B u \rangle \geq \gamma \|(1-\Pi_0)u\|_{L^2}\quad\text{with}~\gamma=\frac{1}{2b^2}\,.
            $$
          \item For the upper bound of the trace $\mathrm{Tr}\left[B\Pi_{\Gamma}\right]$ we use the notations 
          $$
         \dim E_{\pm,b,V^h}=\mathcal{N}_{\pm}=\sum_{p=0}^{2d} \mathcal{N}_{\pm}^{(p)}=\sum_{p'=0}^{d}\mathcal{N}_{\pm}^{(p')}(V)=\dim\mathrm{Ran}\,\tilde{\pi}_h\,.
          $$
          We write:
            \begin{eqnarray*}
              && \hspace{-1.5cm} \real \mathrm{Tr}\left[\frac{1}{2i\pi}\int_{C_0+NO'P'Q} B(z-B)^{-1}\,dz\right] \\
              && = C_0 \mathcal{N}_{\pm} + \real \mathrm{Tr}\left[ \frac{1}{2i\pi}\int_{NO'P'Q} z(z-B_{\pm,b,V^h})^{-1}\,dz  \right] \\
              && = C_0 \mathcal{N}_{\pm} + \real \mathrm{Tr}\left[ \frac{1}{2i\pi} \int_{NO'P'Q} z\mathfrak{D}_z \, dz \right] + \real \mathrm{Tr}\left[ \pi_h \right] \\
              && \leq C_0\mathcal{N}_{\pm} + \frac{C_0}{2\pi}\mathcal{N}_{\pm} (2A^{-2} + 8bA^{7/2}\varrho_h^{-1}\frac{\pi}{2}) + \mathcal{N}_{\pm} \frac{\varrho_{h}}{2} \\
              && \leq 3C_0 \mathcal{N}_{\pm}\\
              && \leq \varepsilon \gamma\,,
            \end{eqnarray*}
            with $\gamma=\frac{1}{2b^2}$ like above and $\varepsilon=6C_0 \mathcal{N}_{\pm}b^2$ which belongs to $]0,\frac{1}{4}[$ when $b^2<\frac{1}{24 C_0 \mathcal{N}_{\pm}}$\,. Remember that $\mathcal{N}_{\pm}=\dim \mathrm{Ran}\,\tilde{\pi}_h $ does not depends on $h\in ]0,h_0]$\,.
          \end{itemize}
          The three above points ensure that all the hypotheses of Lemma~\ref{lem:PTsymmetry} are fulfilled.
          Therefore $\langle \, , \, \rangle_{r} $ is hermitian positive definite form on $\mathrm{Ran}~\Pi_{\Gamma} = E_{\pm,b,V^h} $ and the equivalence of norms \eqref{eq:Epmequi} is a straightforward consequence of \eqref{eq:PTequi}. \\
          The space $E_{\pm,b,V^h}$ is a finite dimensional subspace of $\mathcal{S}(X^h;\mathcal{E}_{\pm}^h)$ endowed with the positive definite hermitian form $\langle~,~\rangle_r$\,. Additionally because $\delta_{\pm,b,V^h}B_{\pm,b,V^h}=B_{\pm,b,V^h} \delta_{\pm,b,V^h}$ on $\mathcal{S}(X^h;\mathcal{E}_{\pm}^h)$ and the same holds when $\delta_{\pm,b,V^h}$ is replaced by $\delta_{\pm,b,V^h}^{*,r}$\,, we deduce that $\delta_{\pm,b,V^h}$ and $\delta_{\pm,b,V^h}^{*,r}$ send $E_{\pm,b,V^h}$ into itself. Thus formula \eqref{eq:BHodge} implies that $B_{\pm,b,V^h}\big|_{E_{\pm,b,V^h}}=[(\delta_{\pm,b,V^h}+\delta_{\pm,b,V^h}^{*,r})\big|_{E_{\pm,b,V^h}}]^2$ is the square of a self-adjoint operator on $(E_{\pm,b,V^h},\langle~,~\rangle_r)$\,. Therefore $B_{\pm,b,V^h}\big|_{E_{\pm,b,V^h}}$ is a self-ajoint non negative operator for $\langle~,~\rangle_r$ and its eigenvalues are non-negative.
\end{proof}
\begin{proof}[Proof of Proposition~\ref{pr:semiGroup}]  The expression of $(I)$ is a consequence of Proposition~\ref{pr:spectrum} because it says that $B_{\pm,b,V^h}^{(p)}\big|_{E_{\pm,b,V^h}^{(p)}}$ is 
diagonalizable. Additionally the $L^2$ dual basis of $(u_{\pm,j}^{(p)})_{1\leq j \leq \mathcal{N}_{\pm}^{(p)}}$ in $E_{\pm,b,V^h}^{(p)}$ is $ (r^*u_{\pm,j}^{(p)})_{1\leq j \leq \mathcal{N}_{\pm}^{(p)}} $ because
$$ \langle u_{\pm,i}^{(p)} , u_{\pm,j}^{(p)} \rangle_{r} = \langle r^* u_{\pm,i}^{(p)} , u_{\pm,j}^{(p)} \rangle_{L^2} = \delta_{ij} \,,$$
while $r^* B_{\pm,b,V^h}r^*=B_{\pm,b,V^h}^{(p)\prime}$ implies that $r^*u_{\pm,j}^{(p)}$ is an eigenvector of $B_{\pm,b,V^h}^{(p) \prime}$\,.\\
Let us check the uniform bound of $\|u_{\pm,j}^{(p)}\|_{\tilde{\cal{W}}^{0,s}}$\,.  For $s=0$ it comes from \eqref{eq:Epmequi} with
$$ \|u_{\pm,j}^{(p)}\|_{L^2} \leq \frac{1}{\sqrt{1-C_0b^2}} \|u_{\pm,j}^{(p)}\|_r = \frac{1}{\sqrt{1-C_0b^2}}\leq 2\,.
$$
For $s>0$ we use the equation
$$ 
[B_{\pm,b,V^h}+A^2\pi_{0,\pm}]u_{\pm,j}^{(p)}=\lambda_{j,\pm}^{(p)}u_{\pm,j}+A^2 \pi_{0,\pm}u_{\pm,j}^{(p)}
$$
where the subelliptic estimate of Proposition~\ref{pr:subellipticEstimateA} implies
$$
C_0^{16/9}\|u_{j,\pm}^{(p)}\|_{\tilde{\cal{W}}^{0,s+2/9}}=A^{16/9}\|u_{j,\pm}^{(p)}\|_{\tilde{\cal{W}}^{0,s+2/9}}\leq C(\lambda_{j,\pm}^{(p)}+A^2)\|u_{\pm,j}^{(p)}\|_{\tilde{\cal{W}}^{0,s}}\leq C(1+C_0^2)\|u_{\pm,j}^{(p)}\|_{\tilde{\cal{W}}^{0,s}}
$$
by choosing $A=C_0$ like in Proposition~\ref{pr:spectrum}\,. A bootstrap argument leads to 
$$
\|u_{j,\pm}^{(p)}\|_{\tilde{\cal{W}}^{0,2/9k}}\leq [C(1+C_0^{2/9})]^{k}\|u_{\pm,j}^{(p)}\|_{L^2}\leq 2 [C(1+C_0^{2/5})]^{k}
$$
for every $k\in\mathbb{N}$ and the general result for $s>0$  follows by interpolation.\\
The integral
$$ \frac{1}{2i\pi}\int_{\Gamma} \frac{e^{-tz}}{z-B_{\pm,b,V^h}^{(p)}}\, dz $$
converges if and only if the two integrals
$ \int_{\Gamma_\pm} e^{-t\real z}\|(z-B_{\pm,b,V^h}^{(p)})^{-1}\|\, |dz|$ converge. With the parametrization of $\Gamma_{\pm}$ given by $ z = u+\varrho_h \pm i( \frac{1}{b^2}(1+u^2)) $ for $u\in [0,+\infty[$\,, we obtain
\begin{eqnarray*}
\int_{\Gamma_\pm} e^{-t\real z}\|(z-B_{\pm,b,V^h}^{(p)})^{-1}\|\, |dz|&\leq&
\int_{0}^{+\infty} \bigg\| \frac{e^{-tz}}{z-B_{\pm,b,V^h}^{(p)}} \bigg\|_{\mathcal{L}(\tilde{\cal W}^{0,s};\tilde{\cal W}^{0,s})} \bigg|1 \pm 2i\frac{u}{b^2}\bigg| \, du 
\\
& \leq & \int_0^{+\infty} \frac{e^{-t(u+\varrho_h)}}{4b\sqrt{b^{-2}(1+u^2)}} \bigg|1\pm i\frac{2u}{b^2}\bigg|\,du \\
& \leq & \frac{e^{-t\varrho_h}}{2b^2}\int_{0}^{+\infty} e^{-tu} \, du \\
& \leq & \frac{e^{-t\varrho_h}}{2b^2}\frac{1}{t}\,.
\end{eqnarray*}
The integral over the line segment $MR$ is estimated by
\begin{eqnarray*}
& & \hspace{-1.5cm} e^{t\varrho_h}\int_{-\frac{1}{b^2}}^{\frac{1}{b^2}} \bigg\|\frac{e^{-tz}}{z-B_{\pm,b,V^h}}\bigg\|_{\mathcal{L}(\tilde{\cal W}^{0,s};\tilde{\cal W}^{0,s})} \,du  \\
& &\leq  C_s \int_{1\leq |u| \leq \frac{1}{b^2}} (|u|^{-1}+bA^{7/2}+A^{-2}) \,du + C_s  \int_{|u|\leq 1}  (A^{-2}+ \frac{1+bA^{7/2}}{\varrho_h^2/4 + u^2} ) \,du \\
& &\leq  2 C_s (-2\ln b + (2bA^{7/2}+A^{-2}) (\frac{1}{b^2}-1) ) +C_s  2(A^{-2} + \frac{1}{\varrho_h}(1+bA^{7/2})\frac{\pi}{2} )
\leq\frac{1}{b^2}\end{eqnarray*}
by using $C_sbA^4\leq  \frac{\varrho_h}{2}\leq\frac{1}{2}$\,, $C_sbA^{7/2}=\frac{C_sbA^4}{A^{1/2}}\leq \frac{1}{10}$ for $A\geq 1$ large enough\,, $b^2|\ln(b)|\leq \frac{1}{10 C_s}$ and $\frac{C_s}{\varrho_h}\leq \frac{1}{bA^4}\leq \frac{1}{b}\leq\frac{1}{10 b^2}$ for $b>0$ small enough.\\
By conclude by adding the two upper bound for $\int_{\Gamma_+}+\int_{\Gamma_{-}}$ and $\int_{[MR]}$\,.
\end{proof}
\subsection{Hodge type structure and accurate spectral estimates}
\label{sec:HodgeAccSpec}
We prove now the accurate comparison of  $\mathrm{Spec}(B_{\pm,b,V^h})\cap\{z\in\mathbb{C}, |z|\leq \varrho_h\}\subset [0,\varrho_h[$ and $\mathrm{Spec}(\frac{1}{2}\Delta_{V^h,1})\cap[0,\varrho_h]=\mathrm{Spec}(\frac{1}{2}\Delta_{V,h})\cap [0,e^{-\frac{c}{h}}]$\,.
\begin{proposition}
\label{pr:accSpec}~\\
  Let $(\lambda_j^{(p)})_{\substack{1\leq j\leq \mathcal{N}_{\pm}^{(p)}\\ 0\leq p\leq 2d}}$ be the eigenvalues of $B_{\pm,b,V^h}$ contained in $[0,\varrho_h[$ and let
  $(\tilde{\lambda}_j^{(p)}(V))_{\substack{1\leq j\leq \mathcal{N}_{\pm}^{(p)}(V)\\ 0\leq p\leq d}}$ be the eigenvalues of $\frac{1}{2}\Delta_{V^h,1}$ contained in $[0,e^{-\frac{c}{h}}]$\,.
  There exists $C_0\geq 1$ such that  for all $A\geq C_0$ and with the additional condition $1\geq \varrho_h\geq C_0 A^4 b\geq C_0^{5}b$\,, the eigenvalues are compared according to 
  $$
\forall p\in \{0,\ldots,2d\}, \forall j\in \{1,\ldots,\mathcal{N}_{\pm}^{(p)}\}\,,\quad
(1+C_0 A^{-1/2})^{-1}\tilde{\lambda}_{j}^{(p-\frac{d}{2}\pm \frac{d}{2})}(V)\leq  \lambda_j^{(p)}\leq (1+C_0 A^{-1/2})\tilde{\lambda}_{j}^{(p-\frac{d}{2}\pm \frac{d}{2})}(V)\,,
$$
where we recall that $\mathcal{N}_{\pm}^{(p)}=\mathcal{N}_{\pm}^{(p-\frac{d}{2}\pm\frac{d}{2})}$ vanishes when $p-\frac{d}{2}\pm \frac{d}{2}\not\in \{0,\ldots,d\}$\,.
\end{proposition}
This result relies on the Hodge type structure of $2B_{\pm,b,V^{h}}\big|_{E_{{\pm,b,h}}}$ and $\Delta_{V^{h},1}\big|_{\mathrm{Ran}1_{[0,\varrho_{h}]}(\Delta_{V^{h},1})}$ and the identification of eigenvalues of those operators with the squares of singular values by following the strategy of \cite{HKN}\cite{Lep}\cite{LNV2} and other related works. We start with three lemmas.
\begin{lemma}
\label{le:hodge}
  Let $E$ be a finite dimensional Hilbert space with the hermitian positive definite form $\left(~,~\right)$\,. Let $\mathbf{d}$ be an operator such that $\mathbf{d}\circ\mathbf{d}=0$ and set $\Delta=(\mathbf{d}+\mathbf{d}^{*})^{2}=\mathbf{d}\mathbf{d}^{*}+\mathbf{d}^{*}\mathbf{d}$\,, where $\mathbf{d}^{*}$ is the adjoint of $\mathbf{d}$ for the scalar product $\left(~,~\right)$\,. The $E$ admits the orthogonal decomposition
$$
E=\ker(\Delta)\mathop{\oplus}^{\perp}\mathrm{Ran}\,\mathbf{d}\mathop{\oplus}^{\perp}\mathrm{Ran}\,\mathbf{d}^{*}=\ker(\mathbf{d})\mathop{\oplus}^{\perp}\mathrm{Ran}\,\mathbf{d}^{*}=\ker(\mathbf{d}^{*})\mathop{\oplus}^{\perp}\mathrm{Ran}\,\mathbf{d}\,.
$$
The eigenvalues of $\Delta$ are the squares of the singular values of $\mathbf{d}$ and equivalently $\mathbf{d}^{*}$\,. More precisely there exists an orthonormal basis $(u_{j})_{1\leq j\leq \dim E}$ such that $(u_{j})_{1\leq j\leq N}$ (resp. $(u_{j})_{N+1\leq j\leq 2N}$) is an orthonormal basis of $\mathrm{Ran}\,\mathbf{d}^{*}$ (resp. $\mathrm{Ran}\,\mathbf{d}$) and $(u_{j})_{2N+1\leq j\leq\dim E}$ is an orthonormal basis of $\ker \Delta$ and
$$
\forall j\in \left\{1,\ldots,N\right\}\,,\quad \mathbf{d}u_{j}=\mu_{j}u_{j+N}\quad \mu_{j}>0\,.
$$
We will set $\mu_{j}=0$ for $j\not\in \left\{1,\ldots,N\right\}$ and write, with an abuse of notation $\mathbf{d}u_{j}=\mu_{j}u_{j+N}$ for all $j\in \left\{1,\ldots,\dim E\right\}$\,.
\end{lemma}
\begin{proof}
  It suffices to notice $\Delta=0\oplus^{\perp}\mathbf{d}\mathbf{d}^{*}\oplus^{\perp}\mathbf{d}^{*}\mathbf{d}$ in the decomposition $E=\ker(\Delta)\mathop{\oplus}^{\perp}\mathrm{Ran}\,\mathbf{d}\mathop{\oplus}^{\perp}\mathrm{Ran}\,\mathbf{d}^{*}$\,. Then one takes for $(u_{j})_{1\leq j\leq N}$ an orthonormal eigenbasis of $\mathbf{d}^{*}\mathbf{d}\big|_{\mathrm{Ran}\,\mathbf{d}}$ with $\mathbf{d}^{*}\mathbf{d}u_{j}=\mu_{j}^{2}u_{j}$ and to set $u_{j+N}=\frac{1}{\mu_{j}}\mathbf{d}u_{j}$\,.
\end{proof}
\begin{definition}
  In a finite dimensional Hilbert space $(E, (~,~))$\,, a basis $\mathcal{B}=(v_{j})_{1\leq j\leq \dim E}$ is $\varepsilon$-orthonormal for $\varepsilon\in ]0,1[$ if $\|((v_{j},v_{k}))_{1\leq j,k\leq \dim E}-\mathrm{Id}_{\mathbb{C}^{\dim E}}\|\leq \varepsilon$\,.\\
  The function $\tau:\sqcup_{n=1}^{\infty}]0,1[^{n}\to ]0,+\infty[$ is defined by
  $\tau(\varepsilon_{1},\ldots,\varepsilon_{n})=\prod_{k=1}^{n}\frac{1+\varepsilon_{k}}{1-\varepsilon_{k}}$\,.
\end{definition}
The following lemma is extracted from Proposition~5.4 in \cite{LNV2}.
\begin{lemma}
\label{le:singval}  Let $\mathcal{B}=(u_{j})_{1\leq j\leq \dim E}$ (resp. $\mathcal{B}'=(v_{j})_{1\leq j\leq \dim E}$) be an $\varepsilon_{1}$- (resp. $\varepsilon_{2}-$) orthonormal basis of the Hilbert space $(E,(~,~))$ for $\varepsilon_{1},\varepsilon_{2}\in ]0,1[$\,. For $B\in \mathcal{L}(E)$\,, let  $(\mu_{j}(B))_{1\leq j\leq \dim E}$ (resp. $(\mu_{j}(\tilde{B}))_{1\leq j\leq \dim E}$) denote the singular values of $B$ (resp. of the matrix $\tilde{B}=((v_{k}\,,\, B u_{j}))_{1\leq j,k\leq \dim E}$)\,, in the usual decreasing order. Then
  $$
\forall j\in \left\{1,\ldots, \dim E\right\}\,,\quad \tau(\varepsilon_{1},\varepsilon_{2})^{-1/2}\tilde{\mu}_{j}\leq \mu_{j}\leq \tau(\varepsilon_{1},\varepsilon_{2})^{1/2}\tilde{\mu}_{j}\,.
$$
\end{lemma}
\begin{proof}[Proof of Proposition~\ref{pr:accSpec}]
  We start the proof for $B_{+,b,V^{h}}$ and the case of $B_{-,b,V^{h}}$ will be recovered in the end by a Poincar{\'e} duality argument.\\
  We do not distinguish the form degree here and recall that the number of eigenvalues of $\frac{1}{2}\Delta_{V^{h},1}$ in $[0,e^{-\frac{c}{h}}]$ and of $B_{+,b,V^{h}}$ in $[0,\varrho_{h}]$ are equal to
  $$
\mathcal{N}_{+}=\sum_{p=0}^{2d}\mathcal{N}_{+}^{(p)}=\sum_{p=0}^{d}\mathcal{N}_{+}(V)=\mathcal{N}_{+}(V)\,.
  $$
Let us set $\tilde{\pi}_{h}=1_{[0,e^{-\frac{c}{h}}]}(\frac{1}{2}\Delta_{V^{h,1}})=1_{[0,2e^{-\frac{c}{h}}]}(\Delta_{V^{h},1})$ and $\pi_{h}=U_{+,\theta}\tilde{\pi}_{h}U_{+,\theta}^{*}$ \,. Because $\Delta_{V^{h},1}=(d_{V^{h},1}+d^{*}_{V^{h},1})^{2}$\,, Lemma~\ref{le:hodge} tells us that there is an orthonormal basis $(\tilde{u}_{j})_{1\leq j\leq \mathcal{N}_{+}}$ such that $\tilde{u}_{j}$ belongs to $\mathrm{Ran}~(d_{V^{h},1}^{*}\big|_{\mathrm{Ran}\,\tilde{\pi}_{h}})$ (resp.  $\mathrm{Ran}~(d_{V^{h},1}\big|_{\mathrm{Ran}\,\tilde{\pi}_{h}})$) for $1\leq j\leq N$ (resp. $N+1\leq j\leq 2N$) and
$$
d_{V^{h},1}\tilde{u}_{j}=
\left\{
    \begin{array}[c]{ll}
      \tilde{\mu}_{j} \tilde{u}_{j+N}&\text{if}~1\leq j\leq N\\
      0&\text{otherwise}\,,
    \end{array}
  \right.$$
  where $\tilde{\mu}_{j}$\,, $1\leq j \leq N$\,, are the non zero singular values of $d_{V^{h},1}\big|_{\mathrm{Ran}\,\tilde{\pi}_{h}}$\,. With the abuse of notation of  Lemma~\ref{le:singval}, it is summarized by  $d_{V^{h},1}\tilde{u}_{j}=\tilde{\mu}_{j}\tilde{u}_{j+N}$ for all $j\in \left\{1,\ldots,\mathcal{N}_{+}\right\}$ with $\tilde{\mu}_{j}=0$ for $j>N$\,.\\
  Let $\pi_{E_{+,b,V^h}}$ be the spectral projector associated with $B_{+,b,V^{h}}$ given by
  $$
\pi_{E_{+,b,V^h}}=\frac{1}{2i\pi}\int_{NO'P'Q}(z-B_{+,b,V^{h}})^{-1}~dz
$$
like in Proposition~\ref{pr:spectrum}.\\
We consider $\mathcal{B}'=(v_{j})_{1\leq j\leq \mathcal{N}_{+}}$ with
\begin{eqnarray*}
  &&
     v_{j}=\pi_{E_{+,b,V^h}}[U_{+,\theta}\tilde{u}_{j}]\,,\\
  \text{and}&& U_{+,\theta}\tilde{u}_{j}= \frac{e^{-\mathcal{H}}}{\pi^{d/4}}[\tilde{u}_{j}](q)\quad,\quad \mathcal{H}=\frac{|p|_{q}^{2}}{2}\,.
\end{eqnarray*}
Let us compute the scalar products $\langle v_{j}\,,\, v_{j'}\rangle_{r}$:
\begin{align*}
  \langle v_{j}\,,\, v_{j'}\rangle_{r}
  &= \langle \pi_{E_{+,b,V^h}}(U_{+,\theta}\tilde{u}_{j})\,,\, \pi_{E_{+,b,V^h}}(U_{+,\theta}\tilde{u}_{j'})\rangle_{r}
                                        \stackrel{\pi_{E_{+,b,V^h}}^{*,r}=\pi_{E_{+,b,V^h}}}{=}\langle U_{+,\theta}\tilde{u}_{j}\,,\, \pi_{E_{+,b,V^h}}(U_{+,\theta}\tilde{u}_{j'})\rangle_{r}
  \\
  &=\langle U_{+,\theta}\tilde{u}_{j}\,,\, r^{*} \pi_{E_{+,b,V^h}}(U_{+,\theta}\tilde{u}_{j'})\rangle\stackrel{r^{*}U_{+,\theta}=U_{+,\theta}}{=}
    \langle U_{+,\theta}\tilde{u}_{j}\,,\,\pi_{E_{+,b,V^h}}(U_{+,\theta}\tilde{u}_{j'})\rangle
\\  &= \delta_{j,j'}
   + \langle U_{+,\theta}\tilde{u}_{j}\,,\,[\pi_{E_{+,b,V^h}}-\pi_{h}](U_{+,\theta}\tilde{u}_{j'})\rangle\,.
\end{align*}
But from \eqref{eq:estimpi1} and \eqref{eq:estimpi2} we deduce
$$
\|\pi_{E_{+,b,V^h}}-\pi_{h}\|\leq 10 C_{0}A^{-2}+(4+2\pi)A^{-1/2}\leq 30 A^{-1/2}\,.
$$
We deduce that $\mathcal{B}'$ is an $\varepsilon$-orthonormal basis of $E_{+,b,V^h}$ with $\varepsilon=30\mathcal{N}_{+}A^{-1/2}\in ]0,1[$ for $A\geq C_{0}$ large enough.\\
Remember $\mu_{0}=\hat{f}_{i}\wedge \mathbf{i}_{f^{i}}$ introduced in Definition~\ref{de:Bismutnot}. We now consider the basis $\mathcal{B}=(u_{j})_{1\leq j\leq \mathcal{N}_{+}}$ with
$$
u_{j}=\pi_{E_{+,b,V^h}}[e^{-\mu_{0}}U_{+,\theta}\tilde{u}_{j}]=v_{j}+\pi_{E_{+,b,V^h}}[(e^{-\mu_{0}}-1)U_{+,\theta}\tilde{u}_{j}]\,.
$$
From $(e^{-\mu_{0}}-1)=\sum_{k=1}^{d}\frac{(-1)^{k}}{k!}(\hat{f}_{i}\wedge \mathbf{i}_{f^{i}})^{k}$ and $\mathrm{Ran}\,U_{+,\theta}=\mathrm{Ran}\,\pi_{0,+}$ we deduce
  $$
(e^{-\mu_{0}}-1)U_{+,\theta}=(1-\pi_{0,+})(e^{-\mu_{0}}-1)U_{+,\theta}\,.
$$
We write now
\begin{align*}
\|\pi_{E_{+,b,V^h}}(1-\pi_{0,+})\|&=\|(1-\pi_{0,+})^{*}\pi_{E_{+,b,V^h}}^{*}\|
=\|(1-\pi_{0,+})r^{*}\pi_{E_{+,b,V^h}}r^{*}\|\\
&=\|r^{*}(1-\pi_{0,+})\pi_{E_{+,b,V^h}}r^{*}\|
=\|(1-\pi_{0,+})\pi_{E_{+,b,V^h}}\|
\end{align*}
but we proved in Proposition~\ref{pr:spectrum} after Lemma~\ref{lem:PTsymmetry} the upper bound
$$
\|(1-\pi_{0,+})\pi_{E_{+,b,V^h}}\|=\vec{d}(E_{+,b,V^h},\mathrm{Ran}\,\pi_{0,+})\leq \sqrt{2 C_{0}'\mathcal{N}_{+}b^{2}}\,.
$$
Remember also the equivalence of norms $\frac{1}{2}\|u\|\leq \|u\|_{r}\leq 2\|u\|$ for $u\in E_{\pm,b,V^h}$\,.
Therefore there exists a constant $C_{0,\mathcal{N}_{+}}\geq 1$ such that $\mathcal{B}$ (and $\mathcal{B}'$) are $\varepsilon$-orthornormal bases of $(E_{+,b,V^h},\langle~,~\rangle_{r})$  with $\varepsilon=C_{0,\mathcal{N}_{+}}A^{-1/2}\in ]0,1[$ for $A\geq 2C_{0,\mathcal{N}_{+}}^{2}$\,.\\
We now apply Lemma~\ref{le:hodge} to
$$
2B_{+,b,V^{h}}\big|_{E_{+,b,V^h}}=(\delta_{+,b,V^{h}}\big|_{E_{+,b,V^h}})(\delta_{+,b,V^{h}}^{*,r}\big|_{E_{+,b,V^h}})+(\delta_{+,b,V^{h}}^{*,r}\big|_{E_{+,b,V^h}})(\delta_{+,b,V^{h}}\big|_{E_{+,b,V^h}})\,,
$$
where we recall \eqref{eq:deltabV}
$$
\delta_{+,b,V^{h}}=e^{-\mu_0}e^{-\mathcal{H}-V^{h}}(K_{b}d^{X} K_{b}^{-1})e^{+\mathcal{H}+V^{h}}e^{\mu_0}\,.
$$
Because $\delta_{+,b,V^{h}}B_{+,b,V^{h}}=B_{+,b,V^{h}}\delta_{+,b,V^{h}}$ on $\mathcal{S}(X^{h};\mathcal{E}_{+}^{h})$ and $E_{\pm,b,V^h}\subset \mathcal{S}(X^{h};\mathcal{E}_{+}^{h})$ we know actually
\begin{align*}
&\delta_{+,b,V^{h}}\pi_{E_{+,b,V^h}}=\pi_{E_{+,b,V^h}}\delta_{+,b,V^{h}}=\pi_{E_{+,b,V^h}}\delta_{+,b,V^{h}}\pi_{E_{+,b,V^h}}
  \\
  \text{and} \quad  &
\delta_{+,b,V^{h}}^{*,r}\pi_{E_{+,b,V^h}}=\pi_{E_{+,b,V^h}}\delta_{+,b,V^{h}}^{*,r}=\pi_{E_{+,b,V^h}}\delta_{+,b,V^{h}}^{*,r}\pi_{E_{+,b,V^h}}\,.
\end{align*}
We now compute $\langle v_{j'}\,,\, \delta_{+,b,V^{h}}u_{j}\rangle_{r}$:
\begin{eqnarray*}
  \langle v_{j'}\,,\, \delta_{+,b,V^{h}}u_{j}\rangle_{r}
    &=&\langle \pi_{E_{+,b,V^h}}(U_{+,\theta}\tilde{u}_{j'})\,, \delta_{+,b,V^{h}}\pi_{E_{+,b,V^h}}(e^{-\mu_{0}}U_{+,\theta}\tilde{u}_{j})\rangle_{r}
    \\
    &=&\langle \pi_{E_{+,b,V^h}}(U_{+,\theta}\tilde{u}_{j'})\,, \delta_{+,b,V^{h}}(e^{-\mu_{0}}U_{+,\theta}\tilde{u}_{j})\rangle_{r}
  \\
  &=&\langle \pi_{E_{+,b,V^h}}(U_{+,\theta}\tilde{u}_{j'})\,, e^{-\mu_0}e^{-\mathcal{H}-V^{h}}(K_{b}d^{X} K_{b}^{-1})e^{+\mathcal{H}+V^{h}} U_{+,\theta}\tilde{u}_{j}\rangle_{r}
  \\
  &=& \langle \pi_{E_{+,b,V^h}}(U_{+,\theta}\tilde{u}_{j'})\,, e^{-\mu_0}U_{+,\theta}[d_{V^{h},1}\tilde{u}_{j}]\rangle_{r}\\
  &=&\tilde{\mu}_{j}\langle \pi_{E_{+,b,V^h}}(U_{+,\theta}\tilde{u}_{j'})\,,\, e^{-\mu_{0}}U_{+,\theta}[\tilde{u}_{j+N}]\rangle_{r}\\
    &=&
        \tilde{\mu}_{j}\langle U_{+,\theta}\tilde{u}_{j'}\,,\, r^* e^{-\mu_{0}}U_{+,\theta}[\tilde{u}_{j+N}]\rangle +\tilde{\mu}_{j}R_{j,j',h}\\
    &=&
        \tilde{\mu}_{j}\langle U_{+,\theta}\tilde{u}_{j'}\,,\, r^{*}e^{-\mu_{0}}U_{+,\theta}[\tilde{u}_{j+N}]\rangle+\tilde{\mu}_{j}R_{j,j',h}\\
  &=& \tilde{\mu}_{j}\langle e^{-\lambda_{0}}r^{*}U_{+,\theta}\tilde{u}_{j'}\,,\, U_{+,\theta}[\tilde{u}_{j+N}]\rangle+\tilde{\mu}_{j}R_{j,j',h}\\
  &=&\tilde{\mu}_{j}\delta_{j',j+N}+\tilde{\mu}_{j}R_{j,j',h}
\end{eqnarray*}
where we used $e^{-\lambda_{0}}r^{*}U_{+,\theta}=e^{-\lambda_{0}}r^{*}\pi_{0,+}U_{+,\theta}=U_{+,\theta}$ in the last identity, where $\tilde{\mu}_{j}=0$ for $j>N$ and where
$$
  |R_{j,j',h}|=\left|\langle [\pi_{E_{+,b,V^h}}-\pi_{h}](U_{+,\theta}\tilde{u}_{j'})\,,\, e^{-\mu_{0}}U_{+,\theta}\tilde{u}_{j+N}\rangle_{r}\right|
\leq 2 \|[\pi_{E_{+,b,V^h}}-\pi_{h}](U_{+,\theta}\tilde{u}_{j'})\|\leq 60 A^{-1/2}\,.
$$
By Gaussian elimination like in \cite{Lep} or equivalently by changing the basis $\mathcal{B}'$ into an other $\varepsilon$-orthonormal basis of $(E_{+,b,V^h},\langle~,~\rangle_{r})$ with $\varepsilon=C_{0,\mathcal{N}_{+}}A^{-1/2}$\,, we deduce with a possibly enlarged constant $C_{0,\mathcal{N}_{+}}$ that the singular values, $\mu_{j}$ of $\delta_{+,b,V^{h}}\big|_{E_{+,b,V^h}}$ satisfy:
$$
\forall j\in \left\{1,\ldots,\mathcal{N}_{+}\right\}\,,\quad (1+C_{0,\mathcal{N}_{+}}A^{-1/2})^{-1}\tilde{\mu}_{j}\leq \mu_{j}\leq (1+C_{0,\mathcal{N}_{+}}A^{-1/2})\tilde{\mu}_{j} 
$$
We deduce the comparison between eigenvalues $(\lambda_{j})_{1\leq j\leq \mathcal{N}_{+}}$ of $B_{\pm,b,h}$ and $(\tilde{\lambda}_{j})_{1\leq j\leq \mathcal{N}_{+}}$ or $\frac{1}{2}\Delta_{V^h,1}$ is deduced from $\lambda_j=\frac{1}{2}\mu_j^2$ and $\tilde{\lambda}_j=\frac{1}{2}\tilde{\mu}_j^2$\,, by doubling the constant $C_0\geq 1$\,.\\
Finally, still in the case of $B_{+,b,V^h}$\,, the degree can be followed by splitting the orthonormal basis $(\tilde{u}_{j})_{1\leq j\leq \mathcal{N}_{+}}$ according to the degree $p\in \{0,\ldots,d\}$ with $\{\tilde{u}_{j}, 1\leq j\leq \mathcal{N}_+\}=\cup_{p=0}^{d}\{\tilde{u}_{j_k}^{(p)}\,,\,, 1\leq k\leq \mathcal{N}_+^{(p)}(V)\}$\,. In particular we notice that $\tilde{u}_j=\tilde{u}_{j_k}^{(p)}$ implies that the total degree of $u_{j}$ and $v_j$ equals $p$ (in this $+$-case).\\

Let us give some details for the Poincar{\'e} duality argument which gives the result in the $-$ case. Attention must be paid on the choice of the Thom form, the unitary map $U_{-,\theta}$\,, when $Q$ is not orientable and $F=Q\times \mathbb{C} \neq (Q\times \mathbb{C})\otimes \mathbf{or}_Q$\,.
Let us split the basis $\tilde{u}_{j,V^h}$ constructed in the $+$ case according to the degree $p\in\left\{0,\ldots,d\right\}$\,, for the potential  and let $\star_Q$ (resp. $\star_X$) denote the Hodge star operator on $Q$ (resp $X$ which is orientable).
We construct the basis $\mathcal{B}$ (resp. $\mathcal{B}'$) with the vectors
\begin{eqnarray*}
&& u_{j}^{(d+p)}=\pi_{E_{-,b,V^h}^{(d+p)}}e^{-\lambda_0} U_{-,\theta}[\star_Q~\tilde{u}_{j,-V^h}^{(d-p)}]
= \pi_{E_{-,b,V^h}^{(d+p)}} \star_X~e^{-\mu_0}U_{+,\theta}[\tilde{u}_{j,-V^h}^{(d-p)}]\,,
\\
\text{resp.}&&  v_{j'}^{(d+p)}=\pi_{E_{-,b,V^h}^{(d+p)}}  U_{-,\theta}[\star_Q~\tilde{u}_{j,-V^h}^{(d-p)}]
= \pi_{E_{-,b,V^h}^{(d+p)}} \star_X~U_{+,\theta}[\tilde{u}_{j,-V^h}^{(d-p)}]\,.
\end{eqnarray*}
The $\frac{C_0}{A^1/2}$-orthonormality of $\mathcal{B}$ and $\mathcal{B}'$ is easily deduced from the $+$ case.
Then we must compute 
\begin{eqnarray*}
\delta_{-,b,V^h}^{*,r} u_j^{(d+p)}&=& e^{-\lambda_0}e^{-\mathcal{H}+V^h}K_b d^{*,X} K_b^{-1} e^{\mathcal{H}-V^h} e^{\lambda_0}u_j^{(d+p)}
\\&=& \star_X~e^{-\mu_0}e^{-\mathcal{H}+V^h}K_b \star_X^{-1}d^{*,X} \star_X K_b^{-1} e^{-\mathcal{H}+V^h} U_{+,\theta}[\tilde{u}_{j,-V^h}^{(d-p)}]
\\
&=& (-1)^{(d-p+1)}\star_X~e^{-\mu_0}e^{-\mathcal{H}+V^h}K_b d^{X} K_b^{-1} e^{\mathcal{H}-V^h} U_{+,\theta}[\tilde{u}_{j,-V^h}^{(d-p)}]\\
&=& (-1)^{(d-p+1)}\star_X e^{-\mu_0} U_{+,\theta}[d_{-V^h,1}~\tilde{u}_{j,-V^h}^{(d-p)}]\\
&=& (-1)^{(d-p+1)}\tilde{\mu}_{j,-V^h}^{(d-p)}\star_X e^{-\mu_0} U_{+,\theta}[\tilde{u}_{k(j),-V^h}^{(d-p+1)}]\,.
\end{eqnarray*}
By taking the $\langle~,~\rangle_r$ scalar product with $v_{j'}^{(d+p')}$ like in the $+$ case, we obtain
$$
\langle v_{j'}^{(d+p')}\,,\, \delta_{-,b,V^h}^{*,r} u_{j}^{(d+p)}\rangle_r = (-1)^{d-p+1}\delta_{p',p-1}\delta_{j',k(j)}\tilde{\mu}_{j,-V^h}^{(d-p)} +\tilde{\mu}_{j,-V^h}^{(d-p)}\mathcal{O}(A^{-1/2})\,.
$$
We deduce that the singular values of $\delta_{-,b,V^h}^{*,r}\big|_{E_{-,b,V^h}^{(d+p)}}$ are comparable, with an $(1+\mathcal{O}(A^{-1/2}))$ factor, with the singular values of $d_{-V^h,1}: \mathrm{Ran}\,1_{[0,2\varrho_h]}(\Delta_{F^+,-V^h,1}^{(d-p)})\to \mathrm{Ran}\,1_{[0,2\varrho_h]}(\Delta_{F^+,-V^h,1}^{(d-p+1)})$\,, where we recall that $\Delta_{F_+,-V^h,1}$ acts on $\mathcal{C}^{\infty}(Q,\Lambda T^*Q\otimes \mathbb{C})$\,. The eigenvalues of $B_{-,b,V^h}^{(d+p)}$ are thus comparable with an
$(1+\mathcal{O}(A^{-1/2}))$ factor, with the eigenvalues of $\frac{1}{2}\Delta_{F_+,-V^h,1}^{(d-p)}$ which by Poincar{\'e} duality for the Witten Laplacian are the eigenvalues of $\frac{1}{2}\Delta_{F_-,+V^h,1}^{(p)}$, where we recall that $\Delta_{F_-,V^h,1}$ acts on $\mathcal{C}^{\infty}(Q,\Lambda T^*Q\otimes \mathbb{C}\otimes \mathbf{or}_Q)$\,.
\end{proof}

\textbf{Acknowledgements:} The two first authors benefited during this work from the support of the French ANR-Project QUAMPROCS "Quantitative Analysis of Metastable Processes". Francis Nier acknowledge the hospitality  of the INRIA-project Matherials and the Cermics in Ecole Nationale des Ponts-et-Chauss\'ees for a sabbatical semester, during which this article was finished.
Francis White is grateful to have received funding from the European Union's Horizon 2020 research and innovation programme under the Marie Sklodowska-Curie grant agreement No 101034255.\euflag

\end{document}